\documentclass{article}
\usepackage{arxiv}
\usepackage[nottoc]{tocbibind}
\newcommand{\Ra}{\Rightarrow}
\newcommand{\p}{\mathfrak{p}}
\newcommand{\m}{\mathfrak{m}}
\newcommand{\R}{\mathbb{R}}
\newcommand{\N}{\mathbb{N}}
\usepackage{latexsym, amssymb, amsfonts,amsmath}
\usepackage{amsthm}
\usepackage{graphicx}
\usepackage[all,knot,arc,import,poly]{xy}
\usepackage[utf8]{inputenc} 
\usepackage[T1]{fontenc}    
\usepackage{hyperref}       
\usepackage{url}            
\usepackage{booktabs}       
\usepackage{amsfonts}       
\usepackage{nicefrac}       
\usepackage{microtype}      
\usepackage{lipsum}
\usepackage{url}
\usepackage{tikz}
\usetikzlibrary{positioning, calc, arrows, decorations.markings, decorations.pathreplacing}
\usepackage{latexsym, amssymb, amsfonts,amsmath}
\usepackage{amsthm}
\newtheorem{theorem}{Theorem}[section]
\newtheorem{lemma}[theorem]{Lemma}
\newtheorem{fact}[theorem]{Fact}
\newtheorem{definition}[theorem]{Definition}
\newtheorem{proposition}[theorem]{Proposition}
\newtheorem{example}[theorem]{Example}
\newtheorem{remark}[theorem]{Remark}
\newtheorem{corollary}[theorem]{Corollary}
\newcommand{\ci}{\mathcal{C}^\infty}
\title{Topics on Smooth Commutative Algebra}
\author{ Jean Cequeira Berni\thanks{\texttt{www.ime.usp.br/$\sim$jeancb/}} \\
  Department of Mathematics\\
  Institute of Mathematics and Statistics\\
  University of S\~{a}o Paulo\\
  \texttt{jeancb@ime.usp.br} \\
   \And
 Hugo Luiz Mariano\thanks{\texttt{www.ime.usp.br/$\sim$hugomar/}} \\
  Department of Mathematics\\
  Institute of Mathematics and Statistics\\
  University of S\~{a}o Paulo\\
  \texttt{hugomar@ime.usp.br} \\
}

\begin{document}
\maketitle

\begin{abstract}
We present, in the same vein as in \cite{moerdijk1986rings} and \cite{rings2}, some results of the so-called ``Smooth (or $\mathcal{C}^{\infty}$) Commutative Algebra'', a version of Commutative Algebra of $\mathcal{C}^{\infty}-$rings instead of ordinary commutative unital rings, looking for similar results to those one finds in the latter, and expanding some others presented in \cite{moerdijk1986rings}. In order to study this ``bridge'', we give an explicit description of an adjunction between the categories $\mathcal{C}^{\infty}{\rm \bf Rng}$ and ${\rm \bf CRing}$. We present and prove many properties of the analog of the radical of an ideal of a ring (namely, the $\mathcal{C}^{\infty}-$radical of an ideal), saturation (which we define as ``smooth saturation'', inspired by \cite{dieudonne11ega}), rings of fractions ($\mathcal{C}^{\infty}-$rings of fractions, defined first by I. Moerdijk and G. Reyes  in \cite{moerdijk1986rings}), local rings (local $\mathcal{C}^{\infty}-$rings), reduced rings ($\mathcal{C}^{\infty}-$reduced $\mathcal{C}^{\infty}-$rings) and others. We also state and prove new results, such as   \textit{ad hoc} ``\textbf{Separation Theorems}'' - similar to the ones we find in Commutative Algebra, and a stronger version (\textbf{Theorem \ref{38}}) of the \textbf{Theorem 1.4} of \cite{moerdijk1986rings}, characterizing every $\mathcal{C}^{\infty}-$ring of fractions.We describe the fundamental concepts of Order Theory for $\mathcal{C}^{\infty}-$rings, proving that every $\mathcal{C}^{\infty}-$ring is semi-real, and we  prove an important result on the strong interplay between the smooth Zariski spectrum  and the real smooth spectrum of a $\mathcal{C}^{\infty}-$ring.
\end{abstract}


\section*{Introduction}

\hspace{0.5cm}In a recent paper, \cite{joyce2010algebraic}, D. Joyce presents the foundations of a version of Algebraic Geometry in which rings or algebras are replaced by $\mathcal{C}^{\infty}-$rings. In this work we explore the ``Commutative Algebraic'' aspects of these rings, emphasizing their similarities and their differences.\\

This work may be regarded as a continuation of the paper \cite{cerqueira2019universal}, in the sense that we keep exploring the algebraic aspects of $\mathcal{C}^{\infty}-$rings, except that here we are concerned with their ``Commutative Algebraic'' aspects. We study some concepts first presented  by I. Moerdijk, N. van Qu\^{e} and G. Reyes in \cite{moerdijk1986rings} and \cite{rings2} (such as the $\mathcal{C}^{\infty}-$radical of an ideal, the $\mathcal{C}^{\infty}-$ring of fractions among others) and define others, expanding some of their results and proving various properties of these objects, comparing them to their Commutative Algebraic parallels. We also introduce some other similar concepts, as the ``smooth saturation'' of a $\mathcal{C}^{\infty}-$ring, which allows us to state a theorem that is similar to the Separation Theorems one finds in Commutative Algebra.\\

We point towards many similarities between Smooth Commutative Algebra and Ordinary Commutative Algebra (for example, in both theories we have a commutativity between ``taking the ring of fractions'' and ``taking quotients''), stressing also some of their differences (for example, the existence of a continuous bijection between the $\mathcal{C}^{\infty}-$version of the prime spectrum and the $\mathcal{C}^{\infty}-$version of the real spectrum - which in general is not the case in Commutative Algebra). \\

In \cite{borisov2018beyond}, Kremnizer and Borisov give a detailed account of six notions of radicals of an ideal of a $\mathcal{C}^{\infty}-$ring, among which we find the $\infty-$radical of an ideal of a $\mathcal{C}^{\infty}-$ring. Here we focus on this concept, that we call `` the $\mathcal{C}^{\infty}-$radical'' of an ideal. This concept first appeared in \cite{moerdijk1986rings}, and it is proper to the theory of $\mathcal{C}^{\infty}-$rings, carrying some differences with respect to the usual notion of radical in ordinary Commutative Algebra - which only makes use of powers of elements. In order to illustrate this difference, we present an example of an ideal of a $\mathcal{C}^{\infty}-$ring which is radical in the ordinary sense (since it is prime) but that is not $\mathcal{C}^{\infty}-$radical (see \textbf{Remark \ref{jurupari}}). \\

The difference between the notions of ``radical'' and ``$\mathcal{C}^{\infty}-$radical'' ideals brings us, alone,  a whole new study of some important concepts, such as $\mathcal{C}^{\infty}-$re\-duced $\mathcal{C}^{\infty}-$rings, the ``smooth Zariski spectrum'' and so on. We are going to see, for example, that the smooth version of the Zariski spectrum presents some crucial differences when compared to the ordinary Zariski spectrum, in its topological as well as in its functorial (and sheaf-theoretic) features.\\

\textbf{Overview of the Paper:}  We begin the first section by giving some preliminaries on the main subject on which we build this version of Commutative Algebra, the category $\mathcal{C}^{\infty}-$rings, presenting some definitions and some of their fundamental constructions (the main definitions can be found in \cite{moerdijk2013models}, and a more detailed account of  $\mathcal{C}^{\infty}-$rings can be found, for example, in \cite{cerqueira2019universal}).\\

We dedicate the  second section to the explicit description of the adjunction between the category of $\mathcal{C}^{\infty}-$rings and the category of ordinary commutative rings (\textbf{Theorem \ref{Madruga}}), which allows us to transport many concepts from one context to the other.\\

In the third section we by recalling the ordinary Commutative Algebraic construction of the ring of fractions and the concept of saturation, motivating the introduction of their $\mathcal{C}^{\infty}$ parallels. In order to prove the existence of the $\mathcal{C}^{\infty}-$ring of fractions (\textbf{Theorem \ref{conc}}) we make use of the notion of the $\mathcal{C}^{\infty}-$ring of $\mathcal{C}^{\infty}-$polynomials, studied in \cite{cerqueira2019universal}, and then we define the concept of ``smooth saturation'' (\textbf{Definition \ref{satlisa}} of \textbf{Section \ref{saty}}), pointing some of its relationships with the ordinary concept of saturation (\textbf{Theorem \ref{lara}}). We state and prove various results about this concept.\\

We prove a result which characterizes any $\mathcal{C}^{\infty}-$ring of fractions (\textbf{Theorem \ref{38}}). We describe some properties of the process of ``taking the $\mathcal{C}^{\infty}-$ring of fractions'' defining a special category of pairs and a specific functor (\textbf{Theorem \ref{otalfuntor}} of \textbf{Section \ref{cop}}). We prove that the processes such as  ``taking the $\mathcal{C}^{\infty}-$coproduct'' and ``taking the $\mathcal{C}^{\infty}-$ring of fractions'', ``taking quotients'' and ``taking the $\mathcal{C}^{\infty}-$ring of fractions''  commute (\textbf{Theorem \ref{comfraccop}} and \textbf{Corollary \ref{Jeq}}, respectivelly), and we present a result which establishes a connection between the $\mathcal{C}^{\infty}-$ring of fractions and directed colimits. \\

In \textbf{Section \ref{dccr}}  we present distinguished classes of $\mathcal{C}^{\infty}-$rings (i.e., $\mathcal{C}^{\infty}-$rings which satisfy some further axioms), as $\mathcal{C}^{\infty}-$fields, $\mathcal{C}^{\infty}-$domains and local $\mathcal{C}^{\infty}-$\-rings. We prove some results connecting the filter of closed subsets of $\mathbb{R}^n$ and the set of $\mathcal{C}^{\infty}-$radical ideals of $\mathcal{C}^{\infty}(\mathbb{R}^n)$ (in fact, a Galois connection [\textbf{Proposition \ref{pseu}}]), as well as various results about them - including the fact that the $\mathcal{C}^{\infty}-$radical of any ideal is again an ideal (\textbf{Proposition \ref{49}}). The authors know of no such proof in the current literature. We make a study of $\mathcal{C}^{\infty}-$reduced $\mathcal{C}^{\infty}-$rings, proving many of their properties, and we state and prove a similar version of the Separation Theorems (\textbf{Theorem \ref{TS}}) one finds in ordinary Commutative Algebra.\\

In \textbf{Section \ref{sheaves}} we make a comprehensive analysis of the $\mathcal{C}^{\infty}$ version of the Zariski spectrum, describing its topology (the ``smooth Zariski topology'', in \textbf{Definition \ref{zahma}}) with detail, in order to clearly state and prove some sheaf-theoretic results. We prove that the Zariski $\mathcal{C}^{\infty}-$spectrum functor (${\rm Spec}^{\infty}: \mathcal{C}^{\infty}{\rm \bf Rng} \to {\rm \bf Top}$) takes finite products to finite coproducts (\textbf{Theorem \ref{esteaca}}). We also present, with details, the construction of a presheaf on the basis (of the $\mathcal{C}^{\infty}-$Zariski topology) of $\mathcal{C}^{\infty}-$rings whose fibers are local $\mathcal{C}^{\infty}-$rings.\\

In \textbf{Section \ref{ot}} we give a survey of some order-theoretic aspects of $\mathcal{C}^{\infty}-$rings, defining fundamental concepts as the ``real $\mathcal{C}^{\infty}-$spectrum'' of a $\mathcal{C}^{\infty}-$ring and its topology (the ``Harrison smooth topology''), proving that every $\mathcal{C}^{\infty}-$ring is semi-real (\textbf{Proposition \ref{doria}}) and we prove an important result which establishes a spectral bijection from the real $\mathcal{C}^{\infty}-$spectrum of a $\mathcal{C}^{\infty}-$ring to its smooth Zariski spectrum (\textbf{Theorem \ref{exv}}).\\

We finish this work with \textbf{Section \ref{versus}}, where we make a brief comparison between the ordinary Commutative Algebraic concepts with the ones presented in this paper.

\section{Preliminaries on $\ci-$Rings}

\hspace{0.5cm}Loosely speaking, $\mathcal{C}^{\infty}-$rings are $\mathbb{R}-$algebras where the composition of smooth functions is defined, satisfying all the equations that hold between these functions. Perhaps the most ``famous''  examples of $\mathcal{C}^{\infty}-$rings are rings of the form $\mathcal{C}^{\infty}(M)$,  where $M$ is a smooth manifold, rings of germs of differentiable functions (used in the Theory of Singularities), the ring of dual numbers, $\mathbb{R}[\varepsilon] \cong \dfrac{\mathbb{R}[X]}{\langle X^2\rangle}$ and Weil algebras, in general.\\

In order to formulate and study the concept of $\mathcal{C}^{\infty}-$ring, we are going to use a first order language $\mathcal{L}$ with a denumerable set of variables (${\rm \bf Var}(\mathcal{L}) = \{ x_1, x_2, \cdots, x_n, \cdots\}$) whose nonlogical symbols are the symbols of $\mathcal{C}^{\infty}-$functions from $\mathbb{R}^m$ to $\mathbb{R}^n$, with $m,n \in \mathbb{N}$, \textit{i.e.}, the non-logical symbols consist only of function symbols, described as follows:\\

For each $n \in \mathbb{N}$, the $n-$ary \textbf{function symbols} of the set $\mathcal{C}^{\infty}(\mathbb{R}^n, \mathbb{R})$, \textit{i.e.}, $\mathcal{F}_{(n)} = \{ f^{(n)} | f \in \mathcal{C}^{\infty}(\mathbb{R}^n, \mathbb{R})\}$. So the set of function symbols of our language is given by:
      $$\mathcal{F} = \bigcup_{n \in \mathbb{N}} \mathcal{F}_{(n)} = \bigcup_{n \in \mathbb{N}} \mathcal{C}^{\infty}(\mathbb{R}^n)$$
      Note that our set of constants is $\mathbb{R}$, since it can be identified with the set of all $0-$ary function symbols, \textit{i.e.}, ${\rm \bf Const}(\mathcal{L}) = \mathcal{F}_{(0)} = \mathcal{C}^{\infty}(\mathbb{R}^0) \cong \mathcal{C}^{\infty}(\{ *\}) \cong \mathbb{R}$.\\

The terms of this language are defined, in the usual way, as the smallest set which comprises the individual variables, constant symbols and $n-$ary function symbols followed by $n$ terms ($n \in \mathbb{N}$).\\

Apart from the functorial definition we gave in the introduction, we have many equivalent descriptions. We focus, first, in the following description of a $\mathcal{C}^{\infty}-$ring in ${\rm \bf Set}$.\\

\begin{definition}\label{cabala} A \textbf{$\mathcal{C}^{\infty}-$structure} on a set $A$ is a pair $ \mathfrak{A} =(A,\Phi)$, where:

$$\begin{array}{cccc}
\Phi: & \bigcup_{n \in \mathbb{N}} \mathcal{C}^{\infty}(\mathbb{R}^n, \mathbb{R})& \rightarrow & \bigcup_{n \in \mathbb{N}} {\rm Func}\,(A^n; A)\\
      & (f: \mathbb{R}^n \stackrel{\mathcal{C}^{\infty}}{\to} \mathbb{R}) & \mapsto & \Phi(f) := (f^{A}: A^n \to A)
\end{array},$$

that is, $\Phi$ interprets the \textbf{symbols}\footnote{here considered simply as syntactic symbols rather than functions.} of all smooth real functions of $n$ variables as $n-$ary function symbols on $A$.
\end{definition}

We call a $\mathcal{C}^{\infty}-$struture $\mathfrak{A} = (A, \Phi)$ a \textbf{$\mathcal{C}^{\infty}-$ring} if it preserves  projections and all equations between smooth functions. We have the following:

\begin{definition}\label{CravoeCanela}Let $\mathfrak{A}=(A,\Phi)$ be a $\mathcal{C}^{\infty}-$structure. We say that $\mathfrak{A}$ (or, when there is no danger of confusion, $A$) is a \textbf{$\mathcal{C}^{\infty}-$ring} if the following is true:\\

$\bullet$ Given any $n,k \in \mathbb{N}$ and any projection $p_k: \mathbb{R}^n \to \mathbb{R}$, we have:

$$\mathfrak{A} \models (\forall x_1)\cdots (\forall x_n)(p_k(x_1, \cdots, x_n)=x_k)$$

$\bullet$ For every $f, g_1, \cdots g_n \in \mathcal{C}^{\infty}(\mathbb{R}^m, \mathbb{R})$ with $m,n \in \mathbb{N}$, and every $h \in \mathcal{C}^{\infty}(\mathbb{R}^n, \mathbb{R})$ such that $f = h \circ (g_1, \cdots, g_n)$, one has:
$$\mathfrak{A} \models (\forall x_1)\cdots (\forall x_m)(f(x_1, \cdots, x_m)=h(g(x_1, \cdots, x_m), \cdots, g_n(x_1, \cdots, x_m)))$$
\end{definition}

\begin{definition}Let $(A, \Phi)$ and $(B,\Psi)$ be two $\mathcal{C}^{\infty}-$rings. A function $\varphi: A \to B$ is called a \textbf{morphism of $\mathcal{C}^{\infty}-$rings} or \textbf{$\mathcal{C}^{\infty}$-homomorphism} if for any $n \in \mathbb{N}$ and any $f: \mathbb{R}^n \stackrel{\mathcal{C}^{\infty}}{\to} \mathbb{R}$ the following diagram commutes:
$$\xymatrixcolsep{5pc}\xymatrix{
A^n \ar[d]_{\Phi(f)}\ar[r]^{\varphi^{(n)}} & B^n \ar[d]^{\Psi(f)}\\
A \ar[r]^{\varphi^{}} & B
}$$
 \textit{i.e.}, $\Psi(f) \circ \varphi^{(n)} = \varphi^{} \circ \Phi(f)$.
\end{definition}

\begin{remark}
Observe that $\mathcal{C}^{\infty}-$structures, together with their morphisms compose a category, that we denote by $\mathcal{C}^{\infty}{\rm \bf Str}$, and that $\mathcal{C}^{\infty}-$rings, together with all the $\mathcal{C}^{\infty}-$homomorphisms between $\mathcal{C}^{\infty}-$rings compose a full subcategory of $\mathcal{C}^{\infty}{\rm \bf Rng}$. In particular, since $\mathcal{C}^{\infty}{\rm \bf Rng}$ is a ``variety of algebras'' (it is a class of $\mathcal{C}^{\infty}-$structures which satisfy a given set of equations), it is closed under substructures, homomorphic images and products, by \textbf{Birkhoff's HSP Theorem}. Moreover:

$\bullet$ $\mathcal{C}^{\infty}{\rm \bf Rng}$ is a concrete category and the forgetful functor, $ U :  \mathcal{C}^{\infty}{\rm \bf Rng}  \to Set$ creates directed inductive colimits. Since $\mathcal{C}^{\infty}{\rm \bf Rng}$ is a variety of algebras, it has all (small) limits and (small) colimits. In particular, it has binary coproducts, that is, given any two $\mathcal{C}^{\infty}-$rings $A$ and $B$, we have their coproduct $A \stackrel{\iota_A}{\rightarrow} A\otimes_{\infty} \stackrel{\iota_B}{\leftarrow} B$;

$\bullet$ Each set $X$ freely generates a $C^\infty$-ring, $L(X)$, as follows:\\
-  for any finite set $X'$ with $\sharp X' = n$ we have $ L(X')= \mathcal{C}^{\infty}(\mathbb{R}^{X'}) \cong \mathcal{C}^\infty(\mathbb{R}^n, \mathbb{R})$ is the free $C^\infty$-ring on $n$ generators, $n \in \mathbb{N}$;\\
- for a general set, $X$, we take $L(X) = \mathcal{C}^{\infty}(\mathbb{R}^X):= \varinjlim_{X' \subseteq_{\rm fin} X} \mathcal{C}^{\infty}(\mathbb{R}^{X'})$;

$\bullet$ Given any $\mathcal{C}^{\infty}-$ring $A$ and a set, $X$, we can freely adjoin the set $X$ of variables to $A$ with the following construction: $A\{ X\}:= A \otimes_{\infty} L(X)$. The elements of $A\{ X\}$ are usually called $\mathcal{C}^{\infty}-$polynomials;

$\bullet$ The congruences of $\mathcal{C}^{\infty}-$rings are classified by their ``ring-theoretical'' ideals;

$\bullet$ Every $\mathcal{C}^{\infty}-$ring is the homomorphic image of some free $\mathcal{C}^{\infty}-$ring determined by some set, being isomorphic to the quotient of a free $\mathcal{C}^{\infty}-$ring by some ideal.\\

\end{remark}

Within the category of $\mathcal{C}^{\infty}-$rings, we have two special subcategories, that we define in the sequel.\\

\begin{definition}A $\mathcal{C}^{\infty}-$ring $A$ is \textbf{finitely generated} whenever there is some $n \in \mathbb{N}$ and some ideal $I \subseteq \mathcal{C}^{\infty}(\mathbb{R}^n)$ such that $A \cong \dfrac{\mathcal{C}^{\infty}(\mathbb{R}^{n})}{I}$. The category of all finitely generated $\mathcal{C}^{\infty}-$rings is denoted by $\mathcal{C}^{\infty}{\rm \bf Rng}_{\rm fg}$.
\end{definition}

\begin{definition}
A $\mathcal{C}^{\infty}-$ring is \textbf{finitely presented} whenever there is some $n \in \mathbb{N}$ and some \underline{finitely generated ideal} $I \subseteq \mathcal{C}^{\infty}(\mathbb{R}^n)$ such that $A \cong \dfrac{\mathcal{C}^{\infty}(\mathbb{R}^{n})}{I}$.\\

Whenever $A$ is a finitely presented $\mathcal{C}^{\infty}-$ring, there is some $n \in \mathbb{N}$ and some $f_1, \cdots, f_k \in \mathcal{C}^{\infty}(\mathbb{R}^n)$ such that:\\

$$ A = \dfrac{\mathcal{C}^{\infty}(\mathbb{R}^n)}{\langle f_1, \cdots, f_k\rangle}$$

The category of all finitely presented $\mathcal{C}^{\infty}-$rings is denoted by $\mathcal{C}^{\infty}{\rm \bf Rng}_{\rm fp}$
\end{definition}

\begin{remark}
The categories $\mathcal{C}^{\infty}{\rm \bf Rng}_{{\rm fg}}$ and $\mathcal{C}^{\infty}{\rm \bf Rng}_{{\rm fp}}$ are closed under initial objects, binary coproducts and binary coequalizers. Thus, they are finitely co-complete categories, that is, they have all finite colimits (for a proof of this fact we refer to the chapter 1 of \cite{berni2018some}).\\

Since $\mathcal{C}^{\infty}{\rm \bf Rng}_{{\rm fp}}$ has all finite colimits, it follows that $\mathcal{C}^{\infty}{\rm \bf Rng}_{{\rm fp}}^{{\rm op}}$ has all finite limits.
\end{remark}

\begin{remark}
An $\mathbb{R}-$algebra $A$ in a category with finite limits, $\mathcal{C}$, may be regarded as a finite product preserving functor from the category ${\rm \bf Pol}$, whose objects are given by ${\rm Obj}\,({\rm \bf Pol}) =\{ \mathbb{R}^n | n \in \mathbb{N}\}$, and whose morphisms are given by polynomial functions between them, ${\rm Mor}\,({\rm \bf Pol}) = \{ \mathbb{R}^m \stackrel{p}{\rightarrow} \mathbb{R}^n | m,n \in \mathbb{N}, p \,\, {\rm polynomial}\}$, to $\mathcal{C}$, that is:

$$A: {\rm \bf Pol} \rightarrow \mathcal{C}.$$

In this sense, an $\mathbb{R}-$algebra $A$ is a functor which interprets all polynomial maps $p: \mathbb{R}^m \to \mathbb{R}^n$, for $m,n \in \mathbb{N}$. More precisely, the categories of $\mathbb{R}-$algebras as defined by the ``Universal Algebra approach'' and by the ``Functorial sense'' provide equivalent categories.

In this vein, one may define a $\mathcal{C}^{\infty}-$ring as a finite product preserving functor from the category ${\cal C}^{\infty}$, whose objects are given by ${\rm Obj}\,(\mathcal{C}^{\infty}) = \{ \mathbb{R}^n | n \in \mathbb{N}\}$ and whose morphisms are given by $\mathcal{C}^{\infty}-$functions between them, ${\rm Mor}\,(\mathcal{C}^{\infty}) = \{ \mathbb{R}^m \stackrel{f}{\rightarrow} \mathbb{R}^n | m,n \in \mathbb{N}, f {\rm smooth}\,\, {\rm function}\}$, \textit{i.e.},\\

$$A: \mathcal{C}^{\infty} \rightarrow \mathcal{C}.$$
\end{remark}



We write ``$A$'' instead of $(A,\Phi)$, partly to simplify the notation and partly because the $\mathcal{C}^{\infty}-$structure will not play an important role as it did in the previous paper (\cite{cerqueira2019universal}).\\

\section{An adjunction between $\mathcal{C}^{\infty}{\rm \bf Rng}$ and ${\rm \bf CRing}$}

We have already seen in \cite{cerqueira2019universal}, that the forgetful functor $U': \mathcal{C}^{\infty}{\rm \bf Rng} \rightarrow {\rm \bf Set}$ has a left adjoint, $L$, and it is a well known fact that the forgetful functor $U'': {\rm \bf CRing} \rightarrow {\rm \bf Set}$ has a left adjoint, so both $\mathcal{C}^{\infty}{\rm \bf Rng}$ and  ${\rm \bf CRing}$  are concrete categories. Considering the forgetful functor $\widetilde{U}: \mathcal{C}^{\infty}{\rm \bf Rng} \to {\rm \bf CRing}$, we have the following diagram:

$$\xymatrixcolsep{5pc}\xymatrix{
\mathcal{C}^{\infty}{\rm \bf Rng} \ar[rr]^{\widetilde{U}} \ar@<1ex>[dr]^{U'}& & {\rm \bf CRing} \ar@<1ex>[dl]^{U''}\\
 & \ar[ul]^{L'} {\rm \bf Set} \ar[ur]^{L''}&
}$$

We would like to obtain a left-adjoint to the functor $\widetilde{U}$, which makes the following diagram commutative:

$$\xymatrixcolsep{5pc}\xymatrix{
\mathcal{C}^{\infty}{\rm \bf Rng} \ar@<0.8ex>[rr]^{\widetilde{U}} \ar@<1ex>[dr]^{U'}& & {\rm \bf CRing} \ar@<1ex>[dl]^{U''} \ar@<1ex>[ll]^{\widetilde{L}}\\
 & \ar[ul]^{L'} {\rm \bf Set} \ar[ur]^{L''}&
}$$

In this section we describe a construction of such a left adjoint.\\

\begin{theorem}\label{Madruga}The forgetful functor:
$$\widetilde{U}: \,\mathcal{C}^{\infty}{\rm \bf Rng} \rightarrow {\rm \bf CRing}$$
has a left adjoint.
\end{theorem}
\begin{proof} First we note that both ${\rm \bf CRing}$ and $\mathcal{C}^{\infty}{\rm \bf Rng}$ are concrete categories and the forgetful functors:
$$U': \mathcal{C}^{\infty}{\rm \bf Rng} \rightarrow {\rm \bf Set}$$
and
$$U'': {\rm \bf CRing} \rightarrow {\rm \bf Set}$$
have left adjoints,  $L': {\rm \bf Set} \to \mathcal{C}^{\infty}{\rm \bf Rng}$ (where, for a given set $E$, we have $L'(E)=\mathcal{C}^{\infty}(\mathbb{R}^{E})$) and $L'': {\rm \bf Set} \to {\rm \bf CRing}$ (where, for a given set $E$, we have $L''(E)=\mathbb{Z}[E]$) so we have, for each $\mathcal{C}^{\infty}-$ring $B$, the  following natural bijection:

$$\mu_{E,B}: {\rm \bf Set}\,(E, U'(B)) \stackrel{\cong}{\rightarrow} \mathcal{C}^{\infty}{\rm \bf Rng}\,(\mathcal{C}^{\infty}(\mathbb{R}^{E}),B) $$

defined as follows:

Given a function $f: E \to U'(B)$, by the universal property of $\imath_E': E \rightarrow U'(\mathcal{C}^{\infty}(\mathbb{R}^{E}))$, there is a unique $\mathcal{C}^{\infty}-$homomorphism $f': \mathcal{C}^{\infty}(\mathbb{R}^{E}) \rightarrow B$, such that the following diagram commutes:

$$\xymatrixcolsep{5pc}\xymatrix{
E \ar[r]^{\imath_E'} \ar[dr]_{f} & U'(\mathcal{C}^{\infty}(\mathbb{R}^{E})) \ar[d]^{U'(f')}\\
 & U'(B)}$$

We define $\mu_{E,B}(f):=f'$, and it is easy to see that it is natural and that the map which takes any $\mathcal{C}^{\infty}-$homomorphism $g: \mathcal{C}^{\infty}(\mathbb{R}^{E}) \rightarrow B$ to the function $U'(g)\circ \imath_E': E \rightarrow U'(B)$ is its inverse, \textit{i.e.}, $\mu_{E,B}^{-1}(g)=U'(g)\circ \imath_E'$.\\

We also have, for any commutative unital ring $R$, a natural bijection:

$$\nu_{E,R}: {\rm \bf Set}\,(E, U''(R)) \stackrel{\cong}{\rightarrow} {\rm \bf CRing}\,(\mathbb{Z}[E],R),$$

which is defined in a similar way, as follows: given a function $f: E \to U''(R)$, by the universal property of $\imath_E'': E \to U''(\mathbb{Z}[E])$, there is a unique homomorphism of rings $f'': \mathbb{Z}[E] \to R$ such that the following diagram commutes:

$$\xymatrixcolsep{5pc}\xymatrix{
E \ar[r]^{\imath_E''} \ar[dr]_{f} & U''(\mathbb{Z}[E]) \ar[d]^{U'(f'')}\\
 & U''(R)}.$$

We define $\nu_{E,R}(f) = f''$, and it is easy to see that it is natural and that the map which takes any commutative unital rings homomorphism $g: \mathbb{Z}[E] \rightarrow R$ to the function $U''(g)\circ \imath_E'': E \rightarrow U''(R)$ is its inverse, \textit{i.e.}, $\nu_{E,R}^{-1}(g)=U''(g)\circ \imath_E''$.\\

Note that the following diagram commutes:

$$\xymatrixcolsep{5pc}\xymatrix{
\mathcal{C}^{\infty}{\rm \bf Rng} \ar[r]^{\widetilde{U}} \ar[dr]_{U'} & {\rm \bf CRing} \ar[d]^{U''}\\
 & {\rm \bf Set}
},$$

that is, $U' = U'' \circ \widetilde{U}$, and we have, for any set $E$ and any $\mathcal{C}^{\infty}-$ring $B$:

$${\rm \bf Set}\,(E,U'(B)) = {\rm \bf Set}\,(E, U''(\widetilde{U}(B))),$$

and we get the following:

\begin{equation*}
  \mathcal{C}^{\infty}{\rm \bf Rng}\,(\mathcal{C}^{\infty}(\mathbb{R}^{E}),B) \stackrel{\mu_{E,B}^{-1}}{\rightarrow} {\rm \bf Set}\,(E, U'(B)) = {\rm \bf Set}\,(E, U''(\widetilde{U}(B))) \stackrel{\nu_{E, \widetilde{U}(B)}}{\rightarrow}  {\rm \bf CRing}\,(\mathbb{Z}[E], \widetilde{U}(B))
\end{equation*}

so, composing these natural bijections yields the natural bijection:

$$\varphi_{\mathbb{Z}[E],B}: \mathcal{C}^{\infty}{\rm \bf Rng}\,(\mathcal{C}^{\infty}(\mathbb{R}^{E}),B) \stackrel{\nu_{E, \widetilde{U}(B)}\circ \mu_{E,B}^{-1}}{\longrightarrow} {\rm \bf CRing}\,(\mathbb{Z}[E], \widetilde{U}(B)).$$

In the case that $\widetilde{U}$ has a left adjoint, $\widetilde{L}$, we must have $L'' \cong \widetilde{L}\circ L'$ (cf. \textbf{Corollary 1}, p. 85 of \cite{mac2013categories}). This suggests us that, in order to compose the left adjoint  we must define a function $\widetilde{L}_0$  on the free objects of ${\rm \bf CRing}$ as follows:\\

$$\widetilde{L}_0(\mathbb{Z}[E]):= \mathcal{C}^{\infty}(\mathbb{R}^{E}).$$

In order to define the left adjoint $\widetilde{L}$, we are going to use \textbf{Theorem 9}, p. 116 of \cite{arbib1975arrows}, that is, we are going to show that $\widetilde{U}: \mathcal{C}^{\infty}{\rm \bf Rng} \to {\rm \bf CRings}$ is a functor with the property that to every commutative unital ring, $R$ there corresponds a free $\mathcal{C}^{\infty}-$ring, call it $(\widetilde{L}_0(R), \gamma_R: R \to \widetilde{U}(\widetilde{L}_0(R)))$. Given any commutative unital rings homomorphism $f: R \to R'$, we define $\widetilde{L}_1(f): \widetilde{L}_0(R) \rightarrow \widetilde{L}_0(R')$ through the universal property of $\gamma_R$: given the map $\gamma_{R'} \circ f: R \rightarrow \widetilde{U}(\widetilde{L}_0(R''))$, there is a unique $\mathcal{C}^{\infty}-$homomorphism $\widetilde{L}_1(f): \widetilde{L}_0(R) \to \widetilde{L}_0(R')$ such that the following diagram commutes:

$$\xymatrixcolsep{5pc}\xymatrix{
R \ar[r]^{\gamma_R} \ar[dr]_{\gamma_{R'}\circ f} & \widetilde{U}(\widetilde{L}_0(R)) \ar[d]^{\widetilde{U}(\widetilde{L}_1(f))}\\
 & \widetilde{U}(\widetilde{L}_0(R'))}$$

We begin with the definition of the free $\mathcal{C}^{\infty}-$rings corresponding to the free commutative unital rings. So, we start by commutative rings of the form $\mathbb{Z}[E]$ for some set $E$.\\

Given any set $E$ and given the map $\imath_E': E \to U''(\widetilde{U}(\mathcal{C}^{\infty}(\mathbb{R}^{E})))$, by the universal property of $\imath_E'': E \to U''(\mathbb{Z}[E])$, there is a unique commutative unital rings homomorphism:

$$\gamma_E: \mathbb{Z}[E] \rightarrow \widetilde{U}(\mathcal{C}^{\infty}(\mathbb{R}^{E}))$$

such that the following triangle commutes:

$$\xymatrixcolsep{5pc}\xymatrix{
E \ar[r]^{\imath_E''} \ar[dr]_{\imath_E'} & U''(\mathbb{Z}[E]) \ar[d]^{U''(\gamma_E)}\\
 & U''(\widetilde{U}(\mathcal{C}^{\infty}(\mathbb{R}^{E})))}.$$

Note that $\gamma_{E}$ has the required universal property. In fact, for any $\mathcal{C}^{\infty}-$ring $A$, one has the following map:

$$\begin{array}{cccc}
    \widehat{\gamma}_{\mathbb{Z}[E],A}: & \mathcal{C}^{\infty}{\rm \bf Rng}\,(\mathcal{C}^{\infty}(\mathbb{R}^{E}),A) & \rightarrow & {\rm \bf CRing}\,(\mathbb{Z}[E], \widetilde{U}(A)) \\
     & f & \mapsto & \widetilde{U}(f)\circ \gamma_E
  \end{array}$$

which is a natural bijection since the following diagram commutes:

$$\xymatrixcolsep{5pc}\xymatrix{
\mathcal{C}^{\infty}{\rm \bf Rng}\,(\mathcal{C}^{\infty}(\mathbb{R}^{E}),A) \ar[r]^{\widehat{\gamma}_{\mathbb{Z}[E],A}} \ar[dr]_{\mu_{E,A}^{-1}} & {\rm \bf CRing}\,(\mathbb{Z}[E], \widetilde{U}(A)) \ar[d]^{\nu_{E, \widetilde{U}(A)}^{-1}}\\
 & {\rm \bf Set}\,(E, U'(A))
},$$

In fact, for every $f \in \mathcal{C}^{\infty}{\rm \bf Rng}\,(\mathcal{C}^{\infty}(\mathbb{R}^{E}),A)$ we have:

\begin{multline*}
(\nu_{E, \widetilde{U}(A)}^{-1}\circ \widehat{\gamma}_{\mathbb{Z}[E],A})(f) = \nu_{E, \widetilde{U}(A)}^{-1}(\widehat{\gamma}_{\mathbb{Z}[E],A}(f)) = \nu_{E, \widetilde{U}(A)}^{-1}(\widetilde{U}(f)\circ \gamma_E) = \\
= U''(\widetilde{U}(f)\circ \gamma_E)\circ \imath_E'' = (U''(\widetilde{U}(f))\circ U''(\gamma_E))\circ \imath_{E''} = U'(f)\circ (U''(\gamma_E)\circ \imath_E'')=\\
= U'(f)\circ \imath_E' = \mu_{E,A}^{-1}(f)
\end{multline*}

and $\widehat{\gamma}_{\mathbb{Z}[E],A}$ is a natural bijection.\\

Now we are going to extend this definition to any commutative unital ring.\\

Given any commutative unital ring of the form $\dfrac{\mathbb{Z}[E]}{I}$, we define $\widehat{I} = \langle \gamma_E[I]\rangle$, the ideal of the $\mathcal{C}^{\infty}-$ring $\mathcal{C}^{\infty}(\mathbb{R}^{E})$ generated by $\gamma_E[I]$, and we define $\gamma_{E,I}$ as the unique commutative unital ring homomorphism such that the following diagram commutes:

\begin{equation*}\xymatrixcolsep{5pc}\xymatrix{
\mathbb{Z}[E] \ar[r]^{\gamma_E} \ar@{->>}[d]_{q_I} & U''(\mathcal{C}^{\infty}(\mathbb{R}^{E})) \ar@{->>}[d]^{q_{\widehat{I}}}\\
\dfrac{\mathbb{Z}[E]}{I} \ar@{-->}[r]^{\gamma_{E,I}} & U''\left( \dfrac{\mathcal{C}^{\infty}(\mathbb{R}^{E})}{\widehat{I}}\right)}
\end{equation*}

where $q_I$ and $q_{\widehat{I}}$ are the canonical quotient homomorphisms.\\

Thus we define:

$$\widetilde{L}_0\left(\dfrac{\mathbb{Z}[E]}{I}\right):= \dfrac{\mathcal{C}^{\infty}(\mathbb{R}^{E})}{\widehat{I}}.$$

Now we are going to prove that the map:

$$\gamma_{E,I}: \dfrac{\mathbb{Z}[E]}{I} \rightarrow \widetilde{U}\left( \dfrac{\mathcal{C}^{\infty}(\mathbb{R}^{E})}{\widehat{I}}\right) $$

has the universal property required, by showing that $\widehat{\gamma}_{E,I}$ is a composition of bijections.\\

Consider the sets:

$$\mathcal{A} = \{ \xymatrix{\mathcal{C}^{\infty}(\mathbb{R}^{E}) \ar[r]^{\varphi} & B} | \varphi[\widehat{I}] = \{ 0\} \} \subseteq \mathcal{C}^{\infty}{\rm \bf Rng}(\mathcal{C}^{\infty}(\mathbb{R}^{E}), B)$$

and:
$$\mathcal{B}= \{ \xymatrix{\mathbb{Z}[E] \ar[r]^{\psi} & \widetilde{U}(B)} | \psi[I] = \{ 0\} \} \subseteq {\rm \bf CRing}(\mathbb{Z}[E], \widetilde{U}(B)),$$

and take:
$$\widehat{\gamma}_E\displaystyle\upharpoonright_{\mathcal{A}} : \mathcal{A} \to \mathcal{B}.$$

Since $\widehat{\gamma}_E\displaystyle\upharpoonright_{\mathcal{A}}[\mathcal{A}] = \mathcal{B}$, we have a bijection between $\mathcal{A}$ and $\mathcal{B}$.

By the \textbf{First Isomorphism Theorem}, however, for every map $\xymatrix{\mathcal{C}^{\infty}(\mathbb{R}^E) \ar[r]^{\varphi} & B}$ such that $\widehat{I} \subseteq \ker \varphi$, i.e., for every element of $\mathcal{A}$ there is a unique element of $\mathcal{C}^{\infty}{\rm \bf Rng}\left( \dfrac{\mathcal{C}^{\infty}(\R^E)}{\widehat{I}},B \right)$, say $\overline{\varphi} : \dfrac{\mathcal{C}^{\infty}(\R^E)}{\widehat{I}} \to B$ such that the following diagram:
  $$\xymatrix{
  \mathcal{C}^{\infty}(\R^E) \ar[d]_{q_{\widehat{I}}} \ar[r]^{\varphi} & B \\
  \dfrac{\mathcal{C}^{\infty}(\R^E)}{\widehat{I}} \ar[ur]^{\overline{\varphi}}
  }$$
commutes, so there is a bijection between these classes. Let $\alpha: \mathcal{A} \to \mathcal{C}^{\infty}{\rm \bf Rng}\left( \dfrac{\mathcal{C}^{\infty}(\R^E)}{\widehat{I}},B \right)$ be such a bijection.\\

By the \textbf{Fundamental Homomorphism Theorem} there is a natural bijection between $\mathcal{B}$ and ${\rm \bf CRing}\left( \dfrac{\mathbb{Z}[E]}{I}, \widetilde{U}(B) \right)$ such that for any $\psi \in \mathcal{B}$ there is a unique $\overline{\psi}$ such that the following diagram:
$$\xymatrix{
\mathbb{Z}[E] \ar[d]^{q_{I}} \ar[r]^{\psi} & \widetilde{U}(B)\\
\dfrac{\mathbb{Z}[E]}{I} \ar[ur]^{\overline{\psi}}
}$$
commutes. Let $\beta: \mathcal{B} \to {\rm \bf CRing}\left( \dfrac{\mathbb{Z}[E]}{I}, \widetilde{U}(B) \right)$ be such a bijection.\\

The map $\widehat{\gamma}_{E,I}:= \beta \circ (\widetilde{\gamma}_{E}\displaystyle\upharpoonright_{\mathcal{A}} ) \circ \alpha^{-1}$ is a natural bijection, since it is a composition of natural bijections:

$$\xymatrixcolsep{5pc}\xymatrix{
\mathcal{A} \ar[r]^{\widehat{\gamma}_{E}\displaystyle\upharpoonright_{\mathcal{A}}} & \mathcal{B} \ar[d]^{\beta}\\
\mathcal{C}^{\infty}{\rm \bf Rng}\left( \dfrac{\mathcal{C}^{\infty}(\mathbb{R}^{E})}{\widehat{I}}, B\right) \ar@{-->}[r]_{\widehat{\gamma}_{E,I}} \ar[u]_{\alpha^{-1}} & {\rm \bf CRing}\left( \dfrac{\mathbb{Z}[E]}{I}, \widetilde{U}(B)\right)}.$$

Now, given any commutative unital ring $R$, we have the counity of the adjunction $L'' \dashv U''$ at $R$:

$$\varepsilon_R'' : L''(U''(R)) \twoheadrightarrow R$$

so we have the isomorphism:

$$\overline{\varepsilon_R''}: \dfrac{\mathbb{Z}[U''(R)]}{\ker \varepsilon_R''} \rightarrow R,$$

which provides us the canonical presentation of $R$ by relations and generators:

$$R \cong \dfrac{\mathbb{Z}[U''(R)]}{\ker \varepsilon_R''},$$

so we define:

$$\widetilde{L}_0(R) = \widetilde{L}_0 \left( \dfrac{\mathbb{Z}[U''(R)]}{\ker \varepsilon_R''}\right) = \dfrac{\mathcal{C}^{\infty}(\mathbb{R}^{U''(R)})}{\widehat{\ker \varepsilon_R''}}$$

together with the universal map:

$$\gamma_R = \gamma_{U'(R), \ker \varepsilon_R''} \circ \overline{\varepsilon_R''}^{-1}: R \rightarrow \widetilde{U}\left( \dfrac{\mathcal{C}^{\infty}(\mathbb{R}^{U''(R)})}{\widehat{\ker \varepsilon_R''}}\right)$$

Applying the \textbf{Theorem 9} (p. 116 of \cite{arbib1975arrows}), it follows that there is a unique way (up to natural isomorphism) to define a left-adjoint, $\widetilde{L}$ to $\widetilde{U}$.

\end{proof}

\begin{remark}Since $\widetilde{L}$, described above, is left-adjoint, it preserves all colimits, and in particular we have:

$$\widetilde{L}\left( \dfrac{\mathbb{Z}[E]}{I}\right) \cong \dfrac{\widetilde{L}(\mathbb{Z}[E])}{\langle \widetilde{\gamma_E}[I]\rangle}$$

Note that $\widetilde{L}$ takes finitely generated (presented) commutative unital rings to finitely generated (presented) $\mathcal{C}^{\infty}-$rings.\\

In fact, let $B = \dfrac{\mathbb{Z}[x_1, \cdots, x_n]}{\langle p_1, \cdots, p_k\rangle}$, with $k,n \in \mathbb{N}$, be a finitely presented commutative unital ring, so it can be expressed by the coequalizer:

$$\xymatrix{\mathbb{Z}[x_1, \cdots, x_k] \ar@<1ex>[r]^{0} \ar@<-1ex>[r]_{P} & \mathbb{Z}[x_1, \cdots, x_n] \ar@{->>}[r] & B},$$

where for every $i \in \{1, \cdots, k \}$, $P(x_i)=p_i$ and $0(x_i)=0$.\\

Since $\widetilde{L}$ preserves free objects and colimits, it follows that $\widetilde{L}(B)$ is given by the coequalizer:

$$\xymatrix{\mathcal{C}^{\infty}(\mathbb{R}^k) \ar@<1ex>[r]^{\widetilde{P}} \ar@<-1ex>[r]_{0} & \mathcal{C}^{\infty}(\mathbb{R}^n) \ar@{->>}[r] & \widetilde{L}(B)\cong \dfrac{\mathcal{C}^{\infty}(\mathbb{R}^n)}{\langle \widetilde{\gamma_E}(p_1), \cdots, \widetilde{\gamma_E}(p_k) \rangle}}$$
\end{remark}

\begin{remark}Since $\mathbb{R} \cong \mathcal{C}^{\infty}(\{ *\})$ is the initial object in $\mathcal{C}^{\infty}{\rm \bf Rng}$, we can consider $\mathcal{C}^{\infty}-$rings as $\mathbb{R}-$algebras. In fact, we have a forgetful functor:

$$\mathcal{U}: \mathcal{C}^{\infty}{\rm \bf Rng} \rightarrow \mathbb{R}-{\rm \bf Alg}.$$

The functor $\mathcal{U}$ also has a left adjoint, whose construction is basically the same as the given in the previous theorem.
\end{remark}

\section{Smooth Ring of Fractions}

\hspace{0.5cm} We begin by giving a description of the fundamental concept of ``smooth ring of fractions'', presenting a slight modification of the axioms given in \cite{moerdijk1986rings}. In order to show that the $\mathcal{C}^{\infty}-$ring of fractions exists in the category of $\mathcal{C}^{\infty}-$rings, we use the $\mathcal{C}^{\infty}-$ring of $\mathcal{C}^{\infty}-$polynomials.\\

We turn to the discussion of how to obtain the ring of fractions of a $\mathcal{C}^{\infty}-$ring $A$ with respect to some of its subsets, $S \subseteq U(A, \Phi)$. Recall that for a commutative ring $(R, +_R, \cdot_R, -_R, 1_R)$ and a multiplicative subset $S \subseteq R$, the ring of fractions is obtained by defining an equivalence relation on $R \times S$, $\equiv$, given by:
$$((r_1,s_1) \sim (r_2, s_2)) \iff (\exists u \in S)(u \cdot (r_1 \cdot s_2 - r_2 \cdot s_1)=0).$$

We denote
$$S^{-1}R \stackrel{\cdot}{=} \dfrac{R \times S}{\sim} = \{ [(r,s)] \stackrel{\cdot}{=} \frac{r}{s} | (r,s) \in R \times S\},$$

and we put a ring structure on it by defining the addition and the multiplication of these fractions in the same way as in Elementary Algebra, that is:
$$\begin{array}{cccc}
    + : & S^{-1}R \times S^{-1}R & \rightarrow & S^{-1}R \\
     & ([(r_1,s_1)], [(r_2,s_2)]) & \mapsto & [(r_1 \cdot_R s_2 +_R r_2 \cdot_R s_1, s_1 \cdot_R s_2)]
  \end{array}$$

$$\begin{array}{cccc}
    \cdot : & S^{-1}R \times S^{-1}R & \rightarrow & S^{-1}R \\
     & ([(r_1,s_1)], [(r_2,s_2)]) & \mapsto & [(r_1 \cdot_R r_2, s_1 \cdot_R s_2)]
  \end{array}$$

$$\begin{array}{cccc}
    - : &  S^{-1}R & \rightarrow & S^{-1}R \\
     & [(r_1,s_1)] & \mapsto & [(-_R r_1, s_1)]
  \end{array}$$

Together with the element $[(1_R,1_R)]$, $(S^{-1}R, +_R, \cdot_R, [(1_R,1_R)])$ is a commutative unital ring. We also have a canonical ring homomorphism that ``embeds'' $R$ into $S^{-1}R$ in such a way that the image of every element of $S$ is invertible, namely:

$$\begin{array}{cccc}
    \eta_S : & R & \rightarrow & S^{-1}R \\
     & r & \mapsto & [(r,1)]
  \end{array}$$

We can prove that $S^{-1}R$ has the following universal property (cf. \textbf{Proposition 3.1} of \cite{atiyah1969introduction}):

\begin{proposition}\label{Nab}Given a ring homomorphism $g: R \to B$  such that $(\forall s \in S)(g(s) \in B^{\times})$, there is a unique ring homomorphism $\widetilde{g}: S^{-1}R \to B$ such that the following triangle commutes:

$$\xymatrixcolsep{5pc}\xymatrix{
A \ar[r]^{\eta_S} \ar[dr]_{g} & S^{-1}R \ar[d]^{\exists ! \widetilde{g}}\\
   & B}$$
\end{proposition}

In order to extend the notion of the ring of fractions to the category $\mathcal{C}^{\infty}{\rm \bf Rng}$ we make use of the universal property described in \textbf{Proposition \ref{Nab}}, as we can see in the following:\\

\begin{definition}\label{Alem}Let $A$ be a $\mathcal{C}^{\infty}-$ring and $S \subseteq A$ one of its subsets. The \index{$\mathcal{C}^{\infty}-$ring of fractions}$\mathcal{C}^{\infty}-$\textbf{ring of fractions} of $A$ with respect to $S$ is a $\mathcal{C}^{\infty}-$ring together with a $\mathcal{C}^{\infty}-$homo\-morphism $\eta_S: A \to A\{ S^{-1}\}$ with the following properties:
\begin{itemize}
  \item[(1)]{$(\forall s \in S)(\eta_S(s) \in (A\{ S^{-1}\})^{\times})$}
  \item[(2)]{If $\varphi: A \to B$ is any $\mathcal{C}^{\infty}-$homomorphism such that for every $s \in S$ we have $\varphi(s) \in B^{\times}$, then there is a unique $\mathcal{C}^{\infty}-$homomorphism $\widetilde{\varphi}: A\{ S^{-1}\} \to B$ such that the following triangle commutes:
      $$\xymatrixcolsep{5pc}\xymatrix{
      A \ar[r]^{\eta_S} \ar[rd]^{\varphi} & A\{ S^{-1}\} \ar[d]^{\widetilde{\varphi}}\\
        & B}$$}
\end{itemize}

By this universal property, the $\mathcal{C}^{\infty}-$ring of fractions is unique, up to (unique) isomorphisms.
\end{definition}

Now we prove the existence of such a $\mathcal{C}^{\infty}-$ring of fractions by constructing it. Recall that we can only make use of the constructions available within the category $\mathcal{C}^{\infty}{\rm \bf Rng}$, such as the free $\mathcal{C}^{\infty}-$ring on a set of generators, their coproduct, their quotients and others described in \textbf{Chapter 1}.\\

Recall that given any set $S$, the  $\mathcal{C}^{\infty}-$ring of ``smooth polynomials'' in the set $S$ of variables over a $\mathcal{C}^{\infty}-$ring $A$ is obtained as follows:\\

We consider $\mathcal{C}^{\infty}(\mathbb{R}^S)$,  the free $\mathcal{C}^{\infty}-$ring on the set $S$ of generators, together with its canonical map, $\jmath_S: S \to \mathcal{C}^{\infty}(\mathbb{R}^S)$. If we denote by:

$$\xymatrixcolsep{5pc}\xymatrix{
A \ar[dr]^{\iota_A} & \\
   & A \otimes_{\infty} \mathcal{C}^{\infty}(\mathbb{R}^S)\\
\mathcal{C}^{\infty}(\mathbb{R}^S) \ar[ur]_{\iota_{\mathcal{C}^{\infty}(\mathbb{R}^S)}}
}$$

the coproduct of $A$ and $\mathcal{C}^{\infty}(\mathbb{R}^S)$, we define:
$$x_s := \iota_{\mathcal{C}^{\infty}(\mathbb{R}^S)}(\jmath_S(s)).$$

By definition, we have:

$$A\{ x_s | s \in S \} := A \otimes_{\infty} \mathcal{C}^{\infty}(\mathbb{R}^S).$$

so we can consider the quotient:

$$\dfrac{A\{x_s | s \in S \}}{\langle \{ x_s \cdot \iota_A(s) - 1 | s \in S \}\rangle},$$

where $\langle \{ x_s \cdot \iota_A(s) - 1 | s \in S\}\rangle$ is the ideal of $A$ generated by $\{ x_s \cdot \iota_A(s) - 1 | s \in S\}$. \\

In these settings we can formulate the following:\\

\begin{theorem}\label{conc}Let $A$ be a $\mathcal{C}^{\infty}-$ring and $S \subseteq A$ be any of its subsets. The \textbf{$\mathcal{C}^{\infty}-$ring of fractions of $A$ with respect to $S$} is given concretely by the $\mathcal{C}^{\infty}-$ring:
$$A\{ S^{-1}\}:= \dfrac{A\{ x_s | s \in S \}}{\langle \{x_s \cdot \iota_A(s) - 1 | s \in S \}\rangle}$$
together with the $\mathcal{C}^{\infty}-$homomorphism:
$$\eta_S:= q \circ \iota_A : A \to A\{ S^{-1}\},$$
where $q: A\{ x_s | s \in S\} \to \dfrac{A\{x_s | s \in S \}}{\langle \{ x_s \cdot \iota_A(s) - 1 | s \in S\}\rangle}$ is the canonical quotient map and $\iota_A: A \to A\{ x_s | s \in S\} = A \otimes_{\infty} \mathcal{C}^{\infty}(\mathbb{R}^S)$ is the canonical coproduct homomorphism corresponding to $A$.
\end{theorem}
\begin{proof}
We are going to show that $q \circ \iota_A: A \rightarrow \dfrac{A\{ x_s | s \in S \}}{\langle \{x_s \cdot s - 1 \}\rangle}$ satisfies the universal property described in the \textbf{Definition \ref{Alem}}.\\

First we prove that for every $s \in S$, $(q \circ \imath_A)(s) \in \left( \dfrac{A\{ x_s | s \in S\}}{\langle \{ x_s \cdot s - 1 | s \in S \}\rangle}\right)^{\times}$.\\

Given $s \in S$, we have that $x_s + \langle \{ x_s \cdot \iota_A(s) - 1 | s \in S \}\rangle$ is the multiplicative inverse of $(q \circ \imath_A)(s) = q(\iota_A(s)) = \iota_A(s) + \langle \{ x_s \cdot \iota_A(s) - 1 | s \in S \}\rangle$, since $x_s \cdot \iota_A(s) - 1 \in \langle \{ x_s \cdot \iota_A(s) - 1 | s \in S \}\rangle$. Thus $\eta_S(s) = (q \circ \imath_A)(s) \in \left( \dfrac{A\{ x_s | s \in S\}}{\langle \{ x_s \cdot \iota_A(s) - 1 | s \in S \}\rangle}\right)^{\times}$.\\

Now we need to check if the $\mathcal{C}^{\infty}-$homomorphism $q \circ \imath_A : A \to \dfrac{A\{ x_s | s \in S\}}{\langle \{ x_s \cdot s - 1 | s \in S \}\rangle}$ satisfies the universal property described in the \textbf{Definition \ref{Alem}}. \\

Let  $B$ be a $\mathcal{C}^{\infty}-$ring and $g: A \to B$ be a $\mathcal{C}^{\infty}-$rings homomorphism such that for every $s \in S$, $g(s) \in B^{\times}$.\\

By the universal property of the free $\mathcal{C}^{\infty}-$ring with the set $S$ of generators, given the restriction of $g: A \to B$ to the function $g\upharpoonright_{S}: S \to B$ and considering the map:
$$\begin{array}{cccc}
    g^{-1}_{S}: & S & \rightarrow & B \\
     & s & \mapsto & g(s)^{-1}
  \end{array}$$
there is a unique $\mathcal{C}^{\infty}-$rings homomorphism $\widetilde{g^{-1}_{S}}: \mathcal{C}^{\infty}(\mathbb{R}^S) \to B$ such that the following diagram commutes:

$$\xymatrixcolsep{5pc}\xymatrix{
S \ar[r]^{\jmath_S} \ar[dr]_{g^{-1}_{S}} & \mathcal{C}^{\infty}(\mathbb{R}^S) \ar[d]^{\exists ! \widetilde{g^{-1}_{S}}} \\
   & B
},$$

that is, $\widetilde{g^{-1}_{S}} \circ \jmath_S = g^{-1}_S $, or equivalently $(\forall s \in S)(g^{-1}_{S}(\jmath_S(s)) = g(s)^{-1})$.\\

By the universal property of the coproduct, given the diagram:

$$\xymatrixcolsep{5pc}
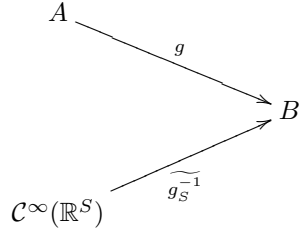
\xymatrix{
A \ar[dr]^{g} & \\
     & B \\
\mathcal{C}^{\infty}(\mathbb{R}^S) \ar[ur]_{\widetilde{g^{-1}_S}}
}$$

there is a unique $\mathcal{C}^{\infty}-$rings homomorphism, $\widetilde{g}: A \otimes_{\infty} \mathcal{C}^{\infty}(\mathbb{R}^S) \to B$ such that the following diagram commutes:

$$\xymatrixcolsep{5pc}
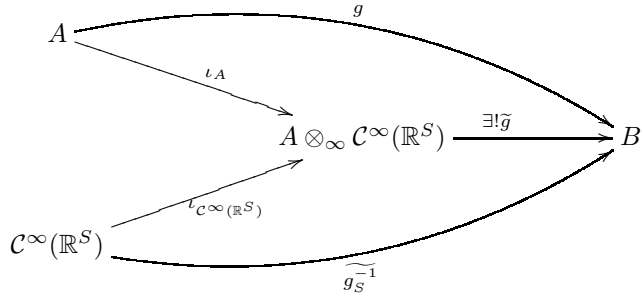
\xymatrix{
A \ar[dr]^{\iota_A} \ar@/^2pc/[rrd]^{g} & & \\
     & A \otimes_{\infty} \mathcal{C}^{\infty}(\mathbb{R}^S) \ar[r]^{\exists ! \widetilde{g}} & B \\
\mathcal{C}^{\infty}(\mathbb{R}^S) \ar[ur]_{\iota_{\mathcal{C}^{\infty}(\mathbb{R}^S)}} \ar@/_2pc/[rru]_{\widetilde{g^{-1}_S}} & &
},$$

that is,

$$\widetilde{g} \circ \iota_{\mathcal{C}^{\infty}(\mathbb{R}^S)} = \widetilde{g}\upharpoonright_S$$
and
\begin{equation}\label{magic}
\widetilde{g} \circ \iota_A = g
\end{equation}

Note that such a $\widetilde{g}$ satisfies:

$$(\forall s \in S)(\widetilde{g}(x_s) = g(s)^{-1}),$$

so for every $s \in S$, $\widetilde{g}(x_s\cdot \iota_A(s) - 1) = \widetilde{g}(x_s)\cdot \widetilde{g}(\iota_A(s)) - \widetilde{g}(1) = g(s)^{-1}\cdot g(s) - 1_B = 0$ and $\langle \{ x_s \cdot \iota_A(s) - 1 | s \in S \}\rangle \subseteq \ker \widetilde{g}$.\\

Let $q: A \otimes_{\infty} \mathcal{C}^{\infty}(\mathbb{R}^S) \to \dfrac{A \otimes_{\infty} \mathcal{C}^{\infty}(\mathbb{R}^S)}{\langle \{ x_s \cdot \iota_A(s) - 1 | s \in S \}\rangle}$ be the canonical quotient map. Since $\langle \{ x_s \cdot \iota_A(s) - 1 | s \in S\}\rangle \subseteq \ker\widetilde{g}$, by the \textbf{Theorem of the Homomorphism}, given the map $\widetilde{g}: A \otimes_{\infty}\mathcal{C}^{\infty}(\mathbb{R}^S) \to B$, there is a unique $\overline{g}: \dfrac{A \otimes_{\infty} \mathcal{C}^{\infty}(\mathbb{R}^S)}{\langle \{ x_s \cdot \iota_A(s) - 1 | s \in S \}\rangle} \to B$ such that the following diagram commutes:

$$\xymatrixcolsep{5pc}
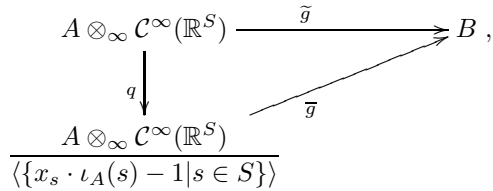
\xymatrix{
A \otimes_{\infty} \mathcal{C}^{\infty}(\mathbb{R}^S) \ar[r]^{\widetilde{g}} \ar[d]_{q} & B \\
\dfrac{A \otimes_{\infty} \mathcal{C}^{\infty}(\mathbb{R}^S)}{\langle \{ x_s \cdot \iota_A(s) - 1 | s \in S \}\rangle} \ar[ur]_{\overline{g}} &
},$$

that is, such that:

\begin{equation}\label{pacoca}
\overline{g} \circ q = \widetilde{g}.
\end{equation}

Note that since $\overline{g}\circ (q \circ \iota_A) = (\overline{g}\circ q) \circ \iota_A \stackrel{\eqref{pacoca}}{=} \widetilde{g} \circ \iota_A \stackrel{\eqref{magic}}{=} g$, the following diagram commutes:

$$\xymatrixcolsep{5pc}\xymatrix{
A \ar[r]^{q \circ \iota_A} \ar[dr]_{g} & \dfrac{A \otimes_{\infty} \mathcal{C}^{\infty}(\mathbb{R}^S)}{\langle \{ x_s \cdot s - 1 | s \in S \}\rangle} \ar[d]^{\overline{g}}\\
 & B}$$

Moreover, the very construction of $\overline{g}$ tells us that it is the unique $\mathcal{C}^{\infty}-$homo\-morphism that makes the above diagram commute.\\

Thus, the pair
\begin{equation*}\left( \dfrac{A\{ x_s | s \in S\}}{\langle \{ x_s \cdot \iota_A(s) - 1 | s \in S \}\rangle}, q \circ \iota_A : A \to \dfrac{A \otimes_{\infty} \mathcal{C}^{\infty}(\mathbb{R}^S)}{\langle \{ x_s \cdot \iota_A(s) - 1 | s \in S \}\rangle} \right)
\end{equation*}
is the $\mathcal{C}^{\infty}-$ring of fractions of $A$ with respect to $S$, up to isomorphism.
\end{proof}

\begin{remark}As we shall see later on, in \textbf{Proposition \ref{fact}}, the universal map $\eta_S: A \to A\{ S^{-1}\}$ is an epimorphism in $\mathcal{C}^{\infty}{\rm \bf Rng}$, even though it is not always a surjective map.
\end{remark}

\subsection{Smooth Saturation}\label{saty}

Within the theory of Commutative Algebra we have the following concept, namely the concept of ``saturation'' of a set:\\

\begin{definition}Let $A$ be any commutative ring with unity and $S \subseteq A$ any of its subsets. Denoting by $\langle S \rangle$ the multiplicative submonoid of $A$ generated by $S$, we define \textbf{the saturation of} $S$ as follows:
$$S^{{\rm sat}} = \{ a \in A | (\exists d \in A)( a \cdot d \in \langle S \rangle) \}.$$

In other words, the saturation of a set $S$ is the set of all divisors of  elements in $\langle S \rangle$.
\end{definition}

One easily sees that the saturation of a subset $S$ of a commutative ring $A$ is equal to the pre-image of the invertible elements of $A[S^{-1}]$ by the canonical map $\eta_S: A \to A\{ S^{-1}\}$, \textit{i.e.}, $S^{{\rm sat}} = \eta_S^{\dashv}[(A[S^{-1}])^{\times}]$. \\

We are going to use this characterization in order to define the \textit{smooth} saturation of a subset $S$ of a $\mathcal{C}^{\infty}-$ring $A$, that we are going to denote by $S^{\infty-{\rm sat}}$. First we need the following:

\begin{proposition}\label{fati}Let $(A,\Phi)$ be a $\mathcal{C}^{\infty}-$ring and $S \subseteq A$ be any subset. If both $(F, \sigma)$ and $(F', \sigma')$ satisfy the \textbf{Definition \ref{Alem}}, then $\sigma^{\dashv}[F^{\times}] = {\sigma'}^{\dashv}[{F'}^{\times}].$
\end{proposition}
\begin{proof}
Since both $(F,\sigma)$ and $(F',\sigma')$ have the universal property described in the \textbf{Definition \ref{Alem}}, there is a unique $\mathcal{C}^{\infty}-$isomorphism $\varphi: F \to F'$, so the following triangle commutes:
$$\xymatrixcolsep{5pc}\xymatrix{
 & F \ar[dd]^{\varphi}_{\cong}\\
A \ar[ur]^{\sigma} \ar[dr]_{\sigma'} & \\
  & F'
}$$

Because $\varphi: F \to F'$ is a $\mathcal{C}^{\infty}-$isomorphism, $\varphi^{\dashv}[{F'}^{\times}] = F^{\times}$, so:
$${\sigma'}^{\dashv}[{F'}^{\times}] = \sigma^{\dashv}[\varphi^{\dashv}[{F'}^{\times}]] = \sigma^{\dashv}[F^{\times}].$$
\end{proof}

Now we give the following:

\begin{definition}\label{satlisa}Let $A$ be a $\mathcal{C}^{\infty}-$ring, $S \subseteq A$ and $(F,\sigma)$ be a ring of fractions of $A$ with respect to $S$. The \index{smooth saturation}\textbf{smooth saturation} of $S$ in $A$ is:
$$S^{\infty-{\rm sat}}:= \{ a \in A | \sigma(a) \in F^{\times}\}.$$
\end{definition}

In virtue of the \textbf{Proposition \ref{fati}}, the set $S^{\infty-{\rm sat}}$ does not depend on any particular choice of the representation of the ring of fractions, rather it depends only on $A$ and $S$.\\

\begin{remark}Since for every $s \in S$, $\eta_S(s) \in (A\{ S^{-1}\})^{\times}$, from now on we are going to use the more suggestive \underline{notation}:

$$\left( \forall s \in S\right)\left(\dfrac{1}{\eta_S(s)} \stackrel{\cdot}{=} \eta_S(s)^{-1}\right),$$

and for any $a \in A$ and $s \in S$ we are going to denote:

$$\dfrac{\eta_S(a)}{\eta_S(s)} \stackrel{\cdot}{=} \eta_S(a)\cdot{\eta_S(s)^{-1}}.$$
\end{remark}

\begin{definition}Let $A$ be a $\mathcal{C}^{\infty}-$ring and let $S \subseteq A^{\times}$ be any subset. The \textbf{smooth saturation of $S$} is $S^{{\rm sat}-\infty} = {\eta_S}^{\dashv}[A\{ S^{-1}\}^{\times}]$, where $\eta_S\,: A \to A\{ S^{-1}\}$ is the canonical map of the ring of fractions of $A$ with respect to $S$.
\end{definition}

\begin{remark}\label{bols}Let $A$ be a $\mathcal{C}^{\infty}-$ring, $S \subseteq A$ and consider the forgetful functor:
$$\begin{array}{cccc}
    \mathcal{U}: & \mathcal{C}^{\infty}{\rm \bf Rng} & \to & \mathbb{R}-{\rm \bf Alg} \\
     & A & \mapsto & \mathcal{U}(A)\\
     & A \stackrel{f}{\to} B & \mapsto & \mathcal{U}(A) \stackrel{\mathcal{U}(f)}{\to} \mathcal{U}(B)\\
  \end{array}$$
We have always:
$$S^{{\rm sat}} \subseteq S^{\infty-{\rm sat}}$$
\end{remark}


\begin{theorem}\label{lara}Let $A$ be a $\mathcal{C}^{\infty}-$ring, $S \subseteq A$ and consider the forgetful functor:
$$\begin{array}{cccc}
    \mathcal{U}: & \mathcal{C}^{\infty}{\rm \bf Rng} & \to & \mathbb{R}-{\rm \bf Alg} \\
     & A & \mapsto & \mathcal{U}(A)\\
     & A \stackrel{f}{\to} B & \mapsto & \mathcal{U}(A) \stackrel{\mathcal{U}(f)}{\to} \mathcal{U}(B)\\
  \end{array}$$
Since $\eta_S^{\infty}: A \to A\{ S^{-1}\}$ is such that $\eta_S^{\infty}[S] \subseteq (A\{ S^{-1}\})^{\times}$, then $(\eta_S^{\infty})[S] \subseteq (\mathcal{U}(A\{ S^{\-1}\}))^{\times}$, so by the universal property of the ring of fractions:
$$\eta_S : \mathcal{U}(A) \to \mathcal{U}(A)[S^{-1}],$$
there is a unique $\R-$algebras homomorphism, ${\rm Can}: \mathcal{U}(A)[S^{-1}] \to \mathcal{U}(A\{ S^{-1}\})$ such that the following diagram commutes:
$$\xymatrix{
\mathcal{U}(A) \ar[r]^{\eta_S} \ar[dr]_{\mathcal{U}(\eta_S^{\infty})} & \mathcal{U}(A)[S^{-1}] \ar@{.>}[d]^{{\rm Can}}\\
  & \mathcal{U}(A\{ S^{-1}\})
}$$
The following assertions are equivalent:
\begin{itemize}
  \item[(1)]{$S^{{\rm sat}} = S^{\infty-{\rm sat}}$;}
  \item[(2)]{${\rm Can}$ is an isomorphism of $\mathbb{R}-$algebra.}
\end{itemize}
\end{theorem}
\begin{proof}
Ad (2) $\to$ (1). Since the above diagram commutes, we have:
$$S^{\infty-{\rm sat}} = {\eta_S^{\infty}}^{\dashv}[(A\{S^{-1}\})^{\times}] = \eta_S^{\dashv}[{\rm Can}^{\dashv}[(A\{S^{-1}\})^{\times}]] \stackrel{(2)}{=} \eta_S^{\dashv}[(A[S^{-1}])^{\times}] = S^{{\rm sat}}$$

Ad (1) $\to$ (2): We already know that ${\rm Can}: \mathcal{U}(A)[S^{-1}] \to \mathcal{U}(A\{S^{-1}\})$ is an $\R-$algebras morphism. Note that since $S^{{\rm sat}} = S^{\infty-{\rm sat}}$, the morphism:
$$\mathcal{U}(\eta_S^{\infty}): \mathcal{U}(A) \to \mathcal{U}(A\{ S^{-1}\})$$
is such that:
\begin{itemize}
  \item[(i)]{$(\forall \varphi' \in \mathcal{U}(A)[S^{-1}])(\exists a \in U(A))(\exists b \in S^{{\rm sat}})(\mathcal{U}(\eta_S^{\infty})(b) \cdot \varphi' = \mathcal{U}(\eta_S^{\infty})(a))$;}
  \item[(ii)]{$(\forall a \in U(A))(\mathcal{U}(\eta_S^{\infty})(a)=0 \to (\exists \lambda' \in S^{{\rm sat}})(\lambda' \cdot a = 0))$;}
  \item[(iii)]{$S \subseteq S^{{\rm sat}}$.}
\end{itemize}
which are precisely the hypotheses of \textbf{Theorem \ref{37}}, so $\mathcal{U}(\eta_S^{\infty}): \mathcal{U}(A) \to \mathcal{U}(A\{ S^{-1}\})$ is isomorphic to the localization. This fact implies that since ${\rm Can}$ is the only ring homomorphism which makes the diagram commute, ${\rm Can}$ must be the unique isomorphism between $\mathcal{U}(A)[S^{-1}]$ and $\mathcal{U}(A\{ S^{-1}\})$
\end{proof}

In what follows we give some properties relating the inclusion relation among the subsets of a $\mathcal{C}^{\infty}-$ring and their smooth saturations.\\

\begin{proposition}Let $A$ be a $\mathcal{C}^{\infty}-$ring and $T, S \subseteq A$ be any of its subsets. Then:
\begin{itemize}
  \item[(i)]{$A^{\times} \subseteq S^{\infty-{\rm sat}}$}
  \item[(ii)]{$S \subseteq S^{\infty-{\rm sat}}$}
  \item[(iii)]{$S \subseteq T$ implies $S^{\infty-{\rm sat}} \subseteq T^{\infty-{\rm sat}}$}
  \item[(iv)]{$S^{\infty-{\rm sat}} = \langle S \rangle^{\infty-{\rm sat}}$, where $\langle S \rangle$ is the submonoid generated by $S$.}
\end{itemize}
\end{proposition}
\begin{proof}
\begin{itemize}
  \item{ Ad  (i): Since $\eta_S : A \to A\{ S^{-1} \}$ is a $\mathcal{C}^{\infty}-$ring homomorphism, it preserves invertible elements, so $\eta_S[A^{\times}] \subseteq (A\{ S^{-1}\})^{\times}$, hence $A^{\times} \subseteq \eta_S^{\dashv}[(A\{ S^{-1}\})^{\times}]$.}
  \item{Ad (ii): It is immediate that $S \subseteq S^{\infty-{\rm sat}}$, since by the very definition of $\eta_S$, $(\forall a \in S)(\eta_S(a) \in (A\{ S^{-1}\})^{\times})$.}
  \item{Ad (iii): Since $S \subseteq T$, $\eta_T[S] \subseteq \eta_T[T] \subseteq (A\{ T^{-1}\})^{\times}$ so, by the universal property of $\eta_S : A \to A\{ S^{-1}\}$ there exists a unique  $\mu_{ST}: A\{ S^{-1}\} \to A\{ T^{-1}\}$ such that the following diagram commutes:
      $$\xymatrix{
      A \ar[r]^{\eta_S} \ar[rd]_{\eta_T} & A\{ S^{-1}\} \ar[d]^{\exists ! \mu_{ST}}\\
                                         & A\{ T^{-1}\}
      }$$
      Hence, $T^{\infty-{\rm sat}} = \eta_T^{\dashv}[(A\{ T^{-1}\})^{\times}] = (\mu_{ST} \circ \eta_S)^{\dashv}[(A\{ T^{-1}\})^{\times}] = \eta_S^{\dashv}[\mu_{ST}^{\dashv}[(A\{ T^{-1}\})^{\times}]]$. Now, since:
      $$\mu_{ST}^{\dashv}[(A\{ T^{-1}\})^{\times}] \supseteq A\{ S^{-1} \}^{\times},$$
      $$T^{\infty-{\rm sat}} = \eta_{S}^{\dashv}[\mu_{ST}^{\dashv}[ A\{ T^{-1}\}] \supseteq \eta_S^{\dashv}[(A\{ S^{-1}\})^{\times}] = S^{\infty-{\rm sat}}.$$
            }
  \item{Ad (iv):  Since $S \subseteq \langle S \rangle$, by the preceding item, $S^{\infty-{\rm sat}} \subseteq \langle S \rangle^{\infty-{\rm sat}}$. We need to show that $\langle S \rangle^{\infty-{\rm sat}} \subseteq S^{\infty-{\rm sat}}$.\\

      We already know that, since $S \subseteq \langle S \rangle$, there exists a unique $\mathcal{C}^{\infty}-$morphism $\mu_{S \langle S \rangle}$ such that the following diagram commutes:
      $$\xymatrix{
      A \ar[r]^{\eta_S} \ar[dr]_{\eta_{\langle S \rangle}} & A\{ S^{-1}\} \ar[d]^{\mu_{S \langle S \rangle}}\\
                                                           & A\{\langle S \rangle^{-1} \}
      }$$

      Given $x \in \langle S \rangle$, there exists some finite subset $S' \subseteq S$ such that $x = \prod S'$. Now, $\eta_S(x) = \prod \eta_S[S']$. Since $S'$ is a finite subset of $S$ we have $\eta_S[S'] \subseteq (A\{ S^{-1} \})^{\times}$. By the universal property of $\eta_{\langle S \rangle}: A \to A\{ \langle S \rangle^{-1} \}$, there exists a unique $\nu_{\langle S \rangle S}: A\{ \langle S \rangle^{-1}\} \to A\{ S^{-1} \}$ such that the following diagram commutes:
      $$\xymatrix{
      A \ar[r]^{\eta_{\langle S \rangle}} \ar[dr]_{\eta_S} & A\{ \langle S \rangle^{-1} \} \ar[d]^{\nu_{\langle S \rangle S}}\\
                & A\{ S^{-1} \}
      }$$
      Now,
      $$\xymatrix{
         & A \{\langle S \rangle^{-1}\} \ar[d]_{\nu_{\langle S \rangle S}} \ar@/^2pc/[dd]^{{\rm id}_{A\{ \langle S \rangle^{-1} \}} }\\
      A \ar@/^/[ur]^{\eta_{\langle S \rangle}} \ar[r]^{\eta_S} \ar@/_/[dr]_{\eta_{\langle S \rangle}} & A\{ S^{-1}\} \ar[d]_{\mu_{S \langle S \rangle}}\\
           & A \{ \langle S \rangle^{-1} \}
      }$$
      It follows that $\mu_{S\langle S \rangle} \circ \nu_{\langle S \rangle S} = {\rm id}_{A\{ \langle A \rangle^{-1} \}}$. Also, we have the following diagram:
      $$\xymatrix{
         & A \{ S^{-1}\} \ar[d]_{\mu_{S \langle S \rangle}} \ar@/^2pc/[dd]^{{\rm id}_{A\{ S^{-1} \}} }\\
      A \ar@/^/[ur]^{\eta_{S}} \ar[r]^{\eta_{\langle S \rangle}} \ar@/_/[dr]_{\eta_{S}} & A\{ \langle S \rangle^{-1}\} \ar[d]_{\nu_{\langle S \rangle S}}\\
           & A \{ S^{-1} \}
      }$$
      so $\nu_{\langle S \rangle S} \circ \mu_{S \langle S \rangle} = {\rm id}_{A\{ S^{-1}\}}$, and $\mu_{S \langle S \rangle} : A\{ S^{-1}\} \to A \{ \langle S \rangle^{-1} \}$ and $\nu_{\langle S \rangle S}: A\{ \langle S \rangle^{-1}\} \to A\{ S^{-1}\}$ are inverse isomorphisms to each other. We conclude, from \textbf{Proposition \ref{Alem}}, that:
      $$S^{\infty-{\rm sat}} = \langle S \rangle^{\infty-{\rm sat}}.$$
       }
\end{itemize}
\end{proof}

Some necessary and sufficient conditions for the $\mathcal{C}^{\infty}-$homomorphism $\eta_S : A \to A\{ S^{-1}\}$ be a $\mathcal{C}^{\infty}-$iso\-morphism is given below:\\

\begin{proposition}Let $A$ be a $\mathcal{C}^{\infty}-$ring and $S \subseteq A$ any of its subsets. The following assertions are equivalent:
\begin{itemize}
  \item[(i)]{$\eta_S : A \to A\{ S^{-1}\}$ is an isomorphism;}
  \item[(ii)]{$S^{\infty-{\rm sat}} \subseteq A^{\times}$;}
  \item[(iii)]{$S^{\infty-{\rm sat}} = A^{\times}$}
\end{itemize}
\end{proposition}
\begin{proof}

Ad (i) $\to$ (iii): Since $\eta_S$ is an isomorphism, both $\eta_S^{-1}$ and $\eta_S$ preserve the invertible elements, so $A^{\times} = \eta_S^{-1}[(A\{S^{-1}\})^{\times}]=\eta_S^{\dashv}[(A\{ S^{-1}\})^{\times}] = S^{{\rm sat}}$.\\

Ad (ii) $\leftrightarrow$ (iii): Since we always have $A^{\times} \subseteq S^{\infty-{\rm sat}}$, by (ii) we conclude that $A^{\times} = S^{\infty-{\rm sat}}$.

Ad (iii) $\to$ (i): Suppose that $S^{\infty-{\rm sat}} = A^{\times}$. We need to show $\eta_{A^{\times}} : A \to A\{ (A^{\times})^{-1}) \}$ is an isomorphism.\\

First we note that ${\rm id}_A : A \to A$ has the universal property of the ring of fractions of $A$ with respect to $A^{\times}$. Indeed, given any $\psi: A \to B$ such that $\psi[A^{\times}] \subseteq B^{\times}$ (i.e., any $\psi$ which is a $\mathcal{C}^{\infty}-$homomorphism), there exists a unique $\mathcal{C}^{\infty}-$homomorphism from $A$ to $B$, namely $\psi: A \to B$, such that the following diagram commutes:
$$\xymatrix{
A \ar[r]^{{\rm id}_A} \ar[rd]_{\psi} & A \ar[d]_{\psi}\\
           & B
}$$
It follows that $({\rm id}_A\, : A \to A) \cong (\eta_{A^{\times}} : A \to  A\{ (A^{\times})^{-1} \})$, since both satisfy the same universal property. Thus, $\eta_{A^{\times}}$ is the composition of a $\mathcal{C}^{\infty}-$isomorphism with ${\rm id}_A$, hence it is a $\mathcal{C}^{\infty}-$isomorphism.


Moreover, by \textbf{Proposition \ref{Alem}}, we conclude that $A^{\times} = \eta_{A^{\times}}[(A\{ (A^{\times})^{-1}\})^{\times}] = (A^{\times})^{\infty-{\rm sat}}$.
\end{proof}

Next we prove that the smooth saturation of the smooth saturation of a set is again the smooth saturation of this set.\\

\begin{proposition}\label{satsatsat}Let $A$ be a $\mathcal{C}^{\infty}-$ring and $S \subseteq A$ be any of its subsets. Then
$$(S^{\infty-{\rm sat}})^{\infty-{\rm sat}} = S^{\infty-{\rm sat}}$$
\end{proposition}
\begin{proof}
Since $S \subseteq S^{\infty-{\rm sat}}$, there exists a unique morphism $\mu_{SS^{\infty-{\rm sat}}}: A\{ S^{-1}\} \to A\{ (S^{\infty-{\rm sat}})^{-1} \}$ such that the following diagram commutes:
$$\xymatrix{
A \ar[r]^{\eta_S} \ar[dr]_{\eta_{S^{\infty-{\rm sat}}}} & A\{ S^{-1}\} \ar[d]^{\mu_{SS^{\infty-{\rm sat}}}}\\
            & A\{(S^{\infty-{\rm sat}})^{-1} \}
}$$

Now, $\eta_S[S^{\infty-{\rm sat}}] \subseteq (A\{ S^{-1}\})^{\times}$ by the very definition of $S^{\infty-{\rm sat}}$, so, by the universal property of $\eta_{S^{\infty-{\rm sat}}}: A \to A\{ (S^{\infty-{\rm sat}})^{-1} \}$, there exist a unique $\nu : A\{ (S^{\infty-{\rm sat}})^{-1}\} \to A\{ S^{-1}\}$ such that the following diagram commutes:

$$\xymatrix{
A \ar[dr]^{\eta_S} \ar[r]_{\eta_{S^{\infty-{\rm sat}}}} & A\{(S^{\infty-{\rm sat}})^{-1}\}  \ar[d]^{\nu}\\
            & A\{ S^{-1}\}
}$$

We have, then, the following commuting diagrams:

$$\begin{array}{cc}
\xymatrix{
          & A\{(S^{\infty-{\rm sat}})^{-1}\} \ar[d]_{\nu} \ar@/^3pc/[dd]^{{\rm id}_{A\{ (S^{\infty-{\rm sat}})^{-1} \}}} \\
A \ar[ur]^{\eta_{S^{\infty-{\rm sat}}}} \ar[r]^{\eta_S} \ar[dr]_{\eta_{S^{\infty-{\rm sat}}}} & A\{ S^{-1}\} \ar[d]_{\mu_{SS^{\infty-{\rm sat}}}}\\
          & A\{ (S^{\infty-{\rm sat}})^{-1} \}} & \xymatrix{
          & A\{S^{-1}\} \ar[d]_{\mu_{SS^{\infty-{\rm sat}}}} \ar@/^3pc/[dd]^{{\rm id}_{A\{ S^{-1} \}}} \\
A \ar[ur]^{\eta_{S}} \ar[r]^{\eta_{S^{\infty-{\rm sat}}}} \ar[dr]_{\eta_{S}} & A\{ (S^{\infty-{\rm sat}})^{-1}\} \ar[d]_{\nu}\\
          & A\{ S^{-1} \}}
          \end{array}$$

So $(\mu_{SS^{\infty-{\rm sat}}})^{-1} = \nu$, and $\mu_{SS^{\infty-{\rm sat}}}$ is an isomorphism. Hence $A\{ S^{-1}\} \cong A\{ (S^{\infty-{\rm sat}})^{-1}\}$, so by \textbf{Proposition \ref{Alem}},

$$(S^{\infty-{\rm sat}})^{\infty-{\rm sat}} = S^{\infty-{\rm sat}}$$
\end{proof}

\begin{proposition}\label{fact} Let $A$ be a $\mathcal{C}^{\infty}-$ring, $S \subseteq A$ be any of its subsets, and $\eta_S : A \to A\{ S^{-1}\}$ be the canonical morphism of the ring of fractions. If $g,h: A\{ S^{-1}\} \to B$ are two morphisms such that $g \circ \eta_S = h \circ \eta_S$ then $g=h$. In other words, $\eta_S : A \to A\{ S^{-1}\}$ is an epimorphism.
\end{proposition}
\begin{proof}
Note that since $g \circ \eta_S$ is such that $(g \circ \eta_S)[S] \subseteq B^{\times}$, there exists a unique morphism $\widetilde{t}: A\{ S^{-1}\} \to B$ such that $\widetilde{t} \circ \eta_S = g \circ \eta_S$. By hypothesis we have $h \circ \eta_S = g \circ \eta_S$, so  $g$ has the property which determines $\widetilde{t}$. Hence $g = \widetilde{t} = h$\\
\end{proof}

\begin{proposition}\label{zero}Let $A$ be a $\mathcal{C}^{\infty}-$ring and $S,T \subseteq A$ two of its subsets. The following assertions are equivalent:
\begin{itemize}
  \item[(i)]{$S^{\infty-{\rm sat}} \subseteq T^{\infty-{\rm sat}}$}
  \item[(ii)]{There is a unique morphism $\mu: A\{ S^{-1}\} \to A\{ T^{-1}\}$ such that the following diagram commutes:
      $$\xymatrix{
      A \ar[dr]^{\eta_T}\ar[r]^{\eta_S} & A\{ S^{-1}\} \ar[d]^{\mu}\\
          & A\{ T^{-1}\}
      }$$}
\end{itemize}
\end{proposition}
\begin{proof}
Ad (i) $\to$ (ii): Since $S^{\infty-{\rm sat}} \subseteq T^{\infty-{\rm sat}}$, by the universal property of $\eta_{S^{\infty-{\rm sat}}}: A \to A\{ (S^{\infty-{\rm sat}})^{-1} \}$ there exists a unique morphism $\mu_{S^{\infty-{\rm sat}}T^{\infty-{\rm sat}}}: A\{ (S^{\infty-{\rm sat}})^{-1}\} \to A\{(T^{\infty-{\rm sat}})^{-1} \}$ such that:

$$\xymatrix{ & A\{(S^{\infty-{\rm sat}})^{-1} \} \ar[dd]^{\mu_{S^{\infty-{\rm sat}}T^{\infty-{\rm sat}}}} \\
A \ar[ur]^{\eta_{S^{\infty-{\rm sat}}}} \ar[dr]^{\eta_{T^{\infty-{\rm sat}}}} &   \\
      & A\{ (T^{\infty-{\rm sat}})^{-1} \} }$$
 commutes.\\

  Since $S \subseteq S^{\infty-{\rm sat}}$ and $T \subseteq T^{\infty-{\rm sat}}$, there are unique isomorphisms $\mu_{SS^{\infty-{\rm sat}}}: A\{ S^{-1}\} \to A\{ {S^{\infty-{\rm sat}}}^{-1}\}$, $\mu_{TT^{\infty-{\rm sat}}}: A\{ T^{-1}\} \to A\{ {T^{\infty-{\rm sat}}}^{-1}\}$ such that the following diagram commutes:

$$\xymatrix{
  & A\{ S^{-1}\} \ar[r]^{\mu_{SS^{\infty-{\rm sat}}}}  & A\{ {S^{\infty-{\rm sat}}}^{-1} \} \ar[dd]^{\mu_{S^{\infty-{\rm sat}}T^{\infty-{\rm sat}}}} \\
A \ar@/^5pc/[urr]^{\eta_{S^{\infty-{\rm sat}}}} \ar@/_5pc/[drr]^{\eta_{T^{\infty-{\rm sat}}}} \ar[ur]^{\eta_S} \ar[dr]_{\eta_T} &      &     \\
       & A\{ T^{-1}\} \ar[r]^{\mu_{TT^{\infty-{\rm sat}}}} & A\{ {T^{\infty-{\rm sat}}}^{-1}\}}$$

Since the upper and the lower triangles commute, we have:

$$\mu_{S^{\infty-{\rm sat}}T^{\infty-{\rm sat}}} \circ \eta_{S^{\infty-{\rm sat}}} = \eta_{T^{\infty-{\rm sat}}}.$$

We define $\mu := (\mu_{TT^{\infty-{\rm sat}}})^{-1}\circ \mu_{S^{\infty-{\rm sat}}T^{\infty-{\rm sat}}} \circ \mu_{SS^{\infty-{\rm sat}}}$. We claim that $\eta_T = \mu \circ \eta_S$.\\

On the one hand we have $\mu_{TT^{\infty-{\rm sat}}} \circ \eta_T = \eta_{T^{\infty-{\rm sat}}}$, and by the very definition of $\mu$:

$$(\mu_{TT^{\infty-{\rm sat}}} \circ \mu)\circ \eta_S = (\mu_{S^{\infty-{\rm sat}}T^{\infty-{\rm sat}}} \circ \mu_{SS^{\infty-{\rm sat}}}) \circ \eta_S$$

We also have:
$$\mu_{S^{\infty-{\rm sat}}T^{\infty-{\rm sat}}} \circ (\mu_{SS^{\infty-{\rm sat}}} \circ \eta_S) = \mu_{TT^{\infty-{\rm sat}}} \circ \eta_T = \eta_{T^{\infty-{\rm sat}}}$$
hence:

$$\mu_{TT^{\infty-{\rm sat}}} \circ \eta_T = \eta_{T^{\infty-{\rm sat}}} = \mu_{TT^{\infty-{\rm sat}}}\circ (\mu \circ \eta_S),$$
and since $\mu_{TT^{\infty-{\rm sat}}}$ is an isomorphism, it follows that:
$$\eta_T = \mu \circ \eta_S.$$

In order to show that $\mu$ is unique, let $\nu: A\{S^{-1}\} \to A\{ T^{-1}\}$ be a homomorphism such that the following diagram commutes:

$$\xymatrix{
  & A\{ S^{-1}\} \ar[dd]^{\nu}\\
A \ar[ur]^{\eta_S} \ar[dr]_{\eta_T} & \\
   & A\{ T^{-1}\}
}$$

By composing both sides of the equation:
$$\nu \circ \eta_S = \eta_T$$
with $\mu_{TT^{\infty-{\rm sat}}}$ we get:
$$\mu_{TT^{\infty-{\rm sat}}} \circ \nu \circ \eta_S = \mu_{TT^{\infty-{\rm sat}}} \circ \eta_T.$$

We know that $\mu_{TT^{\infty-{\rm sat}}}\circ \eta_T = \eta_{T^{\infty-{\rm sat}}}$, so
$$\eta_{T^{\infty-{\rm sat}}} = \mu_{S^{\infty-{\rm sat}}T^{\infty-{\rm sat}}}\circ \eta_{S^{{\rm sat}}} = \mu_{S^{\infty-{\rm sat}}T^{{\rm sat}}} \circ \mu_{SS^{\infty-{\rm sat}}} \circ \eta_S,$$
so
$$\mu_{TT^{\infty-{\rm sat}}} \circ \nu \circ \eta_S = \mu_{S^{\infty-{\rm sat}}T^{\infty-{\rm sat}}} \circ \mu_{SS^{\infty-{\rm sat}}} \circ \eta_S$$

By \textbf{Proposition \ref{fact}}, we conclude:
$$\mu_{TT^{\infty-{\rm sat}}} \circ \nu = \mu_{S^{\infty-{\rm sat}}T^{\infty-{\rm sat}}} \circ \mu_{SS^{\infty-{\rm sat}}}.$$

Composing both sides of the above equation with $(\mu_{TT^{\infty-{\rm sat}}})^{-1}$ by the left we get:
$$\nu = (\mu_{TT^{\infty-{\rm sat}}})^{-1}\circ \mu_{S^{\infty-{\rm sat}}T^{{\rm sat}}} \circ \mu_{SS^{\infty-{\rm sat}}} = \mu$$

\end{proof}

\begin{corollary}\label{cem}The following assertions are equivalent:
\begin{itemize}
  \item[(i)]{$S^{\infty-{\rm sat}} = T^{\infty-{\rm sat}}$}
  \item[(ii)]{There is an isomorphism $\mu: A\{ S^{-1}\} \to A\{ T^{-1}\}$ such that the following diagram commutes:
      $$\xymatrix{
      A \ar[r]^{\eta_S} \ar[dr]^{\eta_T} & A\{ S^{-1}\} \ar[d]^{\mu}\\
         & A\{ T^{-1}\}
      }$$}
  \item[(iii)]{There is a unique isomorphism $\mu: A\{ S^{-1}\} \to A\{ T^{-1}\}$ such that the following diagram commutes:
      $$\xymatrix{
      A \ar[r]^{\eta_S} \ar[dr]^{\eta_T} & A\{ S^{-1}\} \ar[d]^{\mu}\\
         & A\{ T^{-1}\}
      }$$}
\end{itemize}
\end{corollary}
\begin{proof}
Ad (i) $\to$ (iii): Since $S^{\infty-{\rm sat}} = T^{\infty-{\rm sat}}$, we have both $S^{\infty-{\rm sat}} \subseteq T^{\infty-{\rm sat}}$, so by \textbf{Proposition \ref{zero}} there is a unique isomorphism $\nu: A\{ S^{-1}\} \to A\{ T^{-1}\}$ such that the following diagram commutes

$$\xymatrix{
      A \ar[r]^{\eta_S} \ar[dr]^{\eta_T} & A\{ S^{-1}\} \ar[d]^{\nu}\\
         & A\{ T^{-1}\}
      }$$

and $ T^{\infty-{\rm sat}} \subseteq S^{\infty-{\rm sat}}$ so by \textbf{Proposition \ref{zero}} there is a unique homomorphism $\psi: A\{ S^{-1}\} \to A\{ T^{-1}\}$ such that the following diagram commutes

$$\xymatrix{
      A \ar[r]^{\eta_T} \ar[dr]^{\eta_S} & A\{ T^{-1}\} \ar[d]^{\psi}\\
         & A\{ S^{-1}\}
      }.$$

We have the following commutative diagrams:

$$\xymatrix{
    & A\{ S^{-1}\} \ar[d]^{\nu} \ar@/^3pc/[dd]^{{\rm id}_{A\{ S^{-1}\}}}\\
A \ar[ur]^{\eta_S} \ar[r]^{\eta_T} \ar[dr]^{\eta_S} & A\{ T^{-1}\} \ar[d]^{\psi}\\
   & A\{ S^{-1}\}
}$$

$$\xymatrix{
    & A\{ T^{-1}\} \ar[d]^{\psi} \ar@/^3pc/[dd]^{{\rm id}_{A\{ T^{-1}\}}}\\
A \ar[ur]^{\eta_T} \ar[r]^{\eta_S} \ar[dr]^{\eta_T} & A\{S^{-1}\} \ar[d]^{\nu}\\
   & A\{ T^{-1}\}
}$$

It follows that since $\eta_T$ and $\eta_S$ are epimorphisms and  we have:
$$\nu \circ \psi \circ \eta_T = {\rm id}_{A\{ T^{-1}\}} \circ \eta_T \Rightarrow \nu \circ \psi = {\rm id}_{A\{ T^{-1}\}}$$
$$\psi \circ \nu \circ \eta_S = {\rm id}_{A\{ S^{-1}\}} \circ \eta_S \Rightarrow \psi \circ \nu = {\rm id}_{A\{S^{-1}\}}$$
so $\nu$ and $\psi$ are inverse isomorphisms to each other. We can take, then $\mu = \nu$.\\

Ad (ii) $\to$ (i). By \textbf{Proposition \ref{zero}}, since $\mu \circ \eta_S = \eta_T$ we have $S^{\infty-{\rm sat}} \subseteq T^{\infty-{\rm sat}}$ and since $\mu^{-1} \circ \eta_T = \eta_S $ we have $T^{\infty-{\rm sat}} \subseteq S^{\infty-{\rm sat}}$. Hence $S^{\infty-{\rm sat}} = T^{\infty-{\rm sat}}$.\\
\end{proof}

\begin{proposition}Let $A$ be a $\mathcal{C}^{\infty}-$ring and $S,T$ two of its subsets such that $S \subseteq T$. The following assertions are equivalent:
\begin{itemize}
  \item[(i)]{$\mu_{ST}: A\{ S^{-1}\} \to A\{ T^{-1}\}$ is an isomorphism;}
  \item[(ii)]{$S \subseteq T \subseteq S^{\infty-{\rm sat}}$;}
  \item[(iii)]{$T^{\infty-{\rm sat}} = S^{\infty-{\rm sat}}$.}
\end{itemize}
\end{proposition}
\begin{proof}
Ad (ii) $\to$ (iii): Since $S \subseteq T$ we have $S^{\infty-{\rm sat}} \subseteq T^{\infty-{\rm sat}}$, and since $T \subseteq S^{\infty-{\rm sat}}$ we have $T^{\infty-{\rm sat}} \subseteq (S^{\infty-{\rm sat}})^{\infty-{\rm sat}} = S^{\infty-{\rm sat}}$. Hence:
$$S^{\infty-{\rm sat}} \subseteq T^{\infty-{\rm sat}} \subseteq S^{\infty-{\rm sat}}$$
and
$$S^{\infty-{\rm sat}} = T^{\infty-{\rm sat}}.$$

Ad (iii) $\to$ (ii): We always have $T \subseteq T^{{\rm sat}}$, and since $T^{{\rm sat}} = S^{{\rm sat}}$, it follows that $T \subseteq S^{{\rm sat}}$.\\

(i) $\leftrightarrow$ (iii) was established in \textbf{Corollary \ref{cem}}.\\

\end{proof}

\begin{proposition}\label{marina}Let $A$ be a $\mathcal{C}^{\infty}-$ring and $S \subseteq A$. Whenever $\{ S_i\}_{i \in I}$ is a directed system such that:
$$S = \bigcup_{i \in I} S_i$$
we have:
$$S^{\infty-{\rm sat}} = \left( \bigcup_{i \in I} S_i\right)^{\infty-{\rm sat}} = \bigcup_{i \in I} {S_i}^{\infty-{\rm sat}}$$
\end{proposition}
\begin{proof}It is clear that:
$$\bigcup_{i \in I} {S_i}^{\infty-{\rm sat}} \subseteq S^{\infty-{\rm sat}}.$$

In order to prove the other inclusion, we shall use the fact that $A\{S^{-1}\}$ is isomorphic to the vertex of the following directed colimit:

$$\xymatrix{
  & \varinjlim_{i \in I} A\{ {S_i}^{-1}\} &  \\
A\{{S_i}^{-1} \} \ar[rr]^{\alpha_{ij}} \ar@/^/[ur]^{\alpha_i} & & A\{{S_j}^{-1} \} \ar@/_/[ul]_{\alpha_j}
}$$

Note that $\eta_S : A \to A\{ S^{-1}\}$ is such that for any $i \in I$, $\eta_S[S_i] \subseteq \eta_S[S] \subseteq (A\{ S^{-1}\})^{\times}$, so by the universal property of $\eta_{S_i} : A \to A\{ (S_i)^{-1}\}$, there is a unique $\mathcal{C}^{\infty}-$rings homomorphism $\varphi_i: A\{ (S_i)^{-1}\} \to A\{S^{-1}\}$ such that the following triangle commutes:
$$\xymatrix{
A \ar[r]^{\eta_{S_i}} \ar[dr]^{\eta_S} & A\{(S_i)^{-1} \} \ar[d]^{\varphi_i}\\
      & A\{ S^{-1}\}}$$
so
$$\xymatrix{
  & A\{ S^{-1}\} &  \\
A\{{S_i}^{-1} \} \ar[rr]^{\varphi_{ij}} \ar@/^/[ur]^{\varphi_i} & & A\{{S_j}^{-1} \} \ar@/_/[ul]_{\varphi_j}
}$$
commutes for every $i,j \in I$ such that $i \leq j$, since $\eta_{S_i}: A \to A\{ {S_i}^{-1}\}$ is an epimorphism.\\

Thus, by the universal property of the colimit, there is a unique $\mathcal{C}^{\infty}-$homomorphism:

$$\varphi: \varinjlim_{i \in I} A\{ {S_i}^{-1}\} \rightarrow A\{ S^{-1}\}$$

such that:
 $$(\forall i \in I)(\varphi \circ \alpha_i = \varphi_i).$$

On the other hand, given $s \in S = \bigcup_{i \in I}S_i$ there is $i \in I$ such that $s \in S_i$, so $\eta_{S_i}(s) \in A\{ {S_i}^{-1}\}^{\times}$. Taking $\widetilde{\eta}:= \alpha_i \circ \eta_{S_i}: A \rightarrow \varinjlim_{i \in I} A\{ {S_i}^{-1}\}$, we have:

$$\widetilde{\eta}(s) \in \left( \varinjlim_{i \in I} A\{ {S_i}^{-1}\} \right)^{\times}.$$

By the universal property of $\eta_S: A \rightarrow A\{ S^{-1}\}$, there is a unique $\mathcal{C}^{\infty}-$homo\-morphism $\psi: A\{ S^{-1}\} \rightarrow \varinjlim_{i \in I} A\{ {S_i}^{-1}\}$ such that:

$$\xymatrix{
A \ar[r]^{\eta_{S}} \ar[dr]^{\widetilde{\eta}} & A\{S^{-1} \} \ar[d]^{\psi}\\
      & \varinjlim_{i \in I} A\{ {S_i}^{-1}\}}$$

It is easy to see, by the universal properties involved, that $\varphi$ and $\psi$ are inverse $\mathcal{C}^{\infty}-$isomorphism and that:

$$\widetilde{\eta}^{\dashv}\left[ \left( \varinjlim_{i \in I} A\{ {S_i}^{-1}\}\right)^{\times}\right] = \bigcup_{i \in I} {S_i}^{\infty-{\rm sat}}.$$

Thus $S=\bigcup_{i \in I} {S_i}^{\infty-{\rm sat}}$, as we claimed.
\end{proof}






\begin{proposition}\label{lom}Let $A$ be a $\mathcal{C}^{\infty}-$ring. The universal solution to freely adjoining  a variable $x$ to $A$ in the category $\mathcal{C}^{\infty}{\rm {\bf Rng}}$ is isomorphic to the smooth coproduct of $A$ and the free object in one generator of this category:
$$A\{ x \} \cong A \otimes_{\infty} \mathcal{C}^{\infty}(\mathbb{R})$$
\end{proposition}
\begin{proof}Let's first consider the inclusions $\jmath_2: \{ x\} \to A\{ x\}$ and $\eta : \{ x\} \to \mathcal{C}^{\infty}(\mathbb{R})$, this latter given by the definition of free object in one generator. We are going to use the universal property of this free object together with the universal property of the smooth coproduct to show that $A \otimes_{\infty} \mathcal{C}^{\infty}(\mathbb{R})$ has the desired property.\\

In order to show that $A\{x\} \cong A \otimes_{\infty} \mathcal{C}^{\infty}(\mathbb{R})$ we use the universal property of the coproduct and of the localization, one at a time.

$$\xymatrixcolsep{5pc}\xymatrix{
\{x\} \ar[d]^{\jmath_1} \ar[rd]^{\jmath_2}& & \\
\mathcal{C}^{\infty}(\mathbb{R}) \ar@{-->}[r]^{\exists ! \eta} \ar[dr]^{\iota_1} & A\{ x\} & \ar[l]_{\imath} \ar[ld]^{\iota_2} A\\
     & A \otimes_{\infty} \mathcal{C}^{\infty}(\mathbb{R}) \ar@{-->}[u] &  \\
}$$

Since $\mathcal{C}^{\infty}(\mathbb{R})$ is the free object in one generator, $x$, given the inclusion $\jmath_2: \{x\} \to A\{ x \}$ there is a unique $\mathcal{C}^{\infty}-$ring homomorphism $\eta: \mathcal{C}^{\infty}(\mathbb{R}) \to A\{ x\}$ such that the upper left triangle of the above diagram commutes, i.e., such that $\eta \circ \jmath_1 = \jmath_2$. We also have the inclusion map $\imath: A \to A\{ x \}$ in the very definition of $A\{x\}$.\\

Now, by the universal property of the coproduct $\mathcal{C}^{\infty}(\mathbb{R}) \to A \otimes_{\infty} \mathcal{C}^{\infty}(\mathbb{R}) \leftarrow A$, given the $\mathcal{C}^{\infty}-$rings homomorphisms $\eta: \mathcal{C}^{\infty}(\mathbb{R}) \to A\{x\}$ and $ \imath: A \to A\{x\}$ there exists a unique $\mathcal{C}^{\infty}-$homomorphism $\varphi: A \otimes_{\infty} \mathcal{C}^{\infty}(\mathbb{R})$ such that the following diagram commutes:

$$\begin{xy}
  \xymatrix{
        &  \mathcal{C}^{\infty}(\mathbb{R}) \ar[d]^{\iota_1} \ar@/^2pc/[ddr]^{\eta}  &  \\
      A \ar[r]^{\iota_2} \ar@/_2pc/[drr]_{\imath}  &  A \otimes_{\infty} \mathcal{C}^{\infty}(\mathbb{R}) \ar@{-->}[dr]^{\exists ! \varphi}  &  \\
      &  &  A\{ x\}
  }
\end{xy}$$

Conversely, since $\jmath_2: \{x\} \to A\{ x \}$ is the solution to the universal problem of freely adjoining one element $t$ to $A$ and $\mathcal{C}^{\infty}(\mathbb{R})$ is the free object on one generator, there exists a unique $\ell : \mathcal{C}^{\infty}(\mathbb{R}) \to A\{t\}$ such that:
$$\xymatrix{
\{x\} \ar[rd]^{\jmath_1} \ar[d]^{\jmath_2} & \\
\mathcal{C}^{\infty}(\mathbb{R}) \ar@{-->}[r]^{\exists ! \ell} & A\{ x\}
}$$

commutes. We will now consider the following diagram:

$$\xymatrix{
 &\{ x \} \ar[d]^{\jmath_1} \ar@/_3pc/[dd]_{\iota_1 \circ \jmath_1} \ar[rd]^{\jmath_2} & \\
 & \mathcal{C}^{\infty}(\mathbb{R}) \ar@{-->}[r]^{\exists ! \ell} \ar[d]^{\iota_1} & A\{x\} \\
   & A \otimes_{\infty} \mathcal{C}^{\infty}(\mathbb{R}) &
}$$

Now, given the $\mathcal{C}^{\infty}-$ring homomorphism $\iota_1 \circ \jmath_1 : \{x\} \to A\otimes_{\infty} \mathcal{C}^{\infty}(\mathbb{R})$, by the universal property of $\jmath_2 : \{x\} \to A\{x\}$, there exists a unique $\mathcal{C}^{\infty}-$ring homomorphism $\varrho : A\{x\} \to A \otimes_{\infty} \mathcal{C}^{\infty}(\mathbb{R})$ such that the diagram below commutes:

$$\xymatrix{
\{x\} \ar[r]^{\jmath_2} \ar[rd]^{\iota_1 \circ \jmath_1}& A\{x\} \ar@{-->}[d]^{\exists ! \varrho}\\
    & A \otimes_{\infty} \mathcal{C}^{\infty}(\mathbb{R})
}.$$

We have, then, the following diagram:

$$\xymatrix{
 &\{ x \} \ar[d]^{\jmath_1} \ar@/_3pc/[dd]_{\iota_1 \circ \jmath_1} \ar[rd]^{\jmath_2} & \\
 & \mathcal{C}^{\infty}(\mathbb{R}) \ar[r]^{\exists ! \ell} \ar[d]^{\iota_1} & A\{x\} \ar[ld]^{\varrho}\\
   & A \otimes_{\infty} \mathcal{C}^{\infty}(\mathbb{R}) &
}$$

It can be proved that $\varrho \circ \varphi = {\rm id}_{A \otimes_{\infty} \mathcal{C}^{\infty}(\mathbb{R})}$ and $\varphi \circ \varrho = {\rm id}_{A\{x\}}$, hence:
$$A\{x\} \cong A \otimes_{\infty} \mathcal{C}^{\infty}(\mathbb{R}).$$
\end{proof}











Our next goal is to give a characterization of ring of fractions in $\mathcal{C}^{\infty}{\rm \bf Rng}$ using a similar axiomatization one has in Commutative Algebra. In order to motivate it, we first present some characterizations of rings of fractions in ${\rm \bf CRing}$.\\

\begin{theorem}Let $A$ be a commutative ring with unity, $S \subseteq A$ a multiplicative subset, and $\eta: A \to A[S^{-1}]$ its localization. Then:
\begin{itemize}
\item[(i)]{$(\forall \beta \in A[S^{-1}])(\exists c \in S)(\exists d \in A)(\beta \cdot \eta(c) = \eta(d))$;}
\item[(ii)]{$(\forall \beta \in A)(\eta(\beta)=0 \to (\exists c \in S)(c \cdot \beta = 0)$}
\end{itemize}
\end{theorem}
\begin{proof}Ad $(i)$: Given any $\beta \in A[S^{-1}]$ we can write $\beta =\dfrac{a}{s}$ for some $a \in A$ and $s \in S$. Take $c=s$ and $d=a$, so we have:
$$\beta \cdot \eta(c) = \dfrac{a}{s} \cdot \eta(c) = \dfrac{a}{s}\cdot \dfrac{s}{1} = \dfrac{a}{1} = \eta(a),$$

so it suffices to take $d = a$.\\

Ad $(ii)$: Given any $\beta \in A$ such that $\eta(b)=0$ we have, by the very construction of $A[S^{-1}]$, that $(\exists u \in S)(u(\beta\cdot 1 - 0 \cdot 1) = 0)$, so it is enough to take $c=u$ in order to get $c \cdot \beta = u \cdot \beta = 0$.
\end{proof}

Conversely we have:

\begin{proposition}Let $A$ be a commutative ring with unity and $S \subseteq A$. If $\varphi : A \to B$ is a ring homomorphism such that $\varphi[A] \subseteq B$ and:
\begin{itemize}
\item[(i)]{$(\forall \beta \in B)(\exists c \in S)(\exists d \in A)(\beta \cdot \varphi(c) = \varphi(d))$}
\item[(ii)]{$(\forall \beta \in A)(\varphi(\beta)=0 \to (\exists c \in S)(c \cdot \beta = 0))$}
\end{itemize}
then $B \cong A[S^{-1}]$.
\end{proposition}
\begin{proof}Since $\varphi: A \to B$ is such that $\varphi[A] \subseteq B$, by the universal property of the localization there is a unique $\psi: A[S^{-1}] \to B$ such that the diagram below commutes:
$$\xymatrix{
A \ar[r]^{\eta} \ar[rd]^{\varphi} & A[S^{-1}] \ar@{-->}[d]^{\exists ! \psi}\\
  & B
  }$$

Now we claim that $\psi$ is a surjection. Given any $z \in B$, since $\varphi: A \to B$ satisfies $(i)$, there are elements $c \in S$ and $d \in A$ such that $z \cdot \varphi(c) = \varphi(d)$, that is to say $z = \varphi(d)\cdot (\varphi(c))^{-1}$. Taking $w = \eta(d)\cdot (\eta(c))^{-1} \in A[S^{-1}]$ we get $\psi(w) = \psi(\eta(d)\cdot (\eta(c))^{-1}) = \psi(\eta(d))\cdot \psi(\eta(c))^{-1} = \varphi(d)\cdot (\varphi(c))^{-1} = z$.\\

We claim, also, that $\psi$ is an injective ring homomorphism. Let $w \in A[S^{-1}]$ be an element such that $\psi(w)=0$. Since $w \in A[S^{-1}]$ there are $c \in S$ and $d \in A$ such that $w = \eta(d)\cdot (\eta(c))^{-1}$. The condition $\psi(w)=0$ means that $\psi(\eta(d)\cdot (\eta(c))^{-1})= \psi(\eta(d))\cdot \psi((\eta(c)))^{-1} = 0$, so $\psi(\eta(d))=0$, and since $\psi \circ \eta = \varphi$, it follows that $\varphi(d)=0$. By the property $(ii)$, there exists $c' \in S$ such that $c' \cdot d = 0$, so $\eta(c')\cdot \eta(d) = \eta(c' \cdot d) =  0$, and since $\eta(c') \in A[S^{-1}]^{\times}$, it follows that $\eta(d)=0$. Thus $w = \eta(d)\cdot (\eta(c'))^{-1} = 0\cdot (\eta(c'))^{-1} = 0$.\\

It follows that $\psi$ is an isomorphism between $B$ and $A[S^{-1}]$.
\end{proof}

The two preceding theorems gives us the following:

\begin{theorem}\label{37}Let $A$ be a commutative ring with unity and $S \subseteq A$. Then $\varphi: A \to B$ is isomorphic to the localization $\eta: A \to A[S^{-1}]$ if and only if:
\begin{itemize}
\item[(i)]{$(\forall b \in B)(\exists c \in S)(\exists d \in A)(b \cdot \varphi(c) = \varphi(d))$}
\item[(ii)]{$(\forall b \in A)(\varphi(b)=0 \to (\exists c \in S)(c \cdot b = 0))$}
\end{itemize}
hold.
\end{theorem}

We have, thus, obtained a characterization of the localization of a commutative ring. For $\mathcal{C}^{\infty}-$rings we have the analogous result, that generalizes \textbf{Theorem 1.4} of \cite{moerdijk1986rings}, in the sense that it is an equivalence (an ``if and only if'' statement) and that  $S$ needs not to be a singleton:\\

\begin{theorem}\label{38}Let $A$ be a $\mathcal{C}^{\infty}-$ring $\Sigma \subset A$ a set. Then $\varphi: A \to B$ is isomorphic to the smooth localization $\eta: A \to A\{\Sigma^{-1}\}$ if and only if:
\begin{itemize}
\item[(i)]{$(\forall b \in B)(\exists c \in \Sigma^{\infty-{\rm sat}})(\exists d \in A)(b \cdot \varphi(c) = \varphi(d))$}
\item[(ii)]{$(\forall b \in A)(\varphi(b)=0 \to (\exists c \in \Sigma^{\infty-{\rm sat}})(c \cdot b = 0))$}
\end{itemize}
hold.
\end{theorem}

We postpone the proof of this theorem, giving it right after \textbf{Remark \ref{raca}}.\\

\begin{theorem}\label{340}Let $A$, $\widetilde{A}$ be $\mathcal{C}^{\infty}-$rings and let $\eta : A \to \widetilde{A}$ be a $\mathcal{C}^{\infty}-$rings homomorphism such that:
\begin{itemize}
  \item[(i)]{$(\forall d \in \widetilde{A})(\exists b \in A)(\exists c \in A)(\eta(c) \in \widetilde{A}^{\times} \& (d \cdot \eta(c) = \eta(b)))$;}
  \item[(ii)]{$(\forall b \in A)((\eta(b) = 0_{\widetilde{A}}) \to (\exists c \in A)((\eta(c) \in \widetilde{A}^{\times}) \wedge (b \cdot c = 0_A)))$}
\end{itemize}
Then $\eta : A \to \widetilde{A}$ is isomorphic to ${\rm Can}_{S_{\eta}} : A \to A\{ {S_{\eta}}^{-1}\}$, where $S_{\eta} = \eta^{\dashv}[\widetilde{A}^{\times}]$.
\end{theorem}
\begin{proof}
First we show that $\eta: A \to \widetilde{A}$ has the universal property which characterize ${\rm Can}_{S_{\eta}}$.\\

Let $f: A \to B$ be a $\mathcal{C}^{\infty}-$rings homomorphism such that $f[S_{\eta}] \subseteq B^{\times}$. We are going to show there is a unique $\mathcal{C}^{\infty}-$rings homomorphism $\widetilde{f} : \widetilde{A} \to B$ such that the following diagram commutes:
$$\xymatrix{
  A \ar[r]^{\eta} \ar[rd]_{f} & \widetilde{A} \ar@{.>}[d]^{\widetilde{f}}\\
      & B
}$$
i.e., such that $\widetilde{f}\circ \eta = f$.\\

First we notice that:
$$\eta[S_{\eta}] = \eta[\eta^{\dashv}[\widetilde{A}^{\times}]] \subseteq \widetilde{A}^{\times}.$$

\textbf{Candidate and Uniqueness of $\widetilde{f}$:} Let $\widetilde{f}_1, \widetilde{f}_2: \widetilde{A} \to B$ be two $\mathcal{C}^{\infty}-$rings homomorphisms such that the following diagram commutes:

$$\xymatrix{
  A \ar[r]^{\eta} \ar@/_/[dr]_f & \widetilde{A} \ar[d]^{\widetilde{f}_1} \ar@<-1ex>[d]_{\widetilde{f}_2} & \ar[l]_{\eta} A \ar@/^/[dl]^{f}\\
      & B &
}$$
that is to say, such that $\widetilde{f}_1 \circ \eta = f = \widetilde{f}_2 \circ \eta$.\\

Given any $d \in \widetilde{A}$, by the hypothesis (i) there exist $b,c \in A$, $\eta(c) \in \widetilde{A}^{\times}$ such that $d = \dfrac{\eta(b)}{\eta(c)}$, so:
\begin{multline*}\widetilde{f}_1(d) = \widetilde{f}_1(\eta(b)\cdot \eta(c)^{-1}) = \widetilde{f}_1(\eta(b))\cdot \widetilde{f}_1(\eta(c))^{-1}= \\
 = (\widetilde{f}_1 \circ \eta)(b) \cdot ((\widetilde{f}_1 \circ \eta)(c))^{-1} = f(b)\cdot f(c)^{-1} = (\widetilde{f}_2 \circ \eta)(b) \cdot ((\widetilde{f_2} \circ \eta)(c))^{-1} =\\ =\widetilde{f}_2(\eta(b))\cdot \widetilde{f}_2(\eta(c)^{-1}) = \widetilde{f}_2(\eta(b)\cdot\eta(c)^{-1}) = \widetilde{f}_2(d)
\end{multline*}
so we conclude that $(\forall d \in \widetilde{A})(\widetilde{f}_1(d) = \widetilde{f}_2(d))$ and $\widetilde{f}(\frac{\eta(b)}{\eta(c)}) = f(b)\cdot f(c)^{-1}$. Thus $\widetilde{f}_1 = \widetilde{f}_2$.\\

\textbf{Existence of $\widetilde{f}$:} We know that for every $d \in \widetilde{A}$ there exist $b,c \in A$, $\eta(c) \in \widetilde{A}^{\times}$, such that $ d = \eta(b)\cdot \eta(c)^{-1}$. Define the following relation: $\widetilde{f} = \{ (d, f(b)\cdot f(c)^{-1}) | d \in \widetilde{A} \} \subseteq \widetilde{A} \times B$. We claim that $\widetilde{f} $ is a function.\\

\textbf{Claim:} $\widetilde{f}$ is a univocal relation.\\

Given $d \in \widetilde{A}$, let $b,c,b',c' \in A$, $\eta(c), \eta(c') \in \widetilde{A}^{\times}$ be such that $\eta(b)\cdot \eta(c)^{-1} = d = \eta(b')\cdot \eta(c')^{-1}$, so  $\eta(b)\cdot \eta(c') = \eta(b')\cdot \eta(c)$. Then we have $\eta(b \cdot c') = \eta(b) \cdot \eta(c') = \eta(b')\cdot \eta(c)$, so $\eta(b \cdot c' - b' \cdot c)= 0$. By (ii), since $\eta(b \cdot c' - b' \cdot c)= 0$, there is some $a \in A$ such that $\eta(c) \in \widetilde{A}^{\times}$, i.e., $a \in \eta^{\dashv}[\widetilde{A}^{\times}] = S_{\eta}$, and $a \cdot (b\cdot c' - b' \cdot c) = 0$, i.e., $b\cdot c' \cdot a = b' \cdot c \cdot a$. We now have:

$$ f(b) \cdot f(c') \cdot f(a) = f(b \cdot c' \cdot a) = f(b' \cdot c \cdot a) = f(b') \cdot f(c) \cdot f(a).$$

Now since $f[S_{\eta}] \subseteq B^{\times}$ and $a \in S_{\eta}$, from the preceding equations we obtain, by cancelling $f(a)$:
$$f(b)\cdot f(c') = f(b')\cdot f(c)$$
and since $\eta(c), \eta(c') \in \widetilde{A}^{\times}$, $c, c' \in \eta^{\dashv}[\widetilde{A}^{\times}] = S_{\eta}$ so $f(c), f(c') \in B^{\times}$, so:
$$f(b)\cdot f(c)^{-1} = f(b')\cdot f(c')^{-1}.$$

We conclude that $(\forall d \in \widetilde{A})(((d, f(b)\cdot f(c)^{-1}) \in \widetilde{f} ) \wedge ((d, f(b')\cdot f(c')^{-1}) \in \widetilde{f} ) \to f(b) \cdot f(c)^{-1} = f(b') \cdot f(c')^{-1})$.\\

\textbf{Claim:} $\widetilde{f} $ is a total relation.\\

This follows immediately from item (i).\\

We are going to denote $ (d, f(b)\cdot f(c)^{-1}) \in \widetilde{f} $ simply by $\widetilde{f} (d) = f(b)\cdot f(c)^{-1}$, as usually we do for functions.\\

Therefore, there exists exactly one function $\widetilde{f}: \widetilde{A} \to B$ such that $\widetilde{f} \circ \eta = f$. \\

Now we show that $\widetilde{A} \cong A\{ S_{\eta}^{-1}\}$. We begin by noticing that ${\rm Can}_{S_{\eta}}[S_{\eta}] \subseteq A\{ S_{\eta}^{-1}\}^{\times}$ by the very definition of ${\rm Can}_{S_{\eta}}$. For what we have seen above, there exists a unique function $\widetilde{{\rm Can}_{S_{\eta}}}: \widetilde{A} \to A\{ S_{\eta}^{-1} \}$ such that $\widetilde{{\rm Can}_{S_{\eta}}}\circ \eta = {\rm Can}_{S_{\eta}}$. Now, from the universal property of ${\rm Can}_{S_{\eta}}$ there exists a unique $\mathcal{C}^{\infty}-$rings homomorphism:
$$\widehat{\eta}: A\{ S_{\eta}^{-1}\} \to \widetilde{A}$$
such that $\widehat{\eta}\circ {\rm Can}_{S_{\eta}} = \eta$.\\

\textbf{Claim:} $\widehat{\eta}$ is a bijection whose inverse is $\widetilde{{\rm Can}_{S_{\eta}}}$, and that will prove that $\widetilde{{\rm Can}_{S_{\eta}}}$ is a $\mathcal{C}^{\infty}-$rings isomorphism.\\

Now, $(\widehat{\eta}\circ \widetilde{{\rm Can}_{S_{\eta}}}) \circ \eta = \widehat{\eta} \circ (\widetilde{{\rm Can}_{S_{\eta}}} \circ \eta) = \widehat{\eta} \circ {\rm Can}_{S_{\eta}} = \eta = {\rm id}_{\widetilde{A}} \circ \eta$, so:
$$(\widehat{\eta}\circ \widetilde{{\rm Can}_{S_{\eta}}}) \circ \eta = {\rm id}_{\widetilde{A}} \circ \eta.$$

We have seen, however, that there is exactly one function $\widetilde{\varphi}$ such that $\widetilde{\varphi} \circ \eta = \eta$, so it follows that ${\rm id}_{\widetilde{A}} = \widehat{\eta} \circ \widetilde{{\rm Can}_{S_{\eta}}}$.\\

On the other hand,

$$\widetilde{{\rm Can}_{S_{\eta}}} \circ \widehat{\eta} \circ {\rm Can}_{S_{\eta}} = \widetilde{{\rm Can}_{S_{\eta}}} \circ \eta = {\rm Can}_{S_{\eta}} = {\rm id}_{A\{ S_{\eta}^{-1} \}} \circ {\rm Can}_{S_{\eta}}.$$

Once again, by the universal property of ${\rm Can}_{S_{\eta}}$ we have:

$${\rm id}_{A\{S_{\eta}^{-1}\}} = \widetilde{{\rm Can}_{S_{\eta}}} \circ \widehat{\eta}.$$

Hence $\widetilde{{\rm Can}_{S_{\eta}}}$ is the $\mathcal{C}^{\infty}-$rings isomorphism between $\eta$ and ${\rm Can}_{S_{\eta}}$, that is, it is a $\mathcal{C}^{\infty}-$rings isomorphism such that the following diagram commutes:
  $$\xymatrix{A \ar[r]^{\eta} \ar[dr]^{{\rm Can}_{S_{\eta}}} & \widetilde{A} \ar@{>->>}[d]^{\widetilde{{\rm Can}_{S_{\eta}}}}\\
        &  A\{ S_{\eta}^{-1}\}
  }.$$
\end{proof}

In order to smoothly localize larger subsets of $\mathcal{C}^{\infty}(\R^n)$ for some $n \in \mathbb{N}$, say $\Sigma$, which is a set that contains possibly a non-countable amount of elements,  we can proceed as follows. First notice that we can obtain the $\mathcal{C}^{\infty}-$ring of fractions of $\mathcal{C}^{\infty}(\mathbb{R}^n)$ with respect to the singleton $\Sigma = \{f: \R^n \to \R\}$, provided that $f \not\equiv 0$. Whenever $\Sigma = \{ f_1, \cdots, f_k \}$ for some $k \in \N$, inverting $\Sigma$ is equivalent to inverting $\prod \Sigma = f_1 \cdot f_2 \cdots f_{k-1}\cdot f_k$. In the case that $\Sigma$ is infinite, first we decompose it as the union of its finite subsets:

$$\Sigma = \bigcup_{\Sigma' \subseteq_{\rm fin} \Sigma} \Sigma'$$

Note that $\mathcal{S} =\{ \Sigma' \subseteq \Sigma | \Sigma' \,\,\mbox{is}\,\, \mbox{finite}\}$ is partially ordered by the  inclusion relation. Also,  whenever $\Sigma' \subseteq \Sigma''$, since $\eta_{\Sigma''}[\Sigma']\subseteq \eta_{\Sigma''}[\Sigma'']\subseteq (A\{ {\Sigma''}^{-1}\})^{\times}$, by the universal property of $\eta_{\Sigma''}: A \to A\{ {\Sigma''}^{-1}\}$, there is a unique $\mathcal{C}^{\infty}-$homomorphism $\alpha_{\Sigma' \Sigma''}: A\{ {\Sigma'}^{-1}\} \rightarrow A\{ {\Sigma''}^{-1}\}$ such that the following diagram commutes:

$$\xymatrixcolsep{5pc}\xymatrix{
A \ar[r]^{\eta_{\Sigma'}} \ar[dr]_{\eta_{\Sigma''}} & A\{ {\Sigma'}^{-1}\} \ar@{-->}[d]^{\exists ! \alpha_{\Sigma' \Sigma''}}\\
 & A\{ {\Sigma''}^{-1}\}}$$

It is simple to prove, using the ``uniqueness part'' of the $\mathcal{C}^{\infty}-$homomorphism obtained via universal property, that for any finite $\Sigma'$ we have $\alpha_{\Sigma' \Sigma'} = {\rm id}_{A\{ {\Sigma'}^{-1}\}}$, and given any finite $\Sigma', \Sigma''$ and $\Sigma'''$ such that $\Sigma' \subseteq \Sigma'' \subseteq \Sigma'''$,  $\alpha_{\Sigma'' \Sigma'''} \circ \alpha_{\Sigma' \Sigma''} = \alpha_{\Sigma'\Sigma'''}$, so we have an inductive system:
$$\{ \alpha_{\Sigma' \Sigma''}: A\{ {\Sigma'}^{-1}\} \to A\{ {\Sigma''}^{-1}\} | (\Sigma', \Sigma'' \in \mathcal{S}) \& (\Sigma' \subseteq \Sigma'')\}$$

We take, thus:

$$A\{ \Sigma^{-1}\} = \varinjlim_{\Sigma' \subseteq_{\rm fin} \Sigma} A\{ {\Sigma'}^{-'}\} = \varinjlim_{\Sigma' \subseteq_{\rm fin} \Sigma} A \left\{ {\prod \Sigma'}^{-1} \right\}$$

$$\xymatrixcolsep{5pc}\xymatrix{
 & A\{ \Sigma^{-1}\} & \\
A\{ {\Sigma'}^{-'}\} \ar[ur]^{\alpha_{\Sigma'}} \ar[rr]_{\alpha_{\Sigma' \Sigma''}} & & A\{ {\Sigma''}^{-1}\} \ar[ul]_{\alpha_{\Sigma''}}}$$

Hence, given any $\mathcal{C}^{\infty}-$ring $A$ and any $S \subseteq A$, we can construct:

$$\eta_{\Sigma} = \alpha_{\Sigma'}\circ \eta_{\Sigma'}: A \rightarrow A\{ {\Sigma}^{-1}\} $$

It is easy to prove that $\eta_{\Sigma}$ has the universal property which characterizes the $\mathcal{C}^{\infty}-$ring of fractions of $A$ with respect to $\Sigma$.\\

\begin{remark}In the case that $A = \mathcal{C}^{\infty}(\mathbb{R}^n)$ and  $\Sigma = \{ f: \mathbb{R}^n \to \mathbb{R}\}$, we have $\Sigma^{\infty-{\rm sat}} = \{ g \in \mathcal{C}^{\infty}\,(\mathbb{R}^n) | U_f \subseteq U_g\} = \{ g \in \mathcal{C}^{\infty}\,(\mathbb{R}^n) | Z(g) \subseteq Z(f) \}$.
\end{remark}

The following theorem gives us concretely the $\mathcal{C}^{\infty}-$ring of fractions of a finitely generated free $\mathcal{C}^{\infty}-$ring with respect to one element.\\

\begin{theorem}Let $\varphi \in \mathcal{C}^{\infty}(\mathbb{R}^n)$ and $U_{\varphi} = \{ x \in \mathbb{R}^n | \varphi(x)\neq 0\} = {\rm Coz}(\varphi)$. Then:
$$\mathcal{C}^{\infty}(U_{\varphi}) \cong \dfrac{\mathcal{C}^{\infty}(\mathbb{R}^{n+1})}{\langle \{ y \cdot \varphi(x) - 1 \}\rangle}$$
\end{theorem}
\begin{proof}
Let:
$$\begin{array}{cccc}
    \rho: & \mathcal{C}^{\infty}(\mathbb{R}^{n+1}) & \rightarrow & \mathcal{C}^{\infty}(U_{\varphi}) \\
     & f & \mapsto & \begin{array}{cccc}
                       \rho(f): & U_{\varphi} & \rightarrow & \mathbb{R} \\
                        & x & \mapsto & f\left( x, \frac{1}{\varphi(x)}\right)
                     \end{array}
  \end{array}$$

We claim that $\rho$ is surjective, \textit{i.e.}, any $h \in \mathcal{C}^{\infty}(U_{\varphi})$ may be lifted to some $f \in \mathcal{C}^{\infty}(\mathbb{R}^{n+1})$.\\

Let $h \in \mathcal{C}^{\infty}(U_{\varphi})$. Define, for $\varepsilon>0$:

$$\begin{array}{cccc}
    f: & \mathbb{R}^{n+1} & \rightarrow & \mathbb{R} \\
     & (x,y) & \mapsto & \begin{cases}
                           r_{\varepsilon}\left( y - \frac{1}{\varphi(x)}\right)\cdot h(x), & \mbox{if } x \in U_{\varphi} \\
     0, & \mbox{if} |y \cdot \varphi(x) - 1|> \varepsilon \cdot |\varphi(x)|.
                         \end{cases}
  \end{array}$$

Using the fact that $\mathbb{R}$ is an ordered local ring, one easily checks that $f \in \mathcal{C}^{\infty}(\mathbb{R}^{n+1})$. Indeed, either $|\varphi(x)|<\dfrac{1}{\varepsilon + |y|}$ and hence $|y \cdot \varphi(x)-1|> \varepsilon \cdot |\varphi(x)|$ or $|\varphi(x)|> \dfrac{1}{2(\varepsilon + |y|)}$, which implies that $\varphi(x)$ is invertible. Clearly,

$$(\forall x \in U_{\varphi})\left( h(x) = f\left( x, \frac{1}{\varphi(x)}\right)\right).$$

Assume now that $f \in \ker(\rho)$. Using \textbf{Hadamard's Lemma} for $f \in \mathcal{C}^{\infty}(\mathbb{R}^n \times \mathbb{R}, \mathbb{R})$, we have:

$$f(x,t) - f(x,s) = (t-s)\cdot \displaystyle\int_{0}^{1} \dfrac{\partial f}{\partial t}(x, s + (t-s)\cdot u)du$$

we conclude the existence of some $f_1 \in \mathcal{C}^{\infty}(\mathbb{R}^{n+2})$ such that:

$$f(x,y) - f\left( x, \frac{1}{\varphi(x)}\right) = \left( y - \frac{1}{\varphi(x)}\right)\cdot f_i\left(x,y, \frac{1}{\varphi(x)} \right).$$

Define:

$$\begin{array}{cccc}
    \nu: & \mathbb{R}^{n+1} & \rightarrow & [0,\infty[ \\
     & (x,y) & \mapsto & \begin{cases}
                           \dfrac{f(x,y)}{y\cdot \varphi(x)-1}, & \mbox{if } y\cdot \varphi(x) - 1 \neq 0 \\
                           \dfrac{f_1\left( x,y, \frac{1}{\varphi(x)}\right)}{\varphi(x)}, & \mbox{if} \varphi(x)\neq 0 .
                         \end{cases}
  \end{array}$$

Once again, it is easily checked that $\nu \in \mathcal{C}^{\infty}(\mathbb{R}^{n+1})$. Therefore, $f(x,y) = \nu(x,y)\cdot (y \cdot \varphi(x)-1) \in \langle \{ y \cdot \varphi(x) - 1 \}\rangle$.

\end{proof}

Let us now consider the following result, credited to Ortega and Mu\~{n}oz by I. Moerdijk and G. Reyes in \cite{moerdijk1986rings}:

\begin{theorem}\label{OM}Let $U \subseteq \mathbb{R}^n$ be open, and $g \in \mathcal{C}^{\infty}(U)$. Then there are $h,k \in \mathcal{C}^{\infty}(\mathbb{R}^n)$ with $U_k = U$ and $g \cdot k \displaystyle\upharpoonright_{U} \equiv h \displaystyle\upharpoonright_{U}$. where $U_k = \mathbb{R}^n \setminus Z(k)$ and $Z(k) = \{ x \in \mathbb{R}^n | k(x)=0\}$.
\end{theorem}


%

\begin{theorem}\label{cara} Let $A$ be a $\mathcal{C}^{\infty}-$ring and $S \subseteq A$. An element $\lambda = \dfrac{\eta_S(c)}{\eta_S(b)}$ (with $c \in A$ and $b \in S^{\infty-{\rm sat}}$) is invertible in $A\{ S^{-1}\}$ if, and only if, there are elements $d \in S^{\infty-{\rm sat}}$ and $c' \in A$ such that $dc'c \in S^{\infty-{\rm sat}}$, that is,
$$\dfrac{\eta_S(c)}{\eta_S(b)} \in (A\{ S^{-1}\})^{\times} \iff (\exists d \in S^{\infty-{\rm sat}})(\exists c' \in A)(d\cdot c' \cdot c \in S^{\infty-{\rm sat}}).$$
\end{theorem}
\begin{proof}
Suppose $\dfrac{\eta_S(c)}{\eta_S(b)} \in (A\{ S^{-1}\})^{\times}$, so there are $c' \in A$ and $b' \in S^{\infty-{\rm sat}}$ such that:
$$\dfrac{\eta_S(c)}{\eta_S(b)} \cdot \dfrac{\eta_S(c')}{\eta_S(b')} = 1_{A\{ S^{-1}\}} = \eta_S(1_A).$$
$$\eta_S(c \cdot c') = \eta_S(b \cdot b')$$
$$\eta_S(c \cdot c' - b \cdot b') = 0$$
By \textbf{Theorem \ref{38}}, there is some $d \in S^{\infty-{\rm sat}}$ such that:
$$d \cdot (c \cdot c' - b \cdot b') = 0$$
$$d \cdot c \cdot c' = d \cdot b \cdot b' \in S^{\infty-{\rm sat}}$$
where $d \cdot b \cdot b' \in S^{\infty-{\rm sat}}$ because it is a product of elements of $S^{\infty-{\rm sat}}$, which is a submonoid of $A$.\\

Conversely, suppose that $\dfrac{\eta_S(c)}{\eta_S(b)} \in A\{ S^{-1}\}$ with $b \in S^{\infty-{\rm sat}}$ is an element for which there are elements $d \in S^{\infty-{\rm sat}}$ and $c' \in A$ such that $d \cdot c \cdot c' \in S^{\infty-{\rm sat}}$. We have $\eta_S(d \cdot c' \cdot c) \in (A\{ S^{-1}\})^{\times}$ and $b \in S^{\infty-{\rm sat}}$, so $\eta_S(b) \in (A\{ S^{-1}\})^{\times}$, hence
$$\dfrac{\eta_S(d \cdot c' \cdot c)}{\eta_S(b)} \in (A\{ S^{-1}\})^{\times}.$$

Since
$$\dfrac{\eta_S(c)}{\eta_S(b)} \cdot \eta_S(d \cdot c') = \dfrac{\eta_S(d \cdot c' \cdot c)}{\eta_S(b)} \in (A\{ S^{-1}\})^{\times}$$
it follows that $\dfrac{\eta_S(c)}{\eta_S(b)} \in (A\{ S^{-1}\})^{\times}$, for if $\alpha \cdot \beta$ is invertible, then both $\alpha$ and $\beta$ are invertible.
\end{proof}

Now we have the following:

\begin{proposition}\label{pani}Let $U \subseteq \mathbb{R}^n$ be any open subset and define $S_U = \{ g \in \mathcal{C}^{\infty}(\mathbb{R}^n) | U \subseteq U_g\} \subseteq \mathcal{C}^{\infty}(\mathbb{R}^n)$. The $\mathcal{C}^{\infty}-$ring of fractions of $\mathcal{C}^{\infty}(\mathbb{R}^n)$ with respect to the set $S_U$:
$$\eta_{S_U}: \mathcal{C}^{\infty}(\mathbb{R}^n) \to \mathcal{C}^{\infty}(\mathbb{R}^n)\{ {S_U}^{-1}\}$$
is isomorphic to the restriction map:
$$\begin{array}{cccc}
    \rho : & \mathcal{C}^{\infty}(\mathbb{R}^n) & \rightarrow & \mathcal{C}^{\infty}(U) \\
     & h & \mapsto & h \displaystyle\upharpoonright_{U}
  \end{array}$$
\end{proposition}
\begin{proof}
We are going to show that $\rho \cong \eta_{S_U}$ using \textbf{Theorem \ref{340}} with $A = \mathcal{C}^{\infty}(\mathbb{R}^n)$, $\widetilde{A} = \mathcal{C}^{\infty}\,(U)$ and $\eta = \rho$.\\

Note that $\rho[S_U] \subseteq \mathcal{C}^{\infty}(U)^{\times}$.\\

Let's verify the first item, (i):\\

By the \textbf{Theorem 1.3} of \cite{moerdijk1986rings}, given any $g \in \mathcal{C}^{\infty}(U)$, there are $h,k \in \mathcal{C}^{\infty}(\mathbb{R}^n)$ with $U_k = U$, such that $g \cdot k\upharpoonright_U = h\upharpoonright_U$. Since $U=U_k$, $\rho(k) \in \mathcal{C}^{\infty}(U)^{\times}$, and we have, thus:
$$(\rho(k) \in \mathcal{C}^{\infty}(U)^{\times})\&(g \cdot \rho(k) = \rho(h)).$$

For item (ii), suppose that $g \in \mathcal{C}^{\infty}(\mathbb{R}^n)$ is such that $g\upharpoonright_U = 0$, so $U \subseteq Z(g)$. We know, by  a well-known theorem proved by Whitney (see \textbf{Theorem 5.0.5} of \cite{berni2018some}), that given the open subset $U$ of $\mathbb{R}^n$, there is some $h \in \mathcal{C}^{\infty}(\mathbb{R}^n)$ such that $U=U_h$, so $Z(h \cdot g) = Z(h) \cup Z(g) = U \cup (\mathbb{R}^n \setminus U) = \mathbb{R}^n$. We have both:

$$\rho(g) = 0 \in \mathcal{C}^{\infty}(U)$$

and

$$g \cdot h = 0 \in \mathcal{C}^{\infty}(\mathbb{R}^n),$$

so item (ii) is also fulfilled. By \textbf{Theorem \ref{340}}, the result follows.
\end{proof}

\begin{remark}Let $A$ be a $\mathcal{C}^{\infty}-$ring and $a \in A$. In general, the $\mathcal{C}^{\infty}-$ring of fractions of $A$ with respect to $a$ \textbf{is not a local $\mathcal{C}^{\infty}-$ring}. Let us consider the case on which $A = \mathcal{C}^{\infty}(\mathbb{R}^n)$ and $a = f: \mathbb{R}^n \to \mathbb{R}$ is such that $\neg (f \equiv 0)$. By the \textbf{Theorem 1.3} of \cite{moerdijk1986rings}, $A\{ a^{-1}\} \cong \mathcal{C}^{\infty}(\mathbb{R}^n)\{ f^{-1}\}$, and $\mathcal{C}^{\infty}(\mathbb{R}^n)\{ f^{-1}\} \cong \mathcal{C}^{\infty}(U_f)$, where $U_f = {\rm Coz}\,(f) = \mathbb{R}^n \setminus Z(f)$. For every $x \in U_f$, we have a maximal ideal:
$$\mathfrak{m}_x = \{ g \in \mathcal{C}^{\infty}(U_f) | g(x) = 0\},$$
hence a \textit{continuum} of maximal ideals.
\end{remark}

\begin{proposition}\label{star} Let $A$ be a $\mathcal{C}^{\infty}-$ring, $a,b,c \in A$ be three elements of its underlying subset and:
$$\eta_{b,A}^{\infty}: A \to A\{ b^{-1}\}$$
the $\mathcal{C}^{\infty}-$ring of fractions of $A$ with respect to $b$ and
$$\eta_{bc}^{\infty}: A \to A\{ (b\cdot c)^{-1}\}$$
the $\mathcal{C}^{\infty}-$ring of fractions of $A$ with respect to $b \cdot c$.\\

There is an isomorphism between $(A\{ b^{-1}\})\{ \eta_{b,A}^{\infty}(c)^{-1}\}$ and $A\{ (b \cdot c)^{-1} \}$.
\end{proposition}
\begin{proof}Consider the following diagram:
$$\xymatrixcolsep{5pc}\xymatrix{
A \ar[r]^{\eta_{b,A}^{\infty}} \ar@/_2pc/[drr]_{\eta_{b \cdot c}^{\infty}} & A\{ b^{-1}\} \ar[r]^{\eta_{\eta_{b,A}^{\infty}(c)}^{\infty}} & (A\{ b^{-1}\})\{ {\eta_{b,A}^{\infty}(c)}^{-1}\}\\
    &    & A\{ (b \cdot c)^{-1}\}
}$$

Note that $\eta_{b \cdot c}^{\infty}(b) \in (A\{ (b \cdot c)^{-1}\})^{\times}$, for since $\eta_{b \cdot c}^{\infty}(b \cdot c) \in (A\{ (b \cdot c)^{-1}\})^{\times}$, there is some $\psi \in A\{ (b \cdot c)^{-1}\}$ such that:
$$\eta_{b \cdot c}^{\infty}(b \cdot c)\cdot \psi = 1,$$
so
$$\eta_{b \cdot c}^{\infty}(b) \cdot [\eta_{b \cdot c}^{\infty}(c) \cdot \psi] = 1$$
and $\eta_{b \cdot c}^{\infty}(c)\cdot \psi$ is an inverse of $\eta_{b \cdot c}^{\infty}(b)$ in $A\{ (b \cdot c)^{-1}\}$.\\

By the universal property of $\eta_{b,A}^{\infty}: A \to A\{ b^{-1}\}$, there is a unique morphism $\zeta: A\{ b^{-1}\} \to A\{ (b \cdot c)^{-1}\}$ such that the following triangle commutes:
$$\xymatrixcolsep{5pc}\xymatrix{
A \ar[r]^{\eta_{b,A}^{\infty}} \ar[dr]_{\eta_{b \cdot c}^{\infty}} & A\{ b^{-1}\} \ar@{.>}[d]^{\zeta}\\
   & A\{ (b \cdot c)^{-1}\}
}$$

Now we can use a similar argument to conclude that $\eta_{b \cdot c}^{\infty}(c) \in (A\{ (b \cdot c)^{-1}\})^{\times}$. Since the above triangle commutes, it follows that:
$$\zeta(\eta_{b,A}^{\infty}(c)) = \eta_{b \cdot c}^{\infty}(c) \in (A\{ (b \cdot c)^{-1}\})^{\times}.$$

By the universal property of $\eta_{\eta_{b,A}^{\infty}(c)}^{\infty}: A\{ b^{-1}\} \to (A\{ b^{-1} \})\{ {\eta_{b,A}(c)}^{-1}\}$, there is a unique morphism $\theta: (A\{ b^{-1}\})\{ {\eta_{b,A}^{\infty}(c)}^{-1}\} \to A\{ (b \cdot c)^{-1}\}$ such that the following diagram commutes:
$$\xymatrixcolsep{5pc}\xymatrix{
A\{ b^{-1}\} \ar[r]^{\eta_{\eta_{b,A}^{\infty}(c)}^{\infty}} \ar[dr]_{\zeta} & (A\{ b^{-1}\})\{ {\eta_{b,A}^{\infty}(c)}^{-1} \} \ar@{.>}[d]^{\theta}\\
    & A\{ (b \cdot c)^{-1}\}
}$$

hence $\theta$ is the unique $\mathcal{C}^{\infty}-$homomorphism such that the following triangle commutes:

$$\xymatrixcolsep{7pc}\xymatrix{
A \ar[r]^{\eta_{\eta_{b,A}^{\infty}(c)}^{\infty} \circ \eta_{b,A}^{\infty}} \ar[dr]_{\eta_{b \cdot c}^{\infty}} & (A\{ b^{-1}\})\{ {\eta_{b,A}^{\infty}(c)}^{-1} \} \ar[d]^{\theta}\\
    & A\{ (b \cdot c)^{-1}\}
}$$

Now we are going to prove that $\eta_{\eta_{b,A}^{\infty}(c)}^{\infty} \circ \eta_{b,A}^{\infty}(b \cdot c) \in [(A\{ b^{-1}\})\{ \eta_{b,A}^{\infty}(c)^{-1}\}]^{\times}$.\\

We have $\eta_{b,A}^{\infty}(b) \in (A\{ b^{-1}\})^{\times}$, hence $\eta_{\eta_{b,A}^{\infty}}(\eta_{b,A}^{\infty}(b)) \in [(A\{ b^{-1}\})\{ \eta_{b,A}^{\infty}(c)^{-1}\}]^{\times}$, and by the very definition of $\eta_{\eta_{b,A}^{\infty}(c)}^{\infty}: A\{ b^{-1}\} \to (A\{ b^{-1}\})\{ {\eta_{b,A}^{\infty}(c)}^{-1}\}$,
$\eta_{\eta_{b,A}^{\infty}}(\eta_{b,A}^{\infty}(c)) \in [(A\{ b^{-1}\})\{ \eta_{b,A}^{\infty}(c)^{-1}\}]^{\times}$. Since each factor is invertible, it follows that  $\eta_{\eta_{b,A}^{\infty}(c)}^{\infty} \circ \eta_{b,A}^{\infty}(b) \cdot \eta_{\eta_{b,A}^{\infty}(c)}^{\infty} \circ \eta_{b,A}^{\infty}(c) = \eta_{\eta_{b,A}^{\infty}(c)}^{\infty} \circ \eta_{b,A}^{\infty}(b \cdot c)$ is invertible.\\

By the universal property of $\eta_{b \cdot c}^{\infty}: A \to A\{ (b \cdot c)^{-1}\}$, there is a unique morphism $\xi : A\{ (b \cdot c)^{-1}\} \to (A\{ b^{-1}\})\{ \eta_{b,A}^{\infty}(c)^{-1}\}$ such that the following triangle commutes:

$$\xymatrixcolsep{7pc}\xymatrix{
A \ar[r]^{\eta_{b\cdot c}^{\infty}} \ar[dr]_{\eta_{\eta_{b,A}^{\infty}(c)}^{\infty} \circ \eta_{b,A}^{\infty}} & A\{ (b \cdot c)^{-1}\} \ar@{.>}[d]^{\xi}\\
    & (A\{ b^{-1}\})\{ \eta_{b,A}^{\infty}(c)^{-1}\}
}$$

By the uniqueness of the $\mathcal{C}^{\infty}-$homomorphisms $\zeta$ and $\xi$, we conclude $\zeta \circ \xi = {\rm id}_{A\{ (b \cdot c)^{-1}\}}$ and $\xi \circ \zeta = {\rm id}_{(A\{ b^{-1}\})\{ {\eta_{b,A}(c)}^{-1}\}}$, so the result follows.
\end{proof}

In the same line of the previous proposition, we can prove the following:

\begin{proposition}\label{imp} Let $A$ be a $\mathcal{C}^{\infty}-$ring and let $a \in A$ and $\beta \in A\{a^{-1}\}$. Since $\beta = \eta^A_a(b)/\eta^A_a(c)$ for some $b \in A$ and $c \in \{a\}^{\infty-sat}$, then there is a unique $\mathcal{C}^{\infty}-$isomorphism of $A$-algebras:

$$\theta_{ab}: (A\{a^{-1}\})\{\beta^{-1}\} \stackrel{\cong}{\longrightarrow} A\{(a \cdot b)^{-1}\}$$

I.e., $\theta_{ab}: (A\{a^{-1}\})\{\beta^{-1}\} \to A\{(a\cdot b)^{-1}\}$ is a $\mathcal{C}^{\infty}-$rings isomorphism such that the following diagram commutes:

$$\xymatrixcolsep{5pc}\xymatrix{
A \ar[r]^{\eta^A_a} \ar@/_/[drr]_{\eta^A_{a \cdot b}} & A\{a^{-1}\}  \ar[r]^{\eta^{A_a}_{\beta}} & (A\{a^{-1}\})\{\beta^{-1}\} \ar[d]^{\theta_{ab}}\\
  & & A\{(a\cdot b)^{-1}\}
 }$$

that is, $(\eta^A_{a \cdot b}: A \to A\{(a \cdot b)^{-1}\}) \cong (\eta^{A_a}_\beta\circ \eta^A_a : A \to (A\{a^{-1}\})\{\beta^{-1}\})$ in $A \downarrow \mathcal{C}^{\infty}{\rm \bf Rng}_{\rm fp}$. Hence:

$$\langle \{ \eta_{a \cdot b}: A \to A\{(a \cdot b)^{-1}\}|a,b \in A \}\rangle = \langle \{\eta_{\beta}\circ \eta_a : A \to (A\{a^{-1}\})\{\beta^{-1}\}| a,b \in A \}\rangle.$$
\end{proposition}

\begin{remark}\label{raca}In the context of \textbf{Proposition \ref{pani}}, note that if $f \in \mathcal{C}^{\infty}(\mathbb{R}^n)$ is such that $U=  U_f$, we have $S_{U_f} = \{ f\}^{\infty-{\rm sat}}$.
\end{remark}

Now we are ready to prove \textbf{Theorem \ref{38}}.\\

\textbf{Proof of Theorem \ref{38}:} Let $A$ be a $\mathcal{C}^{\infty}-$ring. We are going to prove the result first for free $\mathcal{C}^{\infty}-$rings, then for quotients and finally for colimits.\\

The case where $A = \mathcal{C}^{\infty}(\mathbb{R}^n)$ was proved in \textbf{Proposition \ref{pani}};\\

Suppose, now, that $A = \dfrac{\mathcal{C}^{\infty}(\mathbb{R}^n)}{I}$ for some ideal $I$. By \textbf{Corollary \ref{Jeq}}, it follows that:

$$\dfrac{\mathcal{C}^{\infty}(\mathbb{R}^n)}{I}\lbrace {f + I}^{-1} \rbrace \cong \dfrac{\mathcal{C}^{\infty}(\mathbb{R}^n)\{ f^{-1}\}}{\langle \eta_f[I]\rangle}$$

so items (i) and (ii) of \textbf{Theorem \ref{340}} hold for the quotient.\\

Up to this point, we have proved the result for a finitely generated $\mathcal{C}^{\infty}-$ring,  $A$, and $\Sigma = \{ f\}$. Given any finite $S$, say $S = \{ f_1, \cdots, f_{\ell}\}$, since $A\{ S^{-1}\} = A\{ {\prod_{i = 1}^{\ell}f_i}^{-1}\}$, the result follows too.\\

Let $A$ be a finitely generated $\mathcal{C}^{\infty}-$ring and  $S \subseteq A$ be any set. Write:

$$S = \bigcup_{S' \subseteq_{\rm fin} S}S'$$

Since $A\{ S^{-1}\} \cong \varinjlim_{S' \subseteq S} A\{ {S'}^{-1}\}$, the items (i) and (ii) hold for $A\{ S^{-1}\}$, for a finitely generated $\mathcal{C}^{\infty}-$ring $A$ and any set $S$.\\

Finally, given any $\mathcal{C}^{\infty}-$ring $B$ and any set $S \subseteq B$, write $B$ as a directed colimit of its finitely generated $\mathcal{C}^{\infty}-$subrings

$$\xymatrixcolsep{5pc}\xymatrix{
 & B & \\
B_i \ar@/^/[ur]^{\jmath_i} \ar[rr]_{\beta_{ij}} & & B_j \ar@/_/[ul]_{\jmath_j}}$$

so $B \cong \varinjlim_{B_i \subseteq_{\rm f.g.} B}B_i$ and define $S_i = \jmath_i^{\dashv}[S]$. Since items (i) and (ii) hold for every $B_i\{ {S_i}^{-1}\}$, the same is true for

$$B\{ S^{-1}\} \cong \varinjlim_{B_i \subseteq_{\rm fin} B} B_i\{ {S_i}^{-1}\},$$

and the result follows. \hfill $\square$ \\

\subsection{A Category of Pairs}\label{cop}

Let us now consider the category whose objects are pairs $(A,S)$, where $A$ is any $\mathcal{C}^{\infty}-$ring and $S \subseteq A$ is any of its  subsets, and whose morphisms between $(A,S)$ and $(B,T)$ are precisely the $\mathcal{C}^{\infty}-$morphisms $f: A \to B$ such that $f[S] \subseteq T$. In order to distinguish the latter morphism of the category of pairs from the morphisms of $\mathcal{C}^{\infty}{\rm \bf Rng}$, we shall denote it by $f_{ST}: (A,S) \to (B,T)$. We denote this category by $\mathcal{C}^{\infty}_2$.\\

In this section we are going to define a functor $F: \mathcal{C}^{\infty}_{2} \to \mathcal{C}^{\infty}{\rm \bf Rng}$ with some ``nice'' properties.\\

Given $(A,S) \in {\rm Obj}\,(\mathcal{C}^{\infty}_{2})$, let:
$$F(A,S) = A\{ S^{-1}\}.$$

Also, given $(A,S), (B,T) \in {\rm Obj}\,(\mathcal{C}^{\infty}_{2})$ and $(A,S) \stackrel{f_{ST}}{\rightarrow} (B,T) \in {\rm Mor}\,(\mathcal{C}^{\infty}_{2})$, since $f_{ST}[S] \subseteq T$ and $\eta_T[T] \subseteq {B\{ T^{-1}\}}^{\times}$, we have $(\eta_T \circ f_{ST})[S] \subseteq {B\{ T^{-1}\}}^{\times}$, so by the universal property of $\eta_S: A \to A\{ S^{-1}\}$ there is a unique $\mathcal{C}^{\infty}-$homomorphism:
$$\widetilde{f_{ST}}: A\{ S^{-1}\} \to B\{ T^{-1}\}$$
such that the following diagram commutes:
$$\xymatrixcolsep{5pc}\xymatrix{
A \ar[d]_{f_{ST}} \ar[r]^{\eta_S} & A\{ S^{-1}\} \ar[d]^{\widetilde{f_{ST}}}\\
B \ar[r]^{\eta_T} & B\{ T^{-1}\}}$$

Let:

$$F((A,S) \stackrel{f_{ST}}{\rightarrow} (B,T)) := A\{ S^{-1}\} \stackrel{\widetilde{f_{ST}}}{\rightarrow} B\{ T^{-1}\}.$$

\begin{theorem}\label{otalfuntor}The map:
$$\begin{array}{cccc}
    F: & \mathcal{C}^{\infty}_2 & \rightarrow & \mathcal{C}^{\infty}{\rm \bf Rng} \\
     & (A,S) & \mapsto & A\{ S^{-1}\}\\
     & \xymatrix{(A,S) \ar[r]^{f_{ST}} & (B,T)} & \mapsto & \xymatrix{A\{ S^{-1}\} \ar[r]^{\widetilde{f_{ST}}} & B\{ T^{-1}\}} \\
  \end{array}$$
is a functor.
\end{theorem}
\begin{proof}
Given any object $(A,S) \in {\rm Obj}\,(\mathcal{C}^{\infty}_{2})$ and its identity map ${\rm id}_{(A,S)}: (A,S) \to (A,S)$, $F({\rm id}_{(A,S)})$ is the unique $\mathcal{C}^{\infty}-$homomorphism such that the following rectangle commutes:
$$\xymatrixcolsep{5pc}\xymatrix{
A \ar[r]^{\eta_S} \ar[d]_{{\rm id}_{(A,S)}} & A\{ S^{-1}\} \ar[d]^{F({\rm id}_{(A,S)})} \\
A \ar[r]_{\eta_S} & A\{ S^{-1}\}
}.$$

Since:

$$\xymatrixcolsep{5pc}\xymatrix{
A \ar[r]^{\eta_S} \ar[d]_{{\rm id}_{(A,S)}} & A\{ S^{-1}\} \ar[d]^{{\rm id}_{A\{ S^{-1}\}}} \\
A \ar[r]_{\eta_S} & A\{ S^{-1}\}
}$$
commutes, it follows that $F({\rm id}_{(A,S)}) = {\rm id}_{A\{ S^{-1}\}}$.\\

Let $(A,S), (B,T), (C,R) \in {\rm Obj}\,(\mathcal{C}^{\infty}_{2})$ and $(A,S) \stackrel{f_{ST}}{\rightarrow} (B,T), (B,T) \stackrel{g_{TR}}{\rightarrow} (C,R) \in {\rm Mor}\,(\mathcal{C}^{\infty}_{2})$. By definition, $F(f_{ST}): A\{ S^{-1}\} \to B\{ T^{-1}\}$ is the unique $\mathcal{C}^{\infty}-$homomorphism such that:

$$\xymatrixcolsep{5pc}\xymatrix{
A \ar[r]^{\eta_S} \ar[d]_{f_{ST}} & A\{ S^{-1}\} \ar[d]^{F({f_{ST}})} \\
B \ar[r]_{\eta_T} & B\{ T^{-1}\}
}$$

commutes, and $F(g_{TR}): B\{ T^{-1}\} \to C\{ R^{-1}\}$ is the unique $\mathcal{C}^{\infty}-$homomorphism such that:

$$\xymatrixcolsep{5pc}\xymatrix{
B \ar[r]^{\eta_T} \ar[d]_{g_{TR}} & B\{ T^{-1}\} \ar[d]^{F(g_{TR})} \\
C \ar[r]^{\eta_R} & C\{ R^{-1}\}
}$$

commutes. It follows that the following diagram commutes:

$$\xymatrixcolsep{5pc}\xymatrix{
A \ar[r]^{\eta_S} \ar[d]_{f_{ST}} & A\{ S^{-1}\} \ar[d]^{F({f_{ST}})} \\
B \ar[r]^{\eta_T} \ar[d]_{g_{TR}} & B\{ T^{-1}\} \ar[d]^{F(g_{TR})}\\
C \ar[r]^{\eta_R} & C\{ R^{-1}\}
}$$

so
\begin{equation}
\label{ger}
[F(g_{TR}) \circ F(f_{ST})] \circ \eta_S = \eta_R \circ (g_{TR}\circ f_{ST})
\end{equation}

Also by definition, $F(g_{TR} \circ f_{ST}): A\{ S^{-1}\} \to C\{ R^{-1}\}$ is the unique $\mathcal{C}^{\infty}-$homomorphism such that:

$$\xymatrixcolsep{5pc}\xymatrix{
A \ar[r]^{\eta_S} \ar[d]_{g_{TR} \circ f_{ST}} & A\{ S^{-1}\} \ar[d]^{F(g_{TR} \circ{f_{ST}})} \\
C \ar[r]^{\eta_R} & C\{ R^{-1}\}
}$$

commutes, that is, such that
\begin{equation}
\label{ald}
F(g_{TR} \circ f_{ST}) \circ \eta_S = \eta_R \circ (g_{TR} \circ f_{ST})
\end{equation}

By the uniqueness of the $\mathcal{C}^{\infty}-$homomorphism that satisfies \eqref{ger} and \eqref{ald}, it follows that:
$$F(g_{TR} \circ f_{ST}) = F(g_{TR}) \circ F(f_{ST}).$$

Hence $F$ is a functor.
\end{proof}

\begin{proposition}\label{par}Let $A$ and $B$ be two $\mathcal{C}^{\infty}-$rings and $S \subseteq A$ and $f: A \to B$ a $\mathcal{C}^{\infty}-$homomorphism. By the universal property of $\eta_S : A \to A\{ S^{-1}\}$ we have a unique $\mathcal{C}^{\infty}-$homomorphism $f_S : A\{ S^{-1}\} \to B\{ f[S]^{-1}\}$ such that the following square commutes:
$$\xymatrixcolsep{5pc}\xymatrix{
A \ar[r]^{\eta_S} \ar[d]_{f} & A\{ S^{-1}\} \ar[d]^{\exists ! f_S}\\
B \ar[r]_{\eta_{f[S]}} & B\{ f[S]^{-1}\}
}.$$

The diagram:
$$\xymatrixcolsep{5pc}\xymatrix{
B \ar[dr]^{\eta_{f[S]}}& \\
 & B \{ f[S]^{-1}\}\\
A\{ S^{-1}\} \ar[ur]_{f_{S}}
}$$
is a pushout of the diagram:
$$\xymatrixcolsep{5pc}\xymatrix{
  & B\\
A \ar[ur]^{f} \ar[dr]_{\eta_S} & \\
   & A\{ S^{-1}\}
}$$
\end{proposition}
\begin{proof}
Let $C$ be any $\mathcal{C}^{\infty}-$ring and $g: B \to C$ and $h: A\{ S^{-1}\} \to C$ be two $\mathcal{C}^{\infty}-$homomorphisms such that $g \circ f = h \circ \eta_S$, \textit{i.e.}, such that the following diagram commutes:
$$\xymatrixcolsep{5pc}
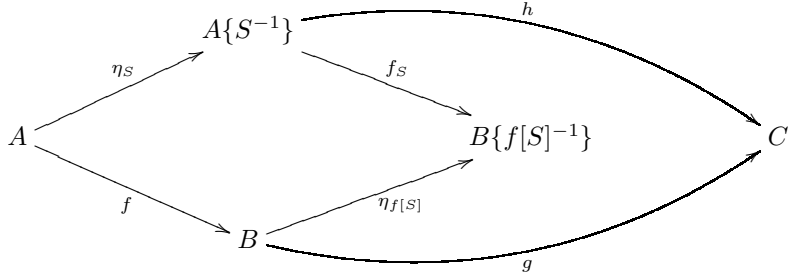
\xymatrix{
 & A\{ S^{-1}\} \ar[dr]^{f_S} \ar@/^2pc/[rrd]^{h}& \\
A \ar[ur]^{\eta_S} \ar[dr]_{f} & & B\{ f[S]^{-1}\}& C \\
   & B \ar[ur]_{\eta_{f[S]}} \ar@/_2pc/[rru]_{g} &
 }$$

 We claim that there is a unique $\mathcal{C}^{\infty}-$homomorphism $u: B\{ f[S]^{-1}\} \to C$ such that the following diagram commutes:

$$\xymatrixcolsep{5pc}
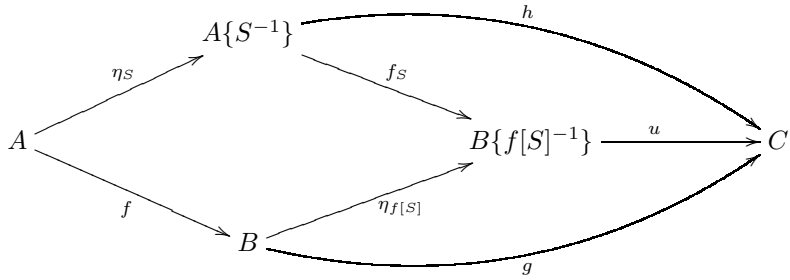
\xymatrix{
 & A\{ S^{-1}\} \ar[dr]^{f_S} \ar@/^2pc/[rrd]^{h}& \\
A \ar[ur]^{\eta_S} \ar[dr]_{f} & & B\{ f[S]^{-1}\} \ar[r]^{u}& C \\
   & B \ar[ur]_{\eta_{f[S]}} \ar@/_2pc/[rru]_{g} &
 }$$

Given any $s \in S$, we have $f_S(s) \in B^{\times}$, so $h(\eta_S(s)) \in C^{\times}$. Since the former diagram commutes, $g(f(s)) = h(\eta_S(s)) \in C^{\times}$ and by the universal property of $\eta_{f[S]}: B \to B\{ f[S]^{-1}\}$ there is a unique $\mathcal{C}^{\infty}-$homomorphism $u: B\{ f[S]^{-1}\} \to C$ such that the following diagram commutes:
$$\xymatrixcolsep{5pc}
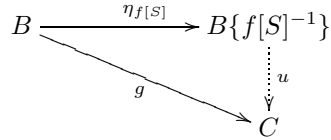
\xymatrix{
B \ar[r]^{\eta_{f[S]}} \ar[dr]_{g} & B\{ f[S]^{-1}\} \ar@{.>}[d]^{u}\\
   & C
}$$

Now we need only to prove that the following triangle commutes:

$$\xymatrixcolsep{5pc}
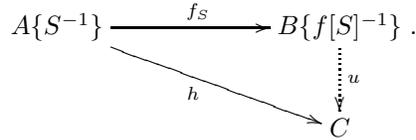
\xymatrix{
A\{ S^{-1}\} \ar[r]^{f_S} \ar[dr]_{h} & B\{ f[S]^{-1}\} \ar@{.>}[d]^{u}\\
   & C
}.$$

Note that, by the universal property of $\eta_S : A \to A\{ S^{-1}\}$ there must exist a unique $\mathcal{C}^{\infty}-$homomorphism $\xi: A\{ S^{-1}\} \to C$ such that:
$$\xymatrixcolsep{5pc}
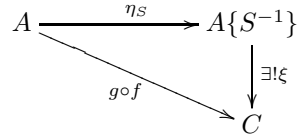
\xymatrix{
A \ar[r]^{\eta_S} \ar[dr]_{g \circ f}& A\{ S^{-1}\} \ar[d]^{\exists ! \xi}\\
  & C
}$$
commutes, and since both $h: A\{ S^{-1}\} \to C$ and $u \circ f_S : A\{ S^{-1}\} \to C$ have this property, it follows that $h = u \circ f_S$.\\

Hence  we have a pushout diagram.
\end{proof}

\begin{corollary}\label{papel}The following rectangle is a pushout:\\

$$\xymatrixcolsep{5pc}\xymatrix{
A  \ar@{-->}[dr]^{\eta_{a_i \cdot a_j}} \ar[r]^{\eta_{a_i}} \ar[d]_{\eta_{a_j}} & A\{ {a_i}^{-1}\} \ar[d]\\
A\{ {a_j}^{-1}\} \ar[r] & A\{(a_i \cdot a_j)^{-1}\}
}$$

\end{corollary}
\begin{proof}
By the previous lemma, the following diagram is a pushout:

$$\xymatrixcolsep{5pc}\xymatrix{
A   \ar[r]^{\eta_{a_i}} \ar[d]_{\eta_{a_j}} & A\{ {a_i}^{-1}\} \ar[d]\\
A\{ {a_j}^{-1}\} \ar[r] & A\{ {a_j}^{-1}\}\{ \eta_{a_j}(a_i)^{-1}\}
}$$

By \textbf{ Proposition \ref{star}}, $A\{ {a_j}^{-1}\}\{ \eta_{a_j}(a_i)^{-1}\} \cong A\{ (a_i \cdot a_j)^{-1}\}$, and the result follows.
\end{proof}

\begin{proposition}Let $(A,S)$ and $(B,T)$ be any two pairs in $\mathcal{C}^{\infty}_2$ and let $\jmath_{ST} : (A,S) \to (B,T)$ be a  $\mathcal{C}^{\infty}-$monomorphism such that $\jmath_{ST}[S^{\infty-{\rm sat}}] = T^{\infty-{\rm sat}}$. Under these circumstances,  $ \widetilde{\jmath_{ST}}: A\{ S^{-1}\} \to B\{ T^{-1}\}$ is also a $\mathcal{C}^{\infty}-$monomorphism.
\end{proposition}
\begin{proof}
Let $\theta \in A\{ S^{-1}\}$ be such that $\widetilde{\jmath_{ST}}(\theta) = 0$. Since $\theta \in A\{ S^{-1}\}$, there are $c \in A$ and $d \in S^{\infty-{\rm sat}}$ such that $\theta = \frac{\eta_S(c)}{\eta_S(d)}$, so:

$$\widetilde{\jmath_{ST}}(\eta_S(c)\cdot \eta_S(d)^{-1}) = \widetilde{\jmath_{ST}}(\eta_S(c))\cdot \widetilde{\jmath_{ST}}(\eta_S(d))^{-1} = \eta_T(\jmath_{ST}(c))\cdot \eta_T(\jmath_{ST}(d))^{-1}$$

Thus

$$\widetilde{\jmath_{ST}}(\eta_S(c)\cdot \eta_S(d)^{-1}) = 0 \Rightarrow \eta_T(\jmath_{ST}(c)) = 0.$$

Since $\eta_T(\jmath_{ST}(c)) = 0$, there is some $u \in T^{\infty-{\rm sat}}$ such that $u \cdot \jmath_{ST}(c) = 0$.\\

By hypothesis, $\jmath_{ST}[S^{\infty-{\rm sat}}]=T^{\infty-{\rm sat}}$, so given $u \in T^{\infty-{\rm sat}}$, there is some $v \in S^{\infty-{\rm sat}}$ such that $u = \jmath_{ST}(v)$. Thus,

$$\jmath_{ST}(v \cdot c) = \jmath_{ST}(v)\cdot \jmath_{ST}(c) = u \cdot \jmath_{ST}(c) = 0.$$

Since, by hypothesis, $\jmath_{ST}$ is a monomorphism, it follows that $v \cdot c = 0$, so $\eta_S(v \cdot c) = \eta_S(v) \cdot \eta_S(c) = 0$ with $v \in S^{\infty-{\rm sat}}$. We have, thus, $\eta_S(c)=0$ and, therefore:

$$\theta = \dfrac{\eta_S(c)}{\eta_S(d)} = 0,$$

and $\widetilde{\jmath_{ST}}$ is a $\mathcal{C}^{\infty}-$monomorphism.
\end{proof}

\begin{theorem}\label{comfraccop}Let $(A_1,S_1)$ and $(A_2,S_2)$ be two pairs in ${\rm Obj}\,(\mathcal{C}^{\infty}_2)$, $\jmath_1: S_1 \hookrightarrow A_1$ and $\jmath_2 : S_2 \hookrightarrow A_2$ be the inclusion maps and let:
$$\xymatrixcolsep{5pc}\xymatrix{
A_1 \ar[dr]^{k_1} & \\
 & A_1 \otimes_{\infty} A_2 \\
A_2 \ar[ur]_{k_2} &
}$$
be the coproduct of $A_1$ and $A_2$ in $\mathcal{C}^{\infty}{\rm \bf Rng}$. We have:
$$(A_1 \otimes_{\infty} A_2)\{ (k_1[S_1]\cup k_2[S_2])^{-1}\} \cong A_1\{ {S_1}^{-1}\}\otimes_{\infty} A_2 \{ {S_2}^{-1} \},$$
that is,
$$F(A_1,S_1)\otimes_{\infty} F(A_2,S_2) \cong F(A_1 \otimes_{\infty} A_2, k_1[S_1]\cup k_2[S_2]).$$
\end{theorem}
\begin{proof}
Let:
$$\xymatrixcolsep{5pc}\xymatrix{
A_1\{ {S_1}^{-1}\} \ar[dr]^{\ell_1} & \\
 & A_1\{ {S_1}^{-1}\} \otimes_{\infty} A_2\{ {S_2}^{-1}\}\\
A_2\{ {S_2}^{-1}\} \ar[ur]_{\ell_2} &
}$$
be the coproduct of $A_1\{ {S_1}^{-1}\}$ and $A_2\{ {S_2}^{-1}\}$.

Let $\eta_{k_1[S_1]\cup k_2[S_2]} : A_1 \otimes_{\infty} A_2 \rightarrow (A_1 \otimes_{\infty} A_2) \{ (k_1[S_1]\cup k_2[S_2])^{-1}\}$ be the canonical map of the ring of fractions of $A_1 \otimes_{\infty} A_2$ with respect to $k_1[S_1]\cup k_2[S_2]$. Since for every $s_1 \in S_1$ we have $k_1(s_1) \in k_1[S_1] \subseteq k_1[S_1]\cup k_2[S_2]$, $(\eta_{k_1[S_1]\cup k_2[S_2]}\circ k_1)(s_1) = \eta_{k_1[S_1]\cup k_2[S_2]}(k_1(s_1)) \in ((A_1 \otimes_{\infty} A_2) \{ k_1[S_1]\cup k_2[S_2]^{-1}\})^{\times}$, so the universal property of $\eta_{S_1}: A_1 \to A_1\{ {S_1}^{-1}\}$ there is a unique $\varphi_1: A_1\{ {S_1}^{-1}\} \to (A_1 \otimes_{\infty} A_2) \{ k_1[S_1]\cup k_2[S_2]^{-1}\}$ such that the following triangle commutes:
$$\xymatrixcolsep{5pc}\xymatrix{
A_1 \ar[r]^{\eta_{S_1}} \ar[dr]_{\eta_{k_1[S_1]\cup k_2[S_2]}\circ k_1} & A\{ {S_1}^{-1}\} \ar[d]^{\varphi_1}\\
 & (A_1 \otimes_{\infty} A_2) \{ k_1[S_1]\cup k_2[S_2]^{-1}\}
},$$

that is, $\varphi_1 \circ \eta_{S_1} = \eta_{k_1[S_1]\cup k_2[S_2]}$.\\

The same argument yields a unique $\varphi_2: A_2\{ {S_2}^{-1}\} \to A_1 \otimes_{\infty} A_2 \{ k_1[S_1]\cup k_2[S_2]^{-1}\}$ such that the following triangle commutes:
$$\xymatrixcolsep{5pc}\xymatrix{
A_2 \ar[r]^{\eta_{S_2}} \ar[dr]_{\eta_{k_1[S_1]\cup k_2[S_2]}\circ k_2} & A\{ {S_2}^{-1}\} \ar[d]^{\varphi_2}\\
 & (A_1 \otimes_{\infty} A_2) \{ k_1[S_1]\cup k_2[S_2]^{-1}\}
},$$

that is, $\varphi_2 \circ \eta_{S_2} = \eta_{k_1[S_1]\cup k_2[S_2]}$.\\

By the universal property of the coproduct of $A_1\{ {S_1}^{-1}\}\otimes_{\infty}A_2\{ {S_2}^{-1}\}$, given the diagram:
$$\xymatrixcolsep{5pc}\xymatrix{
A_1\{ {S_1}^{-1}\} \ar[rd]^{\varphi_1} & \\
 & (A_1 \otimes_{\infty} A_2) \{ k_1[S_1]\cup k_2[S_2]^{-1}\}\\
A_2\{ {S_2}^{-1}\} \ar[ur]_{\varphi_2} &
}$$

there is a unique $\mathcal{C}^{\infty}-$homomorphism $\widetilde{\varphi}: A_1\{ {S_1}^{-1}\}\otimes_{\infty} A_2\{ {S_2}^{-1}\} \to A_1 \otimes_{\infty} A_2 \{ k_1[S_1]\cup k_2[S_2]^{-1}\}$ such that the following diagram commutes:

$$\xymatrixcolsep{5pc}\xymatrix{
A_1\{ {S_1}^{-1}\} \ar[dr]^{\ell_1} \ar@/^2pc/[rrd]^{\varphi_1} & & \\
  & A_1\{ {S_1}^{-1}\} \otimes_{\infty} A_2\{ {S_2}^{-1}\} \ar[r]^{\exists ! \widetilde{\varphi}} & (A_1 \otimes_{\infty} A_2) \{ k_1[S_1]\cup k_2[S_2]^{-1}\}\\
A_2\{ {S_2}^{-1}\} \ar[ur]_{\ell_2} \ar@/_2pc/[urr]_{\varphi_2} & &
},$$

that is, such that $\widetilde{\varphi} \circ \ell_i = \varphi_i$ for $i \in \{ 1,2\}$.\\

Given the diagram:

$$\xymatrixcolsep{5pc}\xymatrix{
A_1 \ar[dr]^{\ell_1 \circ \eta_{S_1}} & \\
 & A_1\{ {S_1}^{-1}\}\otimes_{\infty} A_2\{ {S_2}^{-1}\}\\
A_2 \ar[ur]_{\ell_2 \circ \eta_{S_2}} &
}$$

the universal property of the coproduct $A_1 \otimes_{\infty} A_2$, there is a unique $\mathcal{C}^{\infty}-$homomorphism $\widetilde{\psi}: A_1\otimes_{\infty}A_2 \rightarrow A_1\{ {S_1}^{-1}\}\otimes_{\infty} A\{ {S_2}^{-1}\}$ such that the following diagram commutes:

$$\xymatrixcolsep{5pc}\xymatrix{
A_1 \ar[dr]^{k_1} \ar@/^2pc/[rrd]^{\ell_1 \circ \eta_{S_1}} & & \\
 & A_1 \otimes_{\infty} A_2 \ar[r]^{\exists ! \psi} & A_1\{ {S_1}^{-1}\}\otimes_{\infty} A_2\{ {S_2}^{-1}\}\\
A_2 \ar[ur]_{k_2} \ar@/_2pc/[rru]_{\ell_2 \circ \eta_{S_2}} & & }$$

that is, such that $\psi \circ k_1 = \ell_1 \circ \eta_{S_1}$ and $\psi \circ k_2 = \ell_2 \circ \eta_{S_2}$.\\

Given any $y \in k_1[S_1]\cup k_2[S_2]$, then $y \in k_1[S_1]$ or $y \in k_2[S_2]$. Suppose, without loss of generality, $y \in k_1[S_1]$, so there is some $s_1 \in S_1$ such that $y = k_1(s_1)$. We have:
$$\widetilde{\psi}(y) = \psi(k_1(s_1)) = \ell_1 \circ \eta_{S_1}(s_1) = \ell_1(\eta_{S_1}(s_1)) \in (A_1\{ {S_1}^{-1}\}\otimes_{\infty}A_2\{ {S_2}^{-1}\})^{\times},$$

and the universal property of $\eta_{k_1[S_1]\cup k_2[S_2]}: A_1 \otimes_{\infty} A_2 \to (A_1 \otimes_{\infty} A_2)\{ k_1[S_1]\cup k_2[S_2]^{-1}\}$, there is a unique $\widetilde{\psi}: (A_1 \otimes_{\infty} A_2)\{ k_1[S_1]\cup k_2[S_2]^{-1}\} \to A_1\{ {S_1}^{-1}\}\otimes_{\infty}A_2\{ {S_2}^{-1}\}$ such that the following triangle commutes:

$$\xymatrixcolsep{5pc}\xymatrix{
A_1\otimes_{\infty}A_2 \ar[r]^(0.2){\eta_{k_1[S_1]\cup k_2[S_2]}} \ar[dr]_{\psi} & (A_1 \otimes_{\infty} A_2)\{ k_1[S_1]\cup k_2[S_2]^{-1}\} \ar[d]^{\widetilde{\psi}}\\
  & A_1\{ {S_1}^{-1}\}\otimes_{\infty}A_2\{ {S_2}^{-1}\}}$$

\textbf{Claim:} $$\widetilde{\varphi}\circ \widetilde{\psi} = {\rm id}_{(A_1\otimes_{\infty}A_2)\{(k_1[S_1]\cup k_2[S_2])^{-1}\}}$$

In fact, we are going to prove that

$$(\widetilde{\varphi} \circ \widetilde{\psi})\circ \eta_{k_1[S_1]\cup k_2[S_2]} \circ k_i = \eta_{k_1[S_1]\cup k_2[S_2]} \circ k_i, i=1,2.$$

and since ${\rm id}_{(A_1\otimes_{\infty}A_2)\{k_1[S_1]\cup k_2[S_2]^{-1}\}}$ is the unique $\mathcal{C}^{\infty}-$homomorphism such that:

$$\xymatrixcolsep{5pc}\xymatrix{
A_1\{ {S_1}^{-1}\} \ar[r]^{\eta_{k_1[S_1]\cup k_2[S_2]} \circ k_1} \ar@/_2pc/[dr]_{\eta_{k_1[S_1]\cup k_2[S_2]} \circ k_1}& (A_1\otimes_{\infty}A_2)\{k_1[S_1]\cup k_2[S_2]^{-1}\} \ar[d]^{{\rm id}_{(A_1\otimes_{\infty}A_2)\{k_1[S_1]\cup k_2[S_2]^{-1}\}}} & A_2\{ {S_2}^{-1}\} \ar[l]_{\eta_{k_1[S_1]\cup k_2[S_2]} \circ k_2} \ar@/^2pc/[dl]^{\eta_{k_1[S_1]\cup k_2[S_2]} \circ k_2}\\
 & (A_1\otimes_{\infty}A_2)\{k_1[S_1]\cup k_2[S_2]^{-1}\} &
}$$

commutes, it will follow that $\widetilde{\varphi}\circ \widetilde{\psi} = {\rm id}_{(A_1\otimes_{\infty}A_2)\{k_1[S_1]\cup k_2[S_2]^{-1}\}}$.\\

By definition,
$$\widetilde{\psi} \circ \eta_{k_1[S_1]\cup k_2[S_2]} = \psi,$$
so composing both sides with $k_i$ yields:
$$\widetilde{\psi} \circ \eta_{k_1[S_1]\cup k_2[S_2]} \circ k_i = \psi \circ k_i.$$
Since $\psi$ is such that $\psi \circ k_i = \ell_i \circ \eta_{S_i}$, we get:
$$\widetilde{\psi} \circ \eta_{k_1[S_1]\cup k_2[S_2]} \circ k_i = \ell_i \circ \eta_{S_i}.$$
Composing both sides of the above equation with $\widetilde{\varphi}$ we get:
$$\widetilde{\varphi}\circ ( \widetilde{\psi} \circ \eta_{k_1[S_1]\cup k_2[S_2]} \circ k_i) = \widetilde{\varphi}\circ (\ell_i \circ \eta_{S_i}) = \varphi_i \circ \eta_{S_i},$$
and since $\eta_{k_1[S_1]\cup k_2[S_2]} \circ k_i = \varphi_i \circ \eta_{S_i}$, we have:
$$(\widetilde{\varphi}\circ \widetilde{\psi}) \circ \eta_{k_1[S_1]\cup k_2[S_2]} \circ k_i = \eta_{k_1[S_1]\cup k_2[S_2]}\circ k_i.$$

By the universal property of the coproduct, it follows that $\widetilde{\varphi}\circ \widetilde{\psi}\circ \eta_{k_1[S_1]\cup k_2[S_2]} = \eta_{k_1[S_1]\cup k_2[S_2]}$, and since $\eta_{k_1[S_1]\cup k_2[S_2]}: A_1\otimes_{\infty}A_2 \rightarrow (A_1\otimes_{\infty}A_2)\{ (k_1[S_1]\cup k_2[S_2])^{-1}\}$ is an epimorphism, we have:

$$\widetilde{\varphi}\circ \widetilde{\psi} = {\rm id}_{(A_1\otimes_{\infty}A_2)\{k_1[S_1]\cup k_2[S_2]^{-1}\}}$$

$$\xymatrixcolsep{5pc}\xymatrix{
A_1 \ar[d]^{\eta_{S_1}}\ar[r]^{k_1} & A_1 \otimes_{\infty}A_2 \ar[d]^{\eta_{k_1[S_1]\cup k_2[S_2]}}& A_2 \ar[d]^{\eta_{S_2}} \ar[l]_{k_2}\\
A_1\{{S_1}^{-1}\} \ar[r]^{\varphi_1} \ar[dr]_{\ell_1} \ar@/_2pc/[ddr]_{\varphi_1}  & (A_1\otimes_{\infty}A_2)\{k_1[S_1]\cup k_2[S_2]^{-1} \} \ar[d]^{\widetilde{\psi}} & A_2\{ {S_2}^{-1}\} \ar[dl]^{\ell_2} \ar[l]_{\varphi_2} \ar@/^2pc/[ddl]^{\varphi_2}\\
  & A_1\{ {S_1}^{-1}\}\otimes_{\infty} A_2\{{S_2}^{-1}\} \ar[d]^{\widetilde{\varphi}} & \\
  & (A_1 \otimes_{\infty} A_2)\{k_1[S_1]\cup k_2[S_2]^{-1} \}& }$$

\textbf{Claim:} $\widetilde{\psi} \circ \widetilde{\varphi} = {\rm id}_{A_1\{ {S_1}^{-1}\}\otimes_{\infty}A_2\{ {S_2}^{-1}\}}$.\\

Note that, by the universal property of the coproduct:

$$\xymatrixcolsep{5pc}\xymatrix{
A_1\{ {S_1}^{-1}\} \ar[dr]^{\ell_1} & \\
 & A_1\{ {S_1}^{-1}\} \otimes_{\infty} A_2\{ {S_2}^{-1}\}\\
A_2\{ {S_2}^{-1}\} \ar[ur]_{\ell_2} &
}$$
given the morphisms $\ell_i : A_i\{ {S_i}^{-1}\} \to A_1\{ {S_1}^{-1}\}\otimes_{\infty}A_2\{ {S_2}^{-1}\}$ for $i=1,2$, there is a unique $\mathcal{C}^{\infty}-$homomorphism, namely:
$${\rm id}_{A_1\{ {S_1}^{-1}\}\otimes_{\infty}A_2\{ {S_2}^{-1}\}}: A_1\{ {S_1}^{-1}\}\otimes_{\infty}A_2\{ {S_2}^{-1}\} \to A_1\{ {S_1}^{-1}\}\otimes_{\infty}A_2\{ {S_2}^{-1}\}$$
such that the following diagram commutes:

$$\xymatrixcolsep{5pc}\xymatrix{
A_1\{{S_1}^{-1}\} \ar[r]^{\ell_1} \ar@/_2pc/[dr]_{\ell_1}& A_1\{{S_1}^{-1}\}\otimes_{\infty} A_2\{ {S_2}^{-1}\} \ar[d]^{{\rm id}_{A_1\{{S_1}^{-1}\}\otimes_{\infty} A_2\{{S_2}^{-1}\}}} & A_2\{{S_2}^{-1}\} \ar[l]_{\ell_2} \ar@/^2pc/[dl]^{\ell_2}\\
 & A_1\{{S_1}^{-1}\}\otimes_{\infty}A_2\{{S_2}^{-1}\} & }$$

We are going to show that for $i=1,2$, we have:
$$(\widetilde{\psi}\circ \widetilde{\varphi}) \circ \ell_i = \ell_i.$$

We have:

$$\widetilde{\varphi} \circ \ell_i = \varphi_i$$

and composing both sides of the above equation with $\widetilde{\psi}$ yields:
$$\widetilde{\psi} \circ \widetilde{\varphi} \circ \ell_i = \widetilde{\psi} \circ \varphi_i.$$
Composing both sides of the above equation with $\eta_{S_i}$ yields
$$(\widetilde{\psi}\circ \widetilde{\varphi})\circ \ell_i \circ \eta_{S_i} = \widetilde{\psi}\circ \varphi_i \circ \eta_{S_i}.$$
Since $\varphi_i \circ \eta_{S_i} = \eta_{k_1[S_1]\cup k_2[S_2]} \circ k_i$, we have:
$$(\widetilde{\psi}\circ \widetilde{\varphi})\circ \ell_i \circ \eta_{S_i} = \widetilde{\psi} \circ \eta_{k_1[S_1]\cup k_2[S_2]}\circ k_i$$
and since $\widetilde{\psi} \circ \eta_{k_1[S_1]\cup k_2[S_2]} = \psi$, it follows that
$$(\widetilde{\psi}\circ \widetilde{\varphi})\circ \ell_i \circ \eta_{S_i} =  \psi \circ k_i$$
But $\psi \circ k_i = \ell_i \circ \eta_{S_i}$ for $i=1,2$, so:
$$(\widetilde{\psi}\circ \widetilde{\varphi}) \circ \ell_i \circ \eta_{S_i} = \ell_i \circ \eta_{S_i}$$
and since $\eta_{S_i}$ is an epimorphism for $i=1,2$, we have:
$$(\widetilde{\psi}\circ \widetilde{\varphi}) \circ \ell_i = \ell_i.$$

Hence, by the universal property of the coproduct, $\widetilde{\psi}\circ \widetilde{\varphi} = {\rm id}_{A_1\{ {S_1}^{-1}\}\otimes_{\infty}\{ {S_2}^{-1}\}}$.\\
\end{proof}

\begin{remark}With an analogous proof, we can obtain a more general result: given any set of pairs, $\{ (A_i,S_i) | i \in I\}$, we have:

$$\bigotimes_{\infty}A_i\{ {S_i}^{-1}\} \cong \left( \bigotimes_{\infty} A_i\right)\left\{ \left(\bigcup_{i \in I}k_i[S_i]\right)^{-1}\right\}$$

where $k_i: A_i \rightarrow \otimes_{\infty}A_i$, for each $i$, is the canonical $\mathcal{C}^{\infty}-$homomorphism into the coproduct.
\end{remark}

\begin{theorem}\label{Newt}Let $B$ be the directed colimit of a system $\{ A_{\ell} \stackrel{t_{\ell j}}{\rightarrow} A_j | \ell, j \in I\}$ of $\mathcal{C}^{\infty}-$rings, that is,

$$\xymatrixcolsep{5pc}\xymatrix{
 & B =\varinjlim_{\ell \in I} A_{\ell} & \\
A_i \ar[ur]_{t_i} \ar[rr]^{t_{ij}} & & A_j \ar[ul]^{t_j}
}$$

is a limit co-cone. Given any $u \in B$, there are $j \in I$ and $u_j \in A_j$ such that $t_j(u_j) = u$.\\

Under those circumstances, we have:

$$\varinjlim_{k \geq j} A_{k}\{ u_{k}^{-1} \} \cong B \{ u^{-1}\}$$
where $u_k = t_{jk}(u_j)$.
\end{theorem}
\begin{proof}
Let $\ell \geq k \geq j$. Note that for every  $m \geq j$ we have $u = t_j(u_j) = t_m \circ t_{jm}(u_j):= t_m(u_m)$.\\

Since $\eta_{u_{\ell}} \circ t_{k \ell}(u_k) \in A_{\ell}\{ u_{\ell}^{-1} \}^{\times}$, there is a unique $\mathcal{C}^{\infty}-$homomorphism $s_{k \ell} : A_k\{ u_k^{-1} \} \to A_{\ell}\{ u_{\ell}^{-1}\}$ such that:
$$\xymatrixcolsep{5pc}\xymatrix{
A_k \ar[r]^{\eta_{u_k}} \ar[d]^{t_{k \ell}} & A_k\{u_k^{-1}\} \ar[d]^{\exists ! s_{k \ell}}\\
A_{\ell} \ar[r]^{\eta_{u_{\ell}}} & A_{\ell}\{ u_{\ell}^{-1} \}
}$$

commutes.

Note also that for every $k \geq j$ we have:
$$\eta_u \circ t_k (u_k) = \eta_u(t_k(u_k)) = \eta_u(u) \in B\{ u^{-1}\}^{\times},$$
so by the universal property of $\eta_{u_k}: A_k \to A_k\{u_k^{-1}\}$, there is a unique $s_k: A_k \{ u_k^{-1}\} \to B\{ u^{-1}\}$ such that the following diagram commutes:

$$\xymatrixcolsep{5pc}\xymatrix{
A_k \ar[r]^{\eta_{u_k}} \ar[dr]_{\eta_u \circ t_k} & A_k\{ u_k^{-1}\} \ar[d]^{s_k}\\
    & B\{ u^{-1}\}
}$$

We claim that for every $\ell \geq k \geq j$ the following triangle commutes:

$$\xymatrixcolsep{5pc}\xymatrix{
 & B\{ u^{-1}\} & \\
A_k\{ u_k^{-1}\}\ar[ur]^{s_k} \ar[rr]^{s_{k \ell}} & & A_{\ell}\{ u_{\ell}^{-1}\} \ar[ul]_{s_{\ell}}
}$$

We are going to prove this by dividing our argumentation in two claims.\\

\textbf{Claim 1:} The following triangle commutes:

$$\xymatrixcolsep{5pc}\xymatrix{
A_k \ar[r]^{\eta_{u_k}} \ar[dr]_{\eta_u \circ t_k} & A_k \{ u_k^{-1}\} \ar[d]^{s_{\ell} \circ s_{k \ell}}\\
  & B\{ u^{-1}\}
}$$

We have, for every $\ell \geq k \geq j$:

\begin{multline*}
  (s_{\ell} \circ s_{k \ell}) \circ \eta_{u_k} = s_{\ell} \circ (s_{k\ell} \circ \eta_{u_k}) = s_{\ell} \circ (\eta_{u_{\ell}} \circ t_{k \ell}) = \\
  =(s_{\ell} \circ \eta_{u_{\ell}}) \circ t_{k \ell} = (\eta_u \circ t_{\ell})\circ t_{k \ell} = \eta_u \circ (t_{\ell} \circ t_{k \ell}) = \eta_u \circ t_k = s_k \circ \eta_{u_k}
\end{multline*}

so

$$(s_{\ell} \circ s_{k \ell}) \circ \eta_{u_k} = s_k \circ \eta_{u_k}$$

Since $\eta_{u_k}$ is an epimorphism, it follows that:

$$s_{\ell}\circ s_{k \ell} = s_k$$

thus

$$(\forall k, \ell \geq j)(s_{\ell} \circ s_{k \ell} = s_k)$$

\textbf{Claim 2:} Given any other co-cone $\{ v_k: A_k\{ u_k^{-1}\} \to C | k \geq j \}$ such that for every $\ell \geq k \geq j$:

$$v_{\ell} \circ s_{k \ell} = v_k$$
there is a unique $\mathcal{C}^{\infty}-$homomorphism:
$$\widetilde{v}: B\{ u^{-1}\} \to C$$
such that for every $\ell \geq j$
$$v_{\ell} = \widetilde{v} \circ s_{\ell}$$

In other words, there is a unique $\widetilde{v}: B\{ u^{-1}\} \to C$ such that the following diagram commutes:

$$\xymatrixcolsep{5pc}\xymatrix{
 & C & \\
 & B\{ u^{-1}\} \ar[u]^{\widetilde{v}} & \\
A_k\{u_k^{-1}\} \ar@/^2pc/[uur]^{v_k} \ar[rr]_{s_{k \ell}} \ar[ur]^{s_k} & & A_{\ell}\{ u_{\ell}^{-1}\} \ar@/_2pc/[uul]_{v_{\ell}} \ar[ul]_{s_{\ell}}
}$$

Considering the diagram:

$$\xymatrixcolsep{5pc}\xymatrix{
 & C & \\
 & B\{ u^{-1}\} \ar[u]^{\widetilde{v}} & \\
A_k\{u_k^{-1}\} \ar@/^2pc/[uur]^{v_k} \ar[rr]_{s_{k \ell}} \ar[ur]^{s_k} & & A_{\ell}\{ u_{\ell}^{-1}\} \ar@/_2pc/[uul]_{v_{\ell}} \ar[ul]_{s_{\ell}}\\
A_k \ar[u]_{\eta_{u_k}} \ar[rr]^{t_k\ell} & & A_{\ell} \ar[u]^{\eta_{u_{\ell}}}
}$$

and noting that for every $\ell \geq k \geq j$, $(v_{\ell}\circ \eta_{u_{\ell}})\circ t_{k \ell} = (v_k \circ \eta_{u_k})$, by the universal property of the colimit $B$, there is a unique $\mathcal{C}^{\infty}-$homomorphism $\psi: B \to C$ such that for every $k \geq j$, $\psi \circ t_k = (v_k \circ \eta_{u_k})$

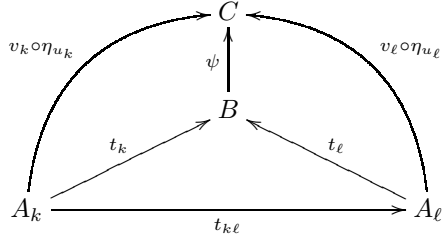
$$\xymatrixcolsep{5pc}\xymatrix{
 & C & \\
 & B \ar[u]^{\psi} & \\
A_k \ar@/^2pc/[uur]^{v_k \circ \eta_{u_k}} \ar[rr]_{t_{k \ell}} \ar[ur]^{t_k} & & A_{\ell} \ar@/_2pc/[uul]_{v_{\ell} \circ \eta_{u_{\ell}}} \ar[ul]_{t_{\ell}}
}$$

Note that $\psi(u) = \psi(t_k(u_k)) = (v_k \circ \eta_{u_k})(u_k) \in C^{\times}$, for every $k \geq j$, and by the universal property of $\eta_u: B \to B\{ u^{-1}\}$, there is a unique $\widetilde{v}: B\{ u^{-1}\} \to C$ such that $\widetilde{v}\circ \eta_u = \psi$.

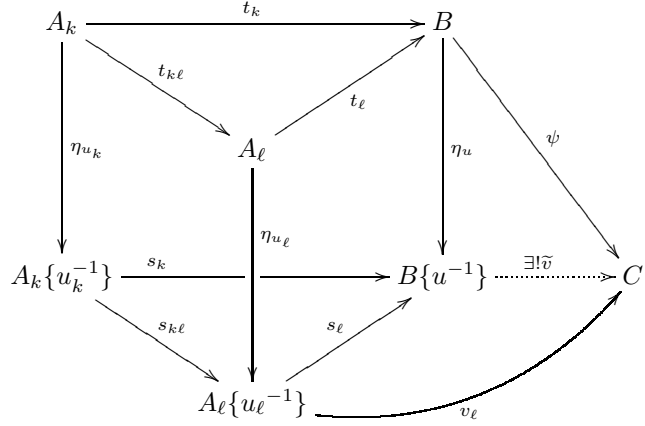
$$\xymatrixcolsep{5pc}\xymatrix @!0 @R=4pc @C=6pc {
   A_k \ar[rr]^{t_k} \ar[rd]^{t_{k\ell}} \ar[dd]^{\eta_{u_k}} && B \ar[dd]^{\eta_u} \ar[rdd]^{\psi} &  \\
   & A_{\ell} \ar[ur]_{t_{\ell}} \ar[dd]^(.35){\eta_{u_{\ell}}}  &  \\
    A_k\{ u_k^{-1}\} \ar[rr]^(.25){s_k} |!{[ur];[dr]}\hole \ar[rd]^{s_{k\ell}} && B\{ u^{-1}\}\ar@{.>}[r]^{\exists ! \widetilde{v}}& C\\
    & A_{\ell}\{ {u_{\ell}}^{-1}\}\ar[ur]^{s_{\ell}} \ar@/_2pc/[urr]_{v_{\ell}}& }$$





\textbf{Claim:} $(\forall \ell \geq j)(\widetilde{v}\circ s_{\ell} = v_{\ell})$.\\

$$\widetilde{v}\circ \eta_u = \psi \Rightarrow (\widetilde{v}\circ \eta_u)\circ t_{\ell} = \psi \circ t_{\ell} \iff \widetilde{v}\circ (\eta_u \circ t_{\ell}) = v_k \circ \eta_{u_{\ell}}  \iff \widetilde{v} \circ (s_{\ell} \circ \eta_{u_{\ell}}) = v_{\ell} \circ \eta_{u_{\ell}}$$

hence:

\begin{equation}\label{sofia}
(\widetilde{v}\circ s_{\ell})\circ \eta_{u_{\ell}} = v_{\ell} \circ \eta_{u_{\ell}}.
\end{equation}

Given any $\alpha \in A_{\ell}\{ u_{\ell}^{-1}\}$, there are $x_{\ell} \in A_{\ell}$ and $y_{\ell} \in \{ u_{\ell}\}^{\infty-{\rm sat}}$ such that $\alpha = \frac{\eta_{u_{\ell}}(x_{\ell})}{\eta_{u_{\ell}}(y_{\ell})}$. \\

We have:

\begin{multline*}(\widetilde{v}\circ s_{\ell})(\alpha) = (\widetilde{v}\circ s_{\ell})\left( \dfrac{\eta_{u_{\ell}}(x_{\ell})}{\eta_{u_{\ell}}(y_{\ell})}\right) = \\
= (\widetilde{v}\circ s_{\ell})(\eta_{u_{\ell}}(x_{\ell})\cdot \eta_{u_{\ell}}(y_{\ell})^{-1}) = (\widetilde{v} \circ s_{\ell} \circ \eta_{u_{\ell}})(x_{\ell})\cdot (\widetilde{v} \circ s_{\ell} \circ \eta_{u_{\ell}})(y_{\ell})^{-1} =\\
= \dfrac{(\widetilde{v}\circ s_{\ell} \circ \eta_{u_{\ell}})(x_{\ell})}{(\widetilde{v}\circ s_{\ell} \circ \eta_{u_{\ell}})(y_{\ell})}
\end{multline*}

and by the equation \eqref{sofia},

$$\dfrac{(\widetilde{v}\circ s_{\ell} \circ \eta_{u_{\ell}})(x_{\ell})}{(\widetilde{v}\circ s_{\ell} \circ \eta_{u_{\ell}})(y_{\ell})} = \dfrac{(v_{\ell} \circ \eta_{u_{\ell}})(x_{\ell})}{(v_{\ell} \circ \eta_{u_{\ell}})(y_{\ell})} = v_{\ell} \left( \dfrac{\eta_{u_{\ell}}(x_{\ell})}{\eta_{u_{\ell}}(y_{\ell})}\right) = v_{\ell}(\alpha)$$

Hence, for every $\alpha \in A_{\ell}\{ {u_{\ell}}^{-1}\}$, $(\widetilde{v}\circ s_{\ell})(\alpha) = v_{\ell}(\alpha)$.


It follows, then, that if $B = \varinjlim_{i \in I} A_i$, then for any $u \in B$, $B\{ u^{-1}\} = \varinjlim_{i \in I} A_i\{ u_i^{-1}\}$.
\end{proof}

\begin{theorem}\label{2137}Let $(I, \leq)$ be a directed poset and let $(A,S)$ be the directed colimit of a system $\{ (A_{\ell}, S_{\ell}) \stackrel{t_{\ell j}}{\rightarrow} (A_j, S_j) | \ell, j \in I\}$ of $\mathcal{C}^{\infty}_{2}$, that is,

$$\xymatrixcolsep{5pc}\xymatrix{
 & (A,S) = \varinjlim_{\ell \in I} (A_{\ell}, S_{\ell}) & \\
(A_i, S_i) \ar[ur]_{t_i} \ar[rr]^{t_{ij}} & & (A_j, S_j) \ar[ul]^{t_j}
}$$

We have:

$$F\left( \varinjlim_{\ell \in I} (A_{\ell}, S_{\ell})\right) \cong \varinjlim_{\ell \in I} F(A_{\ell}, S_{\ell}) = \varinjlim_{\ell \in I} A_{\ell}\{ S_{\ell}^{-1}\}$$

$$\varinjlim_{i \in I} A_i\{ S_i^{-1}\} \cong \left( \varinjlim_{i \in I} A_i\right)\left\{ \left(\varinjlim_{i \in I}S_i\right)^{-1}\right\}$$
\end{theorem}
\begin{proof}(Sketch)\\

Denote by $\alpha_i: A_i \to \varinjlim_{i \in I} A_i$ and $\sigma_i: A_i\{ {S_i}^{-1}\} \to \varinjlim_{i \in I} A_i\{ {S_i}^{-1}\}$ the canonical colimit arrows for each $i \in I$.\\

We make use of the universal property of the colimit of $\mathcal{C}^{\infty}-$rings of fractions in order to show that there is a unique $\mathcal{C}^{\infty}-$homomorphism:

$$\widetilde{\varphi}: \varinjlim_{i \in I} A_i\{ {S_i}^{-1}\} \rightarrow \left( \varinjlim_{i \in I} A_i\right)\left\{ \left( \varinjlim_{i \in I} {S_i} \right)^{-1}\right\}$$

such that the following diagram commutes:

$$\xymatrixcolsep{5pc}\xymatrix{
A_i \ar[r]^{\eta_{S_i}} \ar[d]_{\alpha_i} & A_i\{ {S_i}^{-1}\} \ar[r]^{\sigma_i} & \varinjlim_{i \in I} A_i\{ {S_i}^{-1}\} \ar@{-->}[d]^{\widetilde{\varphi}}\\
\varinjlim_{i \in I}A_i \ar[rr]_{\eta_{\varinjlim_{i \in I}S_i}} & & \left( \varinjlim_{i \in I} A_i\right)\left\{ \left( \varinjlim_{i \in I}S_i\right)^{-1}\right\}}$$

and that there is a unique $\mathcal{C}^{\infty}-$homomorphism:

$$\widetilde{\psi}: \left( \varinjlim_{i \in I} A_i\right)\lbrace \left( \varinjlim_{i \in I}S_i\right)^{-1}\rbrace \rightarrow \varinjlim_{i \in I} A_i\{ {S_i}^{-1}\}$$

such that the following diagram commutes:

$$\xymatrixcolsep{5pc}\xymatrix{
A_i \ar[r]^{\alpha_i} \ar[d]_{\eta_{S_i}} & \varinjlim_{i \in I} A_i \ar[r]^{\eta_{\varinjlim_{i \in I}S_i}} & \left( \varinjlim_{i \in I} A_i\right)\lbrace \left( \varinjlim_{i \in I}S_i\right)^{-1}\rbrace \ar@{-->}[d]^{\widetilde{\psi}} \\
A_i\{ {S_i}^{-1}\} \ar[rr]_{\sigma_i} & & \varinjlim_{i \in I} A_i\{ {S_i}^{-1}\} }$$

It follows, also by a ``uniqueness argument'', that $\widetilde{\varphi}$ and $\widetilde{\psi}$ are inverse $\mathcal{C}^{\infty}-$iso\-morphisms.
\end{proof}

As a consequence of the above proposition we have:

\begin{corollary}In the same context of \textbf{Theorem \ref{2137}}, we have:
  $$\left( \varinjlim_{i \in I} S_i\right)^{\infty-{\rm sat}} = \varinjlim_{i \in I} S_i^{\infty-{\rm sat}} \subseteq \varinjlim_{i \in I}A_i$$
\end{corollary}

Next we present some theorems about quotients of $\mathcal{C}^{\infty}-$rings by their ideals and some  relationships between quotients and $\mathcal{C}^{\infty}-$rings of fractions. We are going to prove that, in the category of $\mathcal{C}^{\infty}-$rings, taking quotients and taking rings of fractions are constructions which commute in a sense that will become clear in the following considerations.\\

Let $A$ be a $\mathcal{C}^{\infty}-$ring, $I$ be any of its ideals, $S \subseteq A$ be any subset of $A$ and consider the canonical map of the ring of fractions of $A$ with respect to $S$:
$$\eta_S : A \to A\{ S^{-1}\},$$
and consider $\widetilde{I_S} = \langle \eta_S[I]\rangle$.\\

Since $(q_{\widetilde{I_S}}\circ \eta_S)[I] \subseteq \widetilde{I_S}$, then by the \textbf{Theorem of Homomorphism} there exists a unique $\eta_{SI} : \dfrac{A}{I} \to \dfrac{A\{ S^{-1}\}}{\widetilde{I_S}}$ such that:

$$\xymatrix{
A \ar[r]^{\eta_S} \ar[d]^{q_I} & A\{ S^{-1}\} \ar[d]^{q_{\widetilde{I_S}}}\\
\dfrac{A}{I} \ar@{.>}[r]^{\eta_{SI}} & \dfrac{A\{ S^{-1}\}}{\widetilde{I_S}}
}$$

commutes.

\begin{proposition}The morphism $\eta_{SI}: \dfrac{A}{I} \to \dfrac{A\{ S^{-1}\}}{\widetilde{I_S}}$ has the same universal property which identifies the ring of quotients of $q_I[S] = S+I$.
\end{proposition}
\begin{proof}
Consider the following diagram:
$$\xymatrix{
A \ar[r]^{\eta_S} \ar[d]^{q_I} & A\{ S^{-1}\} \ar[d]^{q_{\widetilde{I_S}}}\\
\dfrac{A}{I} \ar[dr]^{t}\ar[r]^{\eta_{SI}} & \dfrac{A\{ S^{-1}\}}{\widetilde{I_S}}\\
    & B
}$$
where $t: \dfrac{A}{I} \to B$ is any $\mathcal{C}^{\infty}-$homomorphism such that $t \left[ \left( \dfrac{A}{I}\right)^{\times}\right] \subseteq B^{\times}$. \textit{Ipso facto}, $(t \circ q_I)[S] = t[q_I[S]] \subseteq B^{\times}$, so by the universal property of $\eta_S: A \to A\{ S^{-1}\}$, there exists a unique $\mathcal{C}^{\infty}-$homomorphism $\varphi: A\{ S^{-1}\} \to B$ such that the following diagram commutes:
$$\xymatrix{
A \ar@/_4pc/[ddr]_{t \circ q_I} \ar[r]^{\eta_S} \ar[d]^{q_I} & A\{ S^{-1}\} \ar@/^5pc/[dd]^{\exists ! \varphi} \ar[d]^{q_{\widetilde{I_S}}}\\
\dfrac{A}{I} \ar[rd]^{t}\ar[r]^{\eta_{SI}} & \dfrac{A\{ S^{-1}\}}{\widetilde{I_S}}\\
     & B
}$$

\textbf{Claim:} $\varphi[\widetilde{I_S}] = \{ 0\}$.\\

With regards to this, we note that $(\varphi \circ \eta_S)[I] = (t \circ q_I)[I] = t[\{ 0\}] = \{ 0\}$, so by linearity, $\varphi[\langle \eta_S[I]\rangle] = \varphi[\widetilde{I_S}] = \{ 0\}$.\\

For this reason, and applying the \textbf{Theorem of Homomorphism} we get as a consequence the existence and uniqueness of a $\mathcal{C}^{\infty}-$homomorphism $\overline{\varphi}: \dfrac{A\{ S^{-1}\}}{\widetilde{I_S}} \to B$ such that the following diagram commutes:
$$\xymatrix{
A\{ S^{-1}\} \ar[d]^{q_{\widetilde{I_S}}} \ar[r]^{\varphi} & B \\
\dfrac{A\{ S^{-1}\}}{\widetilde{I_S}} \ar@{.>}[ur]_{\exists ! \overline{\varphi}}
   & B
}$$

\textbf{Claim:} $\overline{\varphi} \circ \eta_{SI} = t$, that is, the following diagram commutes:
$$\xymatrix{
A \ar[r]^{\eta_S} \ar[d]^{q_I} & A\{ S^{-1}\} \ar@/^5pc/[dd]^{\varphi} \ar[d]^{q_{\widetilde{I_S}}}\\
\dfrac{A}{I} \ar[dr]^{t} \ar[r]^{\eta_{SI}} & \dfrac{A\{ S^{-1}\}}{\widetilde{I_S}} \ar[d]^{\overline{\varphi}}\\
    & B}$$

It suffices to show that $\overline{\varphi} \circ \eta_{SI} \circ q_I = t \circ q_I$, for since $q_I$ is an epimorphism we conclude $\overline{\varphi} \circ \eta_{SI} = t$.\\

Indeed, $t \circ q_I = \varphi \circ \eta_S = (\overline{\varphi}\circ q_{\widetilde{I_S}}) \circ \eta_S = \overline{\varphi} \circ \eta_{SI} \circ q_I \Rightarrow t = \overline{\varphi} \circ \eta_{SI}$.\\

Now we claim that if $\psi: \dfrac{A\{ S^{-1}\}}{\widetilde{I_S}} \to B$ is a $\mathcal{C}^{\infty}-$morphism such that:
$$\xymatrix{
\dfrac{A}{I} \ar[dr]^{t} \ar[r]^{\eta_{SI}} & \dfrac{A\{ S^{-1}\}}{\widetilde{I_S}} \ar[d]^{\psi}\\
    & B
}$$
commutes, then $\psi = \varphi$.\\

Verily,
$$\psi \circ \eta_{SI} \circ q_I = t \circ q_I = \varphi \circ \eta_S = \overline{\varphi} \circ q_{\widetilde{I_S}} \circ \eta_S = \overline{\varphi} \circ \eta_{SI} \circ q_I$$
$$\psi \circ q_{\widetilde{I_S}} \circ \eta_S = \psi \circ \eta_{SI} \circ q_I = t \circ q_I = \varphi \circ \eta_S = \overline{\varphi} \circ q_{\widetilde{I_S}} \circ \eta_S$$
hence
$$(\psi \circ q_{\widetilde{I_S}}) \circ \eta_S = (\overline{\varphi} \circ q_{\widetilde{I_S}}) \circ \eta_S$$
and since $\eta_S : A \to A\{ S^{-1}\}$ is a $\mathcal{C}^{\infty}-$epimorphism, we get:
$$\psi \circ q_{\widetilde{I_S}} = \overline{\varphi} \circ q_{\widetilde{I_S}},$$
and because $q_{\widetilde{I_S}}$ is also an epimorphism, we get:
$$\psi = \overline{\varphi}.$$

We have just proved that $\eta_{SI}: \dfrac{A}{I} \to \dfrac{A\{ S^{-1}\}}{\widetilde{I_S}}$ has the same universal property as:
$$\eta_{S+I}: \dfrac{A}{I} \to \left( \dfrac{A}{I}\right)\lbrace \left( S+I \right)^{-1}\rbrace$$
\end{proof}

\begin{proposition}Let $A$ be a $\mathcal{C}^{\infty}-$ring, $I$ any of its ideals and $S \subseteq A$ any of its subsets. Since $(\eta_{S+I} \circ q_I)[A^{\times}] \subseteq \left( \dfrac{A}{I}\lbrace \left( S+I\right)^{-1}\rbrace\right)^{\times}$, there is a unique $\widehat{q}_{IS}$ such that the following diagram commutes:
$$\xymatrixcolsep{5pc}\xymatrix{
A \ar[r]^{\eta_S} \ar[d]^{q_I} & A\{ S^{-1}\} \ar@{.>}[d]^{\widehat{q}_{IS}}\\
\dfrac{A}{I} \ar[r]^{\eta_{S+I}} & \left( \dfrac{A}{I} \right) \lbrace \left( S+I\right)^{-1}\rbrace
}$$

The $\mathcal{C}^{\infty}-$homomorphism $\mu^{-1}: \dfrac{A\{S^{-1} \}}{\widetilde{I_S}} \to \dfrac{A}{I} \lbrace \left( S+I \right)^{-1}\rbrace$ is an isomorphism.
\end{proposition}
\begin{proof}
As a consequence of the uniqueness of the ring of fractions map (up to isomorphism), it follows that there is an isomorphism $\mu: \left( \dfrac{A}{I}\right)\lbrace \left( S + I \right)^{-1}\rbrace \to  \dfrac{A\{ S^{-1}\}}{\widetilde{I_S}}$ such that the following diagrams commute:
$$\begin{array}{lr}
\xymatrixcolsep{3pc}\xymatrix{
\dfrac{A}{I} \ar[r]^(.25){\eta_{S+I}} \ar[dr]^{\eta_{SI}} & \left( \dfrac{A}{I}\right)\lbrace \left( S+I \right)^{-1}\rbrace \ar[d]^{\mu}\\
    & \dfrac{A\{ S^{-1}\}}{\widetilde{I_S}} }
&
\xymatrixcolsep{3pc}\xymatrix{
\dfrac{A}{I} \ar[r]^{\eta_{SI}} \ar[dr]^{\eta_{S+I}} & \dfrac{A\{ S^{-1}\}}{\widetilde{I_S}} \ar[d]^{\mu^{-1}}\\
    & \left( \dfrac{A}{I}\right)\lbrace \left( S+I \right)^{-1}\rbrace
}
\end{array}$$

We claim that the upper right triangles of the following diagrams commute:
$$\begin{array}{lr}
\xymatrixcolsep{3pc}\xymatrix{
A \ar[d]^{q_I} \ar[r]^{\eta_S} & A\{ S^{-1}\} \ar[d]^{\widehat{q}_{IS}} \ar[r]^{q_{\widetilde{I_S}}} & \dfrac{A\{ S^{-1}\}}{\widetilde{I_S}} \ar[ld]^{\mu^{-1}}\\
\dfrac{A}{I} \ar@/_5pc/[rru]^{\eta_{SI}} \ar[r]^{\eta_{S+I}} & \dfrac{A}{I} \lbrace \left( S+I\right)^{-1}\rbrace &
}
&
\xymatrixcolsep{3pc}\xymatrix{
A \ar[d]^{q_I} \ar[r]^{\eta_S} & A\{ S^{-1}\} \ar[d]^{\widehat{q}_{IS}} \ar[r]^{q_{\widetilde{I_S}}} & \dfrac{A\{ S^{-1}\}}{\widetilde{I_S}} \\
\dfrac{A}{I} \ar@/_5pc/[rru]^{\eta_{SI}} \ar[r]^{\eta_{S+I}} & \dfrac{A}{I} \lbrace \left( S+I\right)^{-1}\rbrace \ar[ur]^{\mu} &
}
\end{array}$$

Indeed, since the left square commutes,
$$q_{\widetilde{I_S}} \circ \eta_S = \eta_{SI} \circ q_I = \mu \circ \eta_{S+I} \circ q_I = \mu \circ \widehat{q}_{IS} \circ \eta_S,$$
and since $\eta_S$ is an epimorphism,
$$q_{\widetilde{I_S}}= \mu \circ \widehat{q}_{IS}.$$

Since $\mu^{-1}$ is an isomorphism and $q_{\widetilde{I_S}}$ is an epimorphism, it follows that $\widehat{q}_{IS} = \mu^{-1} \circ q_{\widetilde{q_{I_S}}}$ is an epimorphism. As a result of this, $\widetilde{I_S} = \ker (\widehat{q}_{IS})$, and due to the uniqueness of que quotient morphism, it follows that $\mu^{-1}$ is the quotient morphism, so the result is proved.
\end{proof}

\begin{corollary}\label{Jeq}Let $A$ be a $\mathcal{C}^{\infty}-$ring, $I$ any of its ideals and $S \subseteq A$ any of its subsets. There are unique isomorphisms $\mu: \dfrac{A}{I} \lbrace \left( S + I\right)^{-1}\rbrace \to \dfrac{A\{ S^{-1}\}}{\langle \eta_S[I] \rangle}$ such that the following pentagons commute:
$$\begin{array}{lr}
\xymatrixcolsep{3pc}\xymatrix{
A \ar[d]^{q_I} \ar[r]^{\eta_S} & A\{ S^{-1}\} \ar[r]^{q_{\widetilde{I_S}}} & \dfrac{A\{ S^{-1}\}}{\langle \eta_S[I] \rangle} \\
\dfrac{A}{I} \ar[r]^{\eta_{S+I}} & \dfrac{A}{I} \lbrace \left( S + I\right)^{-1}\rbrace \ar[ur]^{\mu}
}&
\xymatrixcolsep{3pc}\xymatrix{
A \ar[d]^{q_I} \ar[r]^{\eta_S} & A\{ S^{-1}\} \ar[r]^{q_{\widetilde{I_S}}} & \dfrac{A\{ S^{-1}\}}{\langle \eta_S[I] \rangle} \ar[dl]^{\mu^{-1}} \\
\dfrac{A}{I} \ar[r]^{\eta_{S+I}} & \dfrac{A}{I} \lbrace \left( S+I\right)^{-1}\rbrace
}
\end{array}$$
\end{corollary}

Now we analyze the concept of a ``$\mathcal{C}^{\infty}-$radical ideal'' in the theory of $\mathcal{C}^{\infty}-$rings, which plays a similar role to the one played by radical ideals in Commutative Algebra. This concept is presented by I. Moerdijk and G. Reyes in \cite{moerdijk1986rings} in 1986 and explored in more details in \cite{rings2}. \\

Contrary to the concepts of $\mathcal{C}^{\infty}-$fields, $\mathcal{C}^{\infty}-$domains and local $\mathcal{C}^{\infty}-$rings, the concept of a $\mathcal{C}^{\infty}-$radical of an ideal can not be brought from Commutative Algebra via the forgetful functor. Recall that the radical of an ideal $I$ of a commutative unital ring $R$ is given by:

$$\sqrt{I} = \{ x \in R | (\exists n \in \mathbb{N})(x^n \in I)\}$$

and this concept can be characterized by:

$$\sqrt{I} = \bigcap \{ \mathfrak{p} \in {\rm Spec}\,(R) | I \subseteq \mathfrak{p}  \} = \{ x \in R | \left( \dfrac{R}{I}\right)[(x+I)^{-1}] \cong 0\}.$$

We use the latter equality in order to motivate our definition.\\

\begin{definition}\label{defrad}Let $A$ be a $\mathcal{C}^{\infty}-$ring and let $I \subseteq A$ be a proper ideal. The \index{$\mathcal{C}^{\infty}-$radical}\textbf{$\mathcal{C}^{\infty}-$radical of $I$} is given by:

$$\sqrt[\infty]{I}:= \{ a \in A | \left( \dfrac{A}{I}\right)\{ (a+I)^{-1}\} \cong 0\}$$
\end{definition}

\begin{proposition}\label{alba}Let $A$ be a $\mathcal{C}^{\infty}-$ring and let $I \subseteq A$ be any ideal. We have the following equalities:
  $$\sqrt[\infty]{I} = \{ a \in A | (\exists b \in I)\wedge(\eta_a(b) \in (A\{ a^{-1}\})^{\times}) \} = \{ a \in A | I \cap \{ a\}^{\infty-{\rm sat}} \neq \varnothing\}$$
  where $\eta_a : A \to A\{ a^{-1}\}$ is the morphism of fractions with respect to $\{ a\}$.\\
\end{proposition}
\begin{proof}
First we show that:
$$\{ a \in A | (\exists b \in I)(b \in A\{ a^{-1}\}^{\times}) \} \subseteq \sqrt[\infty]{I}.$$

Let $a \in A$ be such that $(\exists b \in I)(b \in A\{ a^{-1}\}^{\times})$. We have the following diagram:
$$\xymatrixcolsep{5pc}\xymatrix{
A \ar[d]^{q_I} \ar[r]^{\eta_A} \ar[dr]_{\eta_{\frac{A}{I}}\circ q_I} & A\{ a^{-1}\}\\
\dfrac{A}{I} \ar[r]_(0.5){\eta_{\frac{A}{I}}} & \left( \dfrac{A}{I}\right)\{ (a+I)^{-1}\}
}$$

Note that the composite $\eta_{\frac{A}{I}}\circ q_I: A \to \left( \dfrac{A}{I}\right)\{ (a+I)^{-1} \}$ maps $a$ to an invertible element of $\left( \dfrac{A}{I}\right)\{(a+I)^{-1}\}$, for
$$(\eta_{\frac{A}{I}} \circ q_I)(a) = \eta_{\frac{A}{I}}(a+I) \in \left( \dfrac{A}{I}\{ (a+I)^{-1}\}\right)^{\times}.$$

By the universal property of $\eta_a : A \to A\{ a^{-1}\}$ there is a unique morphism $q_a : A\{ a^{-1}\} \to \left( \dfrac{A}{I}\right)\{ (a+I)^{-1}\}$ such that the following rectangle commutes:

$$\xymatrix{
A \ar[r]^{\eta_a} \ar[d]^{q_I} & A\{ a^{-1}\} \ar@{.>}[d]^{q_a} \\
\dfrac{A}{I} \ar[r]^(0.2){\eta_{\frac{A}{I}}} & \left( \dfrac{A}{I}\right)\{(a+I)^{-1}\}
}$$
that is to say,
$$q_a \circ \eta_a = \eta_{\frac{A}{I}} \circ q_I.$$

Since there is some $b \in I$ such that $\eta_a(b) \in (A\{ a^{-1}\})^{\times}$, we have:
$$0_{\frac{A}{I}} = \eta_{\frac{A}{I}} \circ q_I(b) = q_a \circ \eta_a(b) = q_a(\eta_a(b)) \in \left( \dfrac{A}{I}\{ (a+I)^{-1}\} \right)^{\times}.$$

It follows that $\dfrac{A}{I}\{ (a+I)^{-1}\} \cong 0$, hence $a \in \sqrt[\infty]{I}$.\\

Conversely, we are going to show that:
$$\sqrt[\infty]{I} \subseteq \{ a \in A | (\exists b \in I)(\eta_a(b) \in A\{ a^{-1}\}^{\times}) \}.$$

Given $a \in \sqrt[\infty]{I}$, we have by definition
$$\dfrac{A}{I}\{ (a+I)^{-1}\} \cong 0,$$
so $1_{\frac{A}{I}} = 0_{\frac{A}{I}}$. Thus, the element $1_A \in A$ is such that $q_a(\eta_a(1_A))=1_{\frac{A}{I}} = 0_{\frac{A}{I}} = \eta_{\frac{A}{I}} \circ  q_I(1_A)$.\\

Since by \textbf{Corollary \ref{Jeq}}, taking quotients and taking rings of fractions commute, we have:
$$0 = \left( \dfrac{A}{I}\right)\{ (a+I)^{-1}\} \cong \dfrac{A\{ a^{-1}\}}{\widehat{I}_a},$$

where $\widehat{I}_a = \ker q_a = \langle \eta_a[I]\rangle = \left\{ \dfrac{\eta_a(i)}{\eta_a(c)} | (i \in I) \& (c \in \{ a\}^{\infty-{\rm sat}})\right\}$  (see  \textbf{Theorem \ref{Cindy}})  , the ideal generated by $\eta_a[I]$. Thus
$$\dfrac{A\{a^{-1}\}}{\ker q_a} = \dfrac{A\{a^{-1}\}}{\langle \eta_a[I]\rangle} \cong 0,$$
so $1 \in \widehat{I}_a$, i.e., there are $i \in I$ and $c \in \{ a \}^{\infty-{\rm sat}}$ such that:
$$1 = \dfrac{\eta_a(i)}{\eta_a(c)}.$$

It follows that $\eta_a(i) = \eta_a(c)$, so $\eta_a(i-c)=0$ and there is some $d \in \{ a\}^{\infty-{\rm sat}}$ such that $d\cdot (i - c) = 0$. Thus,
$$\underbrace{d \cdot i}_{\in I} = \underbrace{d \cdot c}_{\in \{ a\}^{\infty-{\rm sat}}}$$

and $I \cap \{ a\}^{\infty-{\rm sat}} \neq \varnothing$.
\end{proof}

\begin{theorem}\label{Hilfssatz} Let $A$ and $B$ be two $\mathcal{C}^{\infty}-$rings, $S \subseteq A$, $p \twoheadrightarrow B$ a surjective map and $\eta_S : A \to A\{ S^{-1}\}$ and $\eta_{p[S]} : B \to B\{ p[S]^{-1}\}$ be the canonical $\mathcal{C}^{\infty}-$ring of fractions. There is a unique surjective map $q: A\{ S^{-1}\} \twoheadrightarrow B\{ p[S]^{-1}\}$ such that the following square commutes:
$$\xymatrixcolsep{5pc}\xymatrix{
A \ar[r]^{\eta_S} \ar@{->>}[d]^{p} & A\{ S^{-1}\} \ar@{->>}[d]^{q}\\
B \ar[r]^{\eta_{p[S]}} & B\{ p[S]^{-1}\}}$$
\end{theorem}
\begin{proof}
Since for all $s \in S$, $\eta_{p[S]}(p(s)) \in (B\{ p[S]^{-1}\})^{\times}$, there must exist a unique $\mathcal{C}^{\infty}-$homomorphism $q: A\{ S^{-1}\} \to B\{ p[S]^{-1}\}$ such that the following diagram commutes:
$$\xymatrix{
A \ar[r]^{\eta_S} \ar[dr]_{\eta_{p[S]} \circ p}& A\{ S^{-1}\} \ar[d]^{q}\\
   & B\{ p[S]^{-1}\}
}$$
Now we need only to prove the map $q$ is a surjective map.\\

Since $p: A \twoheadrightarrow B$ is surjective, by the \textbf{Theorem of the Isomorphism} there exists a unique $\mathcal{C}^{\infty}-$isomorphism $\Phi: \frac{A}{\ker p} \to B$ such that the following diagram commutes:

$$\xymatrix{
A \ar@{->>}[r]^{p} \ar[d]^{\pi} & B \\
\dfrac{A}{\ker p} \ar@{>->>}[ur]^{\Phi} &
}$$

Note that $p[S] = (\Phi \circ \pi)[S] = \Phi[S + \ker p]$, so $\Phi^{-1}[p[S]] = (\Phi^{-1} \circ \Phi)[S + \ker p] = S + \ker p$, and   $\eta_{S+\ker p}[\Phi^{\dashv}[p[S]]] = \eta_{S + \ker p}[S + \ker p] \subseteq  \left( \left( \frac{A}{\ker p}\right)\{ (S + \ker p)^{-1}\}\right)^{\times}$, so there exists a unique $\mathcal{C}^{\infty}-$homomorphism:

$$\Psi : B\{ p[S]^{-1}\} \to \dfrac{A}{\ker p} \{(S+ \ker p)^{-1}\}  $$

such that:

$$\xymatrixcolsep{5pc}\xymatrix{
B \ar[r]^{\eta_{p[S]}}\ar[d]^{\Phi^{-1}} & B\{ p[S]^{-1}\} \ar[d]^{\psi}\\
\dfrac{A}{\ker p}  \ar[r]^(0.2){\eta_{S + \ker p}} &\left( \frac{A}{\ker p}\right)\{ (S + \ker p)^{-1}\}
}$$
commutes.\\

Conversely, since $\Phi[S + \ker p] = p[S]$ and $\eta_{p[S]}[\Phi[S + \ker p]] = \eta_{p[S]}[p[S]] \subseteq (B\{ p[S]^{-1}\})^{\times}$, there is a unique $\mathcal{C}^{\infty}-$homomorphism:

$$\Xi : \left( \frac{A}{\ker p}\right)\{ (S + \ker p)^{-1}\} \to B\{ p[S]^{-1}\}$$

such that:

$$\xymatrixcolsep{5pc}\xymatrix{
\frac{A}{\ker p} \ar[r]^{\eta_{S + \ker p}} \ar[d]_{\Phi} & \left( \frac{A}{\ker p}\right)\{ (S + \ker p)^{-1}\} \ar[d]^{\Xi}\\
B \ar[r]^{\eta_{p[S]}} & B\{ p[S]^{-1}\}}$$
commutes. \\

It is clear that $\Xi = \Psi^{-1}$, so $B\{ p[S]^{-1}\} \cong \left( \frac{A}{\ker p}\right)\{ (S + \ker p)^{-1}\}$.\\

We have the following commutative diagram:

$$\xymatrixcolsep{5pc}\xymatrix{
A \ar[dd]_{\pi} \ar[r]^{\eta_S} & A\{ S^{-1}\} \ar[dd]^{\varphi} \ar@{->>}[dr]^{q_{\langle \eta_S[\ker p]\rangle}} & \\
 & & \dfrac{A\{ S^{-1}\}}{\langle \eta_S[\ker p]\rangle} \ar@{>->>}[dl]^{\mu^{-1}}\\
\dfrac{A}{\ker p} \ar[r]^{\eta_{S + \ker p}} & \left( \dfrac{A}{\ker p}\right)\{ (S + \ker p)^{-1}\}
}$$

so $\varphi = \mu^{-1} \circ q_{\langle \eta_S[\ker p]\rangle}$, hence it is a composition of surjective $\mathcal{C}^{\infty}-$homomorphisms. \\

$$\xymatrixcolsep{5pc}\xymatrix{
A \ar@{->>}[d]_{p} \ar[r]^{\eta_S} & A\{ S^{-1}\} \ar[d]^{q} \ar@/^3pc/[dd]^{\varphi}\\
B \ar@{>->>}[d]_{\Phi^{-1}} \ar[r]^{\eta_{p[S]}} & B\{ p[S]^{-1}\} \ar@{>->>}[d]^{\Psi}\\
\frac{A}{\ker p} \ar[r]^{\eta_{S + \ker p}} & \left( \frac{A}{\ker p}\right)\{ (S + \ker p)^{-1}\}
}$$

Since $\Psi$ is bijective, $\varphi$ is surjective and $\varphi = \Psi \circ q$ (by the uniqueness of $\varphi$), it follows that $q$ is surjective.\\

\end{proof}





\section{Distinguished Classes of $\mathcal{C}^{\infty}-$rings}\label{dccr}

\hspace{0.5cm}In this section we present some distinguished classes of $\mathcal{C}^{\infty}-$rings, such as $\mathcal{C}^{\infty}-$fields, $\mathcal{C}^{\infty}-$domains, $\mathcal{C}^{\infty}-$local rings, reduced $\mathcal{C}^{\infty}-$rings and (briefly) von Neumann regular $\mathcal{C}^{\infty}-$rings.\\

In \cite{moerdijk1986rings}, we find definitions of $\mathcal{C}^{\infty}-$fields, $\mathcal{C}^{\infty}-$domains and $\mathcal{C}^{\infty}-$local rings. We describe these concepts in the following:

\begin{definition}\label{distintos}Let ${\rm \bf CRing}$ be the category of all commutative unital rings, and consider the forgetful functor $\widetilde{U}: \mathcal{C}^{\infty}{\rm \bf Rng} \to {\rm \bf CRing}$. We say that a $\mathcal{C}^{\infty}-$ring $A$ is:
\begin{itemize}
  \item{a \index{$\mathcal{C}^{\infty}-$field}\textbf{$\mathcal{C}^{\infty}-$field} whenever $\widetilde{U}(A) \in {\rm Obj}\,({\rm \bf CRing})$ is a field;}
  \item{a \index{$\mathcal{C}^{\infty}-$domain}\textbf{$\mathcal{C}^{\infty}-$domain} whenever $\widetilde{U}(A) \in {\rm Obj}\,({\rm \bf CRing})$ is a domain;}
  \item{a \index{local $\mathcal{C}^{\infty}-$ring}\textbf{local $\mathcal{C}^{\infty}-$ring} whenever $\widetilde{U}(A) \in {\rm Obj}\,({\rm \bf CRing})$ is a local ring.}
  \item{a \textbf{von Neumann regular $\mathcal{C}^{\infty}-$ring} whenever $\widetilde{U}(A) \in {\rm Obj}\,({\rm \bf CRing})$ is a von Neumann regular ring.}
\end{itemize}
\end{definition}

Thus a $\mathcal{C}^{\infty}-$field, a $\mathcal{C}^{\infty}-$domain and a local $\mathcal{C}^{\infty}-$ring are $\mathcal{C}^{\infty}-$rings such that their underlying $\R-$algebras are  fields, domains and  local rings in the ordinary sense, respectively.\\

The next lemma gives us a general method to obtain finitely generated local $\mathcal{C}^{\infty}-$rings using \index{prime filter}prime filters on the set of the closed parts of $\mathbb{R}^n$.\\

\begin{lemma}\label{backyardigans}Let $\Phi \subseteq \mathcal{P}(\mathbb{R}^n)$ be a prime filter. Then
$$\varinjlim_{U \in \Phi} \mathcal{C}^{\infty}(U)$$
is a local $\mathcal{C}^{\infty}-$ring.
\end{lemma}
\begin{proof}

First we give an explicit description of $\varinjlim_{U \in \Phi} \mathcal{C}^{\infty}(U)$ as:
$$\varinjlim_{U \in \Phi} \mathcal{C}^{\infty}(U) = \dfrac{\bigsqcup_{U \in \Phi} \mathcal{C}^{\infty}(U)}{ \thicksim },$$
where $ \thicksim $ is the equivalence relation on $\bigsqcup_{U \in \Phi} \mathcal{C}^{\infty}(U)$ given by:
$$(\alpha,U) \thicksim (\beta,V) \iff (\exists W \in \Phi)({\rho}^{U}_{W}(\alpha) = \alpha \upharpoonright_{W} = \beta \upharpoonright_{W}= {\rho}^{V}_{W}(\beta)).$$

$$\xymatrixcolsep{5pc}\xymatrix{
 & \dfrac{\bigsqcup_{U \in \Phi}\mathcal{C}^{\infty}(U)}{ \thicksim } & \\
\mathcal{C}^{\infty}(U) \ar@/^1pc/[ur]^{{\rho}_U} \ar[rr]^{{\rho}^{U}_{V}} & & \mathcal{C}^{\infty}(V) \ar@/_1pc/[ul]
}$$

The sum and the product are defined in a rather obvious fashion:

$$\begin{array}{cccc}
    + : & \left( \frac{\bigsqcup_{U \in \Phi}\mathcal{C}^{\infty}(U)}{ \thicksim }\right)\times \left( \frac{\bigsqcup_{U \in \Phi}\mathcal{C}^{\infty}(U)}{ \thicksim }\right) & \rightarrow &  \frac{\bigsqcup_{U \in \Phi}\mathcal{C}^{\infty}(U)}{ \thicksim } \\
     & ([(\alpha, U)], [(\beta, V)]) & \mapsto & [((\alpha + \beta)\upharpoonright_{U \cap V}, U \cap V)]
  \end{array}$$

and:

$$\begin{array}{cccc}
    \cdot : & \left( \frac{\bigsqcup_{U \in \Phi}\mathcal{C}^{\infty}(U)}{ \thicksim }\right)\times \left( \frac{\bigsqcup_{U \in \Phi}\mathcal{C}^{\infty}(U)}{ \thicksim }\right) & \rightarrow &  \frac{\bigsqcup_{U \in \Phi}\mathcal{C}^{\infty}(U)}{ \thicksim } \\
     & ([(\alpha, U)], [(\beta, V)]) & \mapsto & [((\alpha \cdot \beta)\upharpoonright_{U \cap V}, U \cap V)]
  \end{array}$$

Notice that since $\Phi$ is a proper filter, there is no danger of getting $U \cap V = \varnothing$, since $\varnothing \notin \Phi$, and for every $U, V \in \Phi$ we have $U \cap V \in \Phi$.\\

Let $a = (\alpha,U)$ and $b = (\beta, V) \in \bigsqcup_{U \in \Phi} \mathcal{C}^{\infty}(U)$ be such that $[a] + [b] = [1]$, \textit{i.e.}, $(\exists W \in \Phi)((\alpha + \beta)\upharpoonright_{W} = 1\upharpoonright_{W})$.\\

Let $V_{\alpha} = \{ x \in U | \alpha(x) \neq 0 \}$ and $V_{\beta} = \{ x \in U | \beta(x) \neq 0\}$.\\

\textbf{Claim:} $V_{\alpha} \cup V_{\beta} = W$.\\

If $V_{\alpha} \cup V_{\beta} \subsetneq W$, then there exists some $x_0 \in U$ such that $\alpha(x_0)=0$ and $\beta(x_0) = 0$, and we could not have $\alpha(x_0)+ \beta(x_0) = 1$, hence we have the equality.\\

Now we have $V_{\alpha} \cup V_{\beta} = W \in \Phi$, and since $\Phi$ is a prime filter, either $V_{\alpha} \in \Phi$ or $V_{\beta} \in \Phi$. In the former case, $\tilde{a} = (\alpha, V_{\alpha}) \in \left(\bigsqcup_{U \in \Phi} \mathcal{C}^{\infty}(U)\right)^{\times}$, and since $(\alpha, V_{\alpha}) \thicksim (\alpha, U)$, $[a] = [\tilde{a}] \in \left( \dfrac{\bigsqcup_{U \in \Phi}\mathcal{C}^{\infty}(U)}{\thicksim}\right)^{\times}$, and in the latter case we have, analogously, $[b] \in \left( \dfrac{\bigsqcup_{U \in \Phi}\mathcal{C}^{\infty}(U)}{\thicksim}\right)^{\times}$.
\end{proof}

\begin{proposition}\label{quetiro}Let $A$ be a $\mathcal{C}^{\infty}-$ring and $\p \in {\rm Spec}^{\infty}\,(A)$, that is to say that $\p$ is a $\mathcal{C}^{\infty}-$radical ideal. Then $A_{\p} := A\{ {A \setminus \p}^{-1}\}$ is a local $\mathcal{C}^{\infty}-$ring whose unique maximal ideal is given by $\mathfrak{m}_{\p} = \left\{ \frac{\eta_{A\setminus \p}(x)}{\eta_{A\setminus \p}(y)} |  (x \in \p) \& (y \in A \setminus \p)\right\}$.
\end{proposition}
\begin{proof}
 As for the fact of $A_{\p}$ being a local $\mathcal{C}^{\infty}-$ring, cf. \textbf{Lemma 1.1} of \cite{rings2}.\\

Let $A = \mathcal{C}^{\infty}(\mathbb{R}^n)$. We know that
$$\mathcal{C}^{\infty}(\mathbb{R}^n)_{\p} = \varinjlim_{f \notin \p} \mathcal{C}^{\infty}(U_f)$$
where $U_f = {\rm Coz}\,(f) = \{ x \in \mathbb{R}^n | f(x) \neq 0\}$, so it remains only to prove that:
$$\Phi = \{ U \subseteq \mathbb{R}^n | (\exists f \notin \p)(U = U_f) \}$$
is a prime filter, and by \textbf{Lemma \ref{backyardigans}} the will result follow.\\

We have $\mathbb{R}^n = U_1 \in \Phi$ and for every $f,g \in \mathcal{C}^{\infty}(\mathbb{R}^n)$ we have $U_f \cap U_g = U_{f \cdot g} \in \Phi$. Moreover, if $U_f \subseteq V$, taking the smooth characteristic function of $V$, $\chi_V$ we have $Z(\chi_V) \subseteq Z(f)$, (note that  $\widehat{\p} = \{ Z(h) | h \in \p \}$ is a filter).\\

If $V \cup W \in \Phi$, say $V \cup W = U_f$ for some $f \notin \p$, then writing $V = U_{\chi_V}$ and $W = U_{\chi_W}$, then $Z(f) = Z(\chi_V)\cap Z(\chi_U) = Z({\chi_V}^2 + {\chi_U}^2)$, so ${\chi_V}^2+{\chi_U}^2 \notin \p$, since $\p$ is $\mathcal{C}^{\infty}-$radical, hence ${\chi_V}^2 \notin \p$ or ${\chi_W}^2 \notin \p$, that is, $V \in \Phi$ or $W \in \Phi$, so $\Phi$ is prime.\\

Let us show that $\mathfrak{m}_{\p} = A_{\p} \setminus {A_{\p}}^{\times}$.\\

By the very definition of $A_{\p} = A\{ {A \setminus \p}^{-1}\}$, any element of it is of the form $\frac{\eta_{A \setminus \p}(x)}{\eta_{A \setminus \p}(y)}$ for some $x \in A$ and $y \in {A \setminus \p}^{\infty-{\rm sat}}$. Note that if $x', y' \in A$ are  such that $\frac{\eta_{A \setminus \p}(x')}{\eta_{A \setminus \p}(y')}=\frac{\eta_{A \setminus \p}(x)}{\eta_{A \setminus \p}(y)}$, then $(\exists s \in {A \setminus \p}^{\infty-{\rm sat}} = A \setminus \p)(s \cdot (x' \cdot y - x \cdot y')=0)$, so $x'\cdot \underbrace{s \cdot y}_{\in A \setminus \p} = \underbrace{x \cdot s \cdot y'}_{\in \p}$ and $x' \notin A \setminus \p$, thus $x' \in \p$.\\

Let $\frac{\eta_{A \setminus \p}(x)}{\eta_{A \setminus \p}(y)} \in A_{\p}$ be such that $x \in \p$ and $y \in A \setminus \p$. By \textbf{Theorem \ref{cara}}, it follows that $\frac{\eta_{A\setminus \p}(x)}{\eta_{A \setminus \p}(y)} \in A_{\p} \setminus {A_{\p}}^{\times}$ occurs if, and only if, $(\forall d \in A \setminus \p)(\forall c' \in A)(d \cdot c' \cdot x \in \p)$, which is the case since $x \in \p$ and $\p$ is an ideal. Hence:\\

$$\left\{ \frac{\eta_{A\setminus \p}(x)}{\eta_{A\setminus \p}(y)} |  (x \in \p) \& (y \in A \setminus \p)\right\} \subseteq A_{\p} \setminus {A_{\p}}^{\times}.$$

Conversely, suppose
$$\frac{\eta_{A \setminus \p}(x)}{\eta_{A \setminus \p}(y)} \in A_{\p} \setminus {A_{\p}}^{\times},$$

so by \textbf{Theorem \ref{cara}}, $(\forall d \in A \setminus \p)(\forall c' \in A)(d \cdot c' \cdot x \in \p)$. In particular, taking $c'=1$ yields $(\forall d \in A \setminus \p)(d \cdot x \in \p)$. Since $\p$ is a prime ideal, then either $d \in \p$ (which is not the case since $d \in A \setminus \p$) or $x \in \p$. Hence $x \in \p$.
\end{proof}

\begin{proposition}\label{sizi}Let $A$ be a $\mathcal{C}^{\infty}-$ring. The following assertions are equivalent.
\begin{itemize}
\item[i)]{$A$ is a $\mathcal{C}^{\infty}-$field;}
\item[ii)]{For every subset $S \subseteq A\setminus \{0\}$, the canonical map:
$$\begin{array}{cccc}
{\rm Can}_S : & A & \rightarrow & A\{ S^{-1} \}
\end{array}$$
is a $\mathcal{C}^{\infty}-$ring isomorphism;}
\item[iii)]{For any $a \in A\setminus \{ 0\}$, we have that:
$$\begin{array}{cccc}
{\rm Can}_a : & A & \rightarrow & A\{ a^{-1}\}
\end{array}$$ is a $\mathcal{C}^{\infty}-$isomorphism.}
\end{itemize}
\end{proposition}
\begin{proof}
$i) \Rightarrow ii):$ Since $S \subseteq A^{\times}$, it is easy to see that ${\rm Can}_S: A \to A\{ S^{-1}\}$ is isomorphic to ${\rm id}_A: A \to A$, so ${\rm Can}_S : A \to A\{ S^{-1}\}$ is a $\mathcal{C}^{\infty}-$isomorphism.





$ii) \Rightarrow iii)$\, It is immediate, since $\{ a\} \subseteq A^{\times}$.\\

$iii) \Rightarrow i)$ Suppose that for every $a \in A\setminus \{0\}$, ${\rm Can}_a: A \to A\{a^{-1}\}$ is a $\mathcal{C}^{\infty}-$isomorphism. We must show that $a$ has an inverse in $A$. Let $\theta \in A\{a^{-1}\}$ be the inverse of ${\rm Can}_a(a)$, i.e.,
$$\theta \cdot {\rm Can}_a(a) = 1_{U(A)} = {\rm Can}_a(a) \cdot \theta.$$
Since ${\rm Can}_a$ is an isomorphism, there is a unique $b \in A$ such that ${\rm Can}_a(b) = \theta$. Now, ${\rm Can}_{a}(a)\cdot \theta = {\rm Can}_a(a) \cdot {\rm Can}_a(b) = {\rm Can}_a(a \cdot b) = 1_{U(A)\{a^{-1}\}}$, so $a \cdot b = 1_{A}$, thus $b = a^{-1}$.
\end{proof}

In ordinary Commutative Algebra, given an element $x$ of a ring $R$, we say that $x$ is a nilpotent infinitesimal if and only if there is some $n \in \mathbb{N}$ such that $x^n=0$. Let $A$ be a  $\mathcal{C}^{\infty}-$ring and $a \in A$. D. Borisov and K. Kremnizer in \cite{borisov2018beyond} call $a$ an $\infty-$infinitesimal if, and only if $A\{ a^{-1}\} \cong 0$. The next definition describes the notion of a $\mathcal{C}^{\infty}-$ring being free of $\infty-$infinitesimals - which is analogous to the notion of ``reducedness'', of a commutative ring.\\

\begin{definition}A $\mathcal{C}^{\infty}-$ring $A$ is \index{$\mathcal{C}^{\infty}-$reduced}\textbf{$\mathcal{C}^{\infty}-$reduced} if, and only if, $\sqrt[\infty]{(0)} = (0)$.
\end{definition}

\begin{proposition}Every $\mathcal{C}^{\infty}-$field, $A$, is a $\mathcal{C}^{\infty}-$reduced $\mathcal{C}^{\infty}-$ring.
\end{proposition}
\begin{proof}
Suppose $A$ is a $\mathcal{C}^{\infty}-$field. For every $a \in A$, we have $\dfrac{A}{(0)}\{ (a+(0))^{-1}\} \cong A\{ a^{-1}\}$, so $A\{ a^{-1}\} \cong A$ occurs if, and only if, $a \neq 0$. Hence $\sqrt[\infty]{(0)}=(0)$
\end{proof}

\begin{theorem}\label{Cindy}Let $A$ be a $\mathcal{C}^{\infty}-$ring, $S \subseteq A$, and $I \subset A$ any ideal. Then:
$$\langle \eta_S[I]\rangle = \left\{ \frac{\eta_S(b)}{\eta_S(d)} | b \in I \& d \in S^{\infty-{\rm sat}}\right\}$$
\end{theorem}
\begin{proof}
Given $h \in \langle \eta_S[I]\rangle$, there are $n \in \mathbb{N}$, $b_1, \cdots, b_n \in I$ and $\alpha_1, \cdots, \alpha_n \in A\{S^{-1}\}$ such that:

$$h = \sum_{i=1}^{n} \alpha_i \cdot \eta_{S}(b_i) $$

For each $i \in \{ 1, \cdots, n\}$ there are $c_i \in A$ and $d_i \in S^{\infty-{\rm sat}}$ such that:

$$\alpha_i \cdot \eta_S(d_i) = \eta_S(c_i)$$

so

$$h = \sum_{i=1}^{n} \alpha_i \cdot \eta_{S}(b_i) = \sum_{i=1}^{n} \dfrac{1}{\eta_S(d_i)}\cdot \eta_S(c_i)\cdot \eta_S(b_i),$$

and denoting $b_i':= c_i \cdot b_i \in I$, we get:

$$h = \sum_{i=1}^{n} \dfrac{\eta_S(b_i')}{\eta_S(d_i)}$$

For each $i=1, \cdots, n$, let
$$b_i'' = b_i' \prod_{j \neq i} d_j,$$
so
$$h = \dfrac{\eta_S \left( \prod_{j \neq 1}d_j\right) \eta_S(b_1')}{\eta_S(d_1 \cdots d_n)} + \dfrac{\eta_S \left( \prod_{j \neq 2}d_j \right) \eta_S(b_2')}{\eta_S(d_1 \cdots d_n)} + \cdots + \dfrac{\eta_S \left( \prod_{j \neq n}d_j\right) \eta_S(b_n')}{\eta_S(d_1 \cdots d_n)}$$
hence

$$h \cdot \eta_S(d_1 \cdots d_n)  =\eta_S \left( \prod_{j \neq 1}d_j\right) \eta_S(b_1') + \eta_S \left( \prod_{j \neq 2}d_j\right) \eta_S(b_2')+ \cdots + \eta_S \left( \prod_{j \neq n}d_j\right) \eta_S(b_n').$$

Let $b_i'' := \left( \prod_{j \neq i}d_i\right)\cdot \overbrace{b_i'}^{\in I} \in I$, so we have $h \cdot \eta_S(d_1 \cdots d_n) = \sum_{i=1}^{n} \eta_S(b_i'') = \eta_S\left( \sum_{i=1}^{n}b_i''\right)$

Since $\eta_S(d_1 \cdots d_n) \in A\{ S^{-1}\}^{\times}$, $d = d_1\cdots d_n \in S^{\infty-{\rm sat}}$, so taking $b = \sum_{i=1}^{n} b_i''$, we have $b \in I$, since it is a sum of elements of $I$, we can write:

$$h \cdot \eta_S(d) = \eta_S(b),$$

and $h = \dfrac{\eta_S(b)}{\eta_S(d)}$, with $b \in I$ and $d \in S^{\infty-{\rm sat}}$.

The other way round is immediate.
\end{proof}

As a consequence of \textbf{Theorem \ref{Cindy}}, we have:

\begin{corollary}If $A$ is a reduced $\mathcal{C}^{\infty}-$ring, then we have:
$$(\forall a \in A)((0 \in \{ a\}^{\infty-{\rm sat}}) \leftrightarrow (a=0))$$
\end{corollary}

\begin{proposition}\label{griebel}Let $I \subsetneq \mathcal{C}^{\infty}(\mathbb{R}^n)$ any ideal. Then:
$$\widehat{I} = \{ A \subseteq \mathbb{R}^n | A \, \mbox{is closed and}\, (\exists f \in  I)(A = Z(f))\}$$
is a filter on the set of all the closed subsets of $\mathbb{R}^n$. Moreover, $I$ is a proper ideal if, and only if, $\widehat{I}$ is a proper filter.
\end{proposition}
\begin{proof}
First we note that: $\varnothing \notin \widehat{I}$ if, and only if, there is some $f \in I$ such that $Z(f)=\varnothing$, if, and only if, there is some invertible $f \in I$, which happens if, and only if, $ 1 \in I$ so $I$ is not be proper.\\

Since $0 \in I$, we have $\mathbb{R}^n = Z(0) \in \widehat{I}$.\\

Let $A, B \in \widehat{I}$ and let $f,g \in I$ be smooth functions such that $A = Z(f)$ and $B = Z(g)$, so $f^2 + g^2 \in I $. We claim that:
$$Z(f) \cap Z(g) = Z(f^2 + g^2).$$

Indeed, for every $x \in Z(f^2 + g^2)$ we have $(f^2+g^2)(x) = f^2(x) + g^2(x) = (f(x))^2 + (g(x))^2 = 0$, and since $f(x), g(x) \in \mathbb{R}$, $(f(x))^2 + (g(x))^2 = 0$ implies $f(x) = 0 = g(x)$, so $x \in Z(f)\cap Z(g)$. Conversely, given any $x \in Z(f)\cap Z(g)$, $f(x) = 0 = g(x)$, so $f^2(x) + g^2(x) = 0$, so $x \in Z(f) \cap Z(g)$.\\

Now, since $A \cap B = Z(f^2 + g^2)$ and $f^2 + g^2 \in I$, it follows that $A \cap B \in \widehat{I}$.\\

Finally, if $A \in \widehat{I}$ then there is $f \in I$ such that $A = Z(f)$. If $B \subseteq \mathbb{R}^n$ is closed and and $A \subset B$, then there is a smooth characteristic function $\chi_{B} : \mathbb{R}^n \to \mathbb{R}$ such that $Z(\chi_{B})= B)$. Consider the following smooth function:

$$g = f \cdot \chi_{B} \in I.$$

We notice that $Z(g)=Z(f)\cup Z(\chi_B) = A \cup B = B$, so $B \in \widehat{I}$.\\

From the above considerations it follows that $\widehat{I}$ is a  filter on the closed parts of $\mathbb{R}^n$.
 \end{proof}

\begin{proposition}\label{47}Let $\mathcal{F}$ be any filter on the set of all the closed subsets of $\mathbb{R}^n$. Then:
$$\check{\mathcal{F}} = \{ f \in \mathcal{C}^{\infty}(\mathbb{R}^n) | Z(f) \in \mathcal{F} \}$$
is an ideal of $\mathcal{C}^{\infty}(\mathbb{R}^n)$. Moreover, $\mathcal{F}$ is a proper filter if, and only if, $\check{\mathcal{F}}$ is a proper ideal.
\end{proposition}
\begin{proof}
Note that $0 \in \check{\mathcal{F}}$, since $Z(0)= \mathbb{R}^n \in \mathcal{F}$.\\

We claim that for any $f \in \check{\mathcal{F}}$ and any $g \in \mathcal{C}^{\infty}(\mathbb{R}^n)$ we have $f \cdot g \in \check{\mathcal{F}}$.\\

Indeed, since $f \in \check{\mathcal{F}}$, $Z(f) \in \mathcal{F}$. Since $Z(f) \subseteq Z(f)\cup Z(g)$ and $Z(f) \in \mathcal{F}$, it follows that $f\cdot g \in \check{\mathcal{F}}$.\\

Now, given $f,g \in \check{\mathcal{F}}$, $Z(f), Z(g) \in \mathcal{F}$, and since $\mathcal{F}$ is a filter, $Z(f) \cap Z(g) \in \mathcal{F}$. Since $Z(f)\cap Z(g) \subseteq Z(f+g)$, it follows that $f+ g \in \check{\mathcal{F}}$.\\

Finally, $\varnothing \in \mathcal{F}$ occurs if, and only of there is some invertible $f \in \check{\mathcal{F}}$, thus $1 \in \check{\mathcal{F}}$.
\end{proof}

\begin{remark}Let $I \subseteq \mathcal{C}^{\infty}(\mathbb{R}^n)$ be any ideal. Then we have, by definition, $\widehat{I} = \{ A \subseteq \mathbb{R}^n | (\exists h \in I)(A = Z(h)) \}$, and therefore
$$\check{\widehat{I}} = \{ g \in \mathcal{C}^{\infty}(\mathbb{R}^n) | (\exists h \in I)(Z(g) = Z(h)) \}.$$

Let $\mathcal{F}$ be a proper filter on the set of all closed subsets of $\mathbb{R}^n$. Then:

$$\widehat{\check{\mathcal{F}}} = \{ A \subseteq \mathbb{R}^n | (\exists g \in \check{\mathcal{F}})(A=Z(g))\} =  \{ A \subseteq \mathbb{R}^n | (\exists g \in \check{\mathcal{F}})((Z(g) \in \mathcal{F})\&(A=Z(g)))\} = \mathcal{F}$$

since every closed subset $B \subseteq \mathbb{R}^n$ is a zero set of some smooth function.
\end{remark}




\begin{proposition}\label{49}Let $I \subseteq \mathcal{C}^{\infty}(\mathbb{R}^n)$ be an ideal. Then:
$$\check{\widehat{I}} = \{ g \in \mathcal{C}^{\infty}(\mathbb{R}^n) | (\exists h \in I)(Z(g) = Z(h))\} = \sqrt[\infty]{I}$$
In particular, the $\mathcal{C}^{\infty}-$radical of an ideal $I$ of the free $\mathcal{C}^{\infty}-$ring on finitely many generators, $\mathcal{C}^{\infty}(\mathbb{R}^n)$, is again an ideal.
\end{proposition}
\begin{proof}
Let $g \in \check{\widehat{I}}$, so there is some $h \in I$ such that $Z(g) = Z(h)$, so \textbf{a fortiori}, $Z(g) \supseteq Z(h)$, and $(h \displaystyle\upharpoonright_{U_g})$ is invertible in $\mathcal{C}^{\infty}(U_g)$, where $U_g = \mathbb{R}^n \setminus Z(g)$ (see \textbf{Proposition \ref{pani}}). Since there exists $h \in I$ such that $h \in (\mathcal{C}^{\infty}(\mathbb{R}^n)\{ g \}^{-1})^{\times}$ it follows that $g \in \sqrt[\infty]{I}$.\\

Conversely, given $g \in \sqrt[\infty]{I}$, there must be some $h \in I$ such that $h \in (\mathcal{C}^{\infty}(\mathbb{R}^n)\{g^{-1}\})^{\times}$. But such an $h$ must not be zero when restricted to $U_g$ (otherwise $h$ would not be invertible), so $U_h \supseteq U_g$, and therefore $Z(g) \supseteq Z(h)$. Since $\widehat{I}$ is a filter, $Z(h) \in \widehat{I}$ and $Z(g) \supseteq Z(h)$, it follows that $Z(g) \in \widehat{I}$. But if $Z(g) \in \widehat{I}$ then there is $h' \in I$ such that $Z(g) = Z(h')$ so $g \in \check{\widehat{I}}$.
\end{proof}

The following proposition is a consequence of \textbf{Proposition \ref{49}}

\begin{corollary}\label{Guia}Let $A = \mathcal{C}^{\infty}(\mathbb{R}^n)$ be a finitely generated free $\mathcal{C}^{\infty}-$ring. $I \subseteq \mathcal{C}^{\infty}(\mathbb{R}^n)$ is a $\mathcal{C}^{\infty}-$radical ideal, that is, $\sqrt[\infty]{I}=I$, if, and only if:
$$(\forall g \in \mathcal{C}^{\infty}(\mathbb{R}^n))((g \in I) \leftrightarrow (\exists f \in I)(Z(f)=Z(g))).$$
\end{corollary}

The following proposition is a consequence of a comment made by Moerdijk and Reyes in the page 330 of \cite{rings2}:

\begin{corollary}\label{Gaia}Let $\mathcal{C}^{\infty}(\mathbb{R}^n)$ be the free $\mathcal{C}^{\infty}-$ring and let $I \subseteq \mathcal{C}^{\infty}(\mathbb{R}^n)$ be a finitely generated ideal, that is,
$$I = \langle g_1, \cdots, g_k\rangle,$$
for some $g_1, \cdots, g_k \in \mathcal{C}^{\infty}(\mathbb{R}^n)$.\\

$I$ is a $\mathcal{C}^{\infty}-$radical ideal if, and only if:

$$(\forall x \in Z(g_1, \cdots, g_k))(f(x)=0) \to (f \in I)$$
\end{corollary}
\begin{proof}
Suppose $I = \sqrt[\infty]{I}$ and let $f \in \mathcal{C}^{\infty}(\mathbb{R}^n)$ be such that:
$$(\forall x \in Z(g_1, \cdots, g_k))(f(x)=0).$$

We have $g=g_1^2 + \cdots + g_k^2 \in I$ such that $Z(g)\subseteq Z(f)$, so $Z(f)\in \widehat{I}$ (since $\widehat{I}$ is a filter) and by \textbf{Proposition \ref{49}}, $f \in \check{\widehat{I}} = \sqrt[\infty]{I}=I$.\\

Now, suppose

$$(\forall x \in Z(g_1, \cdots, g_k))(f(x)=0) \to (f \in I).$$

Given $h \in \sqrt[\infty]{I} = \check{\widehat{I}}$, we have $Z(h) \in \widehat{I}$, so there exists some $g \in I$ such that $Z(h)=Z(g)$, thus $(\forall x \in Z(g_1, \cdots, g_k))(h(x)=0)$. By hypothesis, this means that $h \in I$, so $\sqrt[\infty]{I} \subseteq I$. Since $I \subseteq \sqrt[\infty]{I}$ always holds, it follows that $I$ is a $\mathcal{C}^{\infty}-$radical ideal.
\end{proof}

By \textbf{Proposition 4.5} of \cite{moerdijk2013models}, it follows that any finitely generated $\mathcal{C}^{\infty}-$radical ideal $I$ of a free $\mathcal{C}^{\infty}-$ring on $n$ generators, $\mathcal{C}^{\infty}(\mathbb{R}^n)$
we have:

$$(\forall x \in Z(I))(f\upharpoonright_x \in I\upharpoonright_x \to f \in I),$$

where $Z(I):= \bigcap_{f \in I} Z(f).$

As a consequence, we have a rudimentary particular version of the weak \textbf{Nullstellensatz}:

\begin{proposition}\label{Nullstellensatz}For any finitely generated $\mathcal{C}^{\infty}-$radical ideal $I$ of $\mathcal{C}^{\infty}(\mathbb{R}^n)$, we have:
$$1 \in I \iff Z(I)=\varnothing$$
\end{proposition}
\begin{proof}
Let $I = \langle g_1, \cdots, g_k\rangle$ be a $\mathcal{C}^{\infty}-$radical ideal of $\mathcal{C}^{\infty}(\mathbb{R}^n)$, so
$$(\forall x \in Z(g_1, \cdots, g_k))(f(x)=0) \to (f \in I)$$
Suppose $Z(g_1, \cdots,g_k)=\varnothing$, so we have, for each $f \in \mathcal{C}^{\infty}(\mathbb{R}^n)$:

$$\varnothing = Z(g_1, \cdots, g_k) \subseteq Z(f),$$

so

$$(\forall x \in Z(g_1\cdots g_k))(f(x)=0)$$

Thus $\mathcal{C}^{\infty}(\mathbb{R}^n) \subseteq I$ and, in particular, $1 \in I$.

On the other hand, if $1 \in I$, since $1^{\dashv}[\{ 0\}] = \varnothing \in \{ f^{\dashv}[\{ 0\}] \mid f \in I \}$, so $Z(I)=\bigcap \{ f^{\dashv}[\{ 0\}] \mid f \in I\} = \varnothing$.
\end{proof}

\begin{example}As a consequence of the fact pointed out by Moerdijk and Reyes in \cite{rings2}, any countably generated prime ideal of $\mathcal{C}^{\infty}(\mathbb{R}^n)$, for $n \in \mathbb{N}$, is a $\mathcal{C}^{\infty}-$radical ideal.
\end{example}

Let $\mathfrak{F}$ be the set of all the filters on the closed parts of $\mathbb{R}^n$ and $\mathfrak{I}$ be the set of all the ideals of $\mathcal{C}^{\infty}(\mathbb{R}^n)$. We have, so far, established that for every ideal $I \subseteq \mathcal{C}^{\infty}(\mathbb{R}^n)$ we have $\sqrt[\infty]{I} = \check{\widehat{I}}$ and for every filter $\mathcal{F} \in \mathfrak{F}$ we have $\widehat{\check{\mathcal{F}}}$.\\

Consider the following diagram:

$$\xymatrix{\mathfrak{F} \ar@/^/[rr]|{\vee} && \ar@/^/[ll]|{\wedge} \mathfrak{I}}$$

In what follows, we show that $\wedge \vdash \vee$ is a \index{Galois connection}Galois connection.

\begin{proposition}\label{pseu}The  adjunction $\wedge \vdash \vee$ is a covariant Galois connection between the posets $(\mathfrak{F}, \subseteq)$ and $(\mathfrak{I}, \subseteq)$, i.e.,
\begin{itemize}
  \item[a)]{Given $\mathcal{F}_1, \mathcal{F}_2 \in \mathfrak{F}$ such that $\mathcal{F}_1 \subseteq \mathcal{F}_2$ then $\check{\mathcal{F}_1} \subseteq \check{\mathcal{F}_2}$;}
  \item[b)]{Given $I_1, I_2 \in \mathfrak{I}$ such that $I_1 \subseteq I_2$ then $\widehat{I}_1 \subseteq \widehat{I}_2$;}
  \item[c)]{For every $\mathcal{F} \in \mathfrak{F}$ and every $I \in \mathfrak{I}$ we have:
  $$\widehat{I} \subseteq \mathcal{F} \iff I \subseteq \check{\mathcal{F}}.$$}
\end{itemize}
\end{proposition}
\begin{proof}
Ad a):  Given any $f \in \check{\mathcal{F}_1}$ we have $Z(f) \in \mathcal{F}_1$, and since $\mathcal{F}_1 \subseteq \mathcal{F}_2$, $Z(f) \in \mathcal{F}_2$, so $f \in \check{\mathcal{F}_2}$.\\

Ad b) Given $A \in \widehat{I_1}$, there is $f_1 \in I_1$ such that $A = Z(f_1)$. Since $f_1 \in I_1 \subseteq I_2$, $f_1 \in I_2$, so $A = Z(f_1) \in \widehat{I_2}$.\\

Ad c): First we show that $\widehat{I} \subseteq \mathcal{F} \Rightarrow I \subseteq \check{\mathcal{F}}$.\\

Given any $f \in I$, $Z(f) \in \widehat{I}$, and since $\widehat{I} \subseteq \mathcal{F}$, $Z(f) \in \mathcal{F}$ so $f \in \check{\mathcal{F}}$, by definition.\\

Conversely we show that $I \subseteq \check{\mathcal{F}} \Rightarrow \widehat{I} \subseteq \mathcal{F}$.\\

Given $A \in \widehat{I}$, there exists some $h \in I$ such that $A = Z(h)$, so $h \in \check{\mathcal{F}}$ and, by definition, $Z(h) = A \in \mathcal{F}$.
\end{proof}

\begin{proposition}\label{Garland}Let $A = \mathcal{C}^{\infty}(\R^n)$ for some $n \in \mathbb{N}$. The Galois connection $\wedge \vdash \vee$ establishes a bijective correspondence between:
\begin{itemize}
  \item[(a)]{proper filters of $(\mathfrak{F}, \subseteq)$ and proper ideals of $(\mathfrak{I}, \subseteq)$;}
  \item[(b)]{maximal filters of $(\mathfrak{F}, \subseteq)$ and maximal ideals of $(\mathfrak{I}, \subseteq)$;}
  \item[(c)]{prime filters of $(\mathfrak{F}, \subseteq)$ and prime ideals of $(\mathfrak{I}, \subseteq)$.}
  \item[(d)]{filters on the closed parts of $\mathbb{R}^n$, $\mathfrak{F}$, and the set of all $\mathcal{C}^{\infty}-$radical ideals of $\mathcal{C}^{\infty}(\mathbb{R}^n)$, $\mathfrak{I}^{\infty} = \{ I \subseteq \mathcal{C}^{\infty}(\mathbb{R}^n) | \sqrt[\infty]{I} = I \}$.}
\end{itemize}
\end{proposition}
\begin{proof}

Ad (a): By \textbf{Proposition \ref{griebel}}, $I$ is a proper ideal if, and only if, $\check{I}$ is a proper filter, and by \textbf{Proposition \ref{47}}, $F$ is a proper filter if, and only if, $\widehat{\mathcal{F}}$ is a proper ideal. In particular, whenever $I$ is a proper ideal, $\sqrt[\infty]{I}=\widehat{\check{I}}$ is a proper ideal.\\

Ad (b): If $I$ is a maximal ideal then it is proper, so $\check{I}$ is a proper filter. By item (a), $\widehat{\check{I}}$ is a proper ideal. Since $I \subseteq \widehat{\check{I}}$ and $I$ is a maximal ideal, we have $I = \widehat{\check{I}} = \sqrt[\infty]{I}$. In particular, every maximal ideal of $\mathcal{C}^{\infty}(\mathbb{R}^n)$ is a $\mathcal{C}^{\infty}-$radical ideal. \\

On the other hand, for any filter $\mathcal{F}$ we have $\mathcal{F} = \check{\widehat{\mathcal{F}}}$. If $\check{I}\subseteq \mathcal{F}$ for some proper filter $\mathcal{F}$, then $I = \widehat{\check{I}} \subseteq \widehat{\mathcal{F}}$, and since $\widehat{\mathcal{F}}$ is a proper ideal, by (a), it follows that:

$$I = \widehat{\mathcal{F}} \Rightarrow \check{I} = \check{\widehat{\mathcal{F}}} = \mathcal{F},$$

so $\check{I}$ is a maximal filter.\\

Conversely, suppose $\mathcal{F}$ is a maximal filter. If $J$ is some proper ideal of $\mathcal{C}^{\infty}(\mathbb{R}^n)$ such that $\widehat{\mathcal{F}} \subseteq J$, then   $\check{\widehat{\mathcal{F}}} \subseteq \check{J}$, so $\mathcal{F}\subseteq \check{J}$. Since $\check{J}$ is a proper filter, we have $\mathcal{F} = \check{J}$, so $\widehat{\mathcal{F}} = \widehat{\check{J}}$. Since $J \subseteq \widehat{\check{J}}$, we have $\widehat{\mathcal{F}} = \widehat{\check{J}} = J$, and $\widehat{\mathcal{F}}$ is a maximal ideal.\\

Ad (c): Let $\mathcal{F}$ be a prime filter. If $f,g \in \mathcal{C}^{\infty}(\mathbb{R}^n)$ are such that $f \cdot g \in \check{\mathcal{F}}$, then $Z(f\cdot g) = Z(f) \cup Z(g) \in \mathcal{F}$, so we have $Z(f) \in \mathcal{F}$ or $Z(g) \in \mathcal{F}$. Thus, $f \in \check{\mathcal{F}}$ or $g \in \check{\mathcal{F}}$, so $\check{\mathcal{F}}$ is a prime ideal.\\

Let $I$ be a prime ideal of $\mathcal{C}^{\infty}(\R^n)$, that is, if $f,g \in \mathcal{C}^{\infty}(\R^n)$ are such that $f \cdot g \in I$ then $f \in I$ or $g \in I$. We need to show that $\widehat{I} = \{ Z(h) | h \in I\}$ is a prime filter.\\

Let $F,G \subseteq \R^n$ be two closed sets such that $F \cup G \in \widehat{I}$. By definition of $\widehat{I}$, there is some $\varphi: \R^n \to \R$ such that $F\cup G = Z(\varphi)$. Consider:
$$\begin{array}{cccc}
    \widetilde{\varphi}: & \R^n & \rightarrow & \R \\
     & \vec{x} & \mapsto & \varphi^2(\vec{x})
  \end{array}$$
  and we have $F \cup G = Z(\widetilde{\varphi})$.\\

Since $F,G$ are closed subsets of $\R^n$ there are $f,g \in \mathcal{C}^{\infty}(\R^n)$ such that $F=Z(f)$ and $G=Z(g)$. Considering:
$$\begin{array}{cccc}
\widetilde{f}: & \R^n & \rightarrow & \R \\
               & \vec{x} & \mapsto & f^2(\vec{x})
\end{array}$$
and
$$\begin{array}{cccc}
\widetilde{g}: & \R^n & \rightarrow & \R \\
               & \vec{x} & \mapsto & g^2(\vec{x})
\end{array}$$
we have $\widetilde{\varphi}, \widetilde{f}$ and $\widetilde{g}$ non-negative functions, and both $F = Z(\widetilde{f})$ and $G = Z(\widetilde{g})$.\\

Let
$$\begin{array}{cccc}
    \psi : & \R^n & \to & \R \\
     & \vec{x} & \mapsto & \widetilde{f}(\vec{x}) - \widetilde{g}(\vec{x})
  \end{array}$$
  and define $H = \{ \vec{x} \in \R^n | \psi(\vec{x}) \leq 0 \}$ and $K = \{ \vec{x} \in \R^n | \psi(\vec{x}) \geq 0 \}$, so $\mathbb{R}^n = H \cup K$\\

\textbf{Claim:} $H \cap Z(\widetilde{\varphi}) = F$.\\

Given $\vec{x} \in F$, we have $\widetilde{f}(\vec{x})=0$, hence $\psi(\vec{x}) = 0 - \widetilde{g}(\vec{x}) \leq 0$ and $x \in H$. Also, since $Z(\widetilde{f}) \subseteq Z(\widetilde{\varphi})$, $\widetilde{\varphi}(\vec{x})=0$.\\

On the other hand, given $\vec{x} \in H \cap Z(\widetilde{\varphi})$, since $\vec{x} \in Z(\widetilde{\varphi}) = Z(\widetilde{f}) \cup Z(\widetilde{g})$, then either $\vec{x} \in Z(\widetilde{f})$ or $\vec{x} \in Z(\widetilde{g})$. If $\vec{x} \in Z(\widetilde{g})$ then $\widetilde{g}(\vec{x})=0$ and  $\psi(\vec{x}) = \widetilde{f}(\vec{x}) - 0 \leq 0$, thus $\psi(\vec{x}) = \widetilde{f}(\vec{x}) = 0$. Hence $\vec{x} \in Z(\widetilde{f})=F$.\\

Analogously we prove that $K \cap Z(\widetilde{\varphi}) = G$.\\

Let $h,k \in \mathcal{C}^{\infty}(\mathbb{R}^n)$ be non-negative functions such that $K=Z(k)$ and $H= Z(h)$. Note that $Z(\widetilde{\varphi}+\widetilde{h})= F$ ($\widetilde{\varphi}(x)+ \widetilde{h}(x) = 0 \iff \widetilde{\varphi}(x) = \widetilde{h}(x)=0$) and $Z(\widetilde{\varphi}+\widetilde{k})=G$. For every $\vec{x} \in \mathbb{R}^n = Z(h \cdot k) = H \cup K$. Since $\widetilde{\varphi} \in I$ and $I$ is an ideal,  $[(\widetilde{h}+\widetilde{k})\cdot \widetilde{\varphi}] \in I$, so:

$$[(\widetilde{\varphi} + \widetilde{h})\cdot (\widetilde{\varphi}+ \widetilde{k})]= \widetilde{\varphi}^2 + [(\widetilde{h}+\widetilde{k})\cdot \widetilde{\varphi}] \in I.$$

Since $I$ is a prime ideal, either $(\widetilde{\varphi} + \widetilde{h}) \in I$ or $(\widetilde{\varphi}+ \widetilde{k}) \in I$, i.e., $F \in \widehat{I}$ or $G \in \widehat{I}$, hence $\widehat{I}$ is a prime filter.\\

Ad (d): Consider the following maps:
$$\begin{array}{cccc}
    \varphi: & \mathfrak{F} & \rightarrow & \mathfrak{I}^{\infty} \\
     & \mathcal{F} & \mapsto & \check{\mathcal{F}}
  \end{array}$$
and
$$\begin{array}{cccc}
    \psi: & \mathfrak{I}^{\infty}  & \rightarrow & \mathfrak{F} \\
     & I & \mapsto & \widehat{I}
  \end{array}$$

We have already seen that whenever $I$ is an ideal, we have $\sqrt[\infty]{I}=\check{\widehat{I}}$.We must verify that $\varphi$ takes filters to $\mathcal{C}^{\infty}-$radical ideals, i.e., $\mathcal{F} \in \mathfrak{F} \to \check{\mathcal{F}} \in \mathfrak{I}^{\infty}$. Now, for every $\mathcal{F} \in \mathfrak{F}$ we have $\widehat{\check{\mathcal{F}}} = \mathcal{F}$, so $\sqrt[\infty]{\check{\mathcal{F}}} = \stackrel{\vee}{\widehat{\check{\mathcal{F}}}} = \check{\mathcal{F}}$. Thus, $\psi \circ \varphi(\mathcal{F}) = \mathcal{F}$.\\

On the other hand, if $J \in \mathfrak{I}^{\infty}$ then $\varphi \circ \psi(J) = \check{\widehat{J}} = \sqrt[\infty]{J} = J = {\rm id}_{\mathfrak{I}^{\infty}}(J)$. It follows that $\varphi$ and $\psi$ are inverse bijections.
\end{proof}

\begin{corollary} \label{Piqueri}Let $A = \mathcal{C}^{\infty}(\R^n)$ and $\p$ be a prime ideal of $A$. Under those circumstances $\sqrt[\infty]{\p}$ is a prime ideal.
\end{corollary}
\begin{proof}
By the \textbf{Corollary \ref{Garland}}, since $\p$ is a prime ideal, the filter associated with $\p$, $\widehat{\p}$ is a prime filter. Again by \textbf{Corollary \ref{Garland}}, it follows that $\check{\widehat{\p}} = \sqrt[\infty]{\p}$ is a prime ideal.
\end{proof}

\begin{corollary}\label{tumbleweed}The composite $\vee \circ \wedge : \mathfrak{I} \to \mathfrak{I}$ is a \index{closure operator}closure operator, thus we have the following three properties:
\begin{itemize}
  \item[(a)]{$I \subseteq \check{\widehat{I}}$;}
  \item[(b)]{$I_1 \subseteq I_2 \Rightarrow \check{\widehat{I_1}} \subseteq \check{\widehat{I_2}}$;}
  \item[(c)]{If $J = \check{\widehat{I}}$ then $\check{\widehat{J}} = J$.}
\end{itemize}
\end{corollary}

\begin{theorem}\label{yellow}Let $I, I_1, I_2 \subseteq \mathcal{C}^{\infty}(\mathbb{R}^n)$ be ideals. Then:
\begin{itemize}
  \item[(a)]{$\sqrt[\infty]{I}$ is an ideal of $\mathcal{C}^{\infty}(\mathbb{R}^n)$ and $I \subseteq \sqrt[\infty]{I}$;}
  \item[(b)]{$I_1 \subseteq I_2 \Rightarrow \sqrt[\infty]{I_1} \subseteq \sqrt[\infty]{I_2}$}
  \item[(c)]{$\sqrt[\infty]{\sqrt[\infty]{I}} = \sqrt[\infty]{I}$}
\end{itemize}
\end{theorem}
\begin{proof}
Ad (a):  From \textbf{Proposition \ref{49}} we have $\sqrt[\infty]{I} = \check{\widehat{I}}$, and from \textbf{Proposition \ref{47}} it follows that $\check{\widehat{I}}$ is an ideal. Moreover, since $\vee \circ \wedge$ is a closure operator, by item a) of \textbf{Corollary \ref{tumbleweed}} we have $I \subseteq \check{\widehat{I}} = \sqrt[\infty]{I}$.\\

Ad (b): From item (b) of \textbf{Corollary \ref{tumbleweed}}, $I_1 \subseteq I_2 \Rightarrow \check{\widehat{I_1}} \subseteq \check{\widehat{I_2}}$, so $\sqrt[\infty]{I_1} \subseteq \sqrt[\infty]{I_2}$.\\

Ad (c): From item (c) of \textbf{Corollary \ref{tumbleweed}} it follows immediately by the idempotence of $\vee \circ \wedge$ that $\sqrt[\infty]{\sqrt[\infty]{I}} = \sqrt[\infty]{I}$.
\end{proof}

\begin{theorem}\label{Chico}
Let $A$ be a $\mathcal{C}^{\infty}-$ring. We have:
$$\bigcap {\rm Spec}^{\infty}\,(A) = \sqrt[\infty]{(0)}.$$
\end{theorem}
\begin{proof}
In order to prove the inclusion $\sqrt[\infty]{(0)} \subseteq \bigcap {\rm Spec}^{\infty}\,(A)$, we observe that for every $\p \in {\rm Spec}^{\infty}\,(A)$ we have $(0) \subseteq \p$. By item b) of \textbf{Theorem \ref{yellow}}, $\sqrt[\infty]{(0)} \subseteq \sqrt[\infty]{\p} = \p$.\\

We prove the other inclusion by contradiction. \textit{Ab absurdo}, suppose $\bigcap {\rm Spec}^{\infty}\,(A) \nsubseteq \sqrt[\infty]{(0)}$. This implies that there exists some $x \in \p$ such that $x \notin \sqrt[\infty]{(0)}$ and, therefore, $D^{\infty}(x) \nsubseteq D^{\infty}((0)) = \varnothing$, so $D^{\infty}(x) \neq \varnothing$. Taking into account that $D^{\infty}(x) = \{ \p \in {\rm Spec}^{\infty}\,(A) | x \notin \p \}$, we have $(D^{\infty}(x) \neq \varnothing ) \iff ((\exists \p \in {\rm Spec}^{\infty}\,(A))(x \notin \p))$ and we achieve the absurdity (we have simultaneously $(\forall \p \in {\rm Spec}^{\infty}\,(A))(x \in \p)$ and $(\exists \p \in {\rm Spec}^{\infty}\,(A))(x \notin \p)$).
\end{proof}

The following definition will be helpful to prove that whenever $I$ is an ideal of \textbf{any} $\mathcal{C}^{\infty}-$ring $A$, then $\sqrt[\infty]{I}$ is an ideal too.\\

\begin{definition}Let $A$ be a $\mathcal{C}^{\infty}-$ring. We say that $A$ is \textbf{admissible} if for every ideal $I \subseteq A$, $\sqrt[\infty]{I}$ is an ideal in $A$.
\end{definition}

In what follows we sketch a proof that every $\mathcal{C}^{\infty}-$ring $A$ is admissible.\\

\begin{lemma}\label{l1}Let $A$ and $A'$ be  $\mathcal{C}^{\infty}-$rings such that $A \cong A'$. Then if $A$ is an admissible $\mathcal{C}^{\infty}-$ring, $A'$ is admissible too.
\end{lemma}

\begin{lemma}\label{l2}The free $\mathcal{C}^{\infty}-$ring in $n$ generators, $\mathcal{C}^{\infty}(\mathbb{R}^n)$, is an admissible $\mathcal{C}^{\infty}-$ring.
\end{lemma}
\begin{proof}It follows immediately from  \textbf{Proposition \ref{49}}.
\end{proof}

\begin{lemma}\label{l3}Let $A$ be an admissible $\mathcal{C}^{\infty}-$ring and $J \subset A$ be an ideal. Then $\dfrac{A}{J}$ is an admissible $\mathcal{C}^{\infty}-$ring.
\end{lemma}

\begin{lemma}\label{l4}Let $\{ A_i \stackrel{h_{ij}}{\rightarrow} A_j \}$ be a filtered diagram of admissible $\mathcal{C}^{\infty}-$rings. Then $\varinjlim A_i$ is an admissible $\mathcal{C}^{\infty}-$ring.
\end{lemma}

\begin{theorem}\label{little}Every $\mathcal{C}^{\infty}-$ring is admissible.
\end{theorem}
\begin{proof}Let $A$ be any $\mathcal{C}^{\infty}-$ring. We know that every $\mathcal{C}^{\infty}-$ring is a filtered colimit of finitely presented $\mathcal{C}^{\infty}-$rings. Since every finitely presentable $\mathcal{C}^{\infty}-$ring is admissible, the result follows from \textbf{Lemma \ref{l4}}.\\

Let $B$ be any finitely presentable $\mathcal{C}^{\infty}-$ring. We know that there exist some $n \in \mathbb{N}$ and some ideal $J \subset \mathcal{C}^{\infty}(\mathbb{R}^n)$ such that $A \cong \dfrac{\mathcal{C}^{\infty}(\mathbb{R}^n)}{J}$. From \textbf{Lemma \ref{l1}}, if we prove that $\dfrac{\mathcal{C}^{\infty}(\mathbb{R}^n)}{J}$ is admissible, then it follows that $A$ is admissible.\\

From \textbf{Lemma \ref{l2}} we have that $\mathcal{C}^{\infty}(\mathbb{R}^n)$ is admissible, and from \textbf{Lemma \ref{l3}} it follows that $\dfrac{\mathcal{C}^{\infty}(\mathbb{R}^n)}{J}$ is admissible.
\end{proof}

Now we present some properties of taking the $\mathcal{C}^{\infty}-$radical of an ideal.\\

\begin{proposition}Let $A$ be a $\mathcal{C}^{\infty}-$ring, $I,J \subseteq A$ any of its ideals. Then:
\begin{itemize}
\item[(i)]{$I \subseteq J \Rightarrow \sqrt[\infty]{I} \subseteq \sqrt[\infty]{J}$}
\item[(ii)]{$I \subseteq \sqrt[\infty]{I}$}
\end{itemize}
\end{proposition}
\begin{proof}
Ad (i): Given $a \in \sqrt[\infty]{I}$, there is $b \in I$ such that $\eta_a(b) \in (A\{ a^{-1}\})^{\times}$. Since $I \subseteq J$, the same $b$ is a witness of the fact that $a \in \sqrt[\infty]{J}$, for $b \in J$ and $\eta_a(b) \in (A\{ a^{-1}\})$.\\

Ad (ii): We are going to show that $A \setminus \sqrt[\infty]{I} \subseteq A \setminus I$.\\

Given $a \in A \setminus \sqrt[\infty]{I}$, we have that $(\forall b \in I)(\eta_a(b) \notin (A\{ a^{-1}\})^{\times})$, so $\eta_a[I] \cap (A\{ a^{-1}\})^{\times} = \varnothing$. Since $\eta_a(a) \in (A\{ a^{-1}\})^{\times}$, it follows that $\eta_a(a) \notin \eta_a[I]$, so $a \notin I$.
\end{proof}

\begin{proposition}Let $B$ be a $\mathcal{C}^{\infty}-$ring and $J \subseteq B$ any of its ideals. We have the following equality:
$$\sqrt[\infty]{\{ 0_{\frac{B}{J}}\}} = \dfrac{\sqrt[\infty]{J}}{J}$$
\end{proposition}
\begin{proof}
$$a \in \sqrt[\infty]{J} \iff (\exists b \in J)(\eta(b) \in (B\{ a^{-1}\})^{\times}),$$
hence
\begin{multline*}a + J = \overline{a} \in \sqrt[\infty]{\{ 0_{\frac{B}{J}}\}} \iff (\exists \overline{b} \in \{ 0_{\frac{B}{J}}\})(\overline{\eta_a}(\overline{b}) \in \left( \frac{B}{J}\{(a + J)^{-1} \}\right)^{\times}) \iff\\
\iff \dfrac{B}{J}\{(a+J)^{-1} \} \cong \{ 0\} \iff \\
 \iff a \in \sqrt[\infty]{J}\iff  a+ J \in \dfrac{\sqrt[\infty]{J}}{J}
 \end{multline*}

Now, since $J \subseteq \sqrt[\infty]{J}$, if $a'+J = a+J$, then $a \in \sqrt[\infty]{J} \iff a' \in \sqrt[\infty]{J}$.
\end{proof}

\begin{corollary}Let $A$ be a $\mathcal{C}^{\infty}-$ring. We have:
\begin{itemize}
  \item[(a)]{An ideal $J \subseteq A$  is a $\mathcal{C}^{\infty}-$radical ideal if, and only if, $\dfrac{A}{J}$ is a $\mathcal{C}^{\infty}-$reduced $\mathcal{C}^{\infty}-$ring}
  \item[(b)]{A proper prime ideal $\mathfrak{p} \subseteq A$ is $\mathcal{C}^{\infty}-$radical if, and only if, $\dfrac{A}{\mathfrak{p}}$ is a $\mathcal{C}^{\infty}-$reduced $\mathcal{C}^{\infty}-$domain.}
\end{itemize}
\end{corollary}

\begin{remark}\label{jurupari}Not every prime proper ideal of a $\mathcal{C}^{\infty}-$ring $A$ is a $\mathcal{C}^{\infty}-$radical ideal. Consider the ideal of functions which are flat at $0$, $\mathfrak{m}_0^{\infty} \subseteq \mathcal{C}^{\infty}(\mathbb{R})$, which is a maximal - and thus a prime - ideal of $\mathcal{C}^{\infty}(\mathbb{R})$. We claim that this ideal is not a $\mathcal{C}^{\infty}-$radical ideal (cf. \textbf{Example 1.2} of \cite{rings2}).\\

Consider ${\rm id}_{\mathbb{R}}: \mathbb{R} \rightarrow \mathbb{R}$ and $\mathcal{C}^{\infty}(\mathbb{R})\{ {\rm id}_{\mathbb{R}}^{-1}\} \cong_{\varphi} \mathcal{C}^{\infty}(\mathbb{R}\setminus \{ 0\})$. Consider:

$$\begin{array}{lr}
\begin{array}{cccc}
    f: & \mathbb{R}\setminus \{ 0\} & \rightarrow & \mathbb{R} \\
     & x & \mapsto & \begin{cases}
                       0, & \mbox{if }\, x<0 \\
                       1, & \mbox{if}\, x>0.
                     \end{cases}
  \end{array}&
\begin{array}{cccc}
    g: & \mathbb{R}\setminus \{ 0\} & \rightarrow & \mathbb{R} \\
     & x & \mapsto & \begin{cases}
                       1, & \mbox{if }\, x<0 \\
                       0, & \mbox{if}\, x>0.
                     \end{cases}
  \end{array}
\end{array}$$

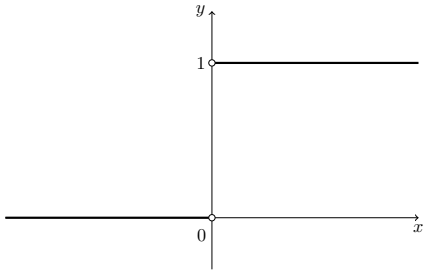
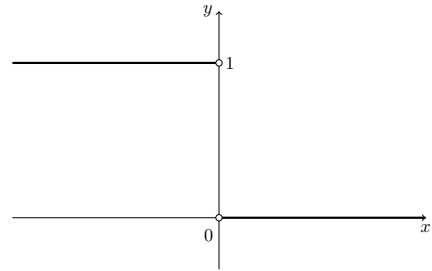
\begin{center}
    \resizebox{0.8\textwidth}{!}{\begin{tikzpicture}[domain=-4:4]
\draw [->] (0,-1) --(0,4) node (yaxis) [left] {$y$};
\draw [->] (-4,0) --(4,0) node (yaxis) [below] {$x$};  
\draw[very thick] (0,3)--(4,3);
\draw[very thick] (-4,0)--(0,0);
\draw[fill=white] (0,0) circle(1.75pt);
\draw[fill=white] (0,3) circle(1.75pt);
\draw (-0.2,-0.1) node[below]{$0$};
\draw (0,3) node[left]{$1$};

\draw [->] (11,-1) --(11,4) node (yaxis) [left] {$y$};
\draw [->] (7,0) --(15,0) node (yaxis) [below] {$x$};  
\draw[very thick] (11,0)--(15,0);
\draw[very thick] (7,3)--(11,3);
\draw[fill=white] (11,0) circle(1.75pt);
\draw[fill=white] (11,3) circle(1.75pt);
\draw (10.8,-0.1) node[below]{$0$};
\draw (11,3) node[right]{$1$};
\end{tikzpicture}}
\end{center}

We have $f+g = 1 \in \mathcal{C}^{\infty}(\mathbb{R}\setminus \{ 0\})$, thus $\varphi(f+g)=\varphi(f)+\varphi(g)=\varphi(1) = 1 \in \mathcal{C}^{\infty}(\mathbb{R})_{\mathfrak{m}_{0}^{\infty}}$. But neither $\varphi(f)$ nor $\varphi(g)$ can be inverted in $\mathcal{C}^{\infty}(\mathbb{R})_{\mathfrak{m}_{0}^{\infty}}$. For, if $\varphi(f)$ were invertible, then $f$ would be invertible and there would be an $h \notin \mathfrak{m}_{0}^{\infty}$ with $U_f \supset U_h$, i.e., $h\upharpoonright_{]-\infty,0[} \cong 0$, contradicting $h \notin \mathfrak{m}_0^{\infty}$; Hence $\mathcal{C}^{\infty}(\mathbb{R})_{\mathfrak{m}_{0}^{\infty}}$ is not a local $\mathcal{C}^{\infty}-$ring.

\end{remark}

\begin{proposition} \label{redi}
Let $A', B'$ be two $\mathcal{C}^{\infty}-$rings and $\jmath: A' \to B'$ be a monomorphism. If $B'$ is $\mathcal{C}^{\infty}-$reduced, then $A'$ is also $\mathcal{C}^{\infty}-$reduced.
\end{proposition}
\begin{proof}
Since $B'$ is reduced, $\{ 0_{B'}\} = \sqrt[\infty]{\{ 0_{B'}\}}$.\\

Suppose, \textit{ab absurdo}, that $(\exists a' \in \sqrt[\infty]{\{ 0_{A'}\}})(a' \neq 0_{A'})$. Consider the following commutative diagram:
$$\xymatrixcolsep{5pc}\xymatrix{
A' \ar[d]^{\jmath} \ar[r]^{\eta_{a'}} & A'\{(a')^{-1} \} \ar[d]_{\jmath_{a'}} \\
B' \ar[r]^{\eta_{\jmath(a')}} & B'\{ \jmath(a')^{-1}\}
}$$

where $\jmath_{a'}$ is the $\mathcal{C}^{\infty}-$homomorphism which is induced by $\eta_{\jmath(a')} \circ \jmath : A' \to B'\{ \jmath(a')^{-1}\}$.\\

Since $a' \neq 0_{A'}$ and $\jmath : A' \to B'$ is injective, it follows that $\jmath(a') \neq 0_{B'}$. Since $B'$ is reduced, $B'\{ \jmath(a')^{-1}\} \ncong \{0\}$. On the other hand, since $a' \in \sqrt[\infty]{\{ 0_{A'}\}}$, $A'\{ a'^{-1}\} \cong \{0\}$. Thus, $\jmath_a : A'\{ a'^{-1}\} \to B'\{ \jmath(a')^{-1}\}$ is a $\mathcal{C}^{\infty}-$rings morphism such that $\jmath_a(1_{A'\{ a'^{-1}\}})=1_{B'\{ \jmath(a')^{-1}\}}$ from the trivial $\mathcal{C}^{\infty}-$ring $A'\{ a'^{-1}\} \cong \{ 0 \}$ into a non-trivial $\mathcal{C}^{\infty}-$ring $B'\{ \jmath(a')^{-1}\}$, which is absurd.\\

Hence, $\sqrt[\infty]{\{ 0_{A'}\}} = \{ 0_{A'}\}$ and $A'$ is reduced.
\end{proof}

Now we consider two explicit situations where we have finitely generated $\mathcal{C}^{\infty}-$fields.

\begin{example}Let $M$ be a compact smooth manifold. By \textbf{Theorem 2.3} of \cite{moerdijk2013models}, $\mathcal{C}^{\infty}(M)\cong \mathcal{C}^{\infty}(\R^k)/J$ for some $k \in \mathbb{N}$ and some finitely generated ideal $J$. Thus, for any  ideal $I \subset \mathcal{C}^{\infty}(M)$, $\mathcal{C}^{\infty}(M)/I$ is a finitely generated $\mathcal{C}^{\infty}-$ring.

Now, if $I \subset \mathcal{C}^{\infty}(M)$ is a maximal ideal of $\mathcal{C}^{\infty}(M)$, then  $I = \mathfrak{m}_{x} = \{ f \in \mathcal{C}^{\infty}(M) \mid f(x)=0\}$ for some unique $x \in M$.

In fact, given any ideal $I \subseteq \mathcal{C}^{\infty}(M)$, one has either $I \subseteq \mathfrak{m}_x$ for some (unique) $x \in M$ or $I=\mathcal{C}^{\infty}(M)$.

Suppose it is not the case that there is some $x \in M$ such that $I \subseteq \mathfrak{m}_x$, \textit{i.e.}, $(\forall x \in M)(I \not\subseteq \mathfrak{m}_x)$. For every $x \in M$ we can find a function $f_x \in I$ such that $f_x(x)\neq 0$. Consider the open covering $\{ M \setminus f_x^{\dashv}[\{ 0\}] \mid x \in M \}$ of $M$, which has a finite sub-covering, say $\{ M \setminus f_{x_1}^{\dashv}[\{ 0\}], \cdots, M \setminus f_{x_r}^{\dashv}[\{ 0\}] \}$. We obtain, thus, the function $f = f_{x_1}^2 + \cdots + f_{x_r}^2 \in I$ such that $(\forall x \in M)(f(x) > 0)$. Hence, $f \in I \cap \mathcal{C}^{\infty}(M)^{\times}$ and $I = \mathcal{C}^{\infty}(M)$. 

As for the uniqueness of $x \in M$, suppose that $I \subset \mathfrak{m}_x$ and let $y \in M$ be such that $x \neq y$.  By the \textbf{Smooth Tietze's Theorem}, there is some $f \in \mathcal{C}^{\infty}(M)$ such that $f(x)=0$ (so $f \in I$ and  $f \in \mathfrak{m}_x$) and $f(y)=1$, so $f \notin \mathfrak{m}_y$. Thus $\mathfrak{m}_x \not \subset \mathfrak{m}_y$ and $I \not\subset \mathfrak{m}_y$.

It follows that whenever $I$ is a  maximal ideal of $\mathcal{C}^{\infty}(M)$ - where $M$ is a compact manifold - there is a unique $x \in M$ such that $I \subseteq \mathfrak{m}_x \subset \mathcal{C}^{\infty}(M)$. Since $I$ is a maximal ideal, then  $I = \mathfrak{m}_x$.

Thus, for every maximal ideal $I \subset \mathcal{C}^{\infty}(M)$, $\mathcal{C}^{\infty}(M)/I \cong \R$ using the fact that the $\mathcal{C}^{\infty}-$homomorphism:

$$\begin{array}{cccc}
    {\rm ev}_x: & \mathcal{C}^{\infty}(M) & \rightarrow & \R \\
     & f & \mapsto & f(x)
\end{array}$$

\noindent is surjective and the \textbf{Fundamental Theorem of the $\mathcal{C}^{\infty}-$Homomorphism} that:

$$\dfrac{\mathcal{C}^{\infty}(M)}{\mathfrak{m}_x} = \dfrac{\mathcal{C}^{\infty}(M)}{\ker {\rm ev}_x} \cong \R$$

Hence, every $\mathcal{C}^{\infty}-$field obtained as a quotient $\mathcal{C}^{\infty}(M)/I$ is isomorphic to $\R$.
\end{example}

\begin{example}
Consider $\mathcal{C}^{\infty}(\mathbb{R})$ together with the ideal of all compactly supported functions:

$$I = \{ f \in \mathcal{C}^{\infty}(\R) \mid {\rm supp} (f)\subset \R \,\, \text{is compact} \}$$

Naturally the constant function $1$ does not belong to $I$, so there is a maximal ideal $\widehat{I} \subset \mathcal{C}^{\infty}(\mathbb{R})$ such that $I \subseteq \widehat{I}$. Also, note that for every $x \in \R$, $I \not\subset \mathfrak{m}_x$. In fact, for every $x \in \R$, the smooth characteristic function $\chi_{]x-1,x+1[}: \R \to \R$ is a compactly supported function which does not belong to $\mathfrak{m}_x$. Since $I \subset \widehat{I}$, $\widehat{I} \neq \mathfrak{m}_x$ for every $x \in \R$.

Now, since $\widehat{I}$ is a maximal ideal different from $\mathfrak{m}_x$ for every $x \in \R$,  $\mathcal{C}^{\infty}(\R)/\widehat{I} \cong \mathbb{F}$ is a finitely generated $\mathcal{C}^{\infty}-$field that is different from $\R$.

An explicit description is given as follows. Let $U \subset \wp(\mathbb{N})$ be a non-principal ultrafilter and let:

$$\widehat{I} = \{ f \in \mathcal{C}^{\infty}(\R) \mid \{ n \in \mathbb{N} \mid f(n)=0\} \in U \} \subset \mathcal{C}^{\infty}(\R)$$

It is straightforward to check that $\widehat{I}$ is an ideal of $\mathcal{C}^{\infty}(\mathbb{R})$. Since $U$ is a non principal ultrafilter, $U$ contains all cofinite subsets of $\mathbb{N}$. Thus, given any $f \in I$ - that is, any $f \in \mathcal{C}^{\infty}(\R)$ with compact support, $K={\rm supp}\,(f)$, since $K$ is limited there is some $n_0 \in \mathbb{N}$ such that $K \subseteq [-n_0,n_0]$, so $(\forall n > n_0)(f(n)=0)$. Hence $\{ n \in \mathbb{N} \mid f(n) = 0\} \subset \mathbb{N}$ is cofinite  and $f \in \widehat{I}$. Thus $I \subset \widehat{I}$.

Finally, in order to show that $\widehat{I}$ is a maximal ideal, we show that $\mathcal{C}^{\infty}(\R)/\widehat{I}$ is a $\mathcal{C}^{\infty}$-field. 

In fact,  if $f + \widehat{I} \neq 0 + \widehat{I}$, then $f \notin \widehat{I}$ and $\{n \in \mathbb{N} \mid f(n)=0\} \notin U$. Since $U$ is an ultrafilter, we have $\{ n \in \mathbb{N} \mid f(n)\neq 0\} \in U$.

Now, since $\mathbb{N} \subset \mathbb{R}$ is discrete, we can take, for every $n \in \mathbb{N}$ such that $f(n)\neq 0$, the open neighbourhood $]n-\frac{1}{2},n+\frac{1}{2}[$
with the smooth characteristic function:

$$\begin{array}{cccc}
    \chi_{]n-\frac{1}{2},n+\frac{1}{2}[}: & \R & \rightarrow & \R  \\
      & x & \mapsto  & \begin{cases}
      e^{1-\frac{1}{1-4(x-n)^2}},\,\, \text{if}\,\, x \in ]n-\frac{1}{2},n+\frac{1}{2}[\\
      0, \,\, \text{otherwise}
      \end{cases}
\end{array}$$

and then glue them up to get the smooth function:

$$\begin{array}{cccc}
    h: & \R & \rightarrow & \R  \\
      & x & \mapsto  & \begin{cases}
      \chi_{]n-\frac{1}{2},n+\frac{1}{2}[}(x),\,\, \text{if}\,\, x \in ]n-\frac{1}{2},n+\frac{1}{2}[\,\, \text{and}\,\,f(n)\neq 0\\
      0\,\, \text{otherwise}
      \end{cases}
\end{array}$$


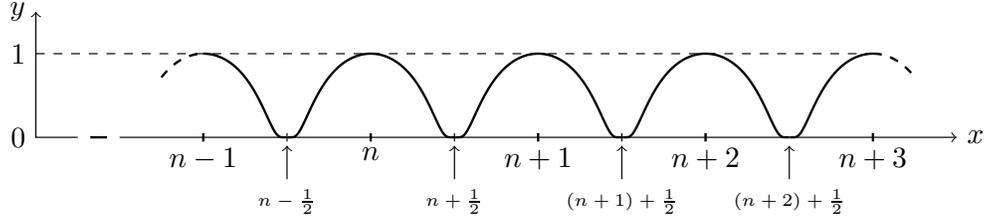
\begin{center}
    \resizebox{0.8\textwidth}{!}{\begin{tikzpicture}[domain=-4:4]
\draw [<->] (-2,1.5) node (yaxis) [left] {$y$}
    |- (-2,0) node (zaxis) [left] {}
    |- (9,0) node (xaxis) [right] {$x$}
    ;
    \draw (-2,1) node[left]{$1$};
    \draw[dashed] (-2,1)--(8,1);
    \draw (2,0) node[below]{$n$};
    \draw (3,-0.5) node[below]{\tiny{$n+\frac{1}{2}$}};
    \draw (-2,0) node[left]{$0$};
    \draw[thick,white] (-1.5,0)--(-1,0);
    \draw[thick] (-1.35,0)--(-1.15,0);
    \draw[->] (3,-0.5)--(3,-0.1);
    \draw[->] (5,-0.5)--(5,-0.1);
    \draw[->] (1,-0.5)--(1,-0.1);
    \draw[->] (7,-0.5)--(7,-0.1);
    \draw (7,-0.5) node[below]{\tiny{$(n+2)+\frac{1}{2}$}};
    \draw (1,-0.5) node[below]{\tiny{$n-\frac{1}{2}$}};
    \draw[thick] (0,-0.05)--(0,0.05);
    \draw (0,0) node[below]{$n-1$};
    \draw (4,0) node[below]{$n+1$};
    \draw (8,0) node[below]{$n+3$};
    \draw[thick] (8,-0.05)--(8,0.05);
    \draw[black, samples=100, smooth, domain=-0.99:1, thick] plot (\x+2, {exp(1-1/(1-\x*\x))});
    \draw[black, samples=100, smooth, domain=-0.99:1, thick] plot (\x+4, {exp(1-1/(1-\x*\x)});
    \draw[black, samples=100, smooth, domain=-0.99:1, thick] plot (\x+6, {exp(1-1/(1-\x*\x)});
    \draw (1,-0.05)--(1,0.05);
    \draw[thick] (2,-0.05)--(2,0.05);
    \draw (3,-0.05)--(3,0.05);
    \draw[thick] (4,-0.05)--(4,0.05);
    \draw (5,-0.05)--(5,0.05);
    \draw[thick] (6,-0.05)--(6,0.05);
    \draw (5,-0.5) node[below]{\tiny{$ (n+1)+\frac{1}{2}$}};
    \draw(6,0) node[below]{$n+2$};
    \draw[black, samples=100, smooth, domain=-0.99:0, thick] plot (\x+8, {exp(1-1/(1-\x*\x)});
    \draw[black,dashed, samples=100, smooth, domain=0:0.5, thick] plot (\x+8, {exp(1-1/(1-\x*\x)});
    \draw[black, samples=100, smooth, domain=0:0.99, thick] plot (\x, {exp(1-1/(1-\x*\x)});
    \draw[black,dashed, samples=100, smooth, domain=-0.5:0, thick] plot (\x, {exp(1-1/(1-\x*\x)});
\end{tikzpicture}}
\end{center}

Now consider:

$$\begin{array}{cccc}
   g: & \R & \rightarrow & \R \\
      & x & \mapsto & \begin{cases}
      \dfrac{h(x)}{f(n)},\,\, \text{if}\,\, x \in ]n-\frac{1}{2},n+\frac{1}{2}[\,\, \text{and}\,\,f(n)\neq 0\\
      0\,\, \text{otherwise}
      \end{cases}
\end{array}$$

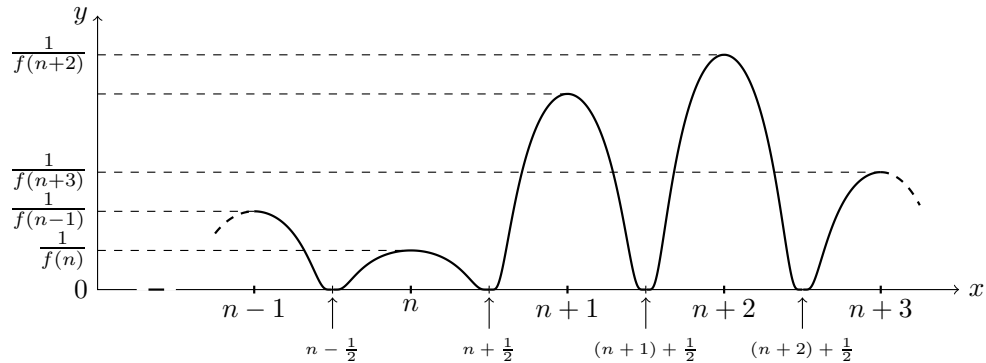
\begin{center}
    \resizebox{0.8\textwidth}{!}{\begin{tikzpicture}[domain=-4:4]
\draw [<->] (-2,3.5) node (yaxis) [left] {$y$}
    |- (-2,0) node (zaxis) [left] {}
    |- (9,0) node (xaxis) [right] {$x$}
    ;
    \draw (-2,1) node[left]{$\frac{1}{f(n-1)}$};
    \draw[dashed] (-2,1)--(0,1);
    \draw (2,0) node[below]{$n$};
    \draw (3,-0.5) node[below]{\tiny{$n+\frac{1}{2}$}};
    \draw (-2,0) node[left]{$0$};
    \draw[thick,white] (-1.5,0)--(-1,0);
    \draw[thick] (-1.35,0)--(-1.15,0);
    \draw[->] (3,-0.5)--(3,-0.1);
    \draw[->] (5,-0.5)--(5,-0.1);
    \draw[->] (1,-0.5)--(1,-0.1);
    \draw[->] (7,-0.5)--(7,-0.1);
    \draw (7,-0.5) node[below]{\tiny{$(n+2)+\frac{1}{2}$}};
    \draw (1,-0.5) node[below]{\tiny{$n-\frac{1}{2}$}};
    \draw[thick] (0,-0.05)--(0,0.05);
    \draw (0,0) node[below]{$n-1$};
    \draw (4,0) node[below]{$n+1$};
    \draw (8,0) node[below]{$n+3$};
    \draw[thick] (8,-0.05)--(8,0.05);
    \draw[black, samples=100, smooth, domain=-0.99:1, thick] plot (\x+2, {0.5*exp(1-1/(1-\x*\x)});
    \draw[dashed] (-2,0.5)--(2,0.5);
    \draw (-2,0.5) node[left]{$\frac{1}{f(n)}$};
    \draw[black, samples=100, smooth, domain=-0.99:1, thick] plot (\x+4, {2.5*exp(1-1/(1-\x*\x)});
    \draw[dashed] (-2,2.5)--(4,2.5);
    \draw[black, samples=100, smooth, domain=-0.99:1, thick] plot (\x+6, {3*exp(1-1/(1-\x*\x)});
    \draw (-2,3) node[left]{$\frac{1}{f(n+2)}$};
    \draw (-2,1.5) node[left]{$\frac{1}{f(n+3)}$};
    \draw[dashed] (-2,3)--(6,3);
    \draw (1,-0.05)--(1,0.05);
    \draw[thick] (2,-0.05)--(2,0.05);
    \draw (3,-0.05)--(3,0.05);
    \draw[thick] (4,-0.05)--(4,0.05);
    \draw (5,-0.05)--(5,0.05);
    \draw[thick] (6,-0.05)--(6,0.05);
    \draw (5,-0.5) node[below]{\tiny{$ (n+1)+\frac{1}{2}$}};
    \draw(6,0) node[below]{$n+2$};
    \draw[black, samples=100, smooth, domain=-0.99:0, thick] plot (\x+8, {1.5*exp(1-1/(1-\x*\x)});
    \draw[black,dashed, samples=100, smooth, domain=0:0.5, thick] plot (\x+8, {1.5*exp(1-1/(1-\x*\x)});
    \draw[dashed] (-2,1.5)--(8,1.5);
    \draw[black, samples=100, smooth, domain=0:0.99, thick] plot (\x, {exp(1-1/(1-\x*\x)});
    \draw[black,dashed, samples=100, smooth, domain=-0.5:0, thick] plot (\x, {exp(1-1/(1-\x*\x)});
\end{tikzpicture}}
\end{center}

\noindent and note that $g \in \mathcal{C}^{\infty}(\R)$.

Also, note that for every $n$ such that $f(n)\neq 0$, we have $h(n)=1$, so $g(n)=1/f(n)$. Thus, since $\{ n \in \mathbb{N} \mid f(n)\cdot g(n) - 1 = 0\} \in U \subset \wp(\mathbb{N})$ (for $\{ n \in \mathbb{N} \mid f(n)\cdot g(n) - 1 = 0\}$ is cofinite), it follows that:

$$f\cdot g - 1 \in \widehat{I},$$

\noindent so $f+\widehat{I} \in (\mathcal{C}^{\infty}(\R)/I)^{\times}$. It follows that $\mathcal{C}^{\infty}(\R)/\widehat{I}$ is a finitely generated $\mathcal{C}^{\infty}-$field which is not isomorphic to $\R$.
\end{example}


As a consequence of \textbf{Proposition \ref{redi}} and of \textbf{Proposition \ref{alba}}, we have the following:

\begin{corollary}\label{medusa}Every $\mathcal{C}^{\infty}-$subring of a $\mathcal{C}^{\infty}-$field is a $\mathcal{C}^{\infty}-$reduced $\mathcal{C}^{\infty}-$domain.
\end{corollary}

Next we show that the directed limit of reduced $\mathcal{C}^{\infty}-$rings is also reduced.\\

\begin{proposition}\label{Xango}Let $(I, \leq)$ be a directed set and suppose that $\{A_i\}_{i \in I}$ is a directed family of $\mathcal{C}^{\infty}-$reduced $\mathcal{C}^{\infty}-$rings. Then
$$B = \varinjlim_{i \in I} A_i$$
is a $\mathcal{C}^{\infty}-$reduced $\mathcal{C}^{\infty}-$ring.
\end{proposition}
\begin{proof}
Let $u \in \sqrt[\infty]{\{ 0_B\}}$, so $B\{ u^{-1}\} \cong 0$, and let $j \in I$ and $u_j \in A_j$ be such that $u = t_j(u_j)$. By \textbf{Theorem \ref{Newt}},
$$B\{ u^{-1}\} \cong \varinjlim_{i \geq j} A_i\{ u_i^{-1}\}.$$

so there is some $\ell \geq j$ such that $A_{\ell}\{ u_{\ell}^{-1}\} \cong 0$. Indeed, $1=0$ in $\varinjlim_{i \geq j} A_i\{ {u_i}^{-1}\}$.\\

Since $A_{\ell}$ is $\mathcal{C}^{\infty}-$reduced, it follows that $u_{\ell} = 0$. Hence, $u = t_{\ell}(u_{\ell}) = t_{\ell}(0) = 0$, and $B$ is $\mathcal{C}^{\infty}-$reduced.
\end{proof}

\begin{theorem}\label{preim}Let $A$ and $B$ be two $\mathcal{C}^{\infty}-$rings, $J \subseteq B$ a $\mathcal{C}^{\infty}-$radical ideal in $B$ and $f: A \to B$ any $\mathcal{C}^{\infty}-$homomorphism. Then $f^{\dashv}[J]$ is a $\mathcal{C}^{\infty}-$radical ideal in $A$.
\end{theorem}
\begin{proof}
Since $\sqrt[\infty]{J} = J$ and $J$ is an ideal, it follows that $f^{\dashv}[\sqrt[\infty]{J}] = f^{\dashv}[J]$ is also an ideal.\\

We always have:
$$f^{\dashv}[J] \subseteq \sqrt[\infty]{f^{\dashv}[J]},$$
so $f^{\dashv}[\sqrt[\infty]{J}] = f^{\dashv}[J] \subseteq \sqrt[\infty]{f^{\dashv}[J]}$. It remains to show the other inclusion, namely:
$$\sqrt[\infty]{f^{\dashv}[J]} \subseteq f^{\dashv}[\sqrt[\infty]{J}].$$

Now, given the morphism $f: A \to B$, we compose it with the projection map:
$$\begin{array}{cccc}
q_J: & B & \twoheadrightarrow & \dfrac{B}{J}\\
     & b & \mapsto & b + J
\end{array},$$
in order to get the following commutative triangle:
$$\xymatrix{
A \ar[r]^{f} \ar[rd]_{q_J \circ f} & B \ar[d]^{q_J}\\
    & \dfrac{B}{J}
}$$

By the \textbf{Theorem of Homomorphism}, there exists a unique morphism $f_J : \dfrac{A}{f^{\dashv}[J]} \rightarrow \dfrac{B}{J}$ such that the following triangle commutes:
$$\xymatrix{
A \ar[r]^{q_J \circ f} \ar[d]_{q_{f^{\dashv}[J]}} & \dfrac{B}{J}\\
   \dfrac{A}{f^{\dashv}[J]} \ar[ur]^{f_J} &
}$$

We note that $f_J : \dfrac{A}{f^{\dashv}[J]} \to \dfrac{B}{J}$ is injective, since:
$$f_J^{\dashv}[0_{\frac{B}{J}}] = \{ a + f^{\dashv}[J] | f_J(a + f^{\dashv}[J]) = J \} = \{ a + f^{\dashv}[J] | f(a) + J = J \} = \{ a + f^{\dashv}[J] | f(a) \in J \} = \{ 0_{\frac{A}{f^{\dashv}[J]}}\}$$

Since $J$ is a $\mathcal{C}^{\infty}-$radical ideal of $B$, it follows by \textbf{Corollary \ref{medusa}} that $\dfrac{\sqrt[\infty]{J}}{J} \cong \{ 0_{\frac{B}{J}}\}$, so $\dfrac{B}{J}$ is $\mathcal{C}^{\infty}-$reduced.   By \textbf{Proposition \ref{redi}}, since $f_J$ is a monomorphism, we conclude that $\dfrac{A}{f^{\dashv}[J]}$ is also reduced, so:
$$\dfrac{\sqrt[\infty]{f^{\dashv}[J]}}{f^{\dashv}[J]} = \sqrt[\infty]{ \{ 0_{\frac{A}{f^{\dashv}[J]}} \} } = \{ 0_{\frac{A}{f^{\dashv}[J]}}\}$$
Hence:
$$\sqrt[\infty]{f^{\dashv}[J]} = f^{\dashv}[J].$$
\end{proof}

We register that, in general, holds:

\begin{proposition}\label{usa4africa}Let $A,B$ be $\mathcal{C}^{\infty}-$rings, $f: A \to B$ a $\mathcal{C}^{\infty}-$homomorphism and $J \subseteq B$ any ideal. Then:
$$\sqrt[\infty]{f^{\dashv}[J]} \subseteq f^{\dashv}[\sqrt[\infty]{J}].$$
\end{proposition}
\begin{proof}
Given the $\mathcal{C}^{\infty}-$homomorphism $f: A \to B$, there exists a unique morphism $f_J: \dfrac{A}{f^{\dashv}[J]} \to \dfrac{B}{J}$ such that the following diagram commutes:
$$\xymatrixcolsep{5pc}\xymatrix{
A \ar[d]_{f} \ar[r]^{q_{f^{\dashv}[J]}} & \dfrac{A}{f^{\dashv}[J]} \ar@{>.>}[d]^{f_J} \\
B \ar[r]^{q_J} & \dfrac{B}{J}
}$$
Consider the rings of fractions:
$$\eta_{a + f^{\dashv}[J]} : \dfrac{A}{f^{\dashv}[J]} \to \left( \dfrac{A}{f^{\dashv}[J]}\right)\{ a + f^{\dashv}[J] \}$$
$$\eta_{f(a)+J}: \left( \dfrac{B}{J}\right) \to \left( \dfrac{B}{J}\right)\{(f(a)+J)^{-1} \}$$
and note that the $\mathcal{C}^{\infty}-$homomorphism:
$$\eta_{f(a)+J} \circ f_J : \dfrac{A}{f^{\dashv}[J]} \to \left( \dfrac{B}{J}\right)\{(f(a)+J)^{-1} \}$$
is such that $(\eta_{f(a)+J}(f_J(a+f^{\dashv}[J]))) = f(a)+J \in \left[ \left( \dfrac{B}{J}\right)\{ (f(a) + J)^{-1}\} \right]^{\times}$. By the universal property of the ring of fractions $\eta_{a + f^{\dashv}[J]}$, there exists a unique morphism $f_{(J,a)}$ such that the following triangle commutes:
$$\xymatrixcolsep{5pc}\xymatrix{
\dfrac{A}{f^{\dashv}[J]} \ar[r]^(0.3){\eta_{a + f^{\dashv}[J]}} \ar[rd]_{\eta_{f(a)+J} \circ f_J} & \left( \dfrac{A}{f^{\dashv}[J]}\right)\{ (a+f^{\dashv}[J]) \} \ar@{.>}[d]^{f_{(J,a)}}\\
   & \left( \dfrac{B}{J}\right) \{ (f(a) + J)^{-1} \}
}$$
so we have the following commutative rectangle:
$$\xymatrixcolsep{5pc}\xymatrix{
\dfrac{A}{f^{\dashv}[J]} \ar[r]^(0.3){\eta_{a + f^{\dashv}[J]}} \ar@{>->}[d]^{f_J} & \left( \dfrac{A}{f^{\dashv}[J]}\right)\{ (a + f^{\dashv}[J])^{-1} \} \ar[d]^{f_{(J,a)}}\\
\dfrac{B}{J} \ar[r]^{\eta_{f(a)+J}} & \left( \dfrac{B}{J}\right)\{ (f(a) + J)^{-1} \}
}$$
Hence, given $a \in \sqrt[\infty]{f^{\dashv}[J]}$, we have $\left( \dfrac{A}{f^{\dashv}[J]}\right)\{ (a + f^{\dashv}[J])^{-1} \} \cong \{ 0\}$, and since the rectangle commutes,
\begin{equation}\label{paul}
(\eta_{f(a)+J} \circ f_J)(a + \sqrt[\infty]{J}) = 0.
\end{equation}

However, we know that $(\eta_{f(a)+J} \circ f_J)(a + f^{\dashv}[J]) \in \left( \frac{B}{J}\right)\{ (f(a) + J)^{-1}\}^{\times}$, so equation \eqref{paul} tells us that $0 \in \left( \frac{B}{J}\right)\{ (f(a) + J)^{-1}\}^{\times}$, hence:
$$\left( \dfrac{B}{J}\right)\{(f(a) + J)^{-1} \} \cong \{ 0 \},$$
so $f(a) \in \sqrt[\infty]{J}$.
\end{proof}

\begin{remark}With exactly the same method used in the proof of \textbf{Theorem \ref{little}}, one proves that whenever $\mathfrak{p} \subseteq A$ is a prime ideal, $\sqrt[\infty]{\mathfrak{p}}$ is also a prime ideal.
\end{remark}

At this point it is natural to look for a $\mathcal{C}^{\infty}-$analog of the Zariski spectrum of a commutative unital ring. With this motivation, we give the following:

\begin{definition}For a $\mathcal{C}^{\infty}-$ring $A$, we define:
$${\rm Spec}^{\infty}\,(A) = \{ \mathfrak{p} \in {\rm Spec}\,(A) | \mathfrak{p} \, \mbox{is} \, \mathcal{C}^{\infty}-\mbox{radical} \}$$
equipped with the smooth Zariski topology defined by its basic open sets:
$$D^{\infty}(a) = \{ \mathfrak{p} \in {\rm Spec}^{\infty}\,(A) | a \notin \mathfrak{p} \} $$
\end{definition}

A more detailed study of the \index{smooth Zariski spectrum}smooth Zariski spectrum will be given in the next chapter.\\

Now we prove some properties of $\mathcal{C}^{\infty}-$radical ideals of a $\mathcal{C}^{\infty}-$ring $A$

\begin{proposition}\label{egito}The following results hold:
\begin{itemize}
\item[(a)]{Suppose that $(\forall \alpha \in \Lambda)(I_{\alpha} \in \mathfrak{I}^{\infty}_A)$. Then $\bigcap_{\alpha \in \Lambda} I_{\alpha} \in \mathfrak{I}^{\infty}_{A}$, that is, if $(\forall \alpha \in \Lambda)(I_{\alpha} \in \mathfrak{I}^{\infty}_A)$, then:
    $$\sqrt[\infty]{\bigcap_{\alpha \in \Lambda}I_{\alpha}} = \bigcap_{\alpha \in \Lambda}I_{\alpha} = \bigcap_{\alpha \in \Lambda} \sqrt[\infty]{I_{\alpha}}$$}
\item[(b)]{If $I \subseteq A$ is any ideal, then $\sqrt[\infty]{I} = \bigcap \{ \mathfrak{p} \in {\rm Spec}^{\infty}(A); I \subseteq \mathfrak{p} \}$}
\item[(c)]{Let $\{ I_{\alpha} | \alpha \in \Sigma \}$ an upward directed family of elements of $\mathfrak{I}^{\infty}_{A}$. Then $\bigcup_{\alpha \in \Sigma} I_{\alpha} \in \mathfrak{I}^{\infty}_{A}$.}
\end{itemize}
\end{proposition}
\begin{proof}
\begin{itemize}
\item[(a)]{Let $u \in \sqrt[\infty]{\bigcap_{\alpha \in \Lambda} I_{\alpha}}$. There exists $b \in \bigcap_{\alpha \in \Lambda} I_{\alpha}$ such that ${\rm Can}_u(b) \in (A\{ u^{-1}\})^{\times}$.\\

    Since $b \in \bigcap_{\alpha \in \Lambda} I_{\alpha}$, then $(\forall \alpha \in \Lambda)(\exists b_{\alpha} \in I_{\alpha})(b_{\alpha = b})({\rm Can}_u(b) \in (A\{ u^{-1}\})^{\times})$, so $u \in \sqrt[\infty]{I_{\alpha}} =  I_{\alpha}$, and therefore $u \in \bigcap_{\alpha \in \Lambda} I_{\alpha}$. This proves that $\sqrt[\infty]{\bigcap_{\alpha \in \Lambda} I_{\alpha}} \subseteq \bigcap_{\alpha \in \Lambda} I_{\alpha}$, so:
    $$\sqrt[\infty]{\bigcap_{\alpha \in \Lambda} I_{\alpha}} = \bigcap_{\alpha \in \Lambda} I_{\alpha},$$
    and
    $$\bigcap_{\alpha \in \Lambda} I_{\alpha} \in \mathfrak{I}^{\infty}_{A}.$$}
\item[(b)]{As we have proved in item e), the arbitrary intersection of $\mathcal{C}^{\infty}-$radical ideals is again $\mathcal{C}^{\infty}-$radical, so given any ideal $I$ there is a unique prime maximal $\mathcal{C}^{\infty}-$radical ideal $\mathfrak{q}$ such that $I \subseteq \mathfrak{q}$, so $I \subseteq \bigcap \{ \mathfrak{p} \in {\rm Spec}^{\infty}(A) | I \subseteq \mathfrak{p} \}$}
\item[(c)]{Given any $u \in \sqrt[\infty]{\bigcup_{\alpha \in \Sigma} I_{\alpha}}$, then there is some $b \in \bigcup_{\alpha \in \Sigma} I_{\alpha}$ such that ${\rm Can}_u(b) \in (A\{ u^{-1}\})^{\times}$. Since $\Sigma$ is upward directed, there is $\alpha_0 \in \Sigma$ such that $b \in I_{\alpha_0}$, so ${\rm Can}_u(b) \in (A\{ u^{-1}\})^{\times}$, therefore $u \in \sqrt[\infty]{I_{\alpha_0}} = I_{\alpha_0} \subseteq \bigcup_{\alpha \in \Sigma} I_{\alpha}$.Thus,

    $$\sqrt[\infty]{\bigcup_{\alpha \in \Sigma} I_{\alpha}} = \bigcup_{\alpha \in \Sigma}I_{\alpha}$$
    }
\end{itemize}
\end{proof}

\begin{remark}Due to the previous proposition, we have that $\mathfrak{I}^{\infty}_{A}$ is a \index{complete Heyting algebra}complete Heyting algebra.
\end{remark}

\begin{lemma}\label{2328}Let $A$ be a $\mathcal{C}^{\infty}-$ring and $S \subset A$ any multiplicative subset. Consider the following posets:
$$({\rm Spec}^{\infty}(A\{ S^{-1}\}), \subseteq)$$
and
$$(\{ \mathfrak{p} \in {\rm Spec}^{\infty}(A) | \mathfrak{p}\cap S = \varnothing \}, \subseteq)$$
The following poset maps:
$$\begin{array}{cccc}
{\rm Can}_S^{*}: & ({\rm Spec}^{\infty}(A\{ S^{-1}\}), \subseteq) & \rightarrow & (\{ \mathfrak{p} \in {\rm Spec}^{\infty}(A) | \mathfrak{p}\cap S = \varnothing \}, \subseteq)\\
                 & Q & \mapsto & {\rm Can}_S^{\dashv}[Q]
\end{array}$$

and

$$\begin{array}{cccc}
{{\rm Can}_S}_{*}: & (\{ \mathfrak{p} \in {\rm Spec}^{\infty}(A) | \mathfrak{p}\cap S = \varnothing \}, \subseteq) & \rightarrow & ({\rm Spec}^{\infty}(A\{ S^{-1}\}), \subseteq)\\
                 & P & \mapsto &\langle {\rm Can}_S[P] \rangle
\end{array}$$
are poset isomorphisms, each one inverse of the other.
\end{lemma}
\begin{proof}See p. 286 of \cite{moerdijk1986rings}.
\end{proof}

\begin{proposition}\label{coisinha}Let $A$ be a $\mathcal{C}^{\infty}-$ring. For every $\mathfrak{p} \in {\rm Spec}^{\infty}(A)$ let $\hat{\mathfrak{p}}$ denote the maximal ideal of $A_{\{ \mathfrak{p}\}} = A\{ {A \setminus \p}^{-1}\}$ and consider:
$${\rm Can}_{\mathfrak{p}}: A \to A_{\{ \mathfrak{p}\}}.$$
We have the following equalities:
$${\rm Can}_{\mathfrak{p}}^{\dashv}[\hat{\mathfrak{p}}] = \mathfrak{p}$$
and
$$({\rm Can}_{\mathfrak{p}}[\mathfrak{p}]) = \hat{\mathfrak{p}}$$
\end{proposition}
\begin{proof}Let us take $S = A \setminus \mathfrak{p}$. Since $\hat{\mathfrak{p}}$ is a maximal ideal, it is the largest element of ${\rm Spec}^{\infty}(A\{ (A \setminus \mathfrak{p})^{-1}\})$. Hence ${\rm Can}_{\mathfrak{p}}^{\dashv}[\hat{\mathfrak{p}}]$ is the largest element of $(\{ \mathfrak{p}' \in {\rm Spec}^{\infty}(A) | \mathfrak{p}'\cap (A\setminus \mathfrak{p}) = \varnothing\})$. Thus, by \textbf{Lemma \ref{2328}}, ${\rm Can}_{\mathfrak{p}}^{\dashv}[\hat{\mathfrak{p}}] = \mathfrak{p}$.
\end{proof}

\begin{proposition}\label{Margarida} If $D$ is a reduced $\mathcal{C}^{\infty}-$domain, then $D\{ a^{-1}\} \cong \{ 0\}$ implies $a=0$.
\end{proposition}
\begin{proof} If $D\{ a^{-1}\} \cong \{ 0\}$ then $\left( \dfrac{D}{(0)}\right)\{ {a + (0)}^{-1} \} \cong \{ 0\}$, so $a \in \sqrt[\infty]{(0)} = (0)$ (for $D$ is reduced), and $a = 0$.
\end{proof}

\begin{proposition}Any free $\mathcal{C}^{\infty}-$ring is a reduced $\mathcal{C}^{\infty}-$ring.
\end{proposition}
\begin{proof}First we prove the result for a free $\mathcal{C}^{\infty}-$ring on finitely many generators.\\

Suppose, by contraposition, that $f$ is not constant equal to zero. Then there is some $f \in \mathcal{C}^{\infty}(\mathbb{R}^k)$ such that $(\exists y_0 \in \mathbb{R}^k)(f(y_0) \neq 0)$ and without loss of generality, suppose $f(y_0)>0$. By the \textbf{Theorem of Sign Permanence}, there is some open subset $V \subseteq \mathbb{R}^k$ such that $y_0 \in V$ and $f \displaystyle\upharpoonright_{V}>0$. It is clear that $V \subseteq U_f = {\rm Coz}\,(f)$, so $f$ is a non-zero element of $\mathcal{C}^{\infty}(U_f) \cong \mathcal{C}^{\infty}(\mathbb{R}^n)\{ f^{-1}\}$, hence $\mathcal{C}^{\infty}(\mathbb{R}^n)\{ f^{-1}\} \ncong 0$.\\

Thus, if $f \in \mathcal{C}^{\infty}(\mathbb{R}^k)$ is such that $\mathcal{C}^{\infty}(\mathbb{R}^k)\{ f^{-1}\} \cong 0$ then $f = 0$, so $\sqrt[\infty]{(0)} = \{ f \in \mathcal{C}^{\infty}(\mathbb{R}^k) | \mathcal{C}^{\infty}(\mathbb{R}^k)\{ f^{-1}\} \cong 0 \} = (0)$.\\

Now consider the free $\mathcal{C}^{\infty}-$rings on the set $E$ of generators. We know that:

$$\mathcal{C}^{\infty}(\mathbb{R}^{E}) = \varinjlim_{E' \subseteq_{\rm fin} E} \mathcal{C}^{\infty}(\mathbb{R}^{E'}),$$

and since the directed colimit of $\mathcal{C}^{\infty}-$reduced $\mathcal{C}^{\infty}-$rings is again a $\mathcal{C}^{\infty}-$reduced $\mathcal{C}^{\infty}-$ring (by \textbf{Proposition \ref{Xango}}) it follows that $\mathcal{C}^{\infty}(\mathbb{R}^{E})$ is a $\mathcal{C}^{\infty}-$reduced $\mathcal{C}^{\infty}-$ring.
\end{proof}

\begin{proposition}\label{Quico}Let $\{ A_i\}_{i \in I}$ be a directed family of $\mathcal{C}^{\infty}-$rings, so we have the diagram:
$$\xymatrixcolsep{5pc}\xymatrix{
 & \varinjlim_{i \in I} A_i & \\
A_i \ar[ur]^{\alpha_i} \ar[rr]_{\alpha_{ij}}& & A_j \ar[ul]_{\alpha_j}
}$$
and let $(\mathfrak{p}_i)_{i \in I}$ be a compatible family of prime $\mathcal{C}^{\infty}-$radical ideals, that is:
$$(\mathfrak{p}_i)_{i \in I} \in \varprojlim_{i \in I} {\rm Spec}^{\infty}(A_i).$$
Under those circumstances,
$$\varinjlim_{i \in I} \mathfrak{p}_i = \bigcup_{i \in I} \alpha_i[\mathfrak{p}_i]$$
is a $\mathcal{C}^{\infty}-$radical prime ideal of $\varinjlim_{i \in I} A_i$.
\end{proposition}
\begin{proof}
  It suffices to prove that
  $$\dfrac{\varinjlim_{i \in I} A_i}{\varinjlim_{i \in I}\mathfrak{p}_i}$$
  is a $\mathcal{C}^{\infty}-$reduced domain.\\

  It is a fact that colimits commute with quotients, so:

  $$\dfrac{\varinjlim_{i \in I} A_i}{\varinjlim_{i \in I}\mathfrak{p}_i} \cong \varinjlim_{i \in I} \dfrac{A_i}{\mathfrak{p}_i}$$

  Now, since every $\mathfrak{p}_i$ is a $\mathcal{C}^{\infty}-$radical prime ideal of $A_i$, we have, for every $i \in I$, that $\dfrac{A_i}{\mathfrak{p}_i}$ is a $\mathcal{C}^{\infty}-$reduced $\mathcal{C}^{\infty}-$domain.\\

  The colimit of $\mathcal{C}^{\infty}-$reduced $\mathcal{C}^{\infty}-$domains is again a $\mathcal{C}^{\infty}-$reduced $\mathcal{C}^{\infty}-$domain, so $\dfrac{\varinjlim_{i \in I} A_i}{\varinjlim_{i \in I}\mathfrak{p}_i}$ is a domain and:
  $$\varinjlim_{i \in I} \mathfrak{p}_i$$
  is a $\mathcal{C}^{\infty}-$radical prime ideal of $\varinjlim_{i \in I} A_i$.
\end{proof}

\subsection{Separation Theorems for Smooth Commutative Algebra}

The main result of this section is:\\

\begin{theorem}[\index{Separation Theorems}\textbf{Separation Theorems}]\label{TS} Let $A$ be a $\mathcal{C}^{\infty}-$ring, $S \subseteq A$ be a subset of $A$ and $I$ be an ideal of $A$. Denote by $\langle S \rangle$ the multiplicative submonoid of $A$ generated by $S$. We have:
\begin{itemize}
  \item[(a)]{If $I$ is a  $\mathcal{C}^{\infty}-$radical ideal, then:

$$I \cap \langle S \rangle = \varnothing \iff {I} \cap S^{\infty-{\rm sat}} = \varnothing$$}
  \item[(b)]{If  $S \subseteq A$ is a $\mathcal{C}^{\infty}$-saturated subset, then:
$$I \cap S = \varnothing \iff \sqrt[\infty]{I} \cap S = \varnothing$$
}
 \item[(c)]{If $\mathfrak{p} \in {\rm Spec}^{\infty}\,(A)$ then $A\setminus \mathfrak{p} = (A \setminus \mathfrak{p})^{\infty-{\rm sat}}$}
 \item[(d)]{If  $S \subseteq A$ is a $\mathcal{C}^{\infty}$-saturated subset, then:
$$I \cap S = \varnothing \iff (\exists \mathfrak{p} \in {\rm Spec}^{\infty}\,(A))((I \subseteq \p)\& (\p \cap S = \varnothing)).$$}
 \item[(e)]{$\sqrt[\infty]{I} = \bigcap \{ \mathfrak{p} \in {\rm Spec}^{\infty}\,(A) | I \subseteq \mathfrak{p} \}$}
\end{itemize}
\end{theorem}
\begin{proof}
Ad (a): Since $\langle S \rangle \subseteq S^{\infty-{\rm sat}}$, it is clear that $(ii) \to (i)$.\\

We are going to show that $(i) \to (ii)$ via contraposition.\\

Suppose there exists some $b \in I \cap S^{\infty-{\rm sat}}$, so $\eta_S(b) \in (B\{ S^{-1}\})^{\times}.$\\

We have:
$$B\{ S^{-1}\} \cong_{\psi} \varinjlim_{S' \stackrel{\subseteq}{{\rm fin}} S} B\{{S'}^{-1} \} \cong_{\varphi} \varinjlim_{S' \stackrel{\subseteq}{{\rm fin}} S} B\left\{{\prod S'}^{-1} \right\},$$
so $(\varphi \circ \psi)(\eta_S(b)) \in \varinjlim_{S' \stackrel{\subseteq}{{\rm fin}} S} B \{ {\prod S'}^{-1}\}$ implies that there is some finite $S'' \subseteq S$ such that $\eta_{S''}(b) \in B\{ {S''}^{-1} \} \cong B\{{\prod S''}^{-1} \}$. Let $a = \prod S''$, so $a \in \langle S \rangle$. We have that $\eta_{a}(b) \in B \{ {S''}^{-1} \}$ implies $\eta_a(b) \in (B\{ a^{-1}\})^{\times}$ and by hypothesis, $b \in I$ so $a \in \sqrt[\infty]{I} = I$. Hence $a \in I \cap \langle S \rangle \neq \varnothing$, and the result is proved.\\

Ad (b): Given $b \in \sqrt[\infty]{I}\cap S$, there must be some $x \in I$ such that $\eta_b(x) \in A\{ b^{-1}\}^{\times}$, so $ x \in \{ b\}^{\infty-{\rm sat}} \subseteq S^{\infty-{\rm sat}} = S$. Thus $x \in I \cap S$ and $I \cap S = \varnothing$. The other way round is immediate since $I \subseteq \sqrt[\infty]{I}$.\\

Ad (c): Since $A\setminus \mathfrak{p} = \langle A \setminus \mathfrak{p}\rangle$, by item (a) we have:

$$\mathfrak{p} \cap (A\setminus \mathfrak{p}) \iff \mathfrak{p}\cap (A \setminus \mathfrak{p})^{\infty-{\rm sat}} = \varnothing$$

so $(A \setminus \mathfrak{p})^{\infty-{\rm sat}} \subseteq A \setminus \mathfrak{p}$. The other inclusion always holds, so $A \setminus \mathfrak{p} = (A \setminus \mathfrak{p})^{\infty-{\rm sat}}$.\\

Ad (d): Consider the set:

$$\Gamma_S := \{ J \in \mathfrak{I}(A) | (I \subseteq J)\& (S \cap J = \varnothing)\}$$

ordered by inclusion. It is straightforward to check that $(\Gamma_S, \subseteq)$ satisfies the hypotheses of \textbf{Zorn's Lemma}. Let $M$ be a maximal member of $\Gamma_S$. We are going to show that $M \in {\rm Spec}^{\infty}\,(A)$.\\

\textbf{Claim:} $M$ is a proper prime ideal of $A$.\\

In fact, $M$ is proper since $1 \in S$ and $S \cap M = \varnothing$.\\

The proof that $M$ is prime is made by contradiction.\\

If $a,a' \notin M$, then by maximality there are $\alpha, \alpha' \in A$ and $m,m' \in M$ such that $m+\alpha \cdot a \in S$ and $m'+\alpha'\cdot a' \in S$. Since $\mathcal{C}^{\infty}-$saturated sets are submonoids, it follows that $(m + \alpha \cdot a)\cdot (m' + \alpha' \cdot a') \in S$. Thus,

$$\underbrace{(m\cdot m' + m \cdot \alpha' \cdot a' + m' \cdot \alpha \cdot a)}_{\in M} + \underbrace{(\alpha \cdot \alpha')\cdot (a \cdot a')}_{\in S} \in S$$

If $a,a' \in M$, we get $M \cap S \neq \varnothing$, a contradiction. Thus, $a,a' \notin M$, and $M$ is a prime ideal.\\

\textbf{Claim:} $M = \sqrt[\infty]{M}$.\\

Since $M \cap S = \varnothing$ and $S$ is $\mathcal{C}^{\infty}-$saturated, by item (b) it follows that $\sqrt[\infty]{M}\cap S = \varnothing$, so $\sqrt[\infty]{M} \in \Gamma_S$. Since $M \subseteq \sqrt[\infty]{M}$, $\sqrt[\infty]{M} \in \Gamma_S$ and $M$ is a maximal element of $\Gamma_S$, it follows that $M = \sqrt[\infty]{M}$.\\

Hence, $M \in {\rm Spec}^{\infty}\,(A)$.\\

Ad (e): Clearly,

$$\sqrt[\infty]{I} \subseteq  \bigcap \{ \mathfrak{p} \in {\rm Spec}^{\infty}\,(A) | I \subseteq \mathfrak{p} \}$$

so we need only to prove the reverse inclusion.\\

Let $a \notin \sqrt[\infty]{I}$, so $S_a = \{ 1, a, a^2, \cdots\} \cap \sqrt[\infty]{I} = \varnothing$. Since $\dfrac{A}{I}\{ (a+I)^{-1}\} \cong \dfrac{A}{I}\{ (a^k+I)^{-1}\}$ for any $k \in \mathbb{N}$ such that $k \geq 1$ and since $\sqrt[\infty]{I}$ is a $\mathcal{C}^{\infty}-$radical ideal, by item (a), we have $(S_a)^{\infty-{\rm sat}}\cap \sqrt[\infty]{I}=\varnothing$, and by item (d), there is some $\mathfrak{p} \in {\rm Spec}^{\infty}\,(A)$ such that $\mathfrak{p} \supseteq \sqrt[\infty]{I} \supseteq I$ such that $a \notin \mathfrak{p}$.
\end{proof}

\begin{proposition} \label{equacaocobertura} Let $A$ be a
$C^\infty$-ring and let
$\{a_i : i \in I\} \subseteq A$. Denote ${\cal I} := \langle \{a_i : i
\in I\}\rangle$. Then the following are equivalent:

(a) ${\rm Spec}^\infty(A) = \bigcup_{i \in I} D^\infty(a_i)$

(b) $A = {\cal I}$

\end{proposition}

\begin{proof}

(b) $\Ra$ (a): Since there exists $\{i_1, \cdots, i_n\} \subseteq I$
such that $ 1_A = \sum_{j = 1}^n \lambda_j. a_{i_j}$
for some $\lambda_1, \cdots, \lambda_n \in A$, then
there is no $\mathfrak{p} \in {\rm Spec}^{\infty}(A)$ such that $\{a_{i_1},
\cdots, a_{i_n} \} \subseteq \mathfrak{p}$, i.e.
${\rm Spec}^\infty(A) \subseteq \bigcup_{i \in I} D^\infty(a_i)$.

(a) $\Ra$ (b): Suppose that $A \neq \mathcal{I}$. Then
$$ A^\times \cap \mathcal{I} = \emptyset.$$

Since $A^\times \subseteq A$ is a $\infty$-saturated subset of $A$
($A^\times = \eta_{1}^{-1} [ (A \{1^{-1}\})^\times]$),
then by the {\bf Separation Theorem
\ref{TS}.(b)},

$$A^\times \cap \sqrt[\infty]{\cal I}= \emptyset.$$

By the {\bf Separation Theorem \ref{TS}.(d)}, there is $\mathfrak{p}
\in {\rm Spec}^\infty(A)$ such that

$$ \sqrt[\infty]{\cal I} \subseteq \mathfrak{p}$$

thus $\mathfrak{p} \in {\rm Spec}^\infty(A) \setminus \bigcup_{i \in I}
D^\infty(a_i)$.

\end{proof}

\begin{proposition}\label{gaviota} Let $A$ be a $\mathcal{C}^{\infty}-$ring. The following assertions are equivalent:
\begin{itemize}
  \item[(1)]{$A$ is a $\mathcal{C}^{\infty}-$reduced  $\mathcal{C}^{\infty}-$domain;}
  \item[(2)]{There is some $\mathcal{C}^{\infty}-$field $\mathbb{K}$ and an injective $\mathcal{C}^{\infty}-$rings homomorphism $$\jmath: A \rightarrowtail \mathbb{K}$$}
  \item[(3)]{$A\{ (A \setminus \{ 0\})^{-1}\}$ is a $\mathcal{C}^{\infty}-$field, $\eta_{A\setminus \{ 0\}}^{\infty}: A \to A\{ (A \setminus \{ 0\})^{-1}\}$ is an injective $\mathcal{C}^{\infty}-$rings homomorphism with the following universal property: for every $\mathcal{C}^{\infty}-$field $\mathbb{K}$ and every injective $\mathcal{C}^{\infty}-$homomorphism $\jmath: A \to \mathbb{K}$ there is a unique $\mathcal{C}^{\infty}-$fields homomorphism:
      $$\overline{\jmath}: A\{ (A\setminus \{0\})^{-1}\} \rightarrowtail \mathbb{K}$$
      such that the following diagram commutes:
      $$\xymatrixcolsep{5pc}\xymatrix{
      A \ar[r]^{\eta_{A \setminus \{ 0 \}}^{\infty}} \ar@{>->}[dr]^{\jmath} & A\{(A \setminus \{ 0\})^{-1} \ar@{.>}[d]^{\overline{\jmath}} \}\\
          & \mathbb{K}
      }$$}
\end{itemize}
\end{proposition}
\begin{proof}
Ad (3) $\to$ (2): Suppose $A\{ (A\setminus \{ 0\})^{-1}\}$ is a $\mathcal{C}^{\infty}-$field with the claimed property.\\

 We always have the following inclusions:
 $$A\setminus \{ 0\} \subseteq (A\setminus \{ 0\})^{{\rm sat}} \subseteq (A\setminus \{ 0\})^{\infty-{\rm sat}}$$
 so we must prove that:
 $$x \notin A\setminus \{ 0\} \to x \notin (A\setminus \{ 0\})^{\infty-{\rm sat}}.$$
 Now, $x \notin A\setminus \{ 0\} \iff x =0$ and $\eta_{A\setminus \{ 0\}}^{\infty}(x)=0$ in a field implies $\eta_{A\setminus \{ 0\}}^{\infty}(x) \notin (A\setminus \{ 0\})^{\times}$, since fields are non-trivial rings ($0$ is not an invertible element). Hence $(A\setminus \{ 0\})^{\infty-{\rm sat}} \subseteq A\setminus \{ 0\}$, so
 $$(A\setminus \{ 0\})^{{\rm sat}} = (A\setminus \{ 0\})^{\infty-{\rm sat}}.$$

 Now we prove $\eta_{A\setminus \{ 0\}}^{\infty}: A \to A\{ (A\setminus \{ 0\})^{-1} \}$ is injective. Let $x \in U(A)$ be such that $\eta_{A\setminus \{ 0\}}^{\infty}(x)=0$. Then $x \notin (A\setminus \{ 0\})^{\infty-{\rm sat}} = A\setminus \{ 0\}$, i.e., $x=0$, so the result follows.

Ad (2) $\to$ (1):  Since $\jmath: A \to \mathbb{K}$ is an injective $\mathcal{C}^{\infty}-$rings homomorphism, we have $\{ 0_{A}\} = \jmath^{\dashv}[\{0_{\mathbb{K}}\}]$. Since $\mathbb{K}$ is a $\mathcal{C}^{\infty}-$field, it is, in particular, a $\mathcal{C}^{\infty}-$reduced $\mathcal{C}^{\infty}-$ring, so $\sqrt[\infty]{(0_{\mathbb{K}})} = (0_{\mathbb{K}})$. We have:
$$\{ 0_A\} = \jmath^{\dashv}[\{0_{\mathbb{K}}\}] = \jmath^{\dashv}[(0_{\mathbb{K}})] = \jmath^{\dashv}[\sqrt[\infty]{(0_{\mathbb{K}})}]$$
Since by \textbf{Theorem \ref{preim}} the pre-image of a $\mathcal{C}^{\infty}-$radical ideal is again a $\mathcal{C}^{\infty}-$radical ideal, we have $\jmath^{\dashv}[\sqrt[\infty]{(0_{\mathbb{K}})}] = \sqrt[\infty]{\jmath^{\dashv}[(0_{\mathbb{K}})]}$, and since $\jmath$ is injective, $\jmath^{\dashv}[(0_{\mathbb{K}})] = (0_{\mathbb{K}})$. Hence:
$$\{ 0_A\} = \jmath^{\dashv}[\{0_{\mathbb{K}}\}] = \jmath^{\dashv}[(0_{\mathbb{K}})] = \jmath^{\dashv}[\sqrt[\infty]{(0_{\mathbb{K}})}] = \sqrt[\infty]{\jmath^{\dashv}[(0_{\mathbb{K}})]} = \sqrt[\infty]{(0_{A})},$$
so $A$ is indeed a $\mathcal{C}^{\infty}-$reduced $\mathcal{C}^{\infty}-$ring.\\

Ad (1) $\to$ (3): Suppose $A$ is a $\mathcal{C}^{\infty}-$reduced $\mathcal{C}^{\infty}-$ring.\\

We claim that:
$$\eta_{A\setminus \{0\}}^{\infty}: A \to A\{{A\setminus \{0\}}^{-1} \}$$
is injective.\\

Let $a \in A$ be such that $\eta_{A\setminus \{0\}}^{\infty}(a)=0$.  By the item (ii), \textbf{Theorem \ref{38}}, there exists some $\lambda \in {A\setminus \{0\}}^{\infty-{\rm sat}}$ with $\lambda \cdot a = 0$ in $A$. By \textbf{Proposition \ref{marina}}, we can write:
$$A\setminus \{0\} = \bigcup_{S' \stackrel{\subseteq}{{\rm fin}} A\setminus \{0\}} S'$$
so given this $\lambda$ there is some finite $S' \subseteq A\setminus \{0\}$ such that $\lambda \in {S'}^{\infty-{\rm sat}}$ and $\lambda \cdot a = 0$. We have, then, $\eta_{S'}^{\infty}(a)=0$.\\

Since $S'$ is a finite set, ${S'}^{\infty-{\rm sat}} = \langle \{\prod S'\} \rangle^{\infty-{\rm sat}}$. Denote $\prod S'$ by $s'$, and we have:

$$\eta_{s'}(a)=0 \rightarrow (A\{ {s'}^{-1}\})\{ \underbrace{\eta_{s'}(a)^{-1}}_{=0}\} \cong \{0\}$$

for $0$ invertible in a $\mathcal{C}^{\infty}-$reduced $\mathcal{C}^{\infty}-$ring implies the ring is trivial. By the
\textbf{Proposition \ref{star}},
$$A\{ s'\cdot a^{-1}\} \cong (A\{ {s'}^{-1}\})\{ \eta_{s'}(a)^{-1}\} \cong \{ 0\} \to s'\cdot a \in \sqrt[\infty]{(0)} \subseteq A.$$
Since $A$ is a $\mathcal{C}^{\infty}-$reduced $\mathcal{C}^{\infty}-$ring, it follows that $s'\cdot a = 0$. Since $S' \subseteq A\setminus \{ 0\}$, $s' = \prod S' \neq 0$, and since $A$ is a $\mathcal{C}^{\infty}-$domain, $(s'\cdot a = 0)\wedge(s' \neq 0) \to a = 0$, so $\eta_{A\setminus \{ 0\}}^{\infty}$ is an injective $\mathcal{C}^{\infty}-$homomorphism.\\

Now we prove $\eta_{A\setminus \{ 0\}}^{\infty}$ has the desired universal property.\\

Let $\mathbb{K}$ be a $\mathcal{C}^{\infty}-$field and $\jmath: A \to \mathbb{K}$ be an injective $\mathcal{C}^{\infty}-$morphism. Under those circumstances,
$$\jmath [A\setminus \{ 0\}] \subseteq \mathbb{K} \setminus \{ 0\} = \mathbb{K}^{\times},$$
so by the universal property of $\eta_{A\setminus \{ 0\}}^{\infty}$ of inverting $A\setminus \{ 0\}$, there is a unique $\mathcal{C}^{\infty}-$homomorphism $\overline{\jmath}: A\{ {A\setminus \{ 0\}}^{-1} \} \to \mathbb{K}$ such that:
$$\xymatrixcolsep{3pc}\xymatrix{
A \ar[r]^{\eta_{A\setminus \{ 0\}}^{\infty}} \ar@{>->}[dr]^{\jmath}  & A\{{A\setminus \{ 0\}}^{-1} \} \ar@{.>}[d]^{\overline{\jmath}}\\
   & \mathbb{K}}$$
commutes.\\

Finally we show that $A\{ {A\setminus \{ 0\}}^{-1} \}$ is a $\mathcal{C}^{\infty}-$field. We are going to show that any non-zero element of $A\{ {A\setminus \{ 0\}}^{-1}\}$ is invertible.\\

By item (i) of \textbf{Theorem \ref{38}}, for every $\varphi \in A\{ {A\setminus \{ 0\}}^{-1}\}$ there are $a \in A$ and $b \in {A\setminus \{ 0\}}^{\infty-{\rm sat}} = A\setminus \{ 0\},$
the last equality occurring because $\eta_{A\setminus \{ 0\}}^{\infty}$ is injective and $A \neq \{ 0\}$. Hence $b \in A\setminus \{ 0\}$ .\\

Since $\varphi \neq 0$, $\eta_{A\setminus \{ 0\}}^{\infty}(a) \neq 0$, and since $\eta_{A\setminus \{ 0\}}^{\infty}$ is injective, this implies $a \neq 0$. We define:
$$\psi := \dfrac{\eta_{A\setminus \{ 0\}}^{\infty}(b)}{\eta_{A\setminus \{ 0\}}^{\infty}(a)},$$
and we note that:
$$\varphi \cdot \psi = 1,$$
so $\varphi$ is invertible.
\end{proof}


\begin{lemma} \label{af1} Let $A$ be a $\mathcal{C}^{\infty}-$ring and $\mathfrak{p}$ be any of its prime ideals. There is a unique $\mathcal{C}^{\infty}-$homomorphism:
$$\alpha_{\p}: A\{ {A\setminus \p}^{-1}\} \to \left( \dfrac{A}{\p} \right)\left\{ {\left( \dfrac{A}{\p}\right) \setminus \{ 0 + \p\}}^{-1}\right\}$$
such that the following rectangle commutes:
$$\xymatrixcolsep{7pc}\xymatrix{
A \ar[d]_{q_{\p}} \ar[r]^{\eta_{A\setminus \p}} & A\{ {A\setminus \p}^{-1} \} \ar@{.>}[d]^{\alpha_{\p}}\\
\dfrac{A}{\p} \ar[r]^(0.3){\eta_{\left( \frac{A}{\p}\right)\setminus \{ 0 + \p \}}} & \left( \frac{A}{\p}\right) \left\{ {\left( \frac{A}{\p}\right)\setminus \{ 0 + \p \}}^{-1}\right\}
}$$
\end{lemma}
\begin{proof}We are going to show that $\eta_{\left( \frac{A}{\p}\right)\setminus \{ 0 + \p \}}[q_{\p}[A\setminus \p]] \subseteq \left[\left( \frac{A}{\p}\right) \left\{ {\left( \frac{A}{\p}\right)\setminus \{ 0 + \p \}}^{-1}\right\} \right]^{\times}$.\\

Since $q_{\p}[\p]=\{0\}$, $q_{\p}[A \setminus \p] = \left(\dfrac{A}{\p}\right) \setminus \{ 0 + \p\}$, and by the very property that defines $\eta_{\left( \frac{A}{\p}\right)\setminus \{ 0 + \p \}}$, we have $\eta_{\left( \frac{A}{\p}\right)\setminus \{ 0 + \p \}}[q_{\p}[A\setminus \p]] \subseteq \left[\left( \frac{A}{\p}\right) \left\{ {\left( \frac{A}{\p}\right)\setminus \{ 0 + \p \}}^{-1}\right\} \right]^{\times}$. By the universal property of $\eta_{A \setminus \p}: A \to A\{ {A \setminus \p}^{-1} \}$, there exists a unique $\mathcal{C}^{\infty}-$homomorphism $\alpha_{\p}$ such that:
$$\xymatrixcolsep{7pc}\xymatrix{
A \ar[r]^{\eta_{A \setminus \p}} \ar[rd]_{\eta_{\left( \frac{A}{\p}\right)\setminus \{ 0+\p\}} \circ q_{\p}} & A\{ {A \setminus \p}^{-1} \} \ar[d]^{\alpha_{\p}}\\
   & \left( \frac{A}{\p}\right) \left\{ {\left( \frac{A}{\p}\right)\setminus \{ 0 + \p \}}^{-1}\right\}
}$$
commutes, so the result follows.
\end{proof}

Note that $\p \in {\rm Spec}^{\infty}\, (A) \iff \left( \dfrac{A}{\p}\right)$ is a reduced $\mathcal{C}^{\infty}-$domain, so whenever $\p \in {\rm Spec}^{\infty}\, (A)$ we have $A\{ {A \setminus \p}^{-1} \}$ a $\mathcal{C}^{\infty}-$local ring. Let $\mathfrak{m}_{\p}$ denote its unique maximal ideal.\\

We note that for $\p \in {\rm Spec}^{\infty}\,(A)$, $\p \cap (A \setminus \p) = \varnothing$. Since $\p$ is a proper prime ideal of $A$, it follows that $A\setminus \p$ is a submonoid of $A$, so by item (a) of \textbf{Theorem \ref{TS}} it follows that:
$$\p \cap (A \setminus \p)^{\infty-{\rm sat}} = \varnothing$$
hence
$$(A \setminus \p)^{\infty-{\rm sat}} \subseteq A \setminus \p$$
so $A \setminus \p$ is a $\mathcal{C}^{\infty}-$saturated set whenever $\p \in {\rm Spec}^{\infty}\, (A)$.

We have the following:

\begin{theorem}\label{Cher}Let $A$ be a $\mathcal{C}^{\infty}-$ring and let $\p \in {\rm Spec}^{\infty}\,(A)$. The ring $A\{ {A \setminus \p }^{-1}\}$ is a local $\mathcal{C}^{\infty}-$ring.
\end{theorem}
\begin{proof}
We are going to show that $(A\{ {A \setminus \p }^{-1}\})\setminus (A\{ {A \setminus \p }^{-1}\})^{\times}$ is an ideal by showing that:
$$(A\{ {A \setminus \p }^{-1}\})\setminus (A\{ {A \setminus \p }^{-1}\})^{\times} = \langle \eta_{A\setminus \p}[\p]\rangle = \{ \eta_{A \setminus \p}(a) \cdot \eta_{A \setminus \p}(b) | a \in \p \, \& \, b \in A \setminus \p \}$$

By \textbf{Theorem \ref{cara}},

$$\lambda = \dfrac{\eta_S(c)}{\eta_S(b)} \in (A\{ S^{-1}\})^{\times} \iff (\exists c' \in A)(\exists d \in (A \setminus \p)^{\infty-{\rm sat}})(d \cdot c' \cdot c \in (A \setminus \p)).$$

If $c \in \p$, then $\lambda$ would not be an invertible element, since for every $d, c' \in A$ we would have $d \cdot c' \cdot c \in \p$, hence $(\forall d \in (A \setminus \p)^{\infty-{\rm sat}})(\forall c \in A)(d \cdot c' \cdot c \notin (A \setminus \p))$. We conclude that $c \in (A \setminus \p)$, so taking $d = 1 \in (A \setminus \p)$ (for $\p$ is a proper prime ideal) and $c=1 \in A$ we have $1 \cdot 1 \cdot c' \in (A \setminus \p)$, so there are, indeed, $d \in (A \setminus \p)^{\infty-{\rm sat}}$ and $c \in A$ such that $d \cdot c' \cdot c = 1 \cdot 1 \cdot c \in (A \setminus \p)^{\infty-{\rm sat}}$.\\

From these considerations, it follows that $\lambda = \dfrac{\eta_S(c)}{\eta_S(b)}$ is non-invertible if, and only if, $c \in \p$, that is:
$$\lambda = \underbrace{\eta_S(c)}_{\in \eta_{A \setminus \p}[\p]} \cdot \overbrace{\eta_S(b)^{-1}}^{\in A\{ {A \setminus \p}^{-1}\}} \in \eta_{A \setminus \p}[\p],$$

so $A\{ {A \setminus \p}^{-1}\}\setminus (A\{ {A \setminus \p}^{-1}\})^{\times}$ is an ideal and $A\{ {A \setminus \p}^{-1}\}$ is a local $\mathcal{C}^{\infty}-$ring.
\end{proof}

\begin{theorem}\label{cheg}Let $A$ be a $\mathcal{C}^{\infty}-$ring and $\p$ one of its $\mathcal{C}^{\infty}-$radical prime ideals. As proven in the previous theorem, $A\{ {A \setminus \p}^{-1} \}$ is a local $\mathcal{C}^{\infty}-$ring, so we denote by $\mathfrak{m}_{\p}$ its unique maximal ideal. There is a unique $\mathcal{C}^{\infty}-$homomorphism:
$$\beta_{\p}: \left( \dfrac{A}{\p}\right) \to \dfrac{A\{ {A \setminus \p}^{-1}\}}{\mathfrak{m}_{\p}}$$
such that the following rectangle commutes:
$$\xymatrix{
A \ar[r]^{q_{\p}} \ar[d]^{\eta_{A\setminus \p}} & \left( \dfrac{A}{\p}\right) \ar@{.>}[d]^{\beta_{\p}}\\
A\{ {A \setminus \p}^{-1} \} \ar[r]_{q_{\mathfrak{m}_{\p}}} & \dfrac{A\{ {A \setminus \p}^{-1}\}}{\mathfrak{m}_{\p}}
}$$
\end{theorem}
\begin{proof}
We are going to show that $\p \subseteq \ker (q_{\mathfrak{m}_{\p}} \circ \eta_{A \setminus \p})$, so the result will follow from the \textbf{Homomorphism Theorem}, i.e., there will be a unique $\beta_{\p}$ such that the following triangle commutes:
$$\xymatrix{
A \ar[r]^{q_{\p}} \ar@/_2pc/[dr]_{q_{\mathfrak{m}_{\p}}\circ \eta_{A \setminus \p}} & \left( \dfrac{A}{\p}\right) \ar[d]^{\beta_{\p}}\\
   & \dfrac{A\{ {A \setminus \p}^{-1}\}}{\mathfrak{m}_{\p}}
}.$$

Since by the proof of \textbf{theorem \ref{Cher}} we have $\mathfrak{m}_{\p} = \langle \eta_{A \setminus \p}[\p]\rangle = \{ \eta_{A \setminus \p}(a) \cdot \eta_{A \setminus \p}(b) | a \in \p \, \& \, b \in (A \setminus \p) \}$, it follows that:
$$q_{\mathfrak{m}_{\p}}[\eta_{A \setminus \p}[\p]] = 0 + \mathfrak{m}_{\p},$$
what proves the result.
\end{proof}

\begin{theorem}\label{Pedro}Let $A$ be a $\mathcal{C}^{\infty}-$ring and $\mathfrak{p}$ be any of its prime ideals. There are unique isomorphisms:
$$\varphi_{\mathfrak{p}} :   \dfrac{A\{ {A\setminus \p }^{-1}\}}{\mathfrak{m}_{\p}} \to \left( \dfrac{A}{\p}\right)\left\{  {\left( \dfrac{A}{\p}\right)\setminus \{ 0 + \p\} }^{-1}\right\}$$
and
$$\psi_{\p}:\left( \dfrac{A}{\p}\right)\left\{  {\left( \dfrac{A}{\p}\right)\setminus \{ 0 + \p\}}^{-1}\right\} \to \dfrac{A\{ {A\setminus \p }^{-1}\}}{\mathfrak{m}_{\p}} $$
such that the following diagram commutes:
$$\xymatrixcolsep{5pc}\xymatrix{
   & A\{ {A\setminus \p}^{-1} \} \ar[r]^{q_{\mathfrak{m}_{\p}}} \ar[ddr]_(0.25){\alpha_{\p}} &  \dfrac{A\{ {A \setminus \p}^{-1} \}}{\mathfrak{m}_{\p}} \ar@<-3pt>[dd]_{\varphi_{\p}}\\
A \ar[ur] \ar[dr] &    &      \\
      & \dfrac{A}{\p} \ar[r]_(0.3){\eta_{\left(\frac{A}{\p} \setminus \{0 + \p \}\right)}} \ar[uur]^(0.25){\beta_{\p}} & \left( \dfrac{A}{\p} \right)\left\{ \left( \dfrac{A}{\p}\right) \setminus \{ 0 + \p\}\right\} \ar@<-3pt>[uu]_{\psi_{\p}}
}$$
\end{theorem}
\begin{proof}
First we are going to show that there exists a unique $\psi_{\p}$ such that:
$$\xymatrixcolsep{5pc}\xymatrix{
\dfrac{A}{\p} \ar[dr]_{\beta_{\p}} \ar[r]^(0.25){\eta_{\frac{A}{\p} \setminus \{ 0 + \p\}}} & \left( \dfrac{A}{\p}\right)\left\{ {\dfrac{A}{\p}\setminus \{ 0 + \p\}}^{-1}\right\} \ar@{.>}[d]^{\psi_{\p}}\\
   & \dfrac{A\{ {A \setminus \p}^{-1} \}}{\mathfrak{m}_{\p}}
}$$
 is a commutative diagram.\\

 Note that $\beta_{\p}\left[ \dfrac{A}{\p}\setminus \{0+ \p \}\right] \subseteq \left( \dfrac{A\{ {A \setminus \p}^{-1}\}}{\m_{\p}}\right)^{\times}$.\\

 By definition, $\beta_{\p}$ is the only $\mathcal{C}^{\infty}-$homomorphism such that the following diagram commutes:
 $$\xymatrix{
A \ar[r]^{q_{\p}} \ar[d]^{\eta_{A\setminus \p}} & \left( \dfrac{A}{\p}\right) \ar@{.>}[d]^{\beta_{\p}}\\
A\{ {A \setminus \p}^{-1} \} \ar[r]_{q_{\mathfrak{m}_{\p}}} & \dfrac{A\{ {A \setminus \p}^{-1}\}}{\mathfrak{m}_{\p}}
}$$
so
$$\beta_{\p}[\{ a + \p | a \notin \p \}] = \{ \eta_{\p}(a) + \m_{\p} | a \in (A \setminus \p) \} \subseteq \left( \dfrac{A\{ {A \setminus \p}^{-1}\}}{\m_{\p}}\right)^{\times}$$
and $\left( \dfrac{A\{ {A \setminus \p}^{-1}\}}{\m_{\p}}\right)^{\times} = \dfrac{A\{ {A \setminus \p}^{-1}\}}{\m_{\p}} \setminus \{ 0 + \m_{\p}\}$, for $a \in (A \setminus \p) \to \eta_{A \setminus \p}(a) \notin \m_{\p} = \langle \eta_{A \setminus \p}[\p]\rangle$. By the universal property of $\eta_{\frac{A}{\p}\setminus \{ 0 + \p\}}$, it follows that there is a unique $\psi_{\p}$ with the desired property.\\

Now we claim that there is a unique $\varphi_{\p}$ such that the following triangle commutes:
$$\xymatrixcolsep{5pc}\xymatrix{
A\{{A \setminus \p}^{-1} \} \ar@/_2pc/[dr]^{\alpha_{\p}} \ar[r]^{q_{\m_{\p}}} & \dfrac{A\{ {A \setminus \p}^{-1}\}}{\m_{\p}} \ar@{.>}[d]^{\varphi_{\p}}\\
    & \left( \dfrac{A}{\p} \right)\left\{ {\dfrac{A}{\p} \setminus \{0+\p \}}^{-1}\right\}
}$$

Note that $\alpha_{\p}[\m_{\p}] \subseteq \{ 0 + \m_{\p}\}$. Indeed, let $\dfrac{\eta_{A\setminus \p}(a)}{\eta_{A\setminus \p}(b)} \in \m_{\p}$, that is, an element such that $a \in \p$ and $b \in A \setminus \p$. By the very definition of $\alpha_{\p}$ given in the \textbf{Lemma \ref{af1}},
$$\alpha_{\p}\left( \dfrac{\eta_{A\setminus \p}(a)}{\eta_{A\setminus \p}(b)}\right) = \dfrac{\eta_{\frac{A}{\p} \setminus \{ 0 + \p\}}\overbrace{(a + \p)}^{\in \p}}{\eta_{\eta_{\frac{A}{\p}} \setminus \{ 0 + \p\}}(b)} = 0+ \m_{\p}.$$
By the \textbf{Homomorphism Theorem}, there is a unique $\varphi_{\p}$ such that the following diagram commutes:

$$\xymatrix{
A\{{A \setminus \p}^{-1} \} \ar@/_2pc/[dr]^{\alpha_{\p}} \ar[r]^{q_{\m_{\p}}} & \dfrac{A\{ {A \setminus \p}^{-1}\}}{\m_{\p}} \ar@{.>}[d]^{\varphi_{\p}}\\
    & \left( \dfrac{A}{\p} \right)\left\{ {\dfrac{A}{\p} \setminus \{0+\p \}}^{-1}\right\}
}$$

Finally, we are going to show that $\varphi_{\p} \circ \psi_{\p} = {\rm id}_{\frac{A\{ {A \setminus \p}^{-1}\}}{\m_{\p}}}$.\\

Consider the following diagram:

$$\xymatrixcolsep{5pc}\xymatrix{
    & A\{ {A \setminus \p}^{-1}\} \ar[r]^{q_{\m_{\p}}} & \dfrac{A\{{A \setminus \p}^{-1} \}}{\m_{\p}} \ar[d]^{\varphi_{\p}} \ar@/^6pc/[dd]^{{\rm id}_{\frac{A\{ {A \setminus \p}^{-1}\}}{\m_{\p}}}} \\
A \ar[ur]^{\eta_{A\setminus \p}} \ar[dr]^{\eta_{A \setminus \p}} \ar@{->>}[r]^{q_{\p}} & \dfrac{A}{\p} \ar[ur]^{\beta_{\p}} \ar[dr]^{\alpha_{\p}} \ar[r]^{\eta_{\frac{A}{\p} \setminus \{ 0 + \p\}}} & \left( \dfrac{A}{\p}\right)\left\{ \dfrac{A}{\p} \setminus \{ 0 + \p\}\right\} \ar[d]^{\psi_{\p}}\\
   & A\{ {A \setminus \p}^{-1}\} \ar[r]^{q_{\m_{\p}}} & \dfrac{A\{ {A \setminus \p}^{-1}\}}{\m_{\p}}}$$

Note that the upper  inner triangle in the above diagram commutes because $q_{\p}$ is an epimorphism and
$$(\varphi_{\p}\circ \beta_{\p}) \circ q_{\p} = \eta_{\frac{A}{\p}\setminus \{ 0 + \p\}} $$

By the uniqueness granted by the \textbf{Homomorphism Theorem}, it follows that $\psi_{\p} \circ \varphi_{\p} = {\rm id}_{\frac{A\{ {A \setminus \p}^{-1}\}}{\m_{\p}}}$.\\

Similarly, by the uniqueness granted by the universal property of the ring of fractions $\eta_{\frac{A}{\p}\setminus \{ 0 + \p\}}$, it follows that:

$$\varphi_{\p} \circ \psi_{\p} = {\rm id}_{\frac{A}{\p}\left\{ {\frac{A}{\p}\setminus \p}^{-1}\right\}}$$

Hence,
$$\dfrac{A\{ {A \setminus \p}^{-1}\}}{\m_{\p}} \cong \left( \dfrac{A}{\p}\right) \left\{ {\dfrac{A}{\p} \setminus \{ 0 + \p\}}^{-1}\right\}$$
\end{proof}

We now present a result that relates the saturation of $S$ and the saturation of $q_I[S]$.\\

\begin{theorem}\label{Skyline}Let $A$ be a $\mathcal{C}^{\infty}-$ring, $S \subseteq A$ and $I$ be an ideal of $A$. Then
$$(\beta + I) \in (S+I)^{\infty-{\rm sat}} \iff (\exists \gamma \in A)(\exists s \in S^{\infty-{\rm sat}})(s \cdot (\beta \gamma - 1_A) \in I)$$
\end{theorem}
\begin{proof}
By definition, $(\beta + I) \in (S+I)^{\infty-{\rm sat}}$ occurs if, and only if, $\eta_{S+I}(\beta +I) \in \left( \frac{A}{I}\right)\{ {S+ I}^{-1}\}^{\times}$.\\

By \textbf{Corollary \ref{Jeq}}, we have the following $\mathcal{C}^{\infty}-$isomorphism:
$$\mu^{-1}: \left( \frac{A}{I} \right) \stackrel{\cong}{\rightarrow} \frac{A\{ S^{-1}\}}{\langle \eta_S[I]\rangle}$$

and since $\mu^{-1}(\eta_{S+I}(\beta + I)) \in  \frac{A\{ S^{-1}\}}{\langle \eta_S[I] \rangle}$, then: $$\mu^{-1}(\eta_{S+I}(\beta)) = \eta_S(\beta) + \langle \eta_S[I] \rangle \frac{A\{ S^{-1}\}}{\langle \eta_S[I]\rangle}.$$

Now, $\mu^{-1}(\eta_{S+I}(\beta+I)) \in \left( \frac{A\{ S^{-1}\}}{\langle \eta_S[I]\rangle}\right)^{\times}$ occurs if, and only if, there is some $\gamma \in A$ such that:
$$(\mu^{-1}(\eta_{S+I}(\beta + I))) \cdot (\eta_S(\gamma) + \langle \eta_S[I] \rangle) = 1_A + \langle \eta_S[I] \rangle$$

$$(\eta_S(\beta) + \langle \eta_S[I]\rangle) \cdot (\eta_S(\gamma) + \langle \eta_S[I] \rangle) = 1_A + \langle \eta_S[I] \rangle$$

or, equivalently, if and only if,
$$\eta_S(\beta \cdot \gamma - 1_A) \in \langle \eta_S[I]\rangle .$$

By \textbf{Corollary \ref{Cindy}}, $\langle \eta_S[I]\rangle = \{ \eta_S(b)\cdot \eta_S(d) | (b \in I) \& (d \in S^{\infty-{\rm sat}})\}$, so
$$\eta_S(\beta \cdot \gamma - 1_A) \in \langle \eta_S[I]\rangle  \iff (\exists i \in I)(\exists d \in S^{\infty-{\rm sat}})(\eta_S(\beta \cdot \gamma - 1_A) = \frac{\eta_S(i)}{\eta_S(d)}),$$
that is, if and only if:
$$\eta_S(\beta \cdot \gamma - 1_A) \cdot \eta_S(d) = \eta_S(i)$$
or
$$\eta_S(d \cdot (\beta \cdot \gamma - 1_A) - i) = 0.$$

By item (ii) of \textbf{Theorem 1.4} of \cite{moerdijk1986rings},

$$(\eta_S(d \cdot (\beta \cdot \gamma - 1_A) - i) = 0) \iff (\exists c \in S^{\infty-{\rm sat}})(c \cdot [d \cdot (\beta \cdot \gamma - 1_A) - i] = 0).$$

and

$$c \cdot [d \cdot (\beta \cdot \gamma - 1_A) - i] = \underbrace{(c \cdot d)}_{\in S^{\infty-{\rm sat}}}[(\beta \cdot \gamma - 1_A) - i] = 0$$

so it is sufficient to take $s = c \cdot d \in S^{\infty-{\rm sat}}$ in order to get:

$$s\cdot (\beta \cdot \gamma - 1) \in I.$$

Conversely, suppose that:

$$(\exists \gamma \in A)(\exists s \in S^{\infty-{\rm sat}})(s \cdot (\beta \cdot \gamma - 1_A) \in I)$$

Let $s \cdot (\beta \cdot \gamma - 1_A) = i \in I$. We have:

$$\eta_S(s \cdot (\beta \cdot \gamma -  1_A)) = \eta_S(i) \in \eta_S[I],$$

hence:

$$\eta_S(\beta \cdot \gamma - 1_A) = \frac{\eta_S(i)}{\eta_S(s)} \in \langle \eta_S[I]\rangle$$

and

$$\eta_S(\beta)\cdot \eta_S(\gamma) - \eta_S(1_A) \in \langle \eta_S[I]\rangle$$
so

$$(\eta_S(\beta) + \langle \eta_S[I]\rangle) \cdot (\eta_S(\gamma) + \langle \eta_S[I]\rangle) = \eta_S(1_A) + \langle \eta_S[I]\rangle,$$

hence

$$\eta_S(\beta) + \langle \eta_S[I]\rangle \in \left( \frac{A\{ S^{-1}\}}{\langle \eta_S[I]\rangle}\right)^{\times}$$

$$\mu^{-1}(\eta_S(\beta) + \langle \eta_S[I]\rangle) \in \mu^{-1} \left[ \left(\frac{A\{ S^{-1}\}}{\langle \eta_S[I]\rangle}\right)^{\times}\right] = \left(\left( \frac{A}{I}\right)\{ {S+I}^{-1} \}\right)^{\times}$$

and

$$\eta_{S+I}(\beta + I) = \mu^{-1}(\eta_S(\beta) + \langle \eta_S[I]\rangle) \in \left(\left( \frac{A}{I}\right)\{ {S+I}^{-1} \}\right)^{\times}$$

hence

$$(\beta+I) \in (S+I)^{\infty-{\rm sat}}.$$

\end{proof}

\begin{proposition}\label{batata}
Let $A$ be $\mathcal{C}^{\infty}-$ring, $I$ be any of its ideals $I$, $q_I : A \to \frac{A}{I}$ be the standard projection map and $a \in A$ be any element. We have a relation between $\{ a\}^{\infty-{\rm sat}}$ and $\{ a+I\}^{\infty-{\rm sat}}$ given by the following implication:
$$(\forall b \in A)((b \in \{ a\}^{\infty-{\rm sat}}) \rightarrow (q_I(b) \in \{ a + I\}^{\infty-{\rm sat}}))$$
\end{proposition}
\begin{proof}
Consider the following diagram:
$$\xymatrixcolsep{5pc}\xymatrix{
A \ar[d]^{q_I} \ar[r]^{\eta_a} & A\{ a^{-1}\} \ar[d]^{\exists ! \varphi}\\
\dfrac{A}{I} \ar[r]^{\eta_{a+I}} & \left( \dfrac{A}{I}\right)\{ (a+I)^{-1}\}
}$$
where $\varphi$ is given by the universal property of $A\{ a^{-1}\}$, since $\eta_{a+I}(q_I(a)) = \eta_{a+I}(a+I) \in \left( \dfrac{A}{I} \{ (a+I)^{-1}\}\right)^{\times}$.\\

Given $b \in \{ a\}^{\infty-{\rm sat}}$, so $\eta_a(b) \in (A\{ a^{-1}\})^{\times}$, we have $\varphi(\eta_a(b)) \in \left( \dfrac{A}{I} \{ (a+I)^{-1}\}\right)^{\times}$ (for, the homomorphic image of invertible elements are invertible). But since the diagram above commutes, it follows that $\eta_{a+I}(b+I) = \eta_{a+I}(q_I(b)) = \varphi(\eta_a(b)) \in \left( \dfrac{A}{I} \{ (a+I)^{-1}\}\right)^{\times}$, so $q_I(b) \in \left( \dfrac{A}{I} \{ (a+I)^{-1}\}\right)^{\times}$ and $b+I \in \{ a + I\}^{\infty-{\rm sat}}$.\\

Now suppose $(s+I) \in \{ a + I\}^{\infty-{\rm sat}}$, so by \textbf{Theorem \ref{Skyline}} there are $\sigma \in \{ a\}^{\infty-{\rm sat}}$ and $\gamma \in A$ such that $\sigma \cdot (s \cdot \gamma - 1) \in I$. By \textbf{Theorem \ref{cara}}, in order to show that $s \in \{ a\}^{\infty-{\rm sat}}$ it suffices to show that these $\sigma \in \{ a\}^{\infty-{\rm sat}}$ and $\gamma \in A$ are such that $\sigma \cdot \gamma \cdot s \in \{ a\}^{\infty-{\rm sat}}$.\\

$$(\sigma \cdot (s \cdot \gamma - 1) \in I) \iff (\sigma \cdot \gamma \cdot s - \sigma \in I).$$

Since $I \cap \{ a\}^{\infty-{\rm sat}} = \varnothing$ (for if it was not the case we would have $\left( \frac{A}{I}\right)\{ (a + I)^{-1}\} \cong \{ 0\}$), it follows that $\sigma \cdot (\gamma \cdot s - 1) \notin \{ a \}^{\infty-{\rm sat}}$. Since $\{ a\}^{\infty-{\rm sat}}$ is a multiplicative set,
$$\sigma \cdot (\gamma \cdot s - 1) \notin \{ a\}^{\infty-{\rm sat}} \to (\sigma \notin \{ a\}^{\infty-{\rm sat}}) \vee (\gamma \cdot s - 1 \notin \{ a\}^{\infty-{\rm sat}})).$$

It is not the case that $\sigma \notin \{ a\}^{\infty-{\rm sat}}$, so we must have $\gamma \cdot s - 1 \notin \{ a\}^{\infty-{\rm sat}}$, that is, $\eta_a(\gamma)\cdot \eta_a(s) - 1 \notin (A\{ a^{-1}\})^{\times}$.\\




\end{proof}

\begin{remark}\label{eu} The converse of the above implication is not always true, that is
$$(\exists b \in A)((q_I(b) \in \{ a + I\}^{\infty-{\rm sat}}) \& (b \notin \{ a\}^{\infty-{\rm sat}})).$$

Consider $A = \mathcal{C}^{\infty}(\R)$,
$$\begin{array}{cccc}
    a: & \R & \to & \R \\
     & x & \mapsto & \begin{cases}
                       e^{-\frac{1}{1 - x^2}}, & \mbox{if } |x|<1 \\
                       0, & \mbox{otherwise}.
                     \end{cases}
  \end{array}$$
  
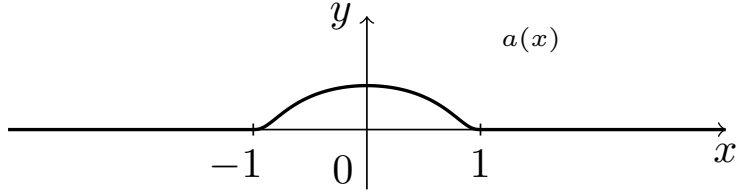
\begin{center}
    \resizebox{0.6\textwidth}{!}{\begin{tikzpicture}[domain=-3:3]
\draw [->] (0,-0.5) --(0,0.95) node (yaxis) [left] {$y$};
\draw [->] (-3,0) --(3,0) node (yaxis) [below] {$x$};  
\draw[-] (-0.95,-0.05)--(-0.95,0.05);
\draw[-] (0.95,-0.05)--(0.95,0.05);
\draw (0.95,-0.05) node[below]{$1$};
\draw (-1.1,-0.05) node[below]{$-1$};
\draw (-0.2,-0.1) node[below]{$0$};
\draw[black, samples=100, smooth, domain=-0.95:0.95, thick] plot (\x,{exp(-1/(1-\x*\x)});
\draw[thick] (-3,0)--(-0.95,0);
\draw[thick] (0.95,0)--(3,0);
\draw (1,0.75)node[right]{\tiny{$a(x)$}};
\end{tikzpicture}}
\end{center}  
  
and $I = \m_{0} = \{ g \in \mathcal{C}^{\infty}(\R) | g(0)=0\}$, which is a maximal ideal. Since $\m_{0}$ is a maximal ideal, $\dfrac{\mathcal{C}^{\infty}(\R)}{\m_{0}}$ is a field, so every non-zero element of it must be invertible and $\dfrac{\mathcal{C}^{\infty}(\R)}{\m_{0}}\{ (a+\m_{0})^{-1}\} \cong \dfrac{\mathcal{C}^{\infty}(\R)}{\m_{0}}$

Consider $h (x) =a (x) - a(\frac{1}{2})$, which is a smooth function. 

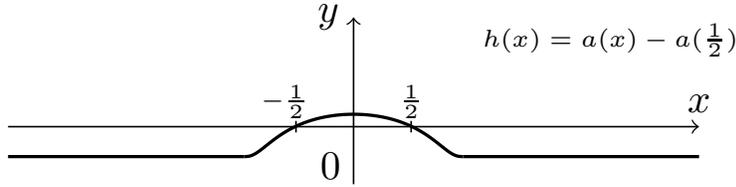
\begin{center}
    \resizebox{0.6\textwidth}{!}{\begin{tikzpicture}[domain=-3:3]
\draw [->] (0,-0.5) --(0,0.95) node (yaxis) [left] {$y$};
\draw [->] (-3,0) --(3,0) node (yaxis) [above] {$x$};  
\draw[-] (-0.5,-0.05)--(-0.5,0.05);
\draw[-] (0.5,-0.05)--(0.5,0.05);
\draw (0.5,-0.05) node[above]{\tiny{$\frac{1}{2}$}};
\draw (-0.6,-0.05) node[above]{\tiny{$-\frac{1}{2}$}};
\draw (-0.2,-0.1) node[below]{$0$};
\draw[shift={(0,-0.26)},black, samples=100, smooth, domain=-0.95:0.95, thick] plot (\x,{exp(-1/(1-\x*\x)});
\draw[thick] (-3,-0.26)--(-0.95,-0.26);
\draw[thick] (0.95,-0.26)--(3,-0.26);
\draw (1,0.75)node[right]{\tiny{$h(x)=a(x)-a(\frac{1}{2})$}};
\end{tikzpicture}}
\end{center}

Since $h(\frac{1}{2})=0$, $h \notin \mathcal{C}^{\infty}(\R)\{ a^{-1}\}$, for $\frac{1}{2} \in {\rm Coz}\,(a)$. However, since $h(0)=a(0)-a(\frac{1}{2}) = e^{-1}-e^{-\frac{4}{3}} \neq 0$, $h \notin \m_{0}$, so $h + \m_{0} \neq 0 + \m_{0} \in \frac{\mathcal{C}^{\infty}(\R)}{\m_{0}}$, hence it is invertible.
\end{remark}

\begin{remark}\label{Rmk15} Let $A$ be a $\mathcal{C}^{\infty}-$ring, $a \in A$ and $b \in A$ be any two elements of $A$ such that $(a)=(b)$. Under those circumstances there is a unique $\mathcal{C}^{\infty}-$isomorphism $\sigma^{a}_{b}: A\{ a^{-1}\} \to A\{ b^{-1}\}$ such that the following diagram commutes:
$$\xymatrix{
A \ar[r]^{\eta_a} \ar[rd]^{\eta_b} & A\{ a^{-1}\} \ar[d]^{\sigma^{b}_{a}}\\
   & A\{ b^{-1}\}
}$$
\end{remark}
\begin{proof}Suppose $a \neq 0$. Since $(a) = (b)$, in particular we have $a \in (b)$, so there is some $y \in A$ such that $a = y \cdot b$. Also, since $b \in (a)$, there is some $z \in A$ such that $b = z \cdot a$.\\

We have $a = y \cdot b = y \cdot (z \cdot a) = (y \cdot z) \cdot a$, so $A\{ a^{-1}\} = A\{ (yza)^{-1} \}$, hence $\eta_a(yza) \in (A\{ a^{-1}\})^{\times}$, i.e., $\eta_a(yz)\cdot \eta_a(a) \in (A\{ a^{-1}\})^{\times}$. Since $\eta_a(a) \in (A\{ a^{-1}\})^{\times}$, it follows that $\eta_a(yz) = \eta_a(y)\cdot \eta_a(z) \in (A\{ a^{-1}\})^{\times}$, so we have both $\eta_a(y) \in (A\{ a^{-1}\})^{\times}$ and $\eta_a(z) \in (A\{ a^{-1}\})^{\times}$.\\

Under those circumstances, $\eta_a(b) = \eta_a(z) \cdot \eta_a(a) \in (A\{ a^{-1}\})^{\times}$. By the universal property of $\eta_b : A \to A\{ b^{-1}\}$, there must exist a unique $\sigma^{b}_{a}: A\{ b^{-1}\} \to A\{ a^{-1}\}$ such that the following diagram commutes:
$$\xymatrix{
A \ar[r]^{\eta_b} \ar[rd]^{\eta_a} & A\{ b^{-1}\} \ar@{.>}[d]^{\exists ! \sigma^{b}_{a}}\\
   & A\{ a^{-1}\}
}$$

Analogously, we conclude that there is a unique $\mathcal{C}^{\infty}-$homomorphism $\sigma^{a}_{b}: A\{ a^{-1}\} \to A\{ b^{-1}\}$ such that:

$$\xymatrix{
A \ar[r]^{\eta_b} \ar[rd]^{\eta_a} & A\{ b^{-1}\} \ar[d]^{\sigma^{a}_{b}}\\
   & A\{ a^{-1}\}
}$$
commutes.\\

By uniqueness, it follows that $(\sigma^{a}_{b})^{-1} = (\sigma^{b}_{a})$, and the result follows.
\end{proof}

\begin{proposition}\label{advo} If $A$ is a $\mathcal{C}^{\infty}-$reduced $\mathcal{C}^{\infty}-$domain, then:
\begin{itemize}
  \item[(i)]{$A\{ {A\setminus \{ 0\}}^{-1}\}$ is a $\mathcal{C}^{\infty}-$field;}
  \item[(ii)]{${\rm Can}_{A \setminus \{ 0\}} : A \to A\{ {A\setminus \{ 0\}}^{-1}\}$ is a $\mathcal{C}^{\infty}-$monomorphism;}
  \item[(iii)]{If $\mathbb{K}$ is a $\mathcal{C}^{\infty}-$field and $\jmath: A \to \mathbb{K}$ is a $\mathcal{C}^{\infty}-$monomorphism, then there is a unique $\mathcal{C}^{\infty}-$homo\-morphism $\widetilde{\jmath}: A\{ {A\setminus \{ 0\}}^{-1}\} \to \mathbb{K}$ such that the following diagram commutes:
      $$\xymatrixcolsep{5pc}\xymatrix{
      A \ar[r]^{{\rm Can}_{A \setminus \{ 0\}}} \ar[dr]_{\jmath} & A\{ {A\setminus \{ 0\}}^{-1}\} \ar@{-->}[d]^{ \exists ! \widetilde{\jmath}}\\
       & \mathbb{K}
            }$$
             }
\end{itemize}
Thus,  ${\rm Frac}(A):= A\{ {A\setminus \{ 0\}}^{-1}\}$ is the $\mathcal{C}^{\infty}-$field of fractions of $A$,
\end{proposition}
\begin{proof}
Ad (i): By item (c) of \textbf{Theorem \ref{TS}}, since $(0) \in {\rm Spec}^{\infty}\,(A)$ then $A \setminus \{ 0\} = (A \setminus \{ 0\})^{\infty-{\rm sat}}$. \\

Given any $\alpha \in A\{ {A\setminus \{ 0\}}^{-1}\}$, there are $x \in A$ and $y \in (A \setminus \{ 0\})^{\infty-{\rm sat}}=A\setminus \{ 0\}$ such that:

$$\alpha = \dfrac{{\rm Can}_{A \setminus \{ 0\}}(x)}{{\rm Can}_{A \setminus \{ 0\}}(y)}.$$

If $\alpha \neq 0$, then $x \neq 0$ and $x \in A \setminus \{ 0\}$. Thus, $\beta = \dfrac{{\rm Can}_{A \setminus \{ 0\}}(y)}{{\rm Can}_{A\setminus \mathfrak{p}}(x)}$ is a multiplicative inverse of $\alpha$.\\

Since every non-zero element of $A\{ {A\setminus \{ 0\}}^{-1}\}$ is invertible, it follows that $A\{ {A\setminus \{ 0\}}^{-1}\}$ is a $\mathcal{C}^{\infty}-$field;\\

Ad (ii): If $a \neq 0$, then $a \in A \setminus \{ 0\}$ so ${\rm Can}_{A\setminus \{ 0\}}(a) \in (A\{ {A\setminus \{ 0\}}^{-1}\})^{\times}$. Since $A\{ {A\setminus \{ 0\}}^{-1}\}$ is a $\mathcal{C}^{\infty}-$field, it is not the trivial $\mathcal{C}^{\infty}-$ring, thus ${\rm Can}_{A\setminus \{ 0\}}(a) \neq 0$. We conclude that if $a \in A \setminus \{ 0\}$ is such that ${\rm Can}_{A \setminus \{ 0\}}(a) = 0$ implies $a=0$.\\

Ad (iii): Note that $\jmath[A\setminus \{ 0\}] \subseteq \mathbb{K}^{\times}$, and the result follows from the universal property of the $\mathcal{C}^{\infty}-$ring of fractions, $A\{ {A\setminus \{ 0\}}^{-1}\}$.
\end{proof}

Now we present a result that says that every $\mathcal{C}^{\infty}-$radical prime ideal is the kernel of  a $\mathcal{C}^{\infty}-$homo\-morphism into a $\mathcal{C}^{\infty}-$field.\\

First we notice that whenever $\mathfrak{p}$ is a $\mathcal{C}^{\infty}-$radical prime ideal of $A$ the localization:
$${\rm Can}_{A\setminus \mathfrak{p}}: A \to A_{\{ \mathfrak{p}\}},$$
where $A_{\{\mathfrak{p}\}} = A\{(A\setminus \mathfrak{p})^{-1}\}$ is a local ring. Let $\hat{\mathfrak{p}}$ be the maximal ideal of $A_{\{\mathfrak{p}\}}$ and $q: A_{\{ \mathfrak{p}\}} \to \dfrac{A_{\{ \mathfrak{p}\}}}{\hat{\mathfrak{p}}}$ be the quotient map into the residual field. Notice that $q$ is a local homomorphism, since $q[\hat{\mathfrak{p}}] = (0)$.\\

\begin{lemma}\label{112}Let $S$ be a multiplicative subset of a $\mathcal{C}^{\infty}-$ring $A$, and write $A\{ S^{-1}\} = \varprojlim_{a \in S} A\{ a^{-1}\}$. There is an inclusion-preserving bijection between the set of $\mathcal{C}^{\infty}-$radical prime ideals in $A\{ S^{-1}\}$ and $\mathcal{C}^{\infty}-$prime ideals in $A$ which are disjoint from $S$.
\end{lemma}
\begin{proof}
A proof of this fact can be found on page 286 of \cite{rings2}.
\end{proof}

We register another technical result:

\begin{lemma}\label{destino}Let $A$ be a $\mathcal{C}^{\infty}-$ring and $e \in A$ an idempotent element. There are unique isomorphisms:
$$A\{ e^{-1}\} \cong \dfrac{A}{(1-e)} \cong A \cdot e := \{ a \cdot e | a \in A\}$$
\end{lemma}
\begin{proof}
Let
$$\begin{array}{cccc}
m_e: & A & \twoheadrightarrow & A\cdot e \\
     & a & \mapsto & a \cdot e
\end{array}$$
and
$$\begin{array}{cccc}
    q & A & \twoheadrightarrow & \dfrac{A}{(1-e)} \\
     & a & \mapsto & a + (1-e)
  \end{array}$$

Since $(1-e)\cdot e = e - e^2 = e - e = 0$, given any $a = x(1-e) \in (1-e)$, $m_e(a) = m_e(x(1-e)) = x(1-e)e = x\cdot 0 = 0$, so $(1-e) \subseteq \ker m_e$.\\

Noting that $(1-e) \subseteq \ker m_e$ and that $m_e$ is a surjective map, by the \textbf{Theorem of the Isomorphism} there is a unique isomorphism $\psi: \dfrac{A}{(1-e)} \to A \cdot e$ such that the following triangle commutes:

$$\xymatrix{
A \ar@{->>}[r]^{m_e} \ar[d]^{q} & A \cdot e \\
\dfrac{A}{(1-e)} \ar@{.>}[ur]^{\exists ! \psi}
}$$

so $\dfrac{A}{(1-e)} \cong A \cdot e$.\\

Finally we show that $A\{ e^{-1}\} \cong \dfrac{A}{(1-e)}$. First we note that since $\eta_e(e) \in (A\{ e^{-1}\})^{\times}$  and $(\eta_e(e))^2 = \eta_e(e^2) = \eta_e(e)$ (that is, $\eta_e(e)$ is an idempotent element of $A\{ e^{-1}\}$), it follows that $\eta_e(e) = 1$.\\

We have $\eta_e(1-e)=\eta_e(1) - \eta_e(e) = 1 - 1 = 0$, so $(1-e) \subseteq \ker \eta_a$, so applying the Theorem of the Homomorphism we get a unique $\mathcal{C}^{\infty}-$homomorphism $\varphi: \dfrac{A}{(1-e)} \to A\{ e^{-1}\}$ such that the following diagram commutes:
$$\xymatrix{
A \ar[r]^{\eta_e} \ar[d]^{q} & A\{ e^{-1}\}\\
\dfrac{A}{(1-e)} \ar@{.>}[ur]^{\exists ! \varphi}
}$$

We have also:
$$q(e)-q(1) = q(e-1) \in (e-1),$$
hence
$$q(e) = q(1) \,\, {\rm in}\,\, \dfrac{A}{(1-e)}.$$

Since $(q(e))^2 = q(e^2)=q(e)$, \textit{i.e.}, it is idempotent, it follows that $q(e) = 1 + (e-1) \in \left( \dfrac{A}{(1-e)}\right)^{\times}$. By the universal property of $\eta_e: A \to A\{ e^{-1}\}$, there is a unique $\mathcal{C}^{\infty}-$homomorphism $\psi: A\{ e^{-1}\} \to \dfrac{A}{(1-e)}$ such that the following diagram commutes:
$$\xymatrix{
A \ar[r]^{\eta_e} \ar[d]^{q} & A\{ e^{-1}\} \ar@{.>}[dl]^{\exists ! \psi}\\
\dfrac{A}{(1-e)}
}$$

By the uniqueness of the arrows $\varphi$ and $\psi$, we conclude that $\varphi \circ \psi = {\rm id}_{A\{ e^{-1}\}}$ and $\psi \circ \varphi = {\rm id}_{\frac{A}{(1-e)}}$, hence $A\{ e^{-1}\} \cong \dfrac{A}{(1-e)}$.
\end{proof}

\section{Sheaves of $\mathcal{C}^{\infty}-$Rings}\label{sheaves}

\hspace{0.5cm} In this chapter we approach $\mathcal{C}^{\infty}-$rings from a Sheaf Theoretic viewpoint. We are going to study a $\mathcal{C}^{\infty}-$ring by means of its ``affine $\mathcal{C}^{\infty}-$scheme''. For this purpose, given an arbitrary $\mathcal{C}^{\infty}-$ring $A$, we construct a spectral topological space, namely its smooth Zariski spectrum, $({\rm Spec}^{\infty}\,(A), {\rm Zar}^{\infty})$ and a sheaf of $\mathcal{C}^{\infty}-$rings, $\Sigma_A: {\rm Open}\,({\rm Spec}^{\infty}\,(A), {\rm Zar}^{\infty})^{\rm op} \to \mathcal{C}^{\infty}{\rm \bf Rng}$ such that for every $a \in A$, $\Sigma_A(D^{\infty}(a)) \cong A\{ a^{-1}\}$ and whose stalks are local $\mathcal{C}^{\infty}-$rings. In this chapter we give a detailed account of this topological space - and we prove, for example, that it is indeed a spectral space (a fact that has been used in the last chapter without proof). A complete proof of this fact could not be found anywhere in the literature.\\

As a motivation we gave for the definition of the smooth version of Zariski spectrum, we have seen that given a $\mathcal{C}^{\infty}-$ring $A$ and a prime ideal, $\mathfrak{p} \subseteq A$, it is not always true that

$$A_{\mathfrak{p}} = \varinjlim_{a \notin \mathfrak{p}} A\{ a^{-1}\}$$

is a local $\mathcal{C}^{\infty}-$ring (cf. \textbf{Example 1.2} of \cite{rings2}). In fact, $A_{\mathfrak{p}}$ is a local $\mathcal{C}^{\infty}-$ring if, and only if, $\sqrt[\infty]{\mathfrak{p}}=\mathfrak{p}$ - so we have defined the smooth Zariski spectrum as consisting only of those $\mathcal{C}^{\infty}-$radical prime ideals. \\

As defined by D. Joyce in \cite{joyce2010algebraic}, a $\mathcal{C}^{\infty}-$ringed space is a pair $(X, \mathcal{O}_X)$, where $X$ is a topological space and $\mathcal{O}_X: {\rm Open}\,(X)^{{\rm op}} \to \mathcal{C}^{\infty}{\rm \bf Rng}$ is a sheaf of $\mathcal{C}^{\infty}-$rings whose stalks are local $\mathcal{C}^{\infty}-$rings. As in Algebraic Geometry, we define a sheaf of (local) $\mathcal{C}^{\infty}-$rings on ${\rm Spec}^{\infty}\,(A)$, for any $\mathcal{C}^{\infty}-$ring, $A$, the ``structure sheaf''.\\

\subsection{The Smooth Zariski Spectrum of a $\mathcal{C}^{\infty}-$Ring $A$: ${\rm Spec}^{\infty}\,(A)$}\label{lacrimosa}

In this section we make a detailed study of the smooth Zariski spectrum. Recall that, given any $\mathcal{C}^{\infty}-$ring $A$,

$${\rm Spec}^{\infty}\,(A):= \{ \mathfrak{p} \in {\rm Spec}\,(\widetilde{U}(A)) | \sqrt[\infty]{\mathfrak{p}} = \mathfrak{p}\}$$

We begin with the following:\\

\begin{definition}Let $A$ be a $\mathcal{C}^{\infty}-$ring.
\begin{itemize}
  \item[1.]{Given any $\mathcal{C}^{\infty}-$radical ideal $\mathfrak{a} \subseteq A$ (not necessarily prime ideal), we define:
      $$Z^{\infty}\,(\mathfrak{a}) := \{ \p \in {\rm Spec}^{\infty}\,(A) | \p \supseteq \mathfrak{a}\}$$
      }
  \item[2.]{Given any element $a \in A$, we define:
  $$D^{\infty}\,(a) := \{ \p \in {\rm Spec}^{\infty}\,(A) | a \notin \mathfrak{p}\}$$}
\end{itemize}
\end{definition}

\begin{proposition}\label{jay}Let $A$ be a $\mathcal{C}^{\infty}-$ring and $\mathfrak{a} \subseteq A$ be any ideal. Then:
$$Z^{\infty}(\mathfrak{a}) = Z^{\infty}(\sqrt[\infty]{\mathfrak{a}}).$$
\end{proposition}
\begin{proof}
Given $\p \in Z^{\infty}(\mathfrak{a})$, $\mathfrak{a} \subseteq \p$. By the item b of the \textbf{Theorem \ref{yellow}}, $\sqrt[\infty]{\mathfrak{a}} \subseteq \sqrt[\infty]{\p} = \p$, so $\p \in Z^{\infty}(\sqrt[\infty]{\mathfrak{a}})$.\\

Conversely, given $\p \in Z^{\infty}(\sqrt[\infty]{\mathfrak{a}})$, then $\p \supseteq \sqrt[\infty]{\mathfrak{a}} \supseteq \mathfrak{a}$, so $\p \in Z^{\infty}(\mathfrak{a})$.
\end{proof}

\begin{proposition}\label{Raj}Let $A$ be a $\mathcal{C}^{\infty}-$ring, $a,b \in A$. The following conditions are equivalent:
\begin{itemize}
  \item[(1)]{$D^{\infty}(a) \subseteq D^{\infty}(b)$;}
  \item[(2)]{$Z^{\infty}(a) \supseteq Z^{\infty}(b)$;}
  \item[(3)]{$a \in \sqrt[\infty]{(b)}$}
  \item[(4)]{$\sqrt[\infty]{(a)} \subseteq \sqrt[\infty]{(b)}$}
\end{itemize}

\end{proposition}
\begin{proof}
The equivalence $(1) \iff (2)$ follows immediately from the definition of $Z^{\infty}$.\\

Ad $(2) \to (4)$.\\

$$\sqrt[\infty]{(a)} = \bigcap Z^{\infty}(a) \subseteq \bigcap Z^{\infty}(b) = \sqrt[\infty]{(b)}$$

Ad $(4) \iff (3)$.\\

$$\sqrt[\infty]{(a)} \subseteq \sqrt[\infty]{(b)} \iff (a) \subseteq \sqrt[\infty]{(b)} \iff a \in \sqrt[\infty]{(b)}$$

Ad $(3) \to (2)$.\\

Let $a \in \sqrt[\infty]{(a)}$. Given $\mathfrak{p} \in Z^{\infty}(b)$, since $\sqrt[\infty]{(b)} = \bigcap \{ \mathfrak{p} \in {\rm Spec}^{\infty}\,(A) | (b) \subseteq \mathfrak{p} \}$, it follows that $\sqrt[\infty]{(b)} \subseteq \mathfrak{p}$. Since $(4) \iff (3)$, we have that $\sqrt[\infty]{(a)} \subseteq \sqrt[\infty]{(b)} \subseteq \mathfrak{p}$ implies $(a) \subseteq \sqrt[\infty]{(a)} \subseteq \mathfrak{p}$. Hence $\mathfrak{p} \in Z^{\infty}(a)$.
\end{proof}

\begin{proposition}\label{namaria}Given any ideals $\mathfrak{a}$ and $\mathfrak{b}$, we have:
$$\sqrt[\infty]{\mathfrak{a}} \subseteq \sqrt[\infty]{\mathfrak{b}} \iff Z^{\infty}(\mathfrak{b}) \subseteq Z^{\infty}(\mathfrak{a}).$$
\end{proposition}
\begin{proof}
In fact, if $\sqrt[\infty]{\mathfrak{a}} \subseteq \sqrt[\infty]{\mathfrak{b}}$, then for every $\p \in {\rm Spec}^{\infty}\,(A)$, whenever $\p \in Z^{\infty}(\sqrt[\infty]{b})$, i.e., whenever $\sqrt[\infty]{b} \subseteq \p$, $\sqrt[\infty]{\mathfrak{a}} \subseteq \p$, so $\p \in Z^{\infty}(\sqrt[\infty]{\mathfrak{a}})$. Thus:
$$Z^{\infty}(\mathfrak{b}) = Z^{\infty}(\sqrt[\infty]{\mathfrak{b}}) \subseteq Z^{\infty}(\sqrt[\infty]{\mathfrak{a}}) = Z^{\infty}(\mathfrak{a}).$$

Conversely, if $Z^{\infty}(\mathfrak{b}) \subseteq Z^{\infty}(\mathfrak{a})$, then:
$$\sqrt[\infty]{\mathfrak{b}} = \bigcap_{\substack{\mathfrak{b} \subseteq \p \\ \p \in {\rm Spec}^{\infty}\,(A)}} \supseteq \bigcap_{\substack{\mathfrak{a} \subseteq \p \\ \p \in {\rm Spec}^{\infty}\,(A)}} \p = \sqrt[\infty]{\mathfrak{a}}.$$
\end{proof}

The next theorem shows us that the sets of the form $Z^{\infty}(\mathfrak{b})$, for some $\mathcal{C}^{\infty}-$radical ideal $\mathfrak{b}$, are the closed subsets of a topology in ${\rm Spec}^{\infty}\,(A)$, that we shall refer to as the \textbf{smooth Zariski topology}.\\

\begin{theorem}\label{para}Let $A$ be a $\mathcal{C}^{\infty}-$ring. Then:
\begin{itemize}
  \item[1.]{$Z^{\infty}\,((0)) = {\rm Spec}^{\infty}\,(A)$ and $Z^{\infty}((1_A)) = \varnothing$;}
  \item[2.]{$Z^{\infty}(\mathfrak{a}) \cup Z^{\infty}(\mathfrak{b}) = Z^{\infty}(\mathfrak{a}\cdot \mathfrak{b}) = Z^{\infty}(\mathfrak{a} \cap \mathfrak{b})$;}
  \item[3.]{$\bigcap_{i \in I}Z^{\infty}(\mathfrak{a}_i) = Z^{\infty}\left( \sum_{i \in I} \mathfrak{a}_i\right)$}
\end{itemize}
where $\sum_{i \in I} \mathfrak{a}_i$ denotes the ideal generated by the family $\{ \mathfrak{a}_i\}_{i \in I}$.
\end{theorem}
\begin{proof}
Ad 1.$Z^{\infty}\,((0)) = \{ \p \in {\rm Spec}^{\infty}\,(A) | (0) \subseteq \p\} = {\rm Spec}^{\infty}\,(A)$ and:
$$Z^{\infty}((1_A)) = \{ \p \in {\rm Spec}^{\infty}\,(A) | (1_A) \subseteq \p\} = \{ \p \in {\rm Spec}^{\infty}\,(A) | 1_A \in \p \} = \varnothing.$$

Hence, by the \textbf{Proposition \ref{jay}}, we can write ${\rm Spec}^{\infty}\,(A) = Z^{\infty}(\sqrt[\infty]{(0)})$ and $\varnothing = Z^{\infty}\,(\sqrt[\infty]{(1_A)})$.\\

Ad 2. Let $\p \in Z^{\infty}(\mathfrak{a}\cap \mathfrak{b})$, so $\mathfrak{a}\cap \mathfrak{b} \subseteq \p$, and since $\mathfrak{a}\cdot \mathfrak{b} \subseteq \mathfrak{a}\cap \mathfrak{b}$, $\mathfrak{a}\cdot \mathfrak{b} \subseteq \mathfrak{a}\cap \mathfrak{b} \subseteq \p$ and $\p \in Z^{\infty}(\mathfrak{a}\cdot \mathfrak{b})$. Since $\p$ is a prime ideal, $\mathfrak{a}\cdot \mathfrak{b} \subseteq \p \Rightarrow \mathfrak{a}\subseteq \p$ or $\mathfrak{b}\subseteq \p$, hence $\p \in Z^{\infty}(\mathfrak{a})\cup Z^{\infty}(\mathfrak{b})$.\\

Now, if $\mathfrak{p} \in Z^{\infty}(\mathfrak{a})\cup Z^{\infty}(\mathfrak{b})$, either $\p \in Z^{\infty}(\mathfrak{a})$ or $\p \in Z^{\infty}(\mathfrak{b})$, so $\mathfrak{a}\cap \mathfrak{b} \subseteq \p$ and $\p \in Z^{\infty}\,(\mathfrak{a}\cap \mathfrak{b})$.\\

Ad 3. We have:
\begin{multline*}
\p \in Z^{\infty} \left( \sum_{i \in I}\mathfrak{a}_i\right) \iff \sum_{i \in I}\mathfrak{a}_i \subseteq \p \iff (\forall i \in I)(\mathfrak{a}_i \subseteq \p) \iff \\
\iff (\forall i \in I)(\p \in Z^{\infty}\,(\mathfrak{a}_i)) \iff \p \in \bigcap_{i \in I} Z^{\infty}\,(\mathfrak{a}_i)
\end{multline*}

Because of the \textbf{Proposition \ref{jay}}, we have the $\mathcal{C}^{\infty}-$radical ideal $\sqrt[\infty]{\sum_{i \in I}\mathfrak{a}_i}$ such that $Z^{\infty}(\sum_{i \in I}\mathfrak{a}_i) = Z^{\infty}\,(\sqrt[\infty]{\sum_{i \in I}\mathfrak{a}_i})$.
\end{proof}

In virtue of the \textbf{Theorem \ref{para}}, the sets of the form ${\rm Spec}^{\infty}\,(A)\setminus Z^{\infty}\,(\p)$ satisfy the axioms of open sets of a topology.\\

We have, thus:

\begin{definition}\label{zahma}Let $A$ be a $\mathcal{C}^{\infty}-$ring. The \index{smooth Zariski topology}\textbf{smooth Zariski topology} on ${\rm Spec}^{\infty}\,(A)$ is:\\

$${\rm Zar}^{\infty}:= \left\{ {\rm Spec}^{\infty}\,(A)\setminus Z^{\infty}\,(\p) | \p \in \mathfrak{I}^{\infty}_{A}\right\}$$
\end{definition}

\begin{definition} Let $A$ be any $\mathcal{C}^{\infty}-$ring. The \textbf{$\mathcal{C}^{\infty}-$spectrum of $A$} is the following topological space:
  $$({\rm Spec}^{\infty}\,(A), {\rm Zar}^{\infty}).$$
\end{definition}

Henceforth we are going to denote the topological space $({\rm Spec}^{\infty}\,(A), {\rm Zar}^{\infty})$ simply by ${\rm Spec}^{\infty}\,(A)$, omitting its topology.\\

\begin{theorem}Let $A$ be a $\mathcal{C}^{\infty}-$ring. The family $\{ D^{\infty}(a) | a \in A\}$ is a basis for the smooth Zariski topology, ${\rm Zar}^{\infty}$.
\end{theorem}
\begin{proof}
Note that given any $\mathcal{C}^{\infty}-$radical ideal, $\mathfrak{a}$ and a $\mathcal{C}^{\infty}-$radical prime ideal $\p$, we have:
$$\mathfrak{a} \nsubseteq \p \iff (\exists a \in \mathfrak{a})(a \notin \p),$$
Hence, any open set of ${\rm Zar}^{\infty}$, say:
$${\rm Spec}^{\infty}\,(A)\setminus Z^{\infty}\,(\mathfrak{a}) = \bigcup_{a \in \mathfrak{a}} D^{\infty}\,(a).$$
\end{proof}

\begin{proposition}Let $A$ be a $\mathcal{C}^{\infty}-$ring and ${\rm Spec}^{\infty}\,(A)$ be its $\mathcal{C}^{\infty}-$spectrum. We have, for any point $\p \in {\rm Spec}^{\infty}\,(A)$:
$${\rm Cl}. \{ \p\} = Z^{\infty}(\p).$$
\end{proposition}
\begin{proof}
By definition, the closure of a set is the intersection of all the closed sets which includes it, so:
$${\rm Cl}. \{ \p\} = \bigcap_{\substack{\p \in Z^{\infty}\,(\mathfrak{a})\\ \mathfrak{a} \in \mathfrak{I}^{\infty}_{A}}} Z^{\infty}(\mathfrak{a}) = Z^{\infty}(\p).$$
\end{proof}

\begin{theorem}\label{maio} Let $A$ be a $\mathcal{C}^{\infty}-$ring. The topological space $({\rm Spec}^{\infty}\,(A), {\rm Zar}^{\infty})$ is such that, given $\p,\mathfrak{q} \in {\rm Spec}^{\infty}\,(A)$ with $\p \neq \mathfrak{q}$, there is either $U \in {\rm Zar}^{\infty}$ such that $\p \in U$ and $\mathfrak{q} \notin U$ or $V \in {\rm Zar}^{\infty}$ such that $\p \notin V$ and $\mathfrak{q} \in V$. In other words, the topological space $({\rm Spec}^{\infty}\,(A), {\rm Zar}^{\infty})$ is a $T_0$ space.
\end{theorem}
\begin{proof}
Given $\p,\mathfrak{q} \in {\rm Spec}^{\infty}\,(A)$ with $\p \neq \mathfrak{q}$, then either $\p \nsubseteq \mathfrak{q}$ or $\mathfrak{q} \nsubseteq \p$. \\

If $\mathfrak{q} \nsubseteq \p$, then there is some $b \in \mathfrak{q}$ such that $b \notin \mathfrak{p}$, so $\mathfrak{q} \notin D^{\infty}(b)$ and $\p \in D^{\infty}(b)$. Thus it suffices to take $U = D^{\infty}\,(b)$.\\

Now, if $\p \nsubseteq \mathfrak{q}$, then there is $a \in \p$ such that $a \notin \mathfrak{q}$, so $\p \notin D^{\infty}\,(a)$ and $\mathfrak{q} \in D^{\infty}(a)$. Thus, it suffices to take $V = D^{\infty}(a)$.
\end{proof}

As a corollary of the \textbf{Proposition \ref{namaria}}, we can prove the following:

\begin{proposition} Given any two $\mathcal{C}^{\infty}-$radical ideals of a $\mathcal{C}^{\infty}-$ring $A$, $I,J \in \mathfrak{I}^{\infty}_{A}$, we have:
$$\sqrt[\infty]{I \cdot J} = \sqrt[\infty]{I \cap J}$$
\end{proposition}
\begin{proof}
Since $I \cdot J \subseteq I \cap J$, we have $\sqrt[\infty]{I \cdot J} \subseteq \sqrt[\infty]{I \cap J}$.\\

\textit{Ab absurdo}, suppose $\sqrt[\infty]{I \cap J} \nsubseteq \sqrt[\infty]{I \cdot J}$, so by \textbf{Proposition \ref{namaria}}, $Z^{\infty}(I \cdot J) \nsubseteq Z^{\infty}(I \cap J)$. This latter condition means that there is some $\p \in {\rm Spec}^{\infty}\,(A)$ such that $I \cdot J \subseteq \p$ and $I \cap J \nsubseteq \p$, that is, there is $a \in I \cap J$ such that $a \notin \p$. Thus, we have $a^2 \in I \cdot J \subseteq \p$, so $a^2 \in \p$ and $a \notin \p$ - which contradicts the fact that the ideal $\p$ is prime. Hence:
$$\sqrt[\infty]{I \cap J} \subseteq \sqrt[\infty]{I \cdot J}.$$
\end{proof}

\begin{theorem}[\textbf{Lemma 1.4} of \cite{rings2}]\label{3111}Let $A$ be a $\mathcal{C}^{\infty}-$ring.\\
\begin{itemize}
  \item[(i)]{$D^{\infty}(a) \subseteq D^{\infty}(b)$ if, and only if, $b \in A\{ a^{-1}\}^{\times}$, if, and only if, $a \in \sqrt[\infty]{(b)}$;}
  \item[(ii)]{Each basic open $D^{\infty}(a)$ of ${\rm Spec}^{\infty}\,(A)$ is compact; in fact $D^{\infty}(a) \subset \bigcup_{i \in I} D^{\infty}(a_i)$ if, and only if, there are finitely many indices $i_1, \cdots, i_n \in I$ such that $D^{\infty}(a) \subseteq D^{\infty}(a_{i_1}^2 + \cdots + a_{i_n}^2)$.}
  \item[(iii)]{The basic opens $D^{\infty}(a)$ form a distributive lattice, with:
  $$D^{\infty}(a)\cap D^{\infty}(b) = D^{\infty}(a\cdot b)\,\,\, \mbox{and}\,\,\, D^{\infty}(a)\cup D^{\infty}(b) = D^{\infty}(a^2 + b^2)$$}
\end{itemize}
\end{theorem}

\begin{definition}Let $(X,\tau)$ be a topological space, $E \subseteq X$ and $x \in X$.\\
\begin{itemize}
  \item[(a)]{$E$ is \index{irreducible set}\textbf{irreducible in $X$} if, and only if, for all closed sets, $F_1,F_2$, in $X$
  $$(E\subseteq F_1 \cup F_2) \rightarrow((E \subseteq F_1)\vee(E \subseteq F_2))$$}
  \item[(b)]{$x$ is a \index{generic point}\textbf{generic point of $E$} if $x \in E$ and $E \subseteq {\rm Cl}. \{ x\}$ }
\end{itemize}
\end{definition}

Note that if $E \subseteq X$ is closed, since the finite intersection of closed sets is closed, $E$ is irreducible in $X$ if, and only if, for all closed sets $F_1$ and $F_2$ in $X$
$$(E=F_1 \cup F_2)\rightarrow((E=F_1)\vee(E=F_2)).$$

Note also that if $E \subseteq X$ is closed, then a point $x \in E$ is generic if, and only if, $E = {\rm Cl}.\{ x\}$.\\

\begin{definition}Let $(X,\tau)$ be a topological space. The space $(X,\tau)$ is a \index{spectral space}\textbf{spectral space} if, and only if:
\begin{itemize}
  \item[(s1)]{$(X, \tau)$ is a compact $T_0$ topological space;}
  \item[(s2)]{$\stackrel{\circ}{\mathcal{K}}\,(X):= \{ U \subseteq X | U \, \mbox{is}\,\, \mbox{compact}\,\, \mbox{and}\,\, \mbox{open}\}$ constitute a basis of open subsets of $X$ that is closed under finite intersections;}
  \item[(s3)]{Every irreducible closed subset $E$ has a unique generic point.}
\end{itemize}
\end{definition}

\begin{remark}The condition (s2) in the definition above can be replaced by the equivalent condition stated as follows: $(X,\tau)$ has some basis of compact open subsets of $X$ that is closed under finite intersections.
\end{remark}

\begin{theorem}\label{diabolin}Let $A$ be any $\mathcal{C}^{\infty}-$ring. The topological space ${\rm Spec}^{\infty}\,(A)$ is a spectral space.
\end{theorem}
\begin{proof}
By the \textbf{Theorem \ref{maio}}, ${\rm Spec}^{\infty}\,(A)$ is a $T_0$ topological space.\\

By item (ii) of the \textbf{Theorem \ref{3111}}, for any $a \in A$, $D^{\infty}\,(a)$ is compact, since given any open cover $\{ D^{\infty}\,(a_i) | i \in I\}$ of $D^{\infty}\,(a)$:
$$D^{\infty}\,(a) \subseteq \bigcup_{i \in I} D^{\infty}\,(a_i)$$
there are finitely many indices $i_1, \cdots, i_n \in I$ such that:
$$D^{\infty}\,(a) \subseteq D^{\infty}(a_{i_1}^2+ \cdots a_{i_n}^2).$$

\textbf{Claim: }$D^{\infty}(a_{i_1}^2+ \cdots a_{i_n}^2) \subseteq \bigcup_{j=1}^{n}D^{\infty}(a_{i_j})$.\\

Given $\p \in D^{\infty}(a_{i_1}^2+ \cdots a_{i_n}^2)$, $a_{i_1}^2+ \cdots a_{i_n}^2 \notin \p$,  there must exist some $k \in \{ 1,2, \cdots, n\}$ such that $a_{i_k}^2 \notin \p$, and since $\p$ is prime, $a_{i_k} \notin \p$. Thus, $\p \in D^{\infty}(a_{i_k}) \subseteq \bigcup_{j=1}^{n}D^{\infty}(a_{i_j})$.\\

It follows that

$$D^{\infty}\,(a) \subseteq \bigcup_{j=1}^{n}D^{\infty}(a_{i_j}).$$

It is also clear, by the item (iii) of \textbf{Lemma 1.4} of \cite{rings2},  that $\mathcal{B}$ is closed under finite intersections, since given $D^{\infty}(a_1), D^{\infty}\,(a_2), \cdots, D^{\infty}\,(a_n) \in \mathcal{B}$,

$$D^{\infty}\,(a_1)\cap \cdots \cap D^{\infty}\,(a_n) = D^{\infty}\,(a_1 \cdot a_2 \cdots a_n) \in \mathcal{B}.$$

Finally, we are going to show that each irreducible closed subset of ${\rm Spec}^{\infty}\,(A)$ has a unique generic point.\\

By the \textbf{Definition \ref{zahma}}, $E$ is closed in ${\rm Spec}^{\infty}\,(A)$ if, and only if, there is some $\mathcal{C}^{\infty}-$radical ideal, $P \in \mathfrak{I}^{\infty}(A)$, such that $E = Z^{\infty}(P) = \{ \mathfrak{p} \in {\rm Spec}^{\infty}\,(A) | P \subseteq \mathfrak{p}\}$.\\

We claim that if $P$ is not a prime ideal, then $E$ is not irreducible, which is equivalent to assert that if $E$ is irreducible then $P$ is prime, via \textit{modus tollens}.\\




Suppose $P$ is not prime, so there are ideals $I,J$ of $A$ such that $I \cdot J \subseteq P$, $I \nsubseteq P$ and $J \nsubseteq P$. \\

We have:

$$E = Z^{\infty}\,(P) \subseteq Z^{\infty}(I)\cup Z^{\infty}\,(J)$$

In fact, if $K \in E = Z^{\infty}\,(P)$, since $I \cdot J \subseteq P$, $I \cdot J \subseteq K$. Because $K$ is prime, $I \cdot J \subseteq K$ implies either $I \subseteq K$ (so $K \in Z^{\infty}\,(I)$) or $J \subseteq K$ (so $K \in Z^{\infty}\,(J)$), hence $K \in Z^{\infty}\,(I)\cup Z^{\infty}\,(J)$.\\

By definition, both $Z^{\infty}\,(I)$ and $Z^{\infty}\,(J)$ are both closed.\\

Finally, we have both $Z^{\infty}(P) \subsetneq Z^{\infty}(I)$ and $ Z^{\infty}(P) \subsetneq Z^{\infty}\,(J)$.\\

If $Z^{\infty}(I) = Z^{\infty}\,(P)$, then $\{ \p \in {\rm Spec}^{\infty}\,(A) | \p \supseteq I \} = \{ \p \in {\rm Spec}^{\infty}\,(A) | \p \supseteq P\}$ and:
$$P = \sqrt[\infty]{P} = \bigcap_{\substack{\p \in {\rm Spec}^{\infty}\,(A)\\P \subseteq \p}}\p = \bigcap_{\substack{\p \in {\rm Spec}^{\infty}\,(A)\\ I \subseteq \p}} \p = \sqrt[\infty]{I} \supseteq I$$
which contradicts our assumption that $I \nsubseteq P$. A similar argument holds for $Z^{\infty}(J)$.\\

Thus we have expressed the closed subset, $E \subseteq {\rm Spec}^{\infty}\,(A)$ as the union of two proper closed subsets, so $E$ is not irreducible.\\

Hence, if $E = Z^{\infty}(P)$ is an irreducible closed subset of ${\rm Spec}^{\infty}\,(A)$, then $P$ is a $\mathcal{C}^{\infty}-$radical prime ideal of $A$.\\

Thus,

$${\rm Cl}. \{ P\}:= \bigcap_{\p \in {\rm Spec}^{\infty}\,(A)} Z^{\infty}\,(\p) = Z^{\infty}\,(P) = E,$$

and $P$ is a generic point of $E$.\\

The uniqueness of the generic point is proved as follows.\\

Suppose, \textit{ab absurdo}, that there is some $P' \in {\rm Spec}^{\infty}\,(A)$ with $P \neq P'$ and ${\rm Cl}. \{ P\} = E = {\rm Cl}. \{ P'\}$. Since $P \neq P'$, we can assume without loss of generality, $P \nsubseteq P'$, so there is some $a \in P$ such that $a \notin P'$, and $P' \in D^{\infty}(a)$, so $P' \notin Z^{\infty}((a))$ and $P \notin D^{\infty}(a)$, so $P \in Z^{\infty}((a))$.\\

Thus:

$$E = {\rm Cl}. \{ P'\} \nsubseteq Z^{\infty}((a))$$
and
$$E = {\rm Cl}. \{ P \} \subseteq Z^{\infty}((a))$$
which is a contradiction. Hence such a $P'$ does not exist and the generic point $P$ is unique.
\end{proof}

\begin{remark}Contrarily to  Commutative Algebra, every compact open set $K$ of the spectral topology ${\rm Spec}^{\infty}\,(A)$ is of the form $D^{\infty}(a)$ for some $a \in A$, since $K = D^{\infty}(a_1) \cup \cdots \cup D^\infty(a_n) = D^\infty(b )$, for some $b \in A$ which is a sum of squares.
\end{remark}

\begin{definition}Let $(X,\tau)$ and $(Y, \sigma)$ be two spectral spaces. A map $f: X \to Y$ is a \textbf{spectral map} if, and only if, for every compact open $K \subseteq Y$, $f^{\dashv}[K] \subseteq X$ is compact open.
\end{definition}

\begin{proposition}Let $A,A'$ be two $\mathcal{C}^{\infty}-$rings and let $f: A \to A'$ be a $\mathcal{C}^{\infty}-$homo\-morphism. The function:
$$\begin{array}{cccc}
    h^{*}: & {\rm Spec}^{\infty}\,(A') & \rightarrow & {\rm Spec}^{\infty}\,(A) \\
     & \mathfrak{p} & \mapsto & h^{\dashv}[\mathfrak{p}]
  \end{array}$$
is a spectral map.
\end{proposition}
\begin{proof}
We are going to show that the inverse image of every basic compact-open subset of ${\rm Spec}^{\infty}\,(A)$ is a basic compact open subset of ${\rm Spec}^{\infty}\,(A')$.\\

Let $D^{\infty}\,(a)$, for some $a \in A$, be a basic open subset of ${\rm Spec}^{\infty}\,(A)$. We have:

\begin{multline*}{h^{*}}^{\dashv}[D^{\infty}(a)] = \{ \mathfrak{p}' \in {\rm Spec}^{\infty}\,(A') | {h^{*}}(\mathfrak{p}') \in D^{\infty}(a) \} = \{ \mathfrak{p}' \in {\rm Spec}^{\infty}\,(A') | h^{\dashv}[\mathfrak{p}'] \in D^{\infty}\,(a) \} = \\
\{ \mathfrak{p}' \in {\rm Spec}^{\infty}\,(A') | a \notin h^{\dashv}[\mathfrak{p}']\} = \{ \mathfrak{p}' \in {\rm Spec}^{\infty}\,(A') | h(a) \notin \mathfrak{p}' \} = D^{\infty}(h(a))
\end{multline*}

so ${h^{*}}^{\dashv}[D^{\infty}(a)]$ is a basic compact open subset of ${\rm Spec}^{\infty}\,(A')$. It follows that $h^{*}$ is a spectral map. In particular, it follows that $h^{*}$ is a continuous function.
\end{proof}

Thus, we are able to define a contravariant functor from the category of $\mathcal{C}^{\infty}-$rings to the category of topological spaces:

$$\begin{array}{cccc}
    {\rm Spec}^{\infty}: & \mathcal{C}^{\infty}{\rm \bf Rng} & \rightarrow & {\rm \bf Top} \\
     & (\xymatrix{A \ar[r]^{h}& B}) & \mapsto & (\xymatrix{{\rm Spec}^{\infty}\,(B) \ar[r]^{h^{*}}& {\rm Spec}^{\infty}\,(A)})
  \end{array}$$




\begin{proposition}\label{Gabi}Let $A$ and $B$ be two $\mathcal{C}^{\infty}-$rings and let $\mathfrak{I}(A)$ be the set of all ideals of $A$ and $\mathfrak{I}(B)$ be the set of all ideals of $B$. Every ideal of the product $A \times B$ has the form $\mathfrak{a}\times \mathfrak{b}$, where $\mathfrak{a}$ is an ideal of $A$ and $\mathfrak{b}$ is an ideal of $B$, so we have the following bijection:

$$\begin{array}{cccc}
    \Phi: & \mathfrak{I}(A) \times \mathfrak{I}(B) & \rightarrow & \mathfrak{I}(A \times B) \\
     & (\mathfrak{a},\mathfrak{b}) & \mapsto & \mathfrak{a} \times \mathfrak{b}
  \end{array}$$
\end{proposition}
\begin{proof}
Let $p_1: A\times B \to A$ and $p_2: A \times B \to B$ be the canonical projections of the product:
$$\xymatrixcolsep{5pc}\xymatrix{
 & A \times B \ar[dl]_{p_1} \ar[dr]^{p_2} & \\
A & & B
}$$
and let $I$ be any ideal of $A \times B$. We have:
$$p_1[I] = \{ a \in A | (\exists b \in B)((a,b) \in I)\} = \{ a \in A | (a,0) \in I\}$$
$$p_2[I] = \{ b \in B | (\exists a \in A)((a,b) \in I)\} = \{ b \in B | (0,b) \in I\}$$

We claim that $p_1[I] \times p_2[I] = I$.\\

Given $(a,b) \in I$, then $p_1(a,b) = a \in p_1[I]$ and $p_2(a,b) = b \in p_2[I]$, so $(a,b) \in p_1[I]\times p_2[I]$. Conversely, given $(a,b) \in p_1[I]\times p_2[I]$, $a \in p_1[I]$ and $b \in p_2[I]$, so $(a,0) \in I$ and $(0,b) \in I$, so $(a,b) = (a,0)+(0,b) \in I$.\\

Since for every $\mathfrak{a}$, ideal of $A$ and for every $\mathfrak{b}$ ideal of $B$, $\mathfrak{a} \times \mathfrak{b}$ is an ideal of $A \times B$, the following map is a bijection:

$$\begin{array}{cccc}
    \Phi: & \mathfrak{I}(A) \times \mathfrak{I}(B) & \rightarrow & \mathfrak{I}(A \times B) \\
     & (\mathfrak{a},\mathfrak{b}) & \mapsto & \mathfrak{a} \times \mathfrak{b}
  \end{array}$$

\end{proof}

\begin{proposition}\label{infi}Let $A,B$ be two reduced $\mathcal{C}^{\infty}-$rings. By the \textbf{Proposition \ref{Gabi}}, the ideals of $A\times B$ are precisely those of the form $\mathfrak{a}\times \mathfrak{b}$, where $\mathfrak{a}$ is an ideal of $A$ and $\mathfrak{b}$ is an ideal of $B$. Moreover, $(\mathfrak{a} \times \mathfrak{b} \in {\rm Spec}^{\infty}\,(A\times B)) \iff ((\mathfrak{a} \in {\rm Spec}^{\infty}\,(A))\&(\mathfrak{b} = B)\vee((\mathfrak{a}=A)\& (\mathfrak{b}\in {\rm Spec}^{\infty}\,(B))))$.
\end{proposition}
\begin{proof}

One can easily see that the map:
$$\begin{array}{cccc}
    \varphi: & A \times B & \twoheadrightarrow & \dfrac{A}{\mathfrak{a}} \times \dfrac{B}{\mathfrak{b}} \\
     & (a,b) & \mapsto & (a + \mathfrak{a}, b + \mathfrak{b})
  \end{array}$$
is a $\mathcal{C}^{\infty}-$rings homomorphism such that $\ker \varphi = \mathfrak{a}\times \mathfrak{b}$. By the \textbf{Theorem of the Isomorphism}, there is a unique $\widetilde{\varphi}: \dfrac{A \times B}{\mathfrak{a} \times \mathfrak{b}} \to \dfrac{A}{\mathfrak{a}}\times \dfrac{B}{\mathfrak{b}}$ such that:
$$\xymatrixcolsep{5pc}\xymatrix{
A \times B \ar[d]_{q_{\mathfrak{a}\times \mathfrak{b}}} \ar@{->>}[r]^{\varphi} & \dfrac{A}{\mathfrak{a}}\times \dfrac{B}{\mathfrak{b}}\\
\dfrac{A \times B}{\mathfrak{a} \times \mathfrak{b}} \ar[ur]_{\widetilde{\varphi}}&
}$$
hence $\widetilde{\varphi}$ is a $\mathcal{C}^{\infty}-$isomorphism:
$$\dfrac{A \times B}{\mathfrak{a}\times \mathfrak{b}} \stackrel{\widetilde{\varphi}}{\cong} \dfrac{A}{\mathfrak{a}}\times \dfrac{B}{\mathfrak{b}}$$

It follows that $\mathfrak{a}\times \mathfrak{b} \in {\rm Spec}^{\infty}\,(A\times B)$ if, and only if, $\dfrac{A \times B}{\mathfrak{a}\times \mathfrak{b}}$ is a $\mathcal{C}^{\infty}-$domain. Since $\dfrac{A \times B}{\mathfrak{a}\times \mathfrak{b}} \stackrel{\widetilde{\varphi}}{\cong} \dfrac{A}{\mathfrak{a}}\times \dfrac{B}{\mathfrak{b}}$, it follows that $\dfrac{A \times B}{\mathfrak{a}\times \mathfrak{b}}$ is a $\mathcal{C}^{\infty}-$domain if,  and only if, its isomorphic image, $\dfrac{A}{\mathfrak{a}} \times \dfrac{B}{\mathfrak{b}}$ is a $\mathcal{C}^{\infty}-$domain.\\

\textbf{Claim:} $\dfrac{A}{\mathfrak{a}} \times \dfrac{B}{\mathfrak{b}}$ is a $\mathcal{C}^{\infty}-$domain if, and only if, one of these factors is $0$ and the other is a $\mathcal{C}^{\infty}-$domain.\\

We prove that $\dfrac{A}{\mathfrak{a}} \times \dfrac{B}{\mathfrak{b}}$ is a $\mathcal{C}^{\infty}-$domain $\Rightarrow$ (($\dfrac{A}{\mathfrak{a}}$ is a $\mathcal{C}^{\infty}-$domain and $\mathfrak{b}=B$)$\vee$ ($\mathfrak{a} = A$ and $\dfrac{B}{\mathfrak{b}}$ is a $\mathcal{C}^{\infty}-$domain)) via \textit{modus tollens}.\\

Suppose we have both $\mathfrak{a} \neq A$ and $\mathfrak{b} \neq B$, so $\dfrac{A}{\mathfrak{a}} \neq \{ 0\}$ and $\dfrac{B}{\mathfrak{b}} \neq \{0\}$, and the $\mathcal{C}^{\infty}-$ring $\dfrac{A}{\mathfrak{a}}\times \dfrac{B}{\mathfrak{b}}$ has two non-trivial idempotent elements, namely $(1_A,0_B)$ and $(0_A, 1_B)$, hence $\dfrac{A \times B}{\mathfrak{a} \times \mathfrak{b}} \cong \dfrac{A}{\mathfrak{a}}\times \dfrac{B}{\mathfrak{b}}$ is not a $\mathcal{C}^{\infty}-$domain.\\

Now we prove that (($\dfrac{A}{\mathfrak{a}}$ is a $\mathcal{C}^{\infty}-$domain and $\mathfrak{b}=B$)$\vee$ ($\mathfrak{a} = A$ and $\dfrac{B}{\mathfrak{b}}$ is a $\mathcal{C}^{\infty}-$domain)) $\Rightarrow$ $\dfrac{A}{\mathfrak{a}}\times \dfrac{B}{\mathfrak{b}}$ is a $\mathcal{C}^{\infty}-$domain.\\

Suppose that $\dfrac{A}{\mathfrak{a}}$ is a reduced $\mathcal{C}^{\infty}-$domain and $\mathfrak{b}=B$, then $\dfrac{A \times B}{\mathfrak{a} \times \mathfrak{b}} \cong \dfrac{A}{\mathfrak{a}} \times \{ 0\} \cong \dfrac{A}{\mathfrak{a}}$ is a reduced $\mathcal{C}^{\infty}-$domain. Now, if $\dfrac{B}{\mathfrak{b}}$ is a $\mathcal{C}^{\infty}-$domain and $\mathfrak{a}=A$, then $\dfrac{A \times B}{\mathfrak{a} \times \mathfrak{b}} \cong \{ 0\} \times \dfrac{B}{\mathfrak{b}} \cong \dfrac{B}{\mathfrak{b}}$ is again a $\mathcal{C}^{\infty}-$domain. \\

We have that $\dfrac{A \times B}{\mathfrak{a} \times \mathfrak{b}}$ is a $\mathcal{C}^{\infty}-$domain if, and only if, $\dfrac{A}{\mathfrak{a}}\times \dfrac{B}{\mathfrak{b}}$ is a $\mathcal{C}^{\infty}-$domain, and that $\dfrac{A}{\mathfrak{a}}\times \dfrac{B}{\mathfrak{b}}$ is a $\mathcal{C}^{\infty}-$domain if, and only if, one (and only one) of the following conditions holds:

\begin{itemize}
\item[(i)]{$\dfrac{A}{\mathfrak{a}} \cong \{ 0\}$ and $\dfrac{B}{\mathfrak{b}}$ is a $\mathcal{C}^{\infty}-$domain;}
\item[(ii)]{$\dfrac{A}{\mathfrak{a}}$ is a $\mathcal{C}^{\infty}-$domain and $\dfrac{B}{\mathfrak{b}} \cong \{ 0\}$;}
\end{itemize}

If (i) is the case, then $\mathfrak{a} = A$ and $\mathfrak{b} \in {\rm Spec}^{\infty}\,(B)$, otherwise $\mathfrak{a} \in {\rm Spec}^{\infty}\,(A)$ and $\mathfrak{b}=B$. Since these two cases are the only possibilities, we have established that:

$$(\mathfrak{a}\times \mathfrak{b} \in {\rm Spec}^{\infty}\,(A\times B)) \iff ((\mathfrak{a} \in {\rm Spec}^{\infty}\,(A)) \& (\mathfrak{b}=B)) \vee ((\mathfrak{a} = A) \& (\mathfrak{b} \in {\rm Spec}^{\infty}\,(B)))$$
\end{proof}

\begin{remark}The topology of ${\rm Spec}^{\infty}(A_i)\sqcup {\rm Spec}^{\infty}\,(A_j)$ is, by definition, the finest one such that:

$$\begin{array}{cccc}
    \imath_k: & {\rm Spec}^{\infty}\,(A_k) & \rightarrow & {\rm Spec}^{\infty}(A_i)\sqcup {\rm Spec}^{\infty}\,(A_j) \\
     & \mathfrak{p}_k & \mapsto & (\mathfrak{p}_k,k)
  \end{array}$$

is continuous for $k=i,j$. Moreover:

$$\xymatrixcolsep{5pc}\xymatrix{
{\rm Spec}^{\infty}\,(A_i) \ar[dr]^{\imath_i} & \\
  & {\rm Spec}^{\infty}\,(A_i)\sqcup {\rm Spec}^{\infty}\,(A_j)\\
{\rm Spec}^{\infty}\,(A_j) \ar[ur]_{\imath_j} &
}$$

is the coproduct of ${\rm Spec}^{\infty}\,(A_i)$ and ${\rm Spec}^{\infty}\,(A_j)$ in the category ${\bf Top}$.
\end{remark}

\begin{proposition}\label{ast}Let $A_i,A_j$ be any two $\mathcal{C}^{\infty}-$rings. \\

The maps:
$$\begin{array}{cccc}
    \varphi_i: & {\rm Spec}^{\infty}\,(A_i) & \rightarrow & {\rm Spec}^{\infty}\,(A_i \times A_j) \\
     & \mathfrak{p}_i & \mapsto & \mathfrak{p}_i \times A_j
  \end{array}$$

and

$$\begin{array}{cccc}
    \varphi_j: & {\rm Spec}^{\infty}\,(A_j) & \rightarrow & {\rm Spec}^{\infty}\,(A_i \times A_j) \\
     & \mathfrak{p}_j & \mapsto & A_i \times \mathfrak{p}_j
  \end{array}$$

are both spectral maps (in particular, they are continuous).\\
\end{proposition}
\begin{proof}
Without loss of generality we prove that $\varphi_i$ is continuous.\\

Let $(a_i, a_j) \in A_i \times A_j$ and let $D^{\infty}(a_i,a_j) := \{ \mathfrak{q} \in {\rm Spec}^{\infty}(A_i \times A_j) | (a_i,a_j) \notin \mathfrak{q}\} $ be any basic open subset of $${\rm Spec}^{\infty}\,(A_i\times A_j) = \{ \mathfrak{p}_i \times A_j | \mathfrak{p}_i \in {\rm Spec}^{\infty}(A_i)\}\stackrel{\cdot}{\cup} \{ A_i \times \mathfrak{p}_j | \mathfrak{p}_j \in {\rm Spec}^{\infty}(A_j)\}$$
We have

$$D^{\infty}(a_i,a_j) = \{ \mathfrak{p}_i \times A_j | \mathfrak{p}_i \in D^{\infty}(a_i)\}\stackrel{\cdot}{\cup} \{ A_i \times \mathfrak{p}_j | \mathfrak{p}_j \in D^{\infty}(a_j)\}$$

Thus,

$$\varphi_i^{\dashv}[D^{\infty}(a_i,a_j)] = \{ \mathfrak{p}_i \in {\rm Spec}^{\infty}\,(A_i) | (a_i,a_j) \notin \mathfrak{p}_i \times A_j \}$$

Since $\{ \mathfrak{p}_i \in {\rm Spec}^{\infty}\,(A_i) | (a_i,a_j) \notin \mathfrak{p}_i \times A_j \} = \{ \mathfrak{p}_i \in {\rm Spec}^{\infty}\,(A_i) | a_i \notin \mathfrak{p}_i \} = D^{\infty}\,(a_i)$, it follows that:

$${\varphi_i}^{\dashv}[D^{\infty}(a_i,a_j)] = D^{\infty}\,(a_i),$$

which is a basic open subset of ${\rm Spec}^{\infty}\,(A_i)$. Hence $\varphi_i$ is an spectral map, and in particular it is a continuous function. An analogous reasoning shows us that $\varphi_j$ is also a continuous map.
\end{proof}

\begin{remark}
By the universal property of ${\rm Spec}^{\infty}\,(A_i)\sqcup {\rm Spec}^{\infty}\,(A_j)$, given the continuous maps $\varphi_i, \varphi_j$ of \textbf{Proposition \ref{ast}}, there is a unique continuous function $\widetilde{\varphi}: {\rm Spec}^{\infty}\,(A_i) \sqcup {\rm Spec}^{\infty}\,(A_j) \to {\rm Spec}^{\infty}\,(A_i \times A_j)$ such that:

$$\xymatrixcolsep{5pc}\xymatrix{
{\rm Spec}^{\infty}\,(A_i)\ar[dr]^{\imath_{A_i}} \ar@/^2pc/[rrd]^{\varphi_i} & & \\
 & {\rm Spec}^{\infty}\,(A_i)\sqcup {\rm Spec}^{\infty}\,(A_j) \ar@{.>}[r]^{\exists ! \widetilde{\varphi}} & {\rm Spec}^{\infty}\,(A_i \times A_j)\\
{\rm Spec}^{\infty}\,(A_j) \ar[ur]_{\imath_{A_j}} \ar@/_2pc/[rru]_{\varphi_{A_j}} & &
}$$

commutes, and since the function:

$$\begin{array}{cccc}
    \varphi: & {\rm Spec}^{\infty}\,(A_i) \sqcup {\rm Spec}^{\infty}\,(A_j) & \rightarrow & {\rm Spec}^{\infty}\,(A_i \times A_j) \\
     & (\mathfrak{p}_i,i) & \mapsto & \mathfrak{p}_i\times A_j\\
     & (\mathfrak{p}_j,j) & \mapsto & A_i \times \mathfrak{p}_j
\end{array}$$

is such that $\varphi \circ \imath_{A_i} = \varphi_{A_i}$ and $\varphi \circ \imath_{A_j} = \varphi_{A_j}$, it follows that $\varphi = \widetilde{\varphi}$, and $\varphi$ is a continuous map.

In virtue of the \textbf{Propositions} \ref{Gabi} and \ref{infi}, it is clear that $\varphi$ is a bijection whose inverse is given by:

$$\begin{array}{cccc}
    \psi: & {\rm Spec}^{\infty}\,(A_i \times A_j) & \rightarrow & {\rm Spec}^{\infty}\,(A_i) \sqcup {\rm Spec}^{\infty}\,(A_j) \\
     & \mathfrak{p}_i\times A_j & \mapsto & (\mathfrak{p}_i, i)\\
     & A_i \times \mathfrak{p}_j & \mapsto & (\mathfrak{p}_j, j)
\end{array}$$

\end{remark}

\begin{theorem}\label{rafiki}Let $A_i,A_j$ be two $\mathcal{C}^{\infty}-$rings. The function:

$$\begin{array}{cccc}
    \varphi: & {\rm Spec}^{\infty}\,(A_i) \sqcup {\rm Spec}^{\infty}\,(A_j) & \rightarrow & {\rm Spec}^{\infty}\,(A_i \times A_j) \\
     & (\mathfrak{p}_i,i) & \mapsto & \mathfrak{p}_i\times A_j\\
     & (\mathfrak{p}_j,j) & \mapsto & A_i \times \mathfrak{p}_j
\end{array}$$

is an homeomorphism. 

\end{theorem}
\begin{proof}

We saw above that $\varphi$ is a spectral/continuous bijection. To show that $\psi = \varphi^{-1} $ is a continuous map is equivalent to prove that $\varphi$ is an open map.\\

Let $D^{\infty}\,(a_k)\times \{ k\}$, $k \in \{ i,j\}$, be any basic open subset of ${\rm Spec}^{\infty}\,(A_i) \sqcup {\rm Spec}^{\infty}\,(A_j)$. We have, for $k=i$ (the case when $k=j$ is analogous):

$$\varphi[D^{\infty}\,(a_i)\times \{ i\}]= \{ \mathfrak{q}_i \times A_j | (\mathfrak{q}_i,i) \in D^{\infty}(a_i) \times \{ i\}\} =\{ \mathfrak{q}_i \times A_j | a_i \notin \mathfrak{q}_i \}$$

and we have:

$$D^{\infty}(a_i,0_j) = \{  \mathfrak{q}_i\times A_j |a_i \notin \mathfrak{q}_i \} \stackrel{\cdot}{\cup} \{ A_i \times \mathfrak{q}_j | 0_j \notin \mathfrak{q}_j \} = \{  \mathfrak{q}_i\times A_j |a_i \notin \mathfrak{q}_i \}$$

Thus $\varphi[D^{\infty}(a_i)\times \{ i\}] = \{  \mathfrak{q}_i\times A_j |a_i \notin \mathfrak{q}_i \} = D^{\infty}(a_i,0_j)$, and $\varphi$ maps (basic) open sets to (basic) open sets. Since $\varphi$ is an open spectral/continuous bijection, it is an homeomorphism.
\end{proof}

By induction, we obtain the following:

\begin{theorem}\label{esteaca}
Let $\{ A_i | i \in \{1,2, \cdots,n \} \}$ be any finite family of $\mathcal{C}^{\infty}-$rings. We have:
$${\rm Spec}^{\infty}\,\left( \prod_{i=1}^{n} A_i\right) \approx \coprod_{i=1}^{n}{\rm Spec}^{\infty}\,(A_i)$$
\end{theorem}

We also register the following consequence of the above theorem:

\begin{theorem}\label{welito}Let $\mathbb{K}$ be any $\mathcal{C}^{\infty}-$field, let $I$ be a finite set and let ${\rm Discr}.\,(I) = (I, \wp(I))$ be the corresponding discrete topological space. We have:
$${\rm Spec}^{\infty}\,(\mathbb{K}^{I}) \approx {\rm Discr}.\,(I)$$
\end{theorem}
\begin{proof}
First note that ${\rm Spec}^{\infty}(\mathbb{K}^I)$ is a finite Boolean space.\\

Indeed, since $\mathbb{K}$ is a $\mathcal{C}^{\infty}-$field, $\mathbb{K}^I$ is a von Neumann regular $\mathcal{C}^{\infty}-$ring, so ${\rm Spec}^{\infty}\,(\mathbb{K}^I) = {\rm Specm}^{\infty}\,(\mathbb{K}^I)$. Since $I$ is a finite set, we have:

$${\rm Spec}^{\infty}\,(\mathbb{K}^I) = \{ \mathfrak{m}_i = \widehat{(\mathbb{K}^I)} | i \in I\}\cup \{ 0\} \cup \{ \mathbb{K}^I\}.$$

Since $\mathbb{K}$ is a $\mathcal{C}^{\infty}-$field, the only $\mathcal{C}^{\infty}-$radical prime ideals of $\mathbb{K}$ are the maximal ones. Consider:
$$\begin{array}{cccc}
    \varphi : & {\rm Discr}.\,(I) & \rightarrow & {\rm Spec}^{\infty}\,(\mathbb{K}^{I}) \\
     & i & \mapsto & \mathfrak{m}_i = \widehat{(\mathbb{K}^{I})_i}:= \mathbb{K}\times \cdots \times \mathbb{K} \times \{ 0\} \times \mathbb{K} \times \cdots \times \mathbb{K}
  \end{array},$$
which is obviously continuous since ${\rm Discr}.\,(I)$ is a discrete space.
\end{proof}

\begin{theorem}\label{tadeu}Let $(I, \preceq)$ be a filtered partially ordered set and let  $\{(A_i,\Phi_i) | i \in I \}$ be a direct filtered system of $\mathcal{C}^{\infty}-$rings, so we have the following colimit diagram:

$$\xymatrixcolsep{5pc}\xymatrix{
 & \varinjlim_{i \in I} (A_i,\Phi_i) & \\
(A_i,\Phi_i) \ar@/^/[ur]^{\alpha_i} \ar[rr]_{\alpha_{ij}} & & (A_j,\Phi_j) \ar@/_/[ul]_{\alpha_j}
}$$

The following map is an homeomorphism
$$\begin{array}{cccc}
    \kappa : & {\rm Spec}^{\infty}\, \left( \varinjlim_{i \in I} A_i\right) & \stackrel{\approx}{\rightarrow} & \varprojlim_{i \in I} {\rm Spec}^{\infty}\,(A_i)\\
       & \mathfrak{p} & \mapsto & (\alpha_i^{\dashv}[\mathfrak{p}])_{i \in I}
\end{array}$$
whose inverse is given by:
$$\begin{array}{cccc}
    \kappa' : & \varprojlim_{i \in I} {\rm Spec}^{\infty}\,(A_i) & \stackrel{\approx}{\rightarrow} & {\rm Spec}^{\infty}\, \left( \varinjlim_{i \in I} A_i\right)\\
       & (\mathfrak{p}_i)_{i \in I} & \mapsto & \varinjlim_{i \in I}\mathfrak{p}_i= \bigcup_{i \in I} \alpha_i[\mathfrak{p}_i]
\end{array}$$
\end{theorem}
\begin{proof}
Since for every $i \in I$, ${\rm Spec}^{\infty}(A_i)$ is a spectral space and the category of spectral spaces and spectral maps is closed under projective limits, it follows that:

$$\varprojlim_{i \in I}{\rm Spec}^{\infty}(A_i)$$

is a spectral space, and for every $i \in I$,

$$\begin{array}{cccc}
    {\alpha_i}^{*}: & \varprojlim_{i \in I}{\rm Spec}^{\infty}(A_i) & \rightarrow & {\rm Spec}^{\infty}(A_i) \\
     & (\mathfrak{p}_i)_{i \in I} & \mapsto & \mathfrak{p}_i
  \end{array}$$

is a spectral map.\\

Also, since each

$$\begin{array}{cccc}
    \alpha_i: & (A_i,\Phi_i) & \rightarrow & \varinjlim_{i \in I}(A_i,\Phi_i) \\
     & a_i & \mapsto & [(a_i,i)]
  \end{array}$$

is a $\mathcal{C}^{\infty}-$homomorphism, it follows that given any $\mathfrak{p} \in {\rm Spec}^{\infty}\left( \varinjlim_{i \in I} (A_i,\Phi_i)\right)$, $\alpha_i^{*}(\mathfrak{p}) = \alpha_i^{\dashv}[\mathfrak{p}] \in {\rm Spec}^{\infty}(A_i)$.\\

We note, thus, that $\kappa$ is a spectral map, since each of its coordinates is spectral and the category of the spectral spaces is closed under projective limits.\\

Since the colimit diagram commutes, we have $\alpha_{ij}^{\dashv}[\mathfrak{p}_j] = \mathfrak{p_i}$, so:

$$\mathfrak{p}_i \subseteq \alpha_{ij}^{\dashv}[\mathfrak{p}_j] \iff \alpha_{ij}[\mathfrak{p}_i] \subseteq \mathfrak{p}_j$$

thus we have the following commutative square:

$$\xymatrixcolsep{5pc}\xymatrix{
\mathfrak{p}_i \ar[d]_{\alpha_{ij}\upharpoonright_{\mathfrak{p}_i}} \ar[r]^{\iota_{\mathfrak{p}_i}^{A_i}} & A_i \ar[d]^{\alpha_{ij}}\\
\mathfrak{p}_j \ar[r]_{\iota_{\mathfrak{p}_j}^{A_j}} & A_j
}$$

Given $(\mathfrak{p}_i)_{i \in I} \in \varprojlim_{i \in I} {\rm Spec}^{\infty}\,(A_i)$, by \textbf{Corollary \ref{Quico}}, $\varinjlim_{i \in I} \mathfrak{p}_i$ is a prime $\mathcal{C}^{\infty}-$radical ideal of $\varinjlim_{i \in I} A_i$, so

$$\begin{array}{cccc}
\kappa': & \varprojlim_{i \in I} {\rm Spec}^{\infty}\,(A_i) & \rightarrow & {\rm Spec}^{\infty}\,\left( \varinjlim_{i \in I} A_i \right)\\
 & (\mathfrak{p}_i)_{i \in I} & \mapsto & \bigcup_{i \in I} \alpha_i[\mathfrak{p}_i]
\end{array}$$

is indeed a function.\\

We have, thus:

$$\kappa' \left[ \varprojlim_{i \in I} {\rm Spec}^{\infty}(A_i,\Phi_i)\right] \subseteq  {\rm Spec}^{\infty}\left(\varinjlim_{i \in I}(A_i,\Phi_i)\right)$$

so

$$\kappa = \alpha \upharpoonright_{{\rm Spec}^{\infty}\left(\varinjlim_{i \in I}(A_i,\Phi_i)\right)}$$

and

$$\kappa' = \alpha' \upharpoonright_{\varprojlim_{i \in I} {\rm Spec}^{\infty}(A_i,\Phi_i)}$$

where $\alpha$ is given in \textbf{Proposition 14} of \cite{cerqueira2019universal} and $\alpha'$ is given in its proof.\\

Since:

$$\alpha\left[ {\rm Spec}^{\infty}\left( \varinjlim_{i \in I} A_i\right) \right] = \varprojlim_{i \in I} {\rm Spec}^{\infty}(A_i)$$

we have:

$$\xymatrixcolsep{7pc}\xymatrix{
\mathfrak{I}\left( \varinjlim_{i \in I} A_i\right) \ar[r]^{\alpha} & \varprojlim_{i \in I} \mathfrak{I}(A_i)\\
{\rm Spec}^{\infty}\left( \varinjlim_{i \in I}A_i\right) \ar@{^{(}->}[u] \ar[r]_{\alpha \upharpoonright_{{\rm Spec}^{\infty}\left( \varinjlim_{i \in I}A_i\right)}}& \varprojlim_{i \in I} {\rm Spec}^{\infty}\,(A_i) \ar@{^{(}->}[u]}$$

Since $\kappa$ and $\kappa'$ are restrictions of inverse bijections, it follows that $\kappa$ and $\kappa'$ are inverse bijections.

\end{proof}

\begin{definition}Given a $\mathcal{C}^{\infty}-$ring $A$, $({\rm Spec}^{\infty}\,(A), {\rm Zar}^{\infty})$ is a spectral space (see \textbf{Theorem \ref{diabolin}}). The \index{constructible topology}\textbf{constructible topology} on ${\rm Spec}^{\infty}\,(A)$, denoted by ${\rm Spec}^{\infty- {\rm const}}\,(A)$  is the smalest boolean topology in this set such that the identity function is continuous/spectral $id : {\rm Spec}^{\infty-{\rm const}}\,(A) \to {\rm Spec}^{\infty}\,(A)$. It can be  constructed by taking the sub-basis

$$\{ D^{\infty}(a) \cap Z^{\infty}(b) | a,b \in A \}.$$

\end{definition}

It is easy to see that, given any $\mathcal{C}^{\infty}-$ring $A$, ${\rm Spec}^{\infty-{\rm const}}(A)$ is indeed a Boolean space.\\

\begin{remark}\label{contraste}Contrary to what occurs in Commutative Algebra, the sub-basis of the constructible topology on ${\rm Spec}^{\infty}\,(A)$ given by $\{ D^{\infty}(a)\cap Z^{\infty}(b) ; a,b \in A\}$ is closed under finite intersections (thus it is a basis), since

$$D^{\infty}(a_1)\cap Z^{\infty}(b_1)\cap \cdots \cap D^{\infty}(a_n)\cap Z^{\infty}(b_n) = D^{\infty}(a_1. \cdots .a_n) \cap Z^\infty(b_1^2 + \cdots + b_n^2).$$

\end{remark}



\subsection{$C^\infty$-rings and $C^\infty$-locally ringed spaces}


\hspace{0.5cm} Just like in Algebraic Geometry, there is a structural sheaf for each $C^\infty$-ring, such that the global section of this sheaf is canonically isomorphic to the original $C^\infty$-ring. The main reference for this section is  \cite{rings2}.\\

\begin{definition}A \index{$\mathcal{C}^{\infty}-$ringed space}\textbf{$\mathcal{C}^{\infty}-$ringed space} is an ordered pair, $(X, \mathcal{O}_X)$, where $X$ is a topological space and:

$$\mathcal{O}_X : {\rm Open}\,(X)^{\rm op} \rightarrow \mathcal{C}^{\infty}{\rm \bf Rng}$$

is a sheaf of $\mathcal{C}^{\infty}-$rings on $X$ such that all the stalks are local $\mathcal{C}^{\infty}-$rings, that is, for every $x \in X$, $\mathcal{O}_{X,x}$ is a local $\mathcal{C}^{\infty}-$ring.
\end{definition}

\begin{remark}Another way of considering these spaces was pointed out by E. J. Dubuc in \cite{dubuc1981c}, as follows: a $\mathcal{C}^{\infty}-$ringed space, $(X, \mathcal{O}_X)$ is a $\mathcal{C}^{\infty}-$ring object in ${\rm Sh}\,(X)$.\end{remark}

\begin{definition}Let $(X, \mathcal{O}_X)$ and $(Y, \mathcal{O}_Y)$ be two $\mathcal{C}^{\infty}-$ringed spaces. A \index{$\mathcal{C}^{\infty}-$ringed space!morphism of $\mathcal{C}^{\infty}-$ringed spaces}\textbf{morphism of $\mathcal{C}^{\infty}-$ringed spaces} is a pair:

$$(f,f^{\sharp}): (X, \mathcal{O}_X) \rightarrow (Y, \mathcal{O}_Y)$$

where $f: X \to Y$ is a continuous map and $f^{\sharp}: f^{\dashv}[\mathcal{O}_Y] \to \mathcal{O}_X$ is a morphism of sheaves of $\mathcal{C}^{\infty}-$rings on $X$ which induces a local $\mathcal{C}^{\infty}-$homomorphism of local $\mathcal{C}^{\infty}-$rings at each stalk,

$$f^{\sharp}_x: f^{\dashv}[\mathcal{O}_{Y,f(x)}] \to \mathcal{O}_{X,x}$$
\end{definition}

The category of all $\mathcal{C}^{\infty}-$locally ringed spaces, together with its morphisms, is going to be denoted by $\mathcal{C}^{\infty}{\rm \bf RngSp}$.\\

Given a $\mathcal{C}^{\infty}-$ring $A$, we define a sheaf on ${\rm Spec}^{\infty}\,(A)$, $\Sigma_A: {\rm Open}\,({\rm Spec}^{\infty}\,(A))^{\rm op} \rightarrow \mathcal{C}^{\infty}{\rm \bf Rng}$, such that for very $a \in A$,

$$\Sigma_A(D^{\infty}(a)) = A\{ a^{-1}\}.$$

In order to define the action of $\Sigma_A$ on the arrows between basic open subsets of ${\rm Spec}^{\infty}\,(A)$, we need some results.\\

\begin{theorem}\label{Maya}Let $A$ be a $\mathcal{C}^{\infty}-$ring and $a,b \in A$. The following conditions are equivalent:
\begin{itemize}
  \item[(1)]{$\eta_a(b) \in (A\{ a^{-1}\})^{\times}$;}
  \item[(2)]{There is a unique $\mathcal{C}^{\infty}-$homomorphism $\sigma_{a,b}: A\{ b^{-1}\} \to A\{ a^{-1}\}$ such that the following triangle commutes:
      $$\xymatrix{
      A \ar[r]^{\eta_b} \ar[rd]_{\eta_a} & A\{ b^{-1}\} \ar@{.>}[d]^{\sigma_{a,b}}\\
         & A\{ a^{-1}\}
      }$$}
\end{itemize}
\end{theorem}
\begin{proof}
Ad $(1) \to (2)$. By the universal property of $\eta_b: A \to A\{ b^{-1}\}$, since $\eta_a(b) \in (A\{ a^{-1}\})^{\times}$ there is a unique $\mathcal{C}^{\infty}-$homomotphism $\sigma_{a,b}: A\{ b^{-1}\} \to A\{ a^{-1}\}$ such that the triangle commutes.\\

Ad $(2) \to (1)$. Suppose there is a unique $\mathcal{C}^{\infty}-$homomorphism $\sigma_{a,b}: A\{ b^{-1}\} \to A\{ a^{-1}\}$ such that $\sigma_{a,b} \circ \eta_b = \eta_a$. Then $(\sigma_{a,b} \circ \eta_b)(b) = \eta_a(b)$, i.e., $\sigma_{a,b}(\eta_b(b)) = \eta_a(b)$. Since $\eta_b(b) \in (A\{ b^{-1}\})^{\times}$, it follows that $\eta_a(b) = \sigma_{a,b}(\eta_b(b)) \in (A\{ a^{-1}\})^{\times}$.
\end{proof}

\begin{theorem}\label{Opash}Let $A$ be a $\mathcal{C}^{\infty}-$ring. Suppose that the equivalent conditions in \textbf{Theorem \ref{Raj}} hold. Then the equivalent conditions of \textbf{Theorem \ref{Maya}} hold.
\end{theorem}
\begin{proof}
We are going to show that for any $a,b \in A$, $D^{\infty}(a) = D^{\infty}(b)$ implies $\eta_a(b) \in (A\{ a^{-1}\})^{\times}$.\\

We know that $\sqrt[\infty]{(b)} = \{ x \in A | (\exists c \in (b))(\eta_x(c) \in (A\{ a^{-1}\})^{\times})\}$, so $a \in \sqrt[\infty]{(b)} \iff (\exists c \in (b))(\eta_a(c) \in (A\{ a^{-1}\})^{\times})$. Now, $c \in (b)$ occurs if, and only if, there is some $\lambda \in A$ such that $c = \lambda \cdot b$. Hence,
$$\eta_a(c) = \eta_a(\lambda \cdot b) = \eta_a(\lambda) \cdot \eta_a(b),$$
and since $\eta_a(c) \in (A\{ a^{-1}\})^{\times}$ it follows that both $\eta_a(\lambda) \in (A\{ a^{-1}\})^{\times}$ and $\eta_a(b) \in (A\{ a^{-1}\})^{\times}$.
\end{proof}

The following two remarks will prove that $\{ \sigma_{a,b} : A\{ a^{-1}\} \to A\{ b^{-1}\}\}_{a,b \in A}$ is a directed system of $\mathcal{C}^{\infty}-$homomorphisms.


\begin{remark}\label{Anusha}Note that whenever we have $D^{\infty}(a) = D^{\infty}(b)$, by \textbf{Theorem \ref{Opash}} we have a unique $\mathcal{C}^{\infty}-$isomorphism between $A\{ a^{-1}\}$ and $A\{ b^{-1}\}$, that follows from the uniqueness of the $\mathcal{C}^{\infty}-$homo\-morphisms $\sigma_{a,b}$ and $\sigma_{b,a}$ described in \textbf{Theorem \ref{Maya}}. In particular, when $a = b$ there exists a unique $\mathcal{C}^{\infty}-$homo\-morphism $\sigma_{a,a}: A\{ a^{-1}\} \to A\{ a^{-1}\}$ such that the following diagram commutes:
$$\xymatrix{
A \ar[r]^{\eta_a} \ar[dr]_{\eta_a} & A\{ a^{-1}\} \ar[d]^{\sigma_{a,a}}\\
  & A\{ a^{-1}\}
}$$
and since ${\rm id}_{A\{ a^{-1}\}}$ also has this property, it follows that $\sigma_{a,a} = {\rm id}_{A\{ a^{-1}\}}$.
\end{remark}

\begin{remark}Let $a,b,c \in A$ be such that $D^{\infty}(a) \subseteq D^{\infty}(b) \subseteq D^{\infty}(c)$. Then the following diagram commutes:
$$\xymatrix{
A\{ a^{-1}\} \ar[r]^{\sigma_{a,b}} \ar[dr]_{\sigma_{a,c}} & A\{ b^{-1}\} \ar[d]^{\sigma_{b,c}}\\
   & A\{ c^{-1}\}
}$$
\end{remark}
\begin{proof}Since $D^{\infty}(a) \subseteq D^{\infty}(b) \subseteq D^{\infty}(c)$, we have that $D^{\infty}(a) \subseteq D^{\infty}(c)$, so by \textbf{Theorem \ref{Maya}} there exists a unique $\mathcal{C}^{\infty}-$homomorphism $\sigma_{a,c}: A\{ a^{-1}\} \to A\{ c^{-1}\}$ such that the following triangle commutes:
$$\xymatrix{
A \ar[r]^{\eta_a} \ar[dr]_{\eta_c} & A\{ a^{-1}\} \ar[d]^{\sigma_{a,c}} \\
  & A\{ c^{-1}\}
}$$
Since:
$$\xymatrix{
 & A\{ a^{-1}\} \ar[d]^{\sigma_{a,b}}\\
 A \ar[ur]^{\eta_a} \ar[r]^{\eta_b} \ar[dr]_{\eta_c} & A\{ b^{-1}\} \ar[d]^{\sigma_{b,c}}\\
   & A\{ c^{-1}\}
}$$
commutes, it follows that $\sigma_{b,c} \circ \sigma_{a,b}$ has the same property that defines $\sigma_{a,c}$, and by its uniqueness, $\sigma_{a,c} = \sigma_{b,c} \circ \sigma_{a,b}$.
\end{proof}

Now we build a pre-sheaf over the basic open subsets of ${\rm Spec}^{\infty}\,(A)$, \textit{i.e.}, the compact open subsets of ${\rm Spec}^{\infty}\,(A)$. Any basic open subset $U$ of ${\rm Spec}^{\infty}\,(A)$ is $U = D^{\infty}(a)$ for some $a \in A$. Given any basic open subset $U$ of ${\rm Spec}^{\infty}\,(A)$, define:
$$\Sigma_A(U):= \varprojlim_{U = D^{\infty}(a)} A\{ a^{-1}\}$$

that is, $\Sigma_A(U)$ is the colimit of the system $\{ A\{ a^{-1}\} \to^{\sigma_{a',a}} A\{ {a'}^{-1}\} | a,a' \in A \,\, {\rm such that}\,\, U = D^{\infty}(a) = D^{\infty}(a') \}$, where $\sigma_{a',a}$ is the $\mathcal{C}^{\infty}-$isomorphism described by \textbf{Remark \ref{Anusha}}.\\

We note that given any basic open subset of ${\rm Spec}^{\infty}\,(A)$, $U$, $\Sigma_A(U)$ has an $A-$algebra structure that is induced by the universal property of $\Sigma_A(U):= \varprojlim_{D^{\infty}(a)=U} A\{ a^{-1}\}$. Indeed, given any $a,a' \in A$ such that $D^{\infty}(a)=U=D^{\infty}(a')$, we have a unique $\mathcal{C}^{\infty}-$homomorphism $\nu_A : A \to \Sigma_A(U)$ such that the following diagram commutes:
$$\xymatrix{
    & A \ar@/_2pc/[ldd]_{\eta_a} \ar@{.>}[d]_{\nu_A} \ar@/^2pc/[rdd]^{\eta_{a'}}&  \\
    & \Sigma_A(U) \ar[ld] \ar[rd] &    \\
A\{ a^{-1}\} \ar[rr]^{\sigma_{a',a}} &   & A\{ {a'}^{-1}\}
}$$

Let $U,V$ be two open subsets of ${\rm Spec}^{\infty}\,(A)$ such that $\imath_{U}^{V}: U \hookrightarrow V$ is the inclusion map, \textit{i.e.}, $U = D^{\infty}(a) \subseteq D^{\infty}(c)=V$, and let $c,a \in A$ be two elements such that $V = D^{\infty}\,(c)$ and $U=D^{\infty}\,(a)$. Since $D^{\infty}(a) \subseteq D^{\infty}(c)$, by \textbf{Theorem \ref{Opash}} there is a unique $\mathcal{C}^{\infty}-$homomorphism $\sigma_{c,a}: A\{ a^{-1}\} \to A\{ c^{-1}\}$, that we shall denote by $\rho^{V}_{U} := \sigma_{c,a}$ such that the following triangle commutes:
$$\xymatrix{
A \ar[r]^{\eta_a} \ar[rd]_{\eta_c} & A\{ a^{-1}\}\\
   & A\{ c^{-1}\} \ar[u]^{\rho^{V}_{U}}
}$$

Let $p^{V}_{c}: \varprojlim_{D^{\infty}(c')=V} A\{ {c'}^{-1}\} \to A\{ c^{-1}\}$ and $p^{U}_{a}: \varprojlim_{D^{\infty}(a')=U} A\{ {a'}^{-1}\} \to A\{ a^{-1}\}$ be the canonical colimit arrows, that are also $\mathcal{C}^{\infty}-$isomorphisms. We have the following diagram:

$$\xymatrixcolsep{5pc}\xymatrix{
\Sigma_A(V) = \varprojlim_{D^{\infty}(c) = V} A\{ c^{-1}\} \ar[r]^{p^{V}_{c}} & A\{ c^{-1}\}\\
\Sigma_A(U) = \varprojlim_{D^{\infty}(a) = U} A\{ a^{-1}\} \ar[r]^{p^{U}_{a}} & A\{ a^{-1 }\} \ar[u]^{\rho^{V}_{U}}
}$$

Since $\{ \sigma_{c',a'} \circ p^{U}_{a'} : \Sigma_A(U) \to A\{ {c'}^{-1}\} | c' \in A \,\, {\rm such \,\, that}\,\, D^{\infty}(c')=V \}$ is a cone, by the universal property of $\Sigma_A(V)$ there exists a unique $\mathcal{C}^{\infty}-$homomorphism such that the following prism commutes for every $a,a',c,c' \in A$ such that $U= D^{\infty}(a) = D^{\infty}(a')$ and $V = D^{\infty}(c) = D^{\infty}(c')$:

$$ \xymatrixrowsep{5pc}\xymatrixcolsep{5pc}\xymatrix @!0 @R=4pc @C=6pc {
    \Sigma_A(U) \ar[rr]^{p^{U}_{a}} \ar[rd]^{p^{U}_{a'}} \ar@{.>}[dd]_{\exists ! \rho^{V}_{U}} && A\{ a^{-1}\} \ar[dd]^{\sigma_{c,a}} \\
    & A\{ {a'}^{-1} \} \ar[ru]^{\sigma_{a',a}} \ar[dd]^{\sigma_{c',a'}} \\
    \Sigma_A(V) \ar[rr] |!{[ur];[dr]}\hole \ar[rd]^{p^{V}_{c'}} && A\{ c^{-1}\} \\
    & A\{ {c'}^{-1}\} \ar[ru]^{\sigma_{c',c}} }$$

By the uniqueness of the arrow $\rho^{V}_{U}$, it follows that given basic open subsets $U,V,W$ of ${\rm Spec}^{\infty}\,(A)$ such that $U \subset V \subset W$,

$$\Sigma_A(\imath^{W}_{V} \circ \imath^{V}_{U}) = \Sigma_A(\imath^{W}_{U}) = \Sigma_A(\imath^{W}_{V}) \circ \Sigma_A(\imath^{V}_{U})$$

and

$\Sigma_A(\imath^{U}_{U}) = {\rm id}^{\Sigma_A(U)}_{\Sigma_A(U)}$.\\

Thus we have defined a functor:
$$\begin{array}{cccc}
    \Sigma_A: & {\rm Open}\,(\mathcal{B}({\rm Spec}^{\infty}\,(A)))^{{\rm op}} & \rightarrow & \mathcal{C}^{\infty}{\rm \bf Rng} \\
     & U & \mapsto & \Sigma_A(U) = \varprojlim_{U=D^{\infty}(a)} A\{ a^{-1}\} \\
     & \imath^{V}_{U}: U \hookrightarrow V & \mapsto & \rho^{V}_{U}: \Sigma_{A}(U) \to \Sigma_A(V)
  \end{array}$$

hence a pre-sheaf of $\mathcal{C}^{\infty}-$rings on the basic open sets of ${\rm Spec}^{\infty}\,(A)$.\\

We extend this definition from the basic open sets to any open set $U$ of ${\rm Spec}^{\infty}(A))$:

$$\Sigma_A(U):= \varprojlim_{D^{\infty}(a) \subseteq U} \Sigma_A(D^{\infty}(a)).$$

Given $U,V \in {\rm Open}\,({\rm Spec}^{\infty}\,(A))$ such that $U \subseteq V$, we denote:

$$\Sigma_A(V) \stackrel{\sigma^{V}_{U}}{\rightarrow} \Sigma_A(U)$$

the $\mathcal{C}^{\infty}-$homomorphism induced by the inclusion $\iota^{V}_{U}: U \hookrightarrow V$, via the universal property of the limit.\\

In \cite{rings2}, we can find a proof of the following:

\begin{proposition}[\textbf{Proposition 1.6} of \cite{rings2}]Let $A$ be a $\mathcal{C}^{\infty}-$ring and let ${\rm Spec}^{\infty}\,(A)$ be its smooth Zariski spectrum. Using the notations of the above considerations,
$$\Sigma_A: {\rm Open}\,({\rm Spec}^{\infty}\,(A))^{\rm op} \rightarrow \mathcal{C}^{\infty}{\rm \bf Rng}$$
is a sheaf and all its stalks are local $\mathcal{C}^{\infty}-$rings. This means that, given any $\mathfrak{p} \in {\rm Spec}^{\infty}(A)$, $A_{\mathfrak{p}} \cong A\{ {A\setminus \mathfrak{p}}^{-1}\}$ is a local $\mathcal{C}^{\infty}-$ring.
\end{proposition}

Summarizing the previous results, the $\mathcal{C}^{\infty}-$Zariski spectrum of a $\mathcal{C}^{\infty}-$ring $A$ can be regarded as a $\mathcal{C}^{\infty}-$ringed space, $({\rm Spec}^{\infty}\,(A), \Sigma_A)$. In \cite{rings2}, we find the following:

\begin{theorem}[\textbf{Theorem 1.7} of \cite{rings2}]\label{cascata}Let $A$ be any $\mathcal{C}^{\infty}-$ring.\\

$$\begin{array}{cccc}
    {\rm Spec}^{\infty}: & \mathcal{C}^{\infty}{\rm \bf Rng} & \rightarrow & \mathcal{C}^{\infty}{\rm \bf RngdSp} \\
     & A & \mapsto & ({\rm Spec}^{\infty}\,(A), \Sigma_A)
  \end{array}$$
is a contravariant functor from the category of $\mathcal{C}^{\infty}-$rings and the category of the $\mathcal{C}^{\infty}-$ringed spaces, which is adjoint to the global sections functor, i.e., for every $\mathcal{C}^{\infty}-$ring $A$ and for every $\mathcal{C}^{\infty}-$ringed space $(X,\mathcal{O}_X)$  there is a natural bijection:

$$\dfrac{A \rightarrow \Gamma(X,\mathcal{O}_X) }{(X,\mathcal{O}_X) \rightarrow ({\rm Spec}^{\infty}\,(A), \Sigma_A)}$$

Moreover, $\Gamma({\rm Spec}^{\infty}(A),\Sigma_A) \cong A$, so ${\rm Spec}^{\infty}$ is a full and faithful functor.
\end{theorem}

\section{Order Theory of $\mathcal{C}^{\infty}-$rings}\label{ot}

The class of $\mathcal{C}^{\infty}-$rings carries good notions of order theory for rings. As pointed out by Moerdijk and Reyes in \cite{rings2}, every $\mathcal{C}^{\infty}-$ring $A$ has a canonical \index{pre-order}pre-order. We describe it in the following:\\

\begin{definition}\label{villa}Given a $\mathcal{C}^{\infty}-$ring $(A,\Phi)$, we write:

$$(\forall a \in A)(\forall b \in A)(a \prec b \iff (\exists c \in A^{\times})(b-a = c^2))$$
Note that if $A \neq \{ 0\}$, $\prec$ is an irreflexive relation, i.e., $(\forall a \in A)(\neg (a \prec a))$.
\end{definition}

According to \cite{moerdijk1986rings}, we have the following facts about the order $\prec$ defined above.\\

\begin{fact}Let $A=\dfrac{\mathcal{C}^{\infty}(\mathbb{R}^{E})}{I}$ for some set $E$. Then, given any $f+I,g+I \in A$, with respect to the relation $\prec$, given in \textbf{Definition \ref{villa}}, we have:
$$f \prec g \iff (\exists \varphi \in I)((\forall x \in Z(\varphi))(f(x)<g(x)))$$
so $\prec$ is compatible with the ring structure which underlies $A$, \textit{i.e.}:
\begin{itemize}
  \item[(i)]{$0 \prec f+I,g+I \Rightarrow 0 \prec (f+I) \cdot (g+I)  $;}
  \item[(ii)]{$0 \prec f+I,g+I \Rightarrow 0 \prec f+g + I $}
\end{itemize}
\end{fact}

\begin{fact}Let $A=\dfrac{\mathcal{C}^{\infty}(\mathbb{R}^{E})}{I}$  for some set $E$ be a $\mathcal{C}^{\infty}-$field, so $I = \sqrt[\infty]{I}$. The relation $\prec$, given in \textbf{Definition \ref{villa}}, is such that:
 $$(\forall f+I \in A)(f+I \neq 0+I \rightarrow (f + I \prec 0)\vee (0 \prec f+I))$$
\end{fact}

We have the following:

\begin{proposition}\label{doria}For any  $\mathcal{C}^{\infty}-$ring $A$, we have:
$$1+ \sum A^2 \subseteq A^{\times},$$

where $\sum A^2 = \{ \sum_{i=1}^{n}a_i^2 | n \in \mathbb{N}, a_i \in A\}$. In particular, every $\mathcal{C}^{\infty}$ ring $A$ is such that its underlying commutative unital ring, $\widetilde{U}(A)$ is a semi-real ring.
\end{proposition}
\begin{proof}
First suppose $A = \mathcal{C}^{\infty}(\mathbb{R}^n)$ for some $n \in \mathbb{N}$. Given any $k \in \mathbb{N}$ and any $k-$tuple, $f_1, \cdots, f_k \in \mathcal{C}^{\infty}(\mathbb{R}^n)$, we have $Z\left( 1 + \sum_{i=1}^{k}f_i^2\right) = \varnothing$, so $1 + \sum_{i=1}^{k}f_i^2 \in \mathcal{C}^{\infty}(\mathbb{R}^n)^{\times}$.\\

Suppose, now, that $A$ is a finitely generated $\mathcal{C}^{\infty}-$ring, that is, that there are $n \in \mathbb{N}$ and an ideal $I \subseteq \mathcal{C}^{\infty}(\mathbb{R}^n)$ such that:
$$A = \dfrac{\mathcal{C}^{\infty}(\mathbb{R}^n)}{I}$$

We have, for every $k \in \mathbb{N}$ and for every $k-$tuple $f_1+I, \cdots, f_k +I \in \dfrac{\mathcal{C}^{\infty}(\mathbb{R}^n)}{I}$:

$$(1+I) + \sum_{i=1}^{k}(f_i+I)^2 = \left( 1 + \sum_{i=1}^{k}f_i^2\right) + I$$

Since $1 + \sum_{i=1}^{k}f_i^2 \in \mathcal{C}^{\infty}(\mathbb{R}^n)^{\times}$, it follows that $\left( 1 + \sum_{i=1}^{k}f_i^2\right) + I \in \left( \dfrac{\mathcal{C}^{\infty}(\mathbb{R}^n)}{I}\right)^{\times}$.\\

Finally, for an arbitrary $\mathcal{C}^{\infty}-$ring $A$, we can always write:

$$A = \varinjlim_{i \in I} A_i$$

for the directed family $\{ A_i | i \in I\}$ consisting of its finitely generated $\mathcal{C}^{\infty}-$subrings. We are going to denote the canonical colimit arrow by $\alpha_i: A_i \to \varinjlim_{i \in I}A_i$, for every $i \in I$.\\

For every $k \in \mathbb{N}$ and for every $k-$tuple, $[a_{i_1}], \cdots, [a_{i_k}] \in \varinjlim_{i \in I}A_i$, there is some $\ell \geq i_1, \cdots, i_k$ such that for every $j \in \{ 1, \cdots, k\}$ $[a_{i_j}] = [\alpha_{\ell}(\alpha_{i_j\ell}(a_{i_j}))]$. Since $A_{\ell}$ is finitely generated, it follows that $1 + \sum_{j=1}^{k} \alpha_{i_j\ell}(a_{i_j})^2 \in A_{\ell}^{\times}$, so
$$1 + \sum_{j=1}^{k}[a_{i_j}]^2 = \alpha_{\ell}\left( 1 + \sum_{j=1}^{k} \alpha_{i_j\ell}(a_{i_j})^2\right) \in \left( \varinjlim_{i \in I}A_i\right)^{\times}.$$
\end{proof}

\begin{fact}\label{famosa}Let $n \in \mathbb{N}$. We have:

$$h + I \in \left( \dfrac{\mathcal{C}^{\infty}(\mathbb{R}^n)}{I}\right)^{\times} \iff (\exists \varphi \in I)(\forall x \in Z(\varphi))(h(x) \neq 0)$$
\end{fact}

Recall that a totally ordered field $(F, \leq)$ is \index{real closed field}\textbf{real closed} if it satisfies:
\begin{itemize}
  \item[(a)]{$(\forall x \in F)(0 < x \rightarrow (\exists y \in F)(x=y^2))$;}
  \item[(b)]{every polynomial of odd degree has, at least, one root;}
\end{itemize}

\begin{remark}A \index{$\mathcal{C}^{\infty}-$polynomial}\textbf{$\mathcal{C}^{\infty}-$polynomial in one variable} is an element of $\mathcal{C}^{\infty}(\mathbb{R}^n)\{ t\}$. More generally, a \textbf{$\mathcal{C}^{\infty}-$polynomial in set $S$ of variables} is an element of $A\{ S\}$.
\end{remark}

As pointed out in \textbf{Theorem 2.10} of \cite{moerdijk1986rings}, we have the following:

\begin{fact}Every $\mathcal{C}^{\infty}-$field, $F$, together with its canonical preorder $\prec$ given in \textbf{Definition \ref{villa}}, is such that $\widetilde{U}(F)$ is a real closed field.
\end{fact}

As a consequence of the above fact, given any polynomial $p \in F[x]$ and any $f,g \in F$ such that $f \prec g$, if $p(f)<p(g)$ then there is some $h \in ]f,g[$ such that $p(h)=0$.\\

We have the $\mathcal{C}^{\infty}-$analog of the notion of ``real closedness'':\\

\begin{definition}Let $(F,\prec)$ be a $\mathcal{C}^{\infty}-$field. We say that $(F,\prec)$ is \index{$\mathcal{C}^{\infty}-$real closed}\textbf{$\mathcal{C}^{\infty}-$real closed} if, and only if:

$$(\forall f \in F\{ x\})((f(0)\cdot f(1)<0)\&(1 \in \langle \{ f, f'\}\rangle \subseteq F\{ x\} ) \rightarrow (\exists \alpha \in ]0,1[ \subseteq F)(f(\alpha)=0))$$
\end{definition}

\begin{fact}As proved in \textbf{Theorem 2.10'} of \cite{moerdijk1986rings}, every $\mathcal{C}^{\infty}-$field is $\mathcal{C}^{\infty}-$real closed.
\end{fact}

From the preceding proposition, we conclude that $\overline{f}\prec \overline{g}$ occurs if, and only if, there is a ``witness'' $\varphi \in I$ such that $(\forall x \in Z(\varphi))(f(x)<g(x))$.

Given any $\mathcal{C}^{\infty}-$ring $A$ and any $\mathfrak{p} \in {\rm Spec}^{\infty}\,(A)$, let:

$$k_{\p}:= \left( \dfrac{A}{\p} \right)\left\{ {\dfrac{A}{\p} \setminus \{ 0 + \p\}}^{-1}\right\},$$

that is, $k_{\p}(A)$ is the $\mathcal{C}^{\infty}-$field obtained by taking the quotient $\dfrac{A}{\p}$:

$$q_{\p}: A \to \dfrac{A}{\p}$$

and then taking its $\mathcal{C}^{\infty}-$ring of fractions with respect to $\dfrac{A}{\p}\setminus \{ 0 + \p\}$,

$$\eta_{\dfrac{A}{\p} \setminus \{ 0 + \p\}}: \dfrac{A}{\p} \rightarrow \left( \dfrac{A}{\p}\right)\left\{ {\dfrac{A}{\p} \setminus \{ 0 + \p\}}^{-1}\right\}.$$

The family of $\mathcal{C}^{\infty}-$fields  $\{ k_{\p}(A)| \p \in {\rm Spec}^{\infty}\,(A)\}$ has the following multi-universal property:

``Given any $\mathcal{C}^{\infty}-$homomorphism $f: A \to \mathbb{K}$, where $\mathbb{K}$ is a $\mathcal{C}^{\infty}-$field, there is a unique $\mathcal{C}^{\infty}-$radical prime ideal $\mathfrak{p}$ and a unique  $\mathcal{C}^{\infty}-$homomorphism $\widetilde{f}: k_{\p}(A) \to \mathbb{K}$ such that the following diagram commutes:
$$\xymatrixcolsep{5pc}\xymatrix{
A \ar[r]^{\alpha_{\p}} \ar[dr]^{f} & k_{\p}(A) \ar@{.>}[d]^{\widetilde{f}}\\
  & K},$$

where $\alpha_{\p} = \eta_{\dfrac{A}{\p} \setminus \{ 0 + \p\}} \circ q_{\p}: A \rightarrow k_{\p}(A)$.''

Thus, given $f: A \to K$, take $\p = \ker(f)$, so $\overline{f}$ is injective, $\dfrac{A}{\p}$ is a $\mathcal{C}^{\infty}-$reduced $\mathcal{C}^{\infty}-$ring. By the universal property of the smooth fraction field $k_{\p}(A)$, there is a unique arrow $\widetilde{f}: k_{\p}(A) \to \mathbb{K}$ such that the following diagram commutes:

$$\xymatrixcolsep{5pc}\xymatrix{
\dfrac{A}{\p} \ar[r]^{\alpha_{\p}} \ar[dr]_{f} & \dfrac{A}{\p} \ar@{-->}[d]^{\widetilde{f}}\\
 & \mathbb{K}
}$$

\begin{definition} \label{rel}
Let $\mathcal{F}$ be the (proper) class of all the $\mathcal{C}^{\infty}-$homomorphisms of  $A$ to some $\mathcal{C}^{\infty}-$field. We define the following relation $\mathcal{R}$: given $h_1: A \to F_1$ and $h_2: A \to F_2$, we say that $h_1$ is related with $h_2$ if, and only if, there is some $\mathcal{C}^{\infty}-$field $\widetilde{F}$ and some $\mathcal{C}^{\infty}-$fields homomorphisms $\mathcal{C}^{\infty}$ $f_1: F_1 \to \widetilde{F}$ and $f_2: F_2 \to \widetilde{F}$ such that the following diagram commutes:

$$\xymatrixcolsep{5pc}\xymatrix{
  & F_1 \ar[dr]^{f_1} & \\
A \ar[ur]^{h_1} \ar[dr]_{h_2} & & \widetilde{F} \\
  & F_2 \ar[ur]_{h_2}
}$$

The relation $\mathcal{R}$ defined above is symmetric and reflexive.
\end{definition}

Let $h_1: A \to F_1$ and $h_2: A \to F_2$ be two $\mathcal{C}^{\infty}-$homomorphisms of  $A$ to the $\mathcal{C}^{\infty}-$fields $F_1,F_2$ such that $(h_1,h_2) \in \mathcal{R}$, and let $f_1: F_1 \to \widetilde{F}$ and $f_2: F_2 \to \widetilde{F}$ be two $\mathcal{C}^{\infty}-$homomorphisms to the $\mathcal{C}^{\infty}-$field $\widetilde{F}$ such that $f_1 \circ h_1= f_2\circ h_2$, so:
$$(f_1 \circ h_1)^{\dashv}[\{ 0\}] = (f_2 \circ h_2)^{\dashv}[\{ 0\}]$$
Then
$$\ker(h_1) = h_1^{\dashv}[\{ 0\}] = h_1^{\dashv}[f_1^{\dashv}[\{ 0\}]] =  h_2^{\dashv}[f_2^{\dashv}[\{ 0\}]] = h_2^{\dashv}[\{ 0\}] = \ker(h_2).$$

The above considerations prove the following:

\begin{proposition}If $h_1: A \to F_1$ and $h_2: A \to F_2$ be two $\mathcal{C}^{\infty}-$homomorphisms from the $\mathcal{C}^{\infty}-$ring  $A$ to the $\mathcal{C}^{\infty}-$fields $F_1,F_2$ such that $(h_1,h_2) \in \mathcal{R}$, then $\ker(h_1)=\ker(h_2)$.
\end{proposition}

The above proposition  has the following immediate consequence:

\begin{corollary}Keeping the same notations of the above result, let $\mathcal{R}^t$ be the transitive closure of $\mathcal{R}$. Then $(h_1,h_2) \in \mathcal{R}^t$ implies $\ker(h_1) = \ker(h_2)$. Thus, $\mathcal{R}^t$ is an equivalence relation on $\mathcal{F}$. We are going to denote the quotient set $\dfrac{\mathcal{F}}{\mathcal{R}^t}$ by $\widetilde{\mathcal{F}}$.
\end{corollary}

Let $F_1, F_2, \widetilde{F}$ be $\mathcal{C}^{\infty}-$fields  and $f_1: F_1 \to \widetilde{F}$ and $f_2: F_2 \to \widetilde{F}$ be two following $\mathcal{C}^{\infty}-$fields homomorphisms. Since $\mathcal{C}^{\infty}-$fields homomorphisms must be injective maps, we have the following:\\

\begin{proposition}The following relation
$$\beta = \{ ([h: A \to F], \ker (h)) | F \, \mbox{is a}\, \mathcal{C}^{\infty}-\mbox{field} \} \subseteq \widetilde{\mathcal{F}} \times {\rm Spec}^{\infty}(A)$$

is a functional relation whose domain is $\widetilde{\mathcal{F}}$.
\end{proposition}
\begin{proof}
Suppose $[h: A \to F_1] = [g: A \to F_2]$, so there are maps $f_1: F_1 \to \widetilde{F}$ and $f_2: F_2 \to \widetilde{F}$ for some $\mathcal{C}^{\infty}-$field $\widetilde{F}$, such that the following diagram commutes:

$$\xymatrix{
    & F \ar[rd]^{f_1} &    \\
A \ar[ur]^{h} \ar[dr]^{g} &  & \widetilde{F}\\
        & F_2 \ar[ur]^{f_2} & \\
}.$$

Now, if $([h: A \to F_1], \ker (h)), ([g: A \to F_2], \ker (g)) \in \beta$ are such that $[h: A \to F_1] = [g: A \to F_2]$, then:

$$f_1 \circ h = f_2 \circ g$$
so
$$\ker(h) =  h^{\dashv}[\{ 0 \}] = h^{\dashv}[f_1^{\dashv}[\{0\}]] = \ker(f_1 \circ h) = \ker (f_2 \circ g) = g^{\dashv}[f_2^{\dashv}[\{ 0\}]] = g^{\dashv}[\{ 0 \}] = \ker(g)
$$
\end{proof}

\begin{definition}Let $A$ be a $\mathcal{C}^{\infty}-$ring. A \index{$\mathcal{C}^{\infty}-$ordering}$\mathcal{C}^{\infty}-$\textbf{ordering} in $A$ is a subset $P \subseteq A$ such that:
\begin{itemize}
  \item[(O1)]{$P + P \subseteq P$;}
  \item[(O2)]{$P \cdot P \subseteq P$;}
  \item[(O3)]{$P \cup (-P) = A$}
  \item[(O4)]{$P \cap (-P) = \mathfrak{p} \in {\rm Spec}^{\infty}\,(A)$}
\end{itemize}
\end{definition}

\begin{definition}Let $A$ be a $\mathcal{C}^{\infty}-$ring. Given a $\mathcal{C}^{\infty}-$ordering $P$ in $A$, the \index{$\mathcal{C}^{\infty}-$support}$\mathcal{C}^{\infty}-$\textbf{support} of $A$ is given by:

$${\rm supp}^{\infty}(P):= P \cap (-P)$$
\end{definition}

\begin{definition}\label{Santo}Let $A$ be a $\mathcal{C}^{\infty}-$ring. The \index{$\mathcal{C}^{\infty}-$real spectrum of  $A$}\textbf{$\mathcal{C}^{\infty}-$real spectrum of  $A$} is given by:

$${\rm Sper}^{\infty}\,(A)=\{ P \subset A | P \, \mbox{is an ordering of the elements of}\,\, A\}$$
together with the (spectral) topology generated by the sets:

$$H^{\infty}(a) = \{ P \in {\rm Sper}^{\infty}\,(A) | a \in P \setminus {\rm supp}^{\infty}\,(P)\}$$

for every $a \in A$. The topology generated by these sets will be called ``\index{smooth Harrison topology}smooth Harrison topology'', and will be denoted by ${\rm Har}^{\infty}$.
\end{definition}

\begin{remark}\label{avio}Given a $\mathcal{C}^{\infty}-$ring $A$, we have a  function given by:

$$\begin{array}{cccc}
    {\rm supp}^{\infty}: & ({\rm Sper}^{\infty}(A), {\rm Har}^{\infty}) & \rightarrow & ({\rm Spec}^{\infty}\,(A), {\rm Zar}^{\infty}) \\
     & P & \mapsto & P \cap (-P)
  \end{array}$$

which is spectral, and thus continuous, since given any $a \in A$, ${{\rm supp}^{\infty}}^{\dashv}[D^{\infty}(a)] = H^{\infty}(a)\cup H^{\infty}(-a)$.
\end{remark}

Contrary to what happens to a general commutative ring $R$, for which the mapping:

$$\begin{array}{cccc}
    {\rm supp} : & {\rm Sper}\,(R) & \rightarrow & {\rm Spec}\,(R) \\
     & P & \mapsto & P \cap (-P)
  \end{array}$$

is seldom surjective or injective, within the category of $\mathcal{C}^{\infty}-$rings ${\rm supp}^{\infty}$ is, as matter of fact, a bijection.\\

In order to prove this fact, we are going to need some preliminary results, given below.\\

\begin{lemma}\label{90}Let $A$ be a $\mathcal{C}^{\infty}-$ring and $\mathfrak{p}$ any  $\mathcal{C}^{\infty}-$radical prime ideal, and let $\hat{\mathfrak{p}}$ be the maximal ideal of $A_{\{ \mathfrak{p}\}}$. Then:
$${\rm Can}_{\mathfrak{p}}^{\dashv}[\hat{\mathfrak{p}}] = \mathfrak{p}.$$
\end{lemma}
\begin{proof}
Let $a \in \mathfrak{p}$, then ${\rm Can}_{\mathfrak{p}}(a) \in \hat{\mathfrak{p}}$, and $\mathfrak{p} \subseteq {\rm Can}_{\mathfrak{p}}^{\dashv}[\hat{\mathfrak{p}}]$. Now, if $a \in A \setminus \mathfrak{p}$ then ${\rm Can}_{\mathfrak{p}}(a) \in \mathcal{U}(A\{ (A\setminus \mathfrak{p})^{-1} \})$. Since $A_{\{ \mathfrak{p} \}}$ is a local ring, $A_{\{\mathfrak{p} \}} = \hat{\mathfrak{p}}\stackrel{\cdot}{\cup} \mathcal{U}(A\{ (A \setminus \mathfrak{p})^{-1} \})$, so ${\rm Can}_{\mathfrak{p}}(a) \in A\{ (A\setminus \mathfrak{p})^{-1} \}\setminus \hat{\mathfrak{p}}$, and therefore ${\rm Can}_{\mathfrak{p}}^{\dashv}[\hat{\mathfrak{p}}] \subseteq \mathfrak{p}$.
\end{proof}

\begin{theorem}\label{1}Let $A$ be a $\mathcal{C}^{\infty}-$ring and $\mathcal{R}^t$ be the relation defined above. The function:
$$\begin{array}{cccc}
\alpha': & {\rm Sper}^{\infty}\,(A) & \rightarrow & \frac{\mathcal{F}}{\mathcal{R}^t}\\
        & P & \mapsto & [\eta_{P \cap (-P)}]
\end{array}$$
is the inverse function of:
$$\begin{array}{cccc}
\beta': & \frac{\mathcal{F}}{\mathcal{R}^t} & \to & {\rm Sper}^{\infty}\,(A)\\
       & [h: A \to K] & \mapsto & h^{\dashv}[K^2]
\end{array}$$
\end{theorem}
\begin{proof}
Note that:
$$(\alpha' \circ \beta')([h: A \to F]) = \alpha'(h^{\dashv}[F^2]) = [\eta_{{\rm supp}\,(h^{\dashv}[F^2])}],$$
where ${\rm supp}^{\infty}(h^{\dashv}[F^2]) = h^{\dashv}[F^2] \cap(- h^{\dashv}[F]^2) = h^{\dashv}[\{0\}] = \ker (h)$.\\

Thus we have:
$$ (\alpha' \circ \beta')([h: A \to F])= [\eta_{\ker(h)}: A \to k_{\ker(h)}(A)].$$

We claim that $\beta'$ is the left inverse function for $\alpha'$, that is:

$$(\forall P \in \, {\rm Sper}^{\infty}(A))((\beta' \circ \alpha')(P) = P).$$

Thus, it will follow that $\alpha'$ is injective and $\beta'$ is surjective.\\

We have $\beta'(\alpha'(P)) = \eta_{{\rm supp}^{\infty}(P)}^{\dashv}[k_{\p}(A)^{2}]$, so we need to show that:

$$\eta_{{\rm supp}(P)}^{\dashv}[k_{\p}(A)^2] = P$$

Let $\mathfrak{p} = {\rm supp}^{\infty}(P)$.\\

\textit{Ab absurdo}, suppose

$$\eta_{\mathfrak{p}}^{\dashv}[k_{\p}(A)^2] \nsubseteq P$$

There must exist some $x \in A$ such that $x \in \eta_{\mathfrak{p}}^{\dashv}[k_{\p}(A)^2]$ and $x \notin P$. We have:

$$\eta_{\mathfrak{p}}(x) \in \left( \frac{A\{ A \setminus {\mathfrak{p}}^{-1}\}}{\widehat{\mathfrak{p}}}\right)^2 \,\,\, {\rm and} \,\,\, x \notin P.$$

Now, since by \textbf{Theorem \ref{Pedro}} (denoting $\mathfrak{m}_{\p}$ by $\widehat{\p}$ instead)  the following diagram commutes:

$$\xymatrixcolsep{5pc}\xymatrix{
 & \dfrac{A\{ {A\setminus \p}^{-1}\}}{\widehat{\p}} \ar@<1ex>[dd]^{\varphi_{\p}}\\
A \ar[ur]^{\eta_{\p}'} \ar[dr]_{\eta_{\p}} & \\
   & \left( \dfrac{A}{\p}\right)\left\{ {\dfrac{A}{\p} \setminus \{ 0 + \p\}}^{-1}\right\} \ar@<1ex>[uu]^{\psi_{\p}}}$$

where $\eta_{\p} = \eta_{\dfrac{A}{\p} \setminus \{ 0 + \p\}} \circ q_{\p}$ and $\eta_{\p}' = q_{\widehat{\p}} \circ \eta_{A\setminus \p}$ and $\varphi_{\p}$ and $\psi_{\p}$ are the isomorphisms described in that theorem. Thus, we have:

$$\eta_{\mathfrak{p}}(x) \in \left( k_{\p}(A)\right)^2 \Rightarrow \psi_{\p}(\eta_{\p}(x)) = \eta_{\p}'(x) \in \left(\dfrac{A\{ {A\setminus \p}^{-1}\}}{\widehat{\p}}\right)^2$$

and

$$\eta_{\mathfrak{p}}(x) \in \left( k_{\p}(A)\right)^2 \Rightarrow (\exists (g+ \widehat{\mathfrak{p}}) \in \frac{A\{ {A \setminus \mathfrak{p}}^{-1}\}}{\widehat{\mathfrak{p}}})(\eta_{\mathfrak{p}}'(x) = g^2 + \widehat{\mathfrak{p}})$$

Since $q_{\widehat{\mathfrak{p}}}$ is surjective, given this $g + \widehat{\mathfrak{p}} \in  \left( \frac{A\{ A \setminus {\mathfrak{p}}^{-1}\}}{\widehat{\mathfrak{p}}}\right)$, there is some $\theta \in A\{ {A \setminus \mathfrak{p}}^{-1}\}$ such that $q_{\widehat{\mathfrak{p}}}(\theta) = g + \widehat{\mathfrak{p}}$.\\

By \textbf{Theorem 1.4}, item (i) of \cite{moerdijk1986rings}, given this $\theta \in A\{ {A \setminus \mathfrak{p}}^{-1}\}$, there are $a \in A$ and $b \in {A\setminus \mathfrak{p}}^{\infty-{\rm sat}}$, that is,
$${\rm Can}_{\mathfrak{p}}(b) \in (A\{ {A \setminus \mathfrak{p}}^{-1}\})^{\times}$$
such that:
$$\theta = \dfrac{{\rm Can}_{\mathfrak{p}}(a)}{{\rm Can}_{\mathfrak{p}}(b)}$$
or equivalently, since $(A \setminus \mathfrak{p})^{\infty-{\rm sat}} = A \setminus \mathfrak{p}$:
\begin{equation}\label{absu} b \notin \mathfrak{p}
\end{equation}

Hence,

$$\eta_{\mathfrak{p}}'(x) = g^2 + \widehat{\mathfrak{p}} = q_{\widehat{\mathfrak{p}}}\left( \frac{{\rm Can}_{\mathfrak{p}}(a)}{{\rm Can}_{\mathfrak{p}}(b)} \right)^2 = \left( \frac{{\rm Can}_{\mathfrak{p}}(a)}{{\rm Can}_{\mathfrak{p}}(b)} \right)^2 + \widehat{\mathfrak{p}}$$

$$\eta_{\mathfrak{p}}'(x)\cdot ({\rm Can}_{\mathfrak{p}}^2(b) + \widehat{\mathfrak{p}}) = {\rm Can}_{\mathfrak{p}}^2(a) + \widehat{\mathfrak{p}}$$

$${\rm Can}_{\mathfrak{p}}(x\cdot b^2 - a^2) \in \widehat{\mathfrak{p}}.$$

$$(x \cdot b^2 - a^2) \in {\rm Can}_{\mathfrak{p}}^{\dashv}[\widehat{\mathfrak{p}}].$$

By \textbf{Lemma \ref{90}}, $\mathfrak{p} = {\rm Can}_{\mathfrak{p}}^{\dashv}[\widehat{\mathfrak{p}}]$, so
$$x \cdot b^2 - a^2 \in \mathfrak{p} \subseteq P$$

Let $y = x \cdot b^2 + (-a^2) \in \mathfrak{p} \subseteq P$. Note that since $x \notin P$, $x \in (-P)\setminus \mathfrak{p}$ and
\begin{equation}\label{londonbeat}
x \cdot b^2 \in (-P)
\end{equation}

Since $y \in P$,
$$x \cdot b^2 = \underbrace{y}_{\in P} + \overbrace{a^2}^{\in P} \in P,$$

\begin{equation}\label{madonna}
  x \cdot b^2 \in P
\end{equation}

By \eqref{londonbeat} and \eqref{madonna}, it follows that $x \cdot b^2 \in P \cap (-P) = \mathfrak{p}$. Since $\mathfrak{p}$ is prime, either $x \in \mathfrak{p}$ or $b^2 \in \mathfrak{p}$. However, since $x \notin P$, \textit{a fortiori}, $x \notin \mathfrak{p}$, so we must have $b^2 \in \mathfrak{p}$. Once again, since $\mathfrak{p}$ is prime, it follows that $b \in \mathfrak{p}$ which contradicts \eqref{absu}. Hence,

$${\eta_{{\rm supp}^{\infty}(P)}'}^{\dashv}\left[\left( \frac{A\{ A \setminus {\rm supp}^{\infty}(P)^{-1}\}}{\widehat{\mathfrak{p}}}\right)^2\right] \subseteq P$$.\\

Now we claim that:\\

$$P \subseteq {\eta_{{\rm supp}(P)}'}^{\dashv}\left[\left( \frac{A\{ A \setminus {\rm supp}^{\infty}(P)^{-1}\}}{\widehat{\mathfrak{p}}}\right)^2\right] $$

Conversely, suppose, \textit{ab absurdo} that

\begin{equation}\label{Adma}
P \nsubseteq {\eta_{{\rm supp}(P)}'}^{\dashv}\left[\left( \frac{A\{ A \setminus {\rm supp}^{\infty}(P)^{-1}\}}{\widehat{\mathfrak{p}}}\right)^2\right]
\end{equation}

so there must exist some $x \in P$ such that

\begin{equation}\label{sarue}
 (\forall (g + \widehat{\mathfrak{p}}) \in \dfrac{A\{{A \setminus \mathfrak{p}}^{-1}\}}{\widehat{\mathfrak{p}}})(\eta_{\mathfrak{p}}'(x) \neq g^2 + \widehat{\mathfrak{p}})
\end{equation}

Equivalently, there must exist some $(h+\widehat{\mathfrak{p}}) \in \frac{A\{ {A \setminus \mathfrak{p}}^{-1}\}}{\widetilde{\mathfrak{p}}}$ such that $\eta_{\mathfrak{p}}(x) = - h^2 + \widehat{\mathfrak{p}}$.\\

Thus, since $q_{\widehat{\mathfrak{p}}}$ is surjective, given such an $h + \widehat{\mathfrak{p}} \in \dfrac{A\{{A \setminus \mathfrak{p}}^{-1}\}}{\widehat{\mathfrak{p}}}$ there must be some $\zeta \in A\{ {A \setminus \mathfrak{p}}^{-1}\}$ such that $q_{\widetilde{\mathfrak{p}}}(\zeta) = h + \widehat{\mathfrak{p}}$. By item (i) of \textbf{Theorem \ref{340}}, there are $c \in A$ and $d \in A$ with:
$${\rm Can}_{\mathfrak{p}}(d) \in A\{ {A \setminus \mathfrak{p}}^{-1}\}^{\times},$$
equivalently
$$d \in (A \setminus \mathfrak{p})^{\infty-{\rm sat}},$$
and since $(A \setminus \mathfrak{p})^{\infty-{\rm sat}} = A \setminus \mathfrak{p}$,
\begin{equation}\label{criat}
d \notin \mathfrak{p}
\end{equation}

such that:
$$\zeta = \dfrac{{\rm Can}_{\mathfrak{p}}(c)}{{\rm Can}_{\mathfrak{p}}(d)}.$$

Hence,

$$\eta_{\mathfrak{p}}(x) = - \dfrac{{\rm Can}_{\mathfrak{p}}^2(c)}{{\rm Can}_{\mathfrak{p}}^2(d)} + \widehat{\mathfrak{p}}$$

$${\rm Can}_{\mathfrak{p}}^2(c) + {\rm Can}_{\mathfrak{p}}(x) \cdot {\rm Can}_{\mathfrak{p}}^2(d) \in \widehat{\mathfrak{p}}$$

$${\rm Can}_{\mathfrak{p}}(c^2 + x \cdot d^2) \in \widehat{\mathfrak{p}}$$

so

$$\underbrace{c^2}_{\in P} + \overbrace{x \cdot d^2}^{\in P} = z = {\rm Can}_{\mathfrak{p}}^{\dashv}[\widehat{\mathfrak{p}}] = \mathfrak{p} \subseteq (-P)$$

$$x\cdot d^2 = z - c^2 \in (-P)$$

Since $x \in P$, we also have $x \cdot d^2 \in P$, hence $x \cdot d^2 \in \mathfrak{p}$ .  Since $\mathfrak{p}$ is prime, either $x \in \mathfrak{p}$ or $d^2 \in \mathfrak{p}$. Now, if $x \in \mathfrak{p}$ then ${\rm Can}_{\mathfrak{p}}(x) = 0^2+\mathfrak{p}$, which contradicts our hypothesis \eqref{sarue}. On the other hand, if $d^2 \in \mathfrak{p}$, then $d \in \mathfrak{p}$, and this contradicts \eqref{criat}. Hence $x \notin \mathfrak{p}$ and $d^2 \notin \mathfrak{p}$, so $x \cdot d^2 \notin \mathfrak{p}$. Thus we achieved an absurdity: $(x \cdot d^2 \in \mathfrak{p}) \& (x \cdot d^2 \notin \mathfrak{p})$. It follows that our premise \eqref{Adma} must be false, so:

$$P \subseteq {\eta_{{\rm supp}(P)}'}^{\dashv}\left[\left( \frac{A\{ A \setminus {\rm supp}^{\infty}(P)^{-1}\}}{\widehat{\mathfrak{p}}}\right)^2\right].$$

Hence $P = {\eta_{{\rm supp}(P)}'}^{\dashv}\left[\left( \frac{A\{ A \setminus {\rm supp}^{\infty}(P)^{-1}\}}{\widehat{\mathfrak{p}}}\right)^2\right]$.\\










Now we need only to show that $\alpha' \circ \beta' = {\rm id}_{\widetilde{\mathcal{F}}}$. \\

Let $[h: A \to F] \in \widetilde{\mathcal{F}}$. We have:

$$(\alpha' \circ \beta')([h: A \to F]) = \alpha'(h^{\dashv}[F^2]) = [\eta_{{\rm supp}^{\infty}(h^{\dashv}[F^2])}: A \to k_{{\rm supp}^{\infty}(h^{\dashv}[F^2])}(A)].$$

It suffices to show that $[h]=[\eta_{{\rm supp}^{\infty}(h^{\dashv}[F^2])}]$. Note that ${\rm supp}^{\infty}(h^{\dashv}[F^2]) = h^{\dashv}[F^2 \cap (-F^2)] = \ker(h)$.\\

By the universal property of the $\mathcal{C}^{\infty}-$field of fractions of $\left( \frac{A}{\ker(h)}\right)$, $k_{\ker(h)}(A)$, since $h \left[ A^{\times}\right] \subseteq F^{\times}$ (for $\mathcal{C}^{\infty}-$homomorphisms preserve invertible elements), there is a unique $\mathcal{C}^{\infty}-$homomorphism $\widetilde{h}: k_{\ker(h)}(A) \to F$ such that the following diagram commutes:

$$\xymatrixcolsep{5pc}\xymatrix{
A \ar[r]^{\eta_{\p}} \ar[rd]^{h} & k_{\ker(h)}(A) \ar[d]^{\widetilde{h}}\\
  & F
}$$

We have, then, the following commutative diagram:

$$\xymatrixcolsep{5pc}\xymatrix{
  & F \ar[dr]^{{\rm id}_F} &  \\
  A \ar[ur]^{h} \ar[dr]_{\eta_{\mathfrak{p}}} & & F\\
    & k_{\ker(h)} \ar@{.>}[ur]_{\widetilde{h}}
}$$

so $[h: A \to F] = [\eta_{\ker(h)}: A \to k_{\ker(h)}]$ and

$$(\alpha' \circ \beta')([h: A \to F]) = [h: A \to F].$$

Hence it follows that $\alpha'$ and $\beta'$ are inverse bijections of each other.
\end{proof}

\begin{theorem}\label{2}The map:
$$\begin{array}{cccc}
\alpha :& {\rm Spec}^{\infty}\,(A) & \rightarrow & \widetilde{\mathcal{F}}\\
        & \mathfrak{p}  & \mapsto & [\eta_{\mathfrak{p}}]= [q \circ {\rm Can}_{\mathfrak{p}}]
\end{array}$$
is a bijection whose inverse is given by:
$$\begin{array}{cccc}
\beta: & \widetilde{\mathcal{F}} & \rightarrow & {\rm Spec}^{\infty}\, (A)\\
       & [h: A  \to F] & \mapsto & \ker(h)
\end{array}$$
\end{theorem}
\begin{proof}
First we are going to show that $\alpha \circ \beta = {\rm id}_{\widetilde{\mathcal{F}}}$, so $\beta$ is an injective map and $\alpha$ is its left inverse, hence it is surjective.\\

Let $[h: A \to F] \in \widetilde{F}$. We have:

$$(\alpha \circ \beta)([h: A \to F]) = \alpha(\ker(h)) = \left[ \eta_{\ker(h)}': A \to \left( \frac{A\{ {A \setminus \ker(h)}^{-1} \}}{\widehat{\ker(h)}}\right)\right]$$

It suffices to show that $[h: A \to F] = \left[ \eta_{\ker(h)}': A \to \left( \frac{A\{ {A \setminus \ker(h)}^{-1} \}}{\widehat{\ker(h)}}\right)\right]$.\\

Since $\left( \frac{A\{ {A \setminus \ker(h)}^{-1} \}}{\widehat{\ker(h)}}\right)$ is (up to $\mathcal{C}^{\infty}-$isomorphism) the $\mathcal{C}^{\infty}-$field of fractions of $\left( \frac{A}{\ker(h)}\right)$, $k_{\ker(h)}(A)$, there is a unique $\mathcal{C}^{\infty}-$homomorphism $\widetilde{h}: \left( \frac{A\{ {A \setminus \ker(h)}^{-1} \}}{\widehat{\ker(h)}}\right) \to F$ such that the following diagram commutes:

$$\xymatrixcolsep{5pc}\xymatrix{
  & \left( \frac{A\{ {A \setminus \ker(h)}^{-1} \}}{\widehat{\ker(h)}}\right) \ar[dr]^{\widetilde{h}} &   \\
A \ar[ur]^{\eta_{{\rm supp}^{\infty}(h^{\dashv}[F^2])'}} \ar[dr]_{h} & & F\\
  & F \ar@{.>}[ur]_{{\rm id}_F} &}$$
and the equality holds, \textit{i.e.},

$$[h: A \to F] = \left[ \eta_{\ker(h)}: A \to \left( \frac{A\{ {A \setminus \ker(h)}^{-1} \}}{\widehat{\ker(h)}}\right)\right]$$

It follows that $\alpha \circ \beta = {\rm id}_{\widetilde{F}}$, $\alpha$ is a surjective map and $\beta$ is an injective map.\\

On the other hand, given $\mathfrak{p} \in {\rm Spec}^{\infty}(A)$, we have:

$$(\beta \circ \alpha)(P) = \beta([\eta_{\mathfrak{p}}: A \to k_{\mathfrak{p}}]) = \ker(\eta_{\mathfrak{p}}) = {\rm Can}_{\mathfrak{p}}^{\dashv}[\widehat{\mathfrak{p}}] = \mathfrak{p}$$

so:

$$(\beta \circ \alpha) = {\rm id}_{{\rm Spec}^{\infty}\,(A)}$$
\end{proof}

as a Corollary of the theorems \textbf{\ref{1}} and \textbf{\ref{2}}, we have:

\begin{lemma}Let $A$ be a $\mathcal{C}^{\infty}-$ring, and define:

$$\begin{array}{cccc}
{\rm supp}^{\infty} : & {\rm Sper}^{\infty}\,(A) & \rightarrow & {\rm Spec}^{\infty}\,(A)\\
      & P & \mapsto & P \cap(-P)
\end{array}$$

The following diagram commutes:

$$\xymatrixcolsep{5pc}\xymatrix{
   & \,\, {\rm Sper}^{\infty}\,(A) \ar@<1ex>[dl]^{\alpha'} \ar[dd]_{{\rm supp}^{\infty}}\\
   \widetilde{\mathcal{F}} \ar@<1ex>[ur]^{\beta'} \ar@<1ex>[dr]^{\beta} &  \\
      & {\rm Spec}^{\infty}\,(A) \ar@<1ex>[ul]^{\alpha}
}$$

that is to say that:
$$\alpha \circ {\rm supp}^{\infty} = \alpha'$$
and
$${\rm supp}^{\infty}\circ \beta' = \beta$$
\end{lemma}
\begin{proof}
Note that if we prove that $\alpha \circ {\rm supp}^{\infty} = \alpha'$, then composing both sides with $\beta'$ yields:
$$(\alpha \circ {\rm supp}^{\infty})\circ \beta' = \alpha' \circ \beta' = {\rm id}_{\widetilde{\mathcal{F}}}$$
so
$$\alpha \circ ({\rm supp}^{\infty} \circ \beta') = {\rm id}_{\widetilde{\mathcal{F}}}$$
and by the uniqueness of the inverse of $\alpha$, it follows that:
$${\rm supp}^{\infty} \circ \beta' = \beta.$$

Now we are going to prove that $\alpha \circ {\rm supp}^{\infty} = \alpha'$.\\

Given $P \in {\rm Sper}^{\infty}\,(A)$ we have:

$$(\alpha \circ {\rm supp}^{\infty})(P) = \alpha({\rm supp}^{\infty}(P)) = [\eta_{{\rm supp}^{\infty}\,(P)}: A \to k_{{\rm supp}^{\infty}\,(P)}(A)] =: \alpha'(P),$$

so the result holds.
\end{proof}

As an important result of the theory of $\mathcal{C}^{\infty}-$rings which distinguishes it from the theory of the rings, we have the following:

\begin{theorem}\label{exv}Let $A$ be a $\mathcal{C}^{\infty}-$ring. The following map:
$$\begin{array}{cccc}
{\rm supp}^{\infty} : & {\rm Sper}^{\infty}\,(A) & \rightarrow & {\rm Spec}^{\infty}\,(A)\\
      & P & \mapsto & P \cap(-P)
\end{array}$$
is a spectral \underline{bijection}.
\end{theorem}
\begin{proof}
In \textbf{Remark \ref{avio}}, we have already seen that ${\rm supp}^{\infty}$ is a spectral function, so we need only to show that is is a bijection.\\

Just note  that ${\rm supp}^{\infty} = \beta \circ \alpha' = \alpha \circ \beta'$, and since ${\rm supp}^{\infty}$ is a composition of bijections, it is a bijection.
\end{proof}

\section{Algebraic Concepts versus Smooth Algebraic Concepts}\label{versus}

In this section we comment some differences between some concepts and results of Commutative Algebra and Smooth Commutative Algebra. Some notions of Smooth Commutative Algebra, such as local $\mathcal{C}^{\infty}-$rings, $\mathcal{C}^{\infty}-$fields, $\mathcal{C}^{\infty}-$domains and von Neumann regular $\mathcal{C}^{\infty}-$rings are obtained via the adjunctions:

$$\xymatrix{ \mathcal{C}^{\infty}{\rm \bf Rng} \ar@<1ex>[r]^{\widetilde{U}}& \ar@<1ex>[l]^{\widetilde{L}} {\rm \bf CRing}}$$

and

$$\xymatrix{ \mathcal{C}^{\infty}{\rm \bf Rng} \ar@<1ex>[r]^{\mathcal{U}}& \ar@<1ex>[l]^{\mathcal{L}} \mathbb{R}-{\rm \bf Alg}}$$

so many results about these concepts are very similar to those of Algebra. As an example, we can cite the fact that $\widetilde{L}$ ``preserves'' the concept of ``object of fractions'', as we see in the following:\\

\begin{remark}Let $A$ be a commutative unital ring and let $a \in A$. The functor $\widetilde{L}: {\rm \bf CRing} \rightarrow \mathcal{C}^{\infty}{\rm \bf Rng}$, which is left adjoint to $\widetilde{U}: \mathcal{C}^{\infty}{\rm \bf Rng} \rightarrow {\rm \bf CRing}$, preserves all colimits, so it takes the ring of fractions $A[a^{-1}]$ to the $\mathcal{C}^{\infty}-$ring of fractions, $B\{ b^{-1}\}$, that is:

$$\widetilde{L}(A[a^{-1}]) \cong \widetilde{L}\left( \dfrac{A\otimes \mathbb{Z}[x]}{\langle \{\iota_A(a)\cdot x - 1 \}\rangle}\right) \cong \dfrac{\widetilde{L}(A)\otimes_{\infty}\mathcal{C}^{\infty}(\mathbb{R})}{\langle \{ \gamma_A(a)\cdot x - 1\}\rangle}$$

For a general $S \subseteq A$, we have:

$$\widetilde{L}(A[S^{-1}]) \cong \widetilde{L}(A)\{\gamma_A[S]^{-1}\}$$
\end{remark}

The ring $A$ to be a domain does not entail that $\widetilde{L}(A)$ is a $\mathcal{C}^{\infty}$-domain. For example, despite $\mathbb{Z}[x]$ is a domain, $\widetilde{L}(\mathbb{Z}[x]) \cong C^\infty(\mathbb{R})$ is not a $\mathcal{C}^{\infty}-$domain - since its underlying ring is not a domain. However, some domains can be taken to $\mathcal{C}^{\infty}-$fields, \textit{e.g.}, $\widetilde{L}(\mathbb{Z}) \cong \mathbb{R}$ and $\widetilde{L}(\mathbb{Q}) \cong \mathbb{R}$.\\

Fields can be taken by $\widetilde{L}$ to the trivial $\mathcal{C}^{\infty}$-ring. Since every $\mathcal{C}^{\infty}-$ring $A$ is such that $1 + \sum A^2 \subseteq A^\times$, any field with non zero characteristic is mapped into the trivial $\mathcal{C}^{\infty}-$ring. More generally, every ring with non zero characteristic is taken to the trivial $\mathcal{C}^{\infty}-$ring.\\

The left adjoint $\widetilde{L}$ also maps free rings to free $\mathcal{C}^{\infty}-$rings, that is, given any set $E$ we have $\widetilde{L}(\mathbb{Z}[E])\cong \mathcal{C}^{\infty}(\mathbb{R}^{E})$. Note that $\mathbb{Z}[E]$ is a reduced ring and that $\mathcal{C}^{\infty}(\mathbb{R}^{E})$ is a $\mathcal{C}^{\infty}-$reduced $\mathcal{C}^{\infty}-$ring.\\


As we have already commented, every $\mathcal{C}^{\infty}-$ring can be regarded as an $\mathbb{R}-$algebra. However, the converse of this claim is not true: $\mathbb{C} \cong \dfrac{\mathbb{R}[x]}{\langle x^2 + 1\rangle}$ is an $\mathbb{R}-$algebra but it is not a $\mathcal{C}^{\infty}-$ring. In fact, every $\mathcal{C}^{\infty}-$ring is semi-real, so the sum of 1 with any sum of squares must be invertible. Thus, since we have $-1 = x^2$ in $\dfrac{\mathbb{R}[x]}{\langle x^2 + 1\rangle}$, we would have $0 + \langle x^2 + 1\rangle$ invertible in $\widetilde{L}\left( \dfrac{\mathbb{R}[x]}{\langle x^2 + 1\rangle}\right) \cong \widetilde{L}(\mathbb{C})$ - hence it must be the trivial $\mathcal{C}^{\infty}-$ring (cf. \textbf{Proposition \ref{doria}}).\\

However, both theories have  ``ad hoc'' notions of radical ideals, rings of fractions and polynomials. Contrary to what occurs in Commutative Algebra, in which every prime ideal is radical, in the smooth case we have seen that not every  prime ideal of a $\mathcal{C}^{\infty}-$ring is $\mathcal{C}^{\infty}-$radical. Due to this fact, we have different notions in these theories, such as ``saturation'', ``Zariski spectrum'', ``reducedness'', ``real spectrum'' and others.\\

In both theories one has very similar results, such as the separation theorems. Perhaps the most ``striking'' difference between Commutative and Smooth Commutative Algebra is the closer relation that the latter holds with Real Algebraic Geometry than the former, in virtue of the spectral bijection between the $\mathcal{C}^{\infty}-$real spectrum and the smooth Zariski spectrum of a $\mathcal{C}^{\infty}-$ring, given in \textbf{Theorem \ref{exv}}.

\bibliography{references}
\end{document}